\documentclass[11pt, reqno]{amsart}
\usepackage{amsfonts,amssymb}
\usepackage{amsmath,amscd}
\usepackage[scr=rsfs]{mathalpha}
 \usepackage[all]{xy}
\usepackage{color}   
\usepackage{multirow}
\usepackage[colorlinks=true, linkcolor=blue, citecolor=red]{hyperref}
\usepackage{multicol}
\usepackage{graphicx}
\usepackage{empheq}
\usepackage{enumitem}
\usepackage{crossreftools} 

\allowdisplaybreaks

\usepackage{tikz}
\usetikzlibrary{arrows, decorations.markings}
\usetikzlibrary{decorations.markings, arrows.meta}
\usetikzlibrary{arrows.meta}
\usetikzlibrary{decorations.pathreplacing,angles,quotes}
\tikzset{bullet/.style={draw,ellipse, text width = 4cm, text centered}}
\tikzset{rec/.style={draw, text width = 3.5 cm, text centered}}
\tikzset{plain/.style={->,>=stealth}}
\tikzset{
  barbarrow/.style={ 
     >={Straight Barb[left,length=5pt, width=5pt]}
  },
  strike through/.style={
    postaction=decorate,
    decoration={
      markings,
      mark=at position 0.5 with {
        \draw[-] (-3pt,-3pt)  --  (3pt, 3pt);
      }
    }
  }
}

\theoremstyle{plain}
\newtheorem{theorem}{Theorem}[section]
\newtheorem{lemma}[theorem]{Lemma}
\newtheorem{proposition}[theorem]{Proposition}
\newtheorem{definition}[theorem]{Definition}
\newtheorem{corollary}[theorem]{Corollary}
\newtheorem{conjecture}[theorem]{Conjecture}

\theoremstyle{remark}
\newtheorem{remark}[theorem]{Remark}
\theoremstyle{remark}

\theoremstyle{remark}
\newtheorem{example}[theorem]{Example}
\theoremstyle{remark}

\newtheoremstyle{dotless}{}{}{\itshape}{}{\bfseries}{}{ }{}
\theoremstyle{dotless}
\newtheorem*{theorem*}{Theorem}
\newtheorem*{lemma*}{Lemma}
\newtheorem*{proposition*}{Proposition}
\newtheorem*{definition*}{Definition}

\newlength{\bibitemsep}
\newlength{\bibparskip}
\let\oldthebibliography\thebibliography
\renewcommand\thebibliography[1]{%
  \oldthebibliography{#1}%
  \setlength{\parskip}{\bibitemsep}%
  \setlength{\itemsep}{\bibparskip}%
}

\newcommand{\cc}{\ensuremath{\mathbb{C}}}
\newcommand{\on}[1]{\operatorname{#1}}

\newcommand{\onsf}[1]{\operatorname{\sf #1}}

\newcommand{\delete}[1]{{}}
\newcommand{\Mod}{\on{\!-Mod}}
\newcommand{\scr}[1]{\mathscr{#1}}
\newcommand{\cal}[1]{\mathcal{#1}}
\newcommand{\mc}[1]{\mathcal{#1}}

\newcommand{\msf}[1]{\mathsf{#1}}
\newcommand{\mf}[1]{\mathfrak{#1}}
\newcommand{\bb}[1]{\mathbb{#1}}

\newcommand{\Ext}{\on{Ext}}
\newcommand{\End}{\on{End}}
\newcommand{\spec}{\onsf{spec}}
\newcommand{\CExt}{\cal Ext}
\newcommand{\Hom}{\on{Hom}}

\newcommand{\Tor}{\on{Tor}}
\newcommand{\CTor}{\cal Tor}

\delete{}

\newcommand{\myhookrightarrow}{\raisebox{-1pt}{ \, \begin{tikzpicture}
  \draw [right hook->] (0, 0) -- (0.7, 0);     
\end{tikzpicture} \, }}

\makeatletter \def\l@subsection{\@tocline{2}{0pt}{1pc}{5pc}{}} \def\l@subsection{\@tocline{2}{0pt}{2pc}{6pc}{}} \makeatother

\title[PD dg algebras and homotopy Lie algebras]{Differential graded algebras with divided powers and homotopy Lie algebras}
\author[A. Caradot]{Antoine Caradot$^1$}
\address{$^1$Hubert Curien Laboratory \\ Jean Monnet University \\ Saint-\'Etienne \\ FRANCE}
\email{antoine.caradot@univ-st-etienne.fr}
\author[Z. Lin]{Zongzhu Lin$^2$}
\address{$^2$Department of Mathematics\\
Kansas State University \\
Manhattan, KS 66506, USA}
\email{zlin@ksu.edu}
\date{June 18, 2026}                     
\keywords{Symmetric tensors, divided powers, homotopy Lie algebras, Ext algebras, complete intersection rings}
\thanks{2020 {\it Mathematics Subject Classification:}
Primary: 13D02, 16W25. Secondary: 14M10, 16T05}

\begin{document}

\begin{abstract}
Given a commutative algebra $A$ and a quotient $A$-algebra $A/I$, we construct a resolution of $A/I$ as an $A$-module such that it is also a differential graded (dg) algebra with divided powers (PD). This construction makes use of symmetric tensors in the symmetric tensor category of dg $A$-modules and does not require a Noetherian assumption on $A$. Moreover, the resolution has many lifting properties which we leverage to study the homotopy Lie algebra associated to the pair $(A,A/I)$, which is defined as the image in the Yoneda algebra $\Ext^*_{A}(A/I,A/I)$ of the cohomology of the PD derivations of this PD dg algebra. Finally we investigate the complete intersection case in more details as well as connect it to the finite generation of the Yoneda algebra. 

\delete{
{\color{blue}
In this article, we study the Ext algebra of a finitely generated complete intersection ring and determine the necessary conditions for it to be strictly graded commutative. We then relax the definition of the commutativity of a graded ring and introduce the weak graded commutativity. We determine the conditions necessary to obtained a weak graded commutative Yoneda algebra. We will also compare the commutativity conditions on the Yoneda algebra with those on its even subalgebra.
}
}

\end{abstract}

\maketitle
\tableofcontents

\section{Introduction}
This paper is part of our effort to understand the geometries attached to a vertex algebra. Corresponding to each vertex algebra $V$, there is a Poisson algebra $R(V)$ which defines a Poisson scheme $X_V=\spec{R(V)}$. In the rational $C_2$-cofinite case, $X_V$ has only one point. In \cite{CJL1}, we defined a cohomological dual $X_{V,x}^!=\spec((\on{Ext}_{R(V)}^*(\msf k_x, \msf k_x))^{ab})$ for each closed point $x\in X_V$ and call it the cohomological variety of $V$ at $x$. While it defines an invariant of the vertex algebra $V$, it is difficult to compute in most cases, in particular when $ x \in X_V$ is not a local complete intersection. In \cite{CJL1}, we estimated a lower bound on the dimension of this variety for many rational $C_2$-cofinite vertex algebras. Our goal is to establish a cohomological support variety theory (\cite{Friedlander-Parshall} for restricted Lie algebras in positive characteristic, \cite{Boe-Kujawa-Nakano} for Lie superalgebras) for representations of vertex algebras in terms of quasicoherent sheaves. In the present paper, we explore a model different from $X_{V,x}^!$. The Yoneda algebra $\on{Ext}_{R(V)}^*(\msf k_x, \msf k_x)$ is actually a Hopf algebra, which is the universal enveloping algebra of a graded Lie algebra, called the homotopy Lie algebra of $R(V)$, in the symmetric monoidal category of graded vector spaces (with Koszul braiding) when $ R(V)$ is Noetherian. Hence the idea of associated varieties of representations of Lie algebras can be applied \cite{Dixmier, Joseph, Bai-Xie, Vogan}. To prepare for this, we first need to explore this Lie algebra from the differential graded context. Since the algebra $ R(V)$ is in general not Noetherian and not complete intersection, we give a more functorial construction of the homotopy Lie algebra as well as its dg enhancement, depending on the construction of the Koszul-Tate resolution, preparing for applications to dg vertex algebras which have been established in \cite{CJL2, CJL3, CL}. One of our objectives is to construct a cohomological duality for vertex algebras, by establishing a Koszul-Tate type of resolutions of vertex algebras. 

\delete{
The Koszul-Tate resolution requires a divided power structure (PD structure) \cite{Roby1, Roby2, Roby3,Andre, Stacks-Project} on the resolution algebra, which is called semi-free in some literature. The free algebras with PD structures have been studied with concrete constructions by Roby in \cite{Roby1,Roby2,Roby3}. One of the approach is to use the functorial construction of subalgebra $\onsf{TS}_{\msf k}(V)$ of symmetric tensors in the shuffle algebra (the space $T_A(V)$ with shuffle product) over a $\msf k$-vector space $V$, or even over free modules over a base commutative ring $R$. This construction has an apparent functorial advantage, automatically making $\onsf{TS}_R(V)$ a cocommutative commutative Hopf algebra with PD structure. In this paper, we extend this construction to the symmetric monoidal category of differential complexes over general commutative rings. This is done in Section \ref{sec:3} with elementary setup in Section \ref{sec:2}. More precisely, in Section \ref{sec:3} we first extend the construction of free PD algebras by Roby in \cite{Roby3} using the shuffle algebra of symmetric tensors $\onsf{TS}_{R}(A/I)$ of a free $R$-module $A/I$ to the symmetric tensor category $\onsf{Cplx}(R)^{free}$ of differential complexes of free $R$-modules. This defines a  functor $\onsf{TS}_{R}^*(-)$ from $\onsf{Cplx}(R)^{free}$ to the category of PD dg algebras, allowing us to construct the Koszul-Tate resolution for an arbitrary commutative ring $A$ with an ideal $ I$ and quotient ring $A/I$ as $A$-module. In most constructions in the literature, we find that either $A$ is assumed to the Noetherian or is a {\color{red}$\mathbb Q$-algebra, from rational homotopy theory (thus the PD structure becomes automatic and unique) \cite{FHT}.} In Tate's construction \cite{Tate}, the approach consists in killing one cocycle at a time, which is good enough to apply when $A$ is Noetherian. Northcott \cite{Northcott} extended Tate's construction by taking limit {\color{red}still over generators.} Gulliksen-Levin's lecture notes \cite{Gulliksen-Levin} also use Northcott's construction. Here we take the generator free approach and apply the functor $ \onsf{TS}_{A}(-)$ to kill the whole homology at a degree at once. This is a Sullivan type of construction from homotopy theory where we attach all cells of a particular dimension at once.  We will rely heavily on the use of the PD structure to extend PD dg algebra homomorphisms as well as derivations. The Koszul-Tate resolution of $A/I$ as $A$-algebra depends on the choice of free $A$-modules $F_{n+1}$ that kill degree $-n$ cohomology of earlier stages. }

The Koszul-Tate resolution requires a divided power structure (PD structure) \cite{Roby1, Roby2, Roby3,Andre, Stacks-Project} on the resolution algebra, which is called semi-free in some literature. The free algebras with PD structures have been studied with concrete constructions by Roby in \cite{Roby1,Roby2,Roby3}. One of the approaches is to use the functorial construction of the subalgebra $\onsf{TS}_A(V)$ of symmetric tensors in the shuffle algebra ($T_A(V)$ with shuffle product) over a free $A$-module $V$. This construction has an apparent functorial advantage, automatically making $\onsf{TS}_A(V)$ a cocommutative commutative Hopf algebra with PD structure. In this paper, we extend this construction to the symmetric monoidal category of differential complexes over general commutative rings. This is done in Section \ref{sec:3} with elementary setup in Section \ref{sec:2}. More precisely, in Section \ref{sec:3} we first extend the construction of free PD algebras by Roby in \cite{Roby3} using the shuffle algebra of symmetric tensors $\onsf{TS}_{A}(M)$ of a free $A$-module $M$ to the symmetric tensor category $\onsf{Cplx}(A)^{free}$ of differential complexes of free $A$-modules. 
This defines a  functor $\onsf{TS}_{A}^*(-):\onsf{Cplx}(A)^{free} \to \onsf{grPDAlg}(A) $ to the category of graded PD algebras over $A$ (see \cite{Andre} for graded PD algebras), sending direct sums to tensor products. 
Thus there is a PBW type basis for each given basis $\mc B$ of $M$ with a given total order. 
This functor automatically defines a graded Hopf algebra with PD structure on $ \onsf{TS}_A(M)$. 
The freeness then makes it possible to extend differentials if $A$ is a PD dg algebra.
We use this functor $ \onsf{TS}_A(-)$  to construct a Koszul-Tate resolution for an arbitrary commutative ring $A$ with an ideal $ I$ and quotient ring $ A/I$ as an $A$-algebra.

In most constructions in the literature, we find that either $A$ is assumed to be Noetherian (most in the context of Noetherian local rings such as in \cite{Avramov, Gulliksen, Gulliksen-Levin, Tate}) or to be a $\mathbb Q$-algebra, from rational homotopy theory \cite{FHT} (thus the PD structure becomes automatic and unique) such as in \cite{Pistalo-Poncin, Brino-Pistalo-Poncin_1,Brino-Pistalo-Poncin_2}. 
In Tate's construction \cite{Tate}, the approach consists in killing one cocycle at a time, which is good enough to apply when $A$ is Noetherian. Northcott \cite{Northcott} extended Tate's construction by still working over generators and taking limits. 
This extension is then used in Gulliksen-Levin's lecture notes \cite{Gulliksen-Levin} as well as Stacks-Project (\cite[Lemma 23.6.9]{Stacks-Project}) and others. 

We take a generator free approach and apply the functor $ \onsf{TS}_{A}(-)$ to kill the entire (co)homology at a degree at once. This is a Sullivan type construction from homotopy theory \cite{Sullivan, FHT} where we attach all cells of a given dimension at once. 
We will rely heavily on the use of the PD structure to extend PD dg algebra homomorphisms as well as derivations. 
The Koszul-Tate resolution $P^*=\varinjlim_{n} P_{n}^*$ of $A/I$ as an $A$-algebra depends on the choice of the free $A$-module cover $\phi_{n+1}: F_{n+1}\to H^{-n}(P^*_n)$ that kills degree $-n$ cohomology of earlier stage $P_n^*$ and the next stage is obtained as $P_{n+1}^*=P_{n}^*\otimes_A \onsf{TS}_{A}(F_{n+1}[n+1])$. 
The differential on $ P^*_{n+1}$ is the unique extension of $d_{P_{n}^*}$ and $ \phi_{n+1}$, making $ P_{n+1}^*$ a PD dg algebra and the natural embedding $ P_{n}^*\to P_{n+1}^*$  a homomorphism of PD dg algebras. 
We remark that, by construction, each $ P_n^*$ is a graded PD Hopf algebra which is both graded commutative and graded cocommutative. 
Thus, forgetting the differential, $P^*$ is a connected graded cocommutative PD Hopf algebra over $A$. 
However it is not a PD dg Hopf algebra (see \cite{Andre}) as the comultiplication will fail to be a chain map, as one might expect.    

It is well-known that the Yoneda algebra $ \onsf{Ext}_A^*(A/I,A/I)$ is isomorphic to $H^*(\mc Hom_{A}^*(P^*,P^*))$, 
not only as graded $A$-modules, 
but also as graded algebras with the composition product on $H^*(\mc Hom_{A}^*(P^*,P^*))$ up to a Koszul sign. 
In Section \ref{sec:4},  we consider the subcomplex $\mc Der^{*,\on{pd}}_A(P^*, P^*)\subseteq \mc Hom_{A}^*(P^*,P^*)$ of the PD derivations. The complex $\mc Der^{*,\on{pd}}_A(P^*, P^*)$
is naturally a dg Lie algebra, which is a dg Lie subalgebra of $\mc Der^{*}_A(P^*, P^*)$.  
The graded Lie algebra $H^*(\mc Der^{*,\on{pd}}_A(P^*, P^*))$ (in the symmetric tensor category of graded $A$-modules) naturally maps to a graded Lie subalgebra in $H^*( \mc Hom_{A}^*(P^*,P^*))$. Its image is defined to be the {\em Homotopy Lie algebra}. 
In the case where $A$ is Noetherian, $I$ is maximal, and the resolution $P^*$ is $I$-minimal (i.e., $d(P^{-n}) \subseteq IP^{-n+1}$), 
the Yoneda algebra $ \onsf{Ext}_{A}^*(A/I,A/I)$ has a connected Hopf algebra structure whose homogeneous primitive elements form a graded Lie algebra, which is classically defined as the homotopy Lie algebra \cite{Avramov}. 
Moreover, the homotopy Lie algebra has a restricted structure, and the Yoneda algebra is the restricted enveloping algebra of the homotopy Lie algebra. 

We could not prove that the graded Lie algebra $H^*(\mc Der^{*,\on{pd}}_A(P^*, P^*))$ does not depend on the resolution $P^*$ for the $A$-algebra $A/I$. We also will not discuss the model category structure on the category of PD dg algebras so that the Koszul-Tate resolution $P^*$ constructed is a cofibrant object. 

The concept of Homotopy Lie algebra was first defined by the homotopy groups of $H$-spaces in \cite{Milnor-Moore} and later interpreted for commutative local Noetherian rings (see \cite{Avramov, Gulliksen, Gulliksen-Levin}). 
We choose the definition using the dg Lie algebra $\mc Der^{*, \on{pd}}_A(P^*, P^*)$,
which has a natural restricted structure given by $D \mapsto D \circ D$. This Lie algebra plays the role of tangent complex in the context of rational homotopy theory and model category, or derived algebraic geometry.
We will not interpret the cotangent complex in terms of the PD dg algebra $P^*$. Cotangent complexes mostly appear in the simplicial context \cite{Illusie}. 
There is a dg approach to cotangent complexes in \cite{Yekutieli}. 
In the paper \cite{Richter}, Richter has proved that divided power structures on the chain complexes will naturally arise from the simplicial commutative algebra structure via normalized Moore complex. 
There is a different notion of homotopy Lie algebra operadic approach in \cite{Hinich-Schechtman} using Lie operad in the homotopy category of differential complexes.
In case $\bb Q\subseteq A$, the homotopy Lie algebra is closely related to the Homotopy Lie algebra described in this paper. 

We have not yet considered the case where $A$ is a Poisson $\msf k$-algebra and $ I$ is a Poisson ideal, but we expect that there is a Koszul-Tate resolution $P^*$ which has a Poisson dg algebra with PD structure so that the augmentation $ P^*\to A/I$ \cite{Cartan-Eilenberg} is a Poisson algebra homomorphism. 
Due to the size of the paper, we also do not mention the algebra of PD dg differential operators on the PD dg algebra $P^*$.  

In Section \ref{sec:5}, we consider the special case when $A$ is a complete intersection as well as several examples arising from simple singularities.
In this case, $P^*$ can constructed as $I$-minimal and the homotopy Lie algebra can be computed as a certain graded Lie algebra concentrated on degree 1 and 2. Conversely, any such graded Lie algebra (with a restricted structure) can be seen as the homotopy Lie algebra of some complete intersection algebra.
In this section, we also consider the finite generation question of Yoneda algebras. This is a classical question regarding finite generation of cohomology rings of groups, restricted Lie algebras, and Hopf algebras (see \cite{Evens, Venkov, Ji-Lin, Friedlander-Parshall, Friedlander-Suslin, Ginzburg-Kumar}).
In contrast to finite groups or finite dimensional restricted Lie algebras, the Yoneda algebra is finitely generated if and only if there is a resolution $P^*$ that is finitely generated as a PD dg algebra, which is equivalent to the homotopy Lie algebra having finite total dimension or to $A$ being a complete intersection ring.

\delete{
Finite generation of cohomology ring is one main questions in many parts of mathematics. 
In this paper, we discuss when the Yoneda algebra $\End_{\mc O_x}(\sf k_x, \sf k_x)$ is a finitely generated algebra over $ \sf k_x$ for a closed  point in  an  algebraic variety. 
Given a coherent sheave  $ \mc F$ on $ X$, we also show that  $H^*(\mc O_x, \mc F_x)$ if a finitely finitely generated module over the Yoneda algebra $\End_{\mc O_x}(\sf k_x, \sf k_x)$.

 Given a finite dimensional augmented algebra $ \epsilon: A\to \sf k$, the cohomology (defined in \cite{Cartan-Eilenberg}) $ H^*(A, \sf k)=\Ext^*_A(\sf k, \sf k)$ has a $\sf k$-algebra structure. One of the question is when this algebra is finitely generated. 
 It has been conjectured that if $ A$ is a finite dimensional Hopf algebra and $ \epsilon $ is counit, then $ H^*(A, \sf k)$ is finitely generated as $ \sf k$-algebra. 
 This conjecture was based on the following earlier known result. When $ A=\sf k G$ for a finite group (for any field), this is a result of Evens\cite{Evens} and Venkov \cite{Venkov}. 
 If  $A=u(\mathfrak{g})$ is the restricted enveloping algebra of a restricted Lie algebra over a field $ \sf k$ of characteristic $p>0$, this was proved by Friedlander and Parshall \cite{Friedlander-Parshall}. 

 Friedlander and Suslin \cite{Friedlander-Suslin} extended above result to any finite  group schemes, whose group algebra (or rather the algebra of distributions at the identity) is exactly a finite dimensional cocommutative algebra $A$. 
 For the finite dimensional Hopf subalgebras arising from quantum groups at roots of unity, Ginzburg and Kumar \cite{Ginzburg-Kumar} also proved the finite generation of the cohomology ring. 
 
 Etingof has a more general conjecture for finite tensor categories. In above cases, the cohomology rings are graded commutative. 
 Thus the even degree part $H^{2*}(A, \sf k)$ is also a finitely generated commutative algebra, and thus defines a conical affine scheme $ \spec(H^{2*}(A, \sf k))$ (or $\onsf{proj}(H^{2*}(A, \sf k))$).
 However, $ \Ext_A^*(\sf k, \sf k)$ is in general not graded commutative.
 The finite generation of $\Ext_A^*(\sf k, \sf k)$ does not imply that $ \Ext_A^{2*}(\sf k, \sf k)$ is finitely generated. 
 In a subsequent paper, we will use complete intersection resolution to find large a commutative quotient of $ \Ext_A^{2*}(\sf k, \sf k)$ in order to estimate the dimension of the corresponding variety for maximal quotient. 
 
 The finite generation of the Yoneda algebra $\Ext_A^{2*}(\sf k, \sf k)$ for a Noetherian local ring $A$ was conjectured by Levin \cite[page 115]{Gulliksen-Levin}. Conter examples were found ?????}
 \delete{
{\color{blue}
 Let $ \cal C$ be an abelian $\msf k$-linear category with enough projective objects and  small limit.
 Let $ \on{Irr}(\cal C)$ be the isomorphism classes of irreducible objects.  Let $L_{\cal C}=\oplus_{L\in \on{Irr}(C)}L$. 
 One is interested in computing the Yoneda $\msf k$-algebra $ \Ext^*_{\cal C}(L_{\cal C}, L_{\cal C})$. More generally, given any object $M$ in $\cal C$, the Yoneda algebra $\Ext^*_{\cal C}(M, M)$ is of interest. 
 This algebra is not commutative in general. 
 The special case when $A$ is a commutative $\msf k$-algebra and $ I\subseteq A$ is an ideal, then taking $M=A/I$ as an $ A$-module, the analogous Yoneda algebra $ \Ext^*_A(A/I, A/I)$  can be computed in the category  $\cal C=A\Mod$ of $A$-modules.
 The commutativity of this algebra is a measurement of smoothness of the embedding of the closed sub-variety defined by the ideal $ I$ in $ \on{Spec}(A)$.  
 }}

 Although we are focusing on the case where $ A$ is a commutative ring and $I$ is an ideal in this paper, the Yoneda algebra has the following geometric formulation. 
Let $ X$ be any $\msf k$-scheme and $Z\subseteq X$ be a closed subscheme. Let $ \cal I \subseteq \cal O_X$ be the ideal sheaf (on $X$) defining $Z$. 
The structure sheaf of $Z$, $\cal O_Z=\cal O_X/\cal I$, is an $\cal O_X$-module supported over $Z$. We then want to compute the graded sheaf of $\cal O_Z$-modules ${\CExt}^*_{\cal O_X}(\cal O_Z, \cal O_Z)$ in the category of $\cal O_X$-modules on $X$. More generally, when $ Z'$ is another closed subscheme of $ X$ with the defining ideal $ \cal J$, we want to compute the graded sheaf ${\CExt}^*_{\cal O_X}(\cal O_Z, \cal O_{Z'})$.
The sheaves ${\CExt}^*_{\cal O_X}(\cal F, \cal G)$ for certain vector bundles $\cal F$ on $Z$ and $ \cal G$ on $Z'$ are expected to have physical interpretations \cite{CKS}.
We note that this sheaf is supported over  the scheme intersection $ Z\times_X Z'$, which is a subscheme of $X$ defined by the ideal $ \cal I+\cal J$ with structure sheaf $ \cal O_Z\otimes_{\cal O_X}\cal O_{Z'}$. The sheaves $ \CTor^{\cal O_X}_*(\cal O_Z, \cal O_{Z'})$ measure the  derived intersection of $ Z$ and $ Z'$ \cite{ACH}. 
When $ x\in Z\cap Z'$, by considering the localisations of $ \cal O_X$, $ \cal O_Z$ and $\cal O_{Z'}$ at $x$, the modules $ \CTor^{\cal O_{X, x}}_i(\cal O_{Z,x}, \cal O_{Z', x})$ provide the intersection multiplicity of $ Z$ and $Z'$ at $x$ as Serre's homological characterisation of intersection theory. 

\delete{
Tate constructed  a  strictly graded commutative differential graded algebra $ (T_*(R/I), d_*)$ that is a resolution of the $R$-module $R/I$. Thus the $\Tor^R_*(R/I, R/J)=H_*(T_*(R/I)\otimes_R R/J)$ and $\Ext_R^*(R/I, R/J)=H^*(\Hom_R(T_*(R/I), R/J)$.
The Yoneda algebras $\Ext_R^*(R/I, R/I)$ and $\Ext_R^*(R/J, R/J)$ act on $\Ext_R^*(R/I, R/J)$ on the left and right respectively. These can be seen in the derived category of  $R$-modules. This is exactly the case for the affine scheme $X=\spec(R)$. 
 }

If $ Z\subseteq X$ is a closed subvariety then $\Ext^1_{\cal O_X}(\cal O_Z, \cal O_Z)$ is the tangent space of the moduli space of subvarieties at the point $[Z]$. 
The (skew)-commutativity of the Yoneda product can be used to construct symplectic structure on the moduli space (see \cite{Kuznetsov-Markushevich}).
If $ [Z]$ is a smooth point of the moduli space, then the Yoneda product is skew commutative. When $Z=\{z\}$ is a closed point, $\Ext^{1}_{\cal O_X }(\msf k_z, \msf k_z)$ is the tangent space $T_zX$. 


\delete{
\subsection {Motivations} As we will see from Tate's construction, the PD structure naturally arises from the following question.
Let $A$ be an associative algebra and $I\subseteq A$ a two-sided ideal, $\bar A=A/I$ is an $A$-bimodule. 
One of the homological questions is to  compute $ \onsf{Ext}^*_A(\bar A,\bar A)$ which has a Yoneda algebra structure.
Since $ \bar A$ is a ring, one would like to construct a resolution of $ \bar A$ in the category of $ A\onsf{-bimod}$ which is a Hopf algebra such that the Yoneda algebra structure lifts to a coalgebra structure on the resolution. 
Dealing with resolutions, one cannot avoid the differential, which will naturally give rise to PD structures. 
 The category $ A\onsf{-bimod}$ is a monoidal category. Assume that there is a graded algebra object $F_*=\bigoplus _{n=0}^{\infty}F_n$ with $ F_0=A$ and $d: F_{n+1}\to F_n$ is a differential of $ A$-bimodules such that 
 $d(xy)=d(x)y+(-1)^{|x|}xd(y)$ for all $x,y\in F$ homogeneous and $dF_n\subseteq IF_{n-1}$. $F_n$ are free left and right $A$-modules.
 Note that the homology is naturally a ring $H_*(F_*, d)=\bar A$.

 The tensor product $ F_*\otimes_A F_*$ is a resolution of $ \bar A\otimes_A \bar A=\bar A$. A comultiplication $ \Delta: F_*\to F_*\otimes F_*$ would lift the Yoneda product as follows: 

We know that for any left $ A$-module $M$, $\onsf{Tor}^A_*(\bar A, M)=H_*(F_*\otimes_A M)$ and $\onsf{Ext}_A^*(\bar A, M)=H_*(\onsf{Hom}_{A\onsf{-Mod}}(F_*, M))$
}
\bigskip

\textsc{Acknowledgements}. The authors would like to thank the anonymous referee for a detailed report containing thoughtful comments that greatly helped in improving the present paper. 
This work started with substantial discussions with Cuipo Jiang on the cohomological varieties for vertex algebras and both authors appreciate her contributions. 
The second author also wants to thank Amnon Yekutieli for sending his preliminary version of \cite{Yekutieli}.

\section{The category of PD dg algebras over a PD dg ring}\label{sec:2}

\subsection{Complexes and graded modules}

Fix a commutative ring $ A$ and consider the categories $ \onsf{Cplx}(A)$ of differential (cochain) complexes of $A$-modules, and $A\onsf{-mod}^{\mathbb{Z}}$ the category of $\mathbb{Z}$-graded $A$-modules. 
Note that $A\onsf{-mod}^{\mathbb{Z}}$ is a symmetric monoidal category with tensor product $\otimes = \otimes_A $ where
\[
(M^* \otimes N^*)^n=\bigoplus_{i+j=n}M^i \otimes N^j
\]
and with braiding
\[
\begin{array}{cccc}
b_{M^*,N^*}:&M^i \otimes N^j & \to & N^j  \otimes M^i \\
& m \otimes n & \mapsto & (-1)^{ij}n \otimes m.
\end{array}
\]
Then $\onsf{Cplx}(A)$ is also a symmetric monoidal category with tensor product of differential complexes with a similar braiding. 
We label the cochain complexes 
\[ \cdots \to M^{i}\stackrel{d^{i}}{\to} M^{i+1}\to \cdots
\]
if the differential map is of degree $1$.  
By simply relabeling $ M_i=M^{-i}$, then each differential cochain complex becomes a differential chain complex with differential $d^{i}=d_{-i}$ of degree $-1$. 
We say that $(M^*,d^*)$ is $A$-free if $M^i$ is a free $A$-module for all $i$, and the category of these complexes is written $\onsf{Cplx}(A)^{free}$.

Finally we have a forgetful functor 
\[
 \onsf{Cplx}(A) \to A\onsf{-mod}^{\mathbb{Z}}
\]
sending $(M^*,d^*)$ to $M^*=\bigoplus_{n \in \mathbb{Z}}M^n$. 
For simplicity reasons, we will later on write $M^*$ for either the complex or the resulting graded $A$-module, depending on the context.

We will also consider the full subcategory $\onsf{Cplx}(A)^+$ of the cochain complexes $ (M^* , d^*)$ with $ M^i=0$ for all $ i>0$.

\subsection{PD dg rings}\label{sec:pd_rings} 

A strictly graded commutative dg ring refers to a graded commutative dg ring satisfying $ xx=0$ for all odd degree elements $x$.  

Let $(R^*, d^*) $ be a strictly graded commutative dg ring with $I \subseteq R$ a dg ideal. We write $I_{ev}=\bigoplus_{n \in \mathbb{Z}}I^{2n}$ and $I_{odd}=\bigoplus_{n \in \mathbb{Z}}I^{2n+1}$ for the even and odd components of $I$. We define $R_{ev}$ and $R_{odd}$ similarly. A divided power (PD) structure on $I$ is a sequence of maps $\gamma_n: I_{ev}\to I_{ev}$ ($n \in \mathbb Z_{+}$) and $\gamma_0:I_{ev} \to R^*$ given by $\gamma_0(x)=1$ satisfying the following conditions:
\begin{enumerate}\setlength{\itemsep}{2pt}
\item \label{axiom_1}  $\gamma_1(x)=x$ for all $x \in I$;
\item \label{axiom_2} $\gamma_n(x)\gamma_m(x)=\binom{n+m}{m} \gamma_{n+m}(x)$ for all $x\in I_{ev}$;
\item \label{axiom_3} $\gamma_n(ax)=a^n\gamma_n(x)$ for all $ a\in R_{ev}$ and $x\in I_{ev}$;
\item \label{axiom_4} $\gamma_n(x+y)=\sum_{i=0}^n\gamma_i(x)\gamma_{n-i}(y)$ for all $ x, y\in I_{ev}$;
\item \label{axiom_5} $\gamma_p( \gamma_q(x))=\frac{(pq)!}{p!(q!)^p}\gamma_{pq}(x)$ for all $x \in I_{ev}$; 
\item \label{axiom_6} $ \gamma_n(I_{2r})\subseteq I_{2rn}$ for $n>0$;
\item \label{axiom_7}$\gamma_n(xy)=0 \text{ for all } x \in R_{odd}, y\in I_{odd} \text{ homogeneous of odd degrees and } n\geq 2$;
\item \label{axiom_8} $d(\gamma_n(x))=\gamma_{n-1}(x)d(x)$. 
\end{enumerate}
A PD dg ring is a triple $(R^*,I,\gamma)$ where $R^*$ is a dg ring, $I$ is a dg ideal of $R^*$, and $\gamma$ is a PD structure on $I$. We will call $I$ the PD ideal of $(R^*,I,\gamma)$. When the ideal $I$ and the PD structure $\gamma$ are clear, we will simply write $R^*$ for the PD dg ring.

The condition \eqref{axiom_7} on the odd degree was imposed by Andr\'e in \cite{Andre}. We recall that the PD structure is heavily dependent on the ideal $I$. The use of the term ``PD'' comes from the French ``puissances divis\'ees'', the term under which the notion was introduced (see \cite{Roby2}). 

It follows from the definition that:
\begin{enumerate}
 \item Given any PD dg ring $R^*$, for any $x \in I_{ev}$, then $n! \gamma_n(x)=x^n$ for all $n \geq 0$.
 
 \item If $R$ is a graded commutative dg $\bb Q$-algebra, i.e., $\bb Q\subset R_0$, then for any dg ideal $I \subset R$ there exists a unique PD dg structure on $R$ given by $\gamma_n(x)=\frac{x^n}{n!}$ for $x \in I_{ev}$. 

 \item Any strictly graded commutative dg ring $R^*$ is a PD dg ring with $I=\{0\}$. In particular, this is true when $R^*$ is concentrated in degree $0$, which is a commutative ring. \label{note_ring_dg}
\end{enumerate}

Let $(R^*, d^*)$ be a dg ring in $\onsf{Cplx}(A)$. Set $B^*(R^*)\subseteq Z^*(R^*)\subseteq R^*$ the coboundary and cocycle subcomplexes of $R^*$. Then $ Z^*(R^*)Z^*(R^*)\subseteq Z^*(R^*)$ by the Leibniz rule, rule that also implies $ d(1)=0$. Thus $Z^*(R^*)$ is a dg subring of $R^*$ (with trivial differential) and  $B^*(R^*)$ is a two-sided ideal of $Z^*(R^*)$. 
Therefore the  cohomology $H^*(R^*)$ is dg ring (with trivial differential).

Note that $Z^*(I)=I\cap Z^*(R^*)$ is a dg ideal of $Z^*(R^*)$. It follows from \eqref{axiom_8} in the definition of the PD structure that $ \gamma_n(Z^*(I)_{ev})\subseteq Z^*(I)_{ev}$. Thus $(Z^*(R^*), Z^*(I), \gamma) $ is a PD dg ring. \delete{Since $B^*(R^*)$ is an ideal of $Z^*(R^*)$, we also have $\gamma_n(I_{ev}\cap B^*(R^*))\subseteq I_{ev}\cap B^*(R^*)$ following $n\gamma_n(d(x))=\gamma_{n-1}(d(x))d(x)$ if $ \bb Q\subseteq R$.} 
Set $I(H^*(R^*))$ the image of $Z^*(I)$ in $H^*(R^*)$, which is clearly an ideal. We want to assume that the PD structure on $(Z^*(R^*), Z^*(I), \gamma)$ induces a PD structure  $(H^*(R^*), I(H^*(R^*)), \bar \gamma)$. This requires 
\begin{align}
\gamma_n(z+b)=\gamma_n(z)+\sum_{i=1}^n\gamma_{n-i}(z)\gamma_{i}(b)\in \gamma_n(z)+I_{ev}\cap B^*(R^*)
\end{align}
for $z\in I_{ev}\cap Z^*(R^*) $ and $b\in I_{ev}\cap B^*(R^*)$, i.e., $ B^*(R^*)$ should be a PD dg ideal of $ Z^*(R^*)$. Since
 $n\gamma_{n}(d(r))=\gamma_{n-1}(d(r))d(r)$ by Axioms \eqref{axiom_1} and \eqref{axiom_2}, then  $B^*(R^*)$ is a PD dg ideal of $ Z^*(R^*)$ if $R^*$ is a $\bb Q$-algebra.  

\delete{ Add a remark on characteristic zero and the free one.

Also in the argument of $\gamma_{p}(x \star y)=0
$ for odd $x,y$. We can not reduce to $p=2$ unless $ \bb Q\subseteq R^*$.}

Given two PD dg rings $(R', I',\gamma')$ and 
$(R'', I'',\gamma'')$, with the ideal $ I=\onsf{Im}(R'\otimes_{\bb Z} I''+I'\otimes_{\bb Z} R'')\subseteq R=R'\otimes_{\bb Z} R''$, 
there exists (cf. \cite[Prop.1.7.10]{Gulliksen-Levin}) a unique PD structure  $(R,I,\gamma) $ such that the obvious dg ring homomorphisms $\iota':(R',I',\gamma')\to(R,I, \gamma)$ and $\iota'': (R'',I'',\gamma'')\to(R,I, \gamma)$ are PD dg ring homomorphisms. 
Using Axiom \eqref{axiom_3}, we can make $\gamma_n$ explicit. If $ x'\in R'$ and $ x''\in R''$ are homogeneous, then $ (x'\otimes x'')=(x'\otimes 1)(1\otimes x'')=(-1)^{|x'||x''|}(1\otimes x'')(x'\otimes 1)$. For $ x''\in I''_{ev}$ and $x' \in R'_{ev}$, we have 
\[
\begin{array}{rcl}
\gamma_n(x'\otimes x'') & = & (x'\otimes 1)^{n}\gamma_n(1\otimes x'') \\[5pt]
& = & (x'\otimes 1)^{n}\gamma_n \circ \iota''(x'') \\[5pt]
& = & (x'\otimes 1)^{n} \iota'' \circ \gamma_n''(x'') \\[5pt]
& = & (x'\otimes 1)^{n}(1\otimes \gamma''_n(x'')) \\[5pt]
& = & x'^{n}\otimes \gamma''_n(x'')
\end{array}
\]
because $\iota''$ is a PD homomorphism. Likewise, for $ x'\in I'_{ev}$ and $x'' \in R''_{ev}$ we have
\begin{align}
\gamma_n(x'\otimes x'')=&(-1)^{n|x'||x''|}(1\otimes x'')^{n}( \gamma'_n(x')\otimes 1) \notag \\[5pt]
=&(-1)^{n|x'||x''|}(-1)^{n|x'|n|x''|}\gamma'_n(x')\otimes x''^{n}\notag \\[5pt]
=&\gamma'_n(x')\otimes x''^{n}. \label{pd_tensor_first}
\end{align}
Moreover, $\gamma_n(R'_{odd} \otimes I''_{odd}+I'_{odd} \otimes R''_{odd})=0$ for $n >1$ because of Axiom \eqref{axiom_7}.

\delete{
{\color{red}
For this to work, we need to modify Axiom \eqref{axiom_7} to
\[
\gamma_n(xy)=0 \text{ for all } x \in R_{odd}, y\in I_{odd} \text{ homogeneous of odd degrees and } n\geq 2.
\]
Otherwise we cannot compute $\gamma_n(R'_{odd} \otimes I''_{odd})$ using Axiom \eqref{axiom_7} OKAY.
}
}

\delete{\color{red} Why do we have $x'^{n}\otimes \gamma''_n(x'')=\gamma'_n(x')\otimes x''^{n}$? yes because $\iota'$ and $\iota''$ are PD homomorphisms}

Therefore the category of PD dg rings is a tensor category. Let $F$ be the forgetful functor from the category of PD dg rings to the category of dg rings. This makes $F$ into a strictly monoidal functor. However, the tensor product is not the coproduct (cf. \cite[Remark 23.3.6]{Stacks-Project}). 
In fact, there is a homomorphism of PD dg rings
\[
(R_1^*,I_1^*, \gamma)\coprod_{(R_0^*,I_0^*, \gamma)}(R_2^*,I_2^*, \gamma) \to (R_1^*,I_1^*, \gamma)\otimes_{(R_0^*,I_0^*, \gamma)}(R_2^*,I_2^*, \gamma)
\]
and there is a natural homomorphism of dg rings   
\[F((R_1^*,I_1^*, \gamma))\otimes_{F((R_0^*,I_0^*, \gamma))}F((R_2^*,I_2^*, \gamma))\to F\left((R_1^*,I_1^*, \gamma)\coprod_{(R_0^*,I_0^*, \gamma)}(R_2^*,I_2^*, \gamma)\right)
\]
which is not an isomorphism (\cite[Remark 23.3.6]{Stacks-Project}).

The example in (\cite[Remark 23.3.6]{Stacks-Project}) suggests that this failure can be avoided if one requires a certain flatness condition, in particular, freeness condition. Thus working on resolutions becomes necessary in many applications.  

We note that the braiding in the tensor category of complexes of abelian groups induces a symmetric tensor category structure on the category of commutative dg rings. The same braiding defines a symmetric tensor structure on the category of PD dg rings. We will need this symmetric tensor category structure later to work on Hopf algebras.  

\delete{(Therefore one can define PD dg scheme and PD dg algebraic group, PD dg Hopf algebra)}

The following argument will become useful later on. Let $ A'$ be a commutative $A$-algebra. We can regard $A'$ as a PD dg $A$-algebra $ (A',\{0\}, 0) $. Let $ \onsf{PDdgAlg}(A)$ be the category of PD dg rings with PD dg ring homomorphisms $(A, \{0\},0) \to (R^*, I^*, \gamma)$. Morphisms of $ \onsf{PDdgAlg}(A)$ are morphisms of PD dg rings which are $A$-linear. In particular, $ \onsf{PDdgAlg}(A)$ contains the category of commutative $A$-algebras as a full subcategory.

If $ A\to A'$ is a ring homomorphism and $(A^*,I,\gamma)$ is a PD dg $A$-algebra, we define the base extension $A^*_{A'}=A^* \otimes_{A} A'$. It has an ideal $I_{A'}=\onsf{Im}(I\otimes A') \subset A^*_{A'}$. Moreover, $\gamma:I_{ev} \to I_{ev}$ extends to $\gamma_{A'}:(I_{A'})_{ev} \to (I_{A'})_{ev}$ by $(\gamma_{A'})_n(x \otimes \alpha)=\gamma_n(x) \otimes \alpha^n$ for all $x \in I_{ev}, \alpha \in A'$.

\begin{lemma}\label{lem:base_change}
If $ A\to A'$ is a ring homomorphism, then for any PD dg $A$-algebra $(A^*,I,\gamma)$, the PD dg $A'$-algebra structure $(A^*_{A'},I_{A'}, \gamma_{A'})$ is isomorphic to the tensor product of PD dg $A$-algebras $(A^*,I,\gamma)$ and $(A',\{0\},0)$.
\end{lemma}

\begin{lemma}\label{lem:torsion_free}
 Let $A^*$ be a dg $\mathbb Z$-algebra, which is torsion free as $\mathbb Z$-module. Let $I \subset A^*$ be a dg ideal. Assume that there exist maps $\gamma_n:I_{ev} \to I_{ev}$ ($n \geq 1$) such that $x^n=n!\gamma_n(x)$ for all $x \in I_{ev}$. Then $\gamma_n$ satisfies the relations (1)-(6). Moreover, for any commutative $\mathbb Z$-algebra $A$, the maps $(\gamma_{A})_n$ also satisfy (1)-(6).
\end{lemma}

\begin{proof}
The ring homomorphism $A^* \to \mathbb{Q} \otimes_{\mathbb Z}A^*$ is injective as $A^*$ is torsion free. The ring homomorphism $A^* \to A^*_{A}$ commutes with $\gamma$ and $\gamma_{A}$.
\end{proof}

\delete{\begin{proposition}
If a set $S$ of elements of odd degree generates $A$ as subalgebra over $ A^{ev}$, then Axiom \eqref{axiom_6} is a consequence of others provided $\gamma_{p}(xy)=0$ for all $ x, y \in S $ and for all $p>1$.   
\end{proposition}

\begin{proof}
    {\color{red} PROBLEM DUE TO SUM DIVIDED POWER OF A SUM.}
\end{proof}
 }

\section{Koszul-Tate Resolutions for arbitrary commutative rings}\label{sec:3}
\delete{
{\color{red}  Let $ R$ be a commutative ring and $ I\subseteq R$ be an ideal. The Koszul-Tate resolution (\cite{Tate}) is a strictly graded commutative differential $R$-algebra $R\langle T\rangle=\onsf{TS}_A(T)$ where $T$ should a complex of free $R$-modules and $\onsf{TS}_A(T)$ is the strictly commutative  differential graded $R$-algebra such that $H_*(\onsf{TS}_A(T))=R/I$ concentrated in degree $0$. In particular $B_i(\onsf{TS}_A(T))\subseteq I\onsf{TS}_A(T)_i$. 
 $T$ should have dg Lie coalgebra structure, 
 which should be the cotangent complex (in derived which defines the coalgebra structure on $ \onsf{TS}_A(T)$ and the dual of $T$ should be the tangent Lie algebra?

 It known given a vector space $V$, that $ T_A^*(V)$ has a comultiplication corresponding using shuffle. (Roby's contruction $T_A^*(V)$ with a shuffle product

 $V$ can be thought as Lie coalgebra with $ \delta: V\to (\bb C +V)\otimes (\bb C+V)$ such that $ \delta(v)=1\otimes v+v\otimes 1$.

 In the smooth case, $\onsf{TS}_A(V)=\wedge^*(V)$.   Want to know for complete intersection case. 
 This is in Andre's paper of Hopf algebra with PD structure. i.e., understand the coalgebra structure of tate resolution of $R/I$ in the complete intersection case. coderivation of a coalgebra. 

 Using the symmetric tensor construction in the grade case will clearly tell us two different versions of exterior algebras. 
 In general, given a free $\msf k$-module $V$, the exterior algebra $ \wedge^*_{\msf k}(V)$ is coinvariant algebra of the tensor algebra $ T_A^*(V)$ under the symmetric group $ S_n$ actions on $V^{\otimes n}$ by the standard permutation action tensored with the rank one sign representation. 
 Thus $\wedge^*(V)=\onsf{Sym}_{\msf k}(V[-1])$ with $ V $ regarded as a graded $ \msf k$-module concentrated in degree $0$. 
 In $\wedge^*_{\msf k}(V)$ one does not have $ v\wedge v=0$, but only has $2 v\wedge v=0$. 
 However, $v^2=0$ in  $\onsf{TS}_{\msf k}(V[-1])$ is automatic. 
 This clearly tells the difference between $\onsf{TS}_{\msf k}(V[-1]) $ and $\wedge^*_{\msf k}(V)$, despite in the literature, very often, these two were confused. 
 Here $ \onsf{TS}_{\msf k}(V[-1])$ is the invariants. See \cite{Loday-Vallette} for discussions of differences in the operadic setting, despite that the two are the same over fields of characteristic zero.  } 
}

\subsection{Symmetric tensors of differential complexes}

Let $A$ be a commutative ring. We briefly recall the symmetric strict monoidal category $\onsf{Cplx}(A)$ of differential (cochain) complexes  of $ A$-modules with morphisms being chain maps.  The tensor product over $A$ of two complexes $(X^*, d^*)\otimes (Y^*, d^*)$ is defined by 
\[ 
((X^*, d^*)\otimes(Y^*, d^*))^n=\bigoplus _{i+j=n}X^i\otimes Y^j
\]
with differential $d^{X\otimes Y}(x\otimes y)=d(x)\otimes y+(-1)^{|x|}x\otimes d(y)$ for all homogeneous $x\in X^*$ and $ y\in Y^*$. In this section we will use $\otimes=\otimes_{A} $ unless there is a confusion with other tensor products.

The braiding $b_{X, Y}: X^*\otimes Y^*\to Y^*\otimes X^*$ is defined by 
\[ b_{X,Y}(x\otimes y)=(-1)^{|x||y|}y\otimes x
\]
for all homogeneous $x\in X^*$ and $y\in Y^*$. The map 
$b_{X,Y}$ is clearly a chain map of chain complexes:
\begin{align*}
b_{X,Y}d^{X\otimes Y}(x\otimes y)&=(-1)^{(|x|+1)|y|}y\otimes d(x) +(-1)^{|x|}(-1)^{|x|(|y|+1)} d(y)\otimes x\\
&=(-1)^{|x||y|}(d(y)\otimes x+(-1)^{|y|}y\otimes d(x))=d^{Y\otimes X}b_{X,Y}(x\otimes y).
\end{align*}
An associative algebra object $(A^*, m, 1)$ (with multiplication $m: A^*\otimes A^*\to $ being a chain map) in $\onsf{Cplx}(A)$ is called a dg algebra over $A$. We use $ \onsf{DGA}(A)$ to denote the category of all dg algebras over $A$. A dg algebra $(A^*, m, 1)$ is called {\em strictly graded commutative}  if  
\begin{align}\label{diag:commutativity}
\xymatrix{A^*\otimes A^*\ar[dr]_m\ar[rr]^{b_{A, A}}&& A^*\otimes A^*\ar[dl]^{m}\\
&A^*&
}
\end{align} 
is a commutative diagram {\em and} if $m(x, x)=0$ for all homogeneous $ x\in X^{odd}$.   Let $ \onsf{scDGA}(A)$ denote the category of strictly graded commutative dg algebras in $\onsf{Cplx}(A)$. If $ (A^*, d^A)$ and $ (B^*, d^B)$ are two strictly graded commutative dg algebras, then $ A^*\otimes B^*$ is also a strictly graded commutative dg algebra. We recall the multiplication $m_{A\otimes B}$ is defined as the following:
\[ (A^*\otimes  B^*)\otimes (A^*\otimes B^*)\xrightarrow{1\otimes b_{B, A}\otimes 1}(A^*\otimes A^*)\otimes (B^*\otimes B^*)\xrightarrow{m_A\otimes m_B}A^*\otimes B^*.\]

If $\onsf{cAlg}(A)$ is the category of commutative $A$-algebras, then there are embeddings  $\onsf{cAlg}(A) \hookrightarrow \onsf{scDGA}(A) \hookrightarrow \onsf{PDdgA}(A)$ as full subcategories. This is clear from \eqref{note_ring_dg} in Section \ref{sec:pd_rings}.

 Let $ (X^*, d^*)$ be an object in $\onsf{Cplx}(A)$. Then $T_A^n(X^*)=X^*\otimes\cdots\otimes X^*$ is a differential complex. The symmetric group $ \mf S_n$ acts on $T_A^n(X^*)$ as automorphisms of chain complexes as follows:
\begin{align}\label{action_sigma}
\sigma(x_1\otimes\cdots\otimes x_n)=(-1)^{\ell(\sigma; (x_1, \dots, x_n))}(x_{\sigma^{-1}(1)}\otimes \cdots \otimes x_{\sigma^{-1}(n)}).
\end{align}
Here $\ell(\sigma;(x_1, \dots, x_n))=\sum_{(i, j)\in\onsf{Inv}(\sigma^{-1})}|x_i||x_j|$ and $ \onsf{Inv}(\sigma)=\{ (i< j) \;|\; \sigma(i)> \sigma (j)\}$ is the set of inversions of $\sigma$. This group action is generated by the braiding $(i, i+1)\mapsto b_{i, i+1}=1^{\otimes (i-1)}\otimes b_{X,X}\otimes 1^{\otimes(n-i-1)}$. As $b_{X,X}$ is a chain map, then so is $ \sigma: T_A^n(X^*)\to T_A^n(X^*)$.

Let $\onsf{TS}_A^n(X^*)=T_A^n(X^*)^{\mf S_n}$ be the $A$-submodule of $\mf S_n$-fixed points. Thus $\onsf{TS}_A^n(X^*)$
 is closed under differentials since the $ \mf S_n$-action commutes with the differentials, making $\onsf{TS}_A^n(X^*)$ a subcomplex of $T_A^n(X^*)$.  It is obvious that $ \onsf{TS}_A^1(X^*)=T_A^1(X^*)=X^*.$
 
 Using induction, one can check that, for any sequence of homogeneous elements $(x_1,\dots, x_n)$ in $X^*$ and $ \sigma, \tau\in \mf S_n$, we have
 \[\ell(\sigma\tau; (x_1, \cdots, x_n))\equiv \ell(\sigma; (x_1, \cdots, x_n))+\ell(\tau; (x_{\sigma^{-1}(1)}, \cdots, x_{\sigma^{-1}(n)})) \pmod 2.
 \]


For $m, n \in \mathbb{Z}_{>0}$, we set $\mf S(m,n)=\mf S_{m+n}/(\mf  S_m\times \mf  S_n)$. The relative trace map $\onsf{Tr}_{\on{}(m,n)}:\onsf{TS}_A^{m}(X^*) \otimes \onsf{TS}_A^{n}(X^*) \to \onsf{TS}_A^{m+n}(X^*)$ (see \cite[Chap. IV, \S 5]{Bourbaki_algebra}) is defined by
\[
\onsf{Tr}_{(m,n)}(f \otimes g)=\sum_{\sigma \in \mf S(m,n)}\sigma (f \otimes g).
\]
Note that, for any $ (\sigma_m, \sigma_n) \in \mf S_m\times \mf S_n\subseteq \mf S_{m+n}$, 
\[\sigma \circ(\sigma_m, \sigma_n) (f\otimes g)=\sigma(f\otimes g)
\]
as $f$ and $g$ are symmetric tensors themselves. Hence $\onsf{Tr}_{\mf S(m,n)}(f \otimes  g) \in \onsf{TS}_A^{m+n}(X^*)$.
With this we define a multiplication (shuffle product) 
\begin{align} \label{TS-star-product}
\star: \onsf{TS}_A^m(X^*)\otimes\onsf{TS}_A^n(X^*)\to \onsf{TS}_A^{n+m}(X^*)
\end{align} as follows: for $f\in \onsf{TS}_A^m(X^*)$ and $g\in \onsf{TS}_A^n(X^*)$,
\begin{align} \label{shuffle-product} f\star g=\onsf{Tr}_{\mf S(m,n)}(f \otimes g).
\end{align}

We note that the definition of $\star$-product in \eqref{shuffle-product} does not apply to elements in $T_A(X^*)$ directly. However, this product can be extended to all elements in $ T_A(X^*)$ as follows. 
There is a canonical choice of minimal coset representatives $ \sigma\in \mf S(m,n)$, called $(m,n)$-shuffles (see \cite[Chap. IV, \S 5.3]{Bourbaki_algebra}), such that
\[ \sigma(1)<\cdots<\sigma(m)\; \text{ and } \; \sigma(m+1)<\cdots<\sigma(m+n).
\]
Let $ \on{Sh}(m,n)\subseteq \mf S_{m+n}$ be the set of all $(m,n)$-shuffles. 
Then 
\begin{align}
  f \star g=\sum_{\sigma \in \on{Sh}(m,n)}\sigma (f \otimes g)  
\end{align}
 for all $f\in T_A^m(X^*)$ and $g\in T_A^n(X^*)$. 
 We will also refer to this product on $T_A(X^*)$ as shuffle product.  It is a standard argument that $ (T_A(X^*), \star)$ is an associative dg algebra. 
 In fact, the argument applies to any strict symmetric monoidal category in place of $ \onsf{Cplx}$ by using the obvious bijective maps $ \on{Sh}(l+m, n)\times \on{Sh}(l,m)\to \on{Sh}(l,m,n)$ and $\on{Sh}(l, m+n)\times \on{Sh}(m,n)\to \on{Sh}(l,m,n) $. To distinguish with standard tensor multiplication, we use $ T_A(X^*)_{Sh}=(T_A(X^*), \star) $ to denote the dg algebra with respect to the $\star$-product following the notation of Roby \cite{Roby3} (in fact a slightly modified notation to avoid confusion with the base ring $A$).

\begin{proposition}\label{prop:scdgal}
 Let $X^*$ be an object in $\onsf{Cplx}(A)$. Then $ \onsf{TS}_A(X^*)=\bigoplus _{n=0}^{\infty}\onsf{TS}_A^n(X^*)$ is an $ \bb N$-graded associative dg $A$-subalgebra of $ T_A(X^*)_{Sh}=(T_A(X^*), \star)$ under the multiplication $ \star$. Moreover, we have  functors
 \[
T_A(-): \onsf{Cplx}(A) \to \onsf{DGA}(A)\; \text{ and }\; T_A(-)_{Sh}: \onsf{Cplx}(A) \to \onsf{DGA}(A)
 \]
 such that $ \onsf{TS}_A(-)$ is a subfunctor of $T_A(-)_{Sh}$.
\end{proposition}

\begin{proof}
We know that the natural product $m_{T}: T_A^m(X^*)\otimes T_A^n(X^*)\to T_A^{m+n}(X^*)$ of the tensor algebra is a chain map by:
\begin{align*}
&dm_{T}((x_1\otimes\cdots \otimes x_m) \otimes (x_{m+1} \otimes \cdots \otimes x_{m+n})) \\
&=d(x_1\otimes\cdots \otimes x_m \otimes x_{m+1} \otimes \cdots \otimes x_{m+n}) \\
&=\sum_{i=1}^{m+n}(-1)^{\sum_{j=1}^{i-1}|x_j|}
x_1 \otimes\cdots \otimes d(x_i)\otimes \cdots \otimes x_{m+n}\\
&=\left(\sum_{i=1}^{m}(-1)^{\sum_{j=1}^{i-1}|x_j|}
x_1 \otimes\cdots \otimes d(x_i)\otimes \cdots \otimes x_{m}\right) \otimes x_{m+1} \otimes \cdots \otimes x_{m+n}\\
& \quad +(-1)^{\sum_{j=1}^m|x_j|}\sum_{i=1}^{n}
(-1)^{\sum_{j=1}^{i-1}|x_{m+j}|}
(x_1 \otimes\cdots \otimes x_m)\otimes (x_{m+1}\otimes \cdots \otimes d(x_{m+i})\otimes \cdots \otimes x_{m+n}) \\
&=d(x_1 \otimes \cdots \otimes x_m) \otimes (x_{m+1} \otimes \cdots \otimes x_{m+n})\\
& \quad +(-1)^{|x_1 \otimes \cdots \otimes x_m|}(x_1 \otimes \cdots \otimes x_m) \otimes d(x_{m+1} \otimes \cdots \otimes x_{m+n}), \\
&=m_{T}(d(x_1 \otimes \cdots \otimes x_m) \otimes (x_{m+1} \otimes \cdots \otimes x_{m+n}))\\
& \quad +m_{T}\left((-1)^{|x_1 \otimes \cdots \otimes x_m|}(x_1 \otimes \cdots \otimes x_m) \otimes d(x_{m+1} \otimes \cdots \otimes x_{m+n})\right)\\
&=m_{T}d((x_1\otimes\cdots \otimes x_m) \otimes (x_{m+1} \otimes \cdots \otimes x_{m+n})).
\end{align*}
Thus the algebra $(T_A(X^*), m_{T})$ is a dg $A$-algebra. 

Moreover, the action of $\mf S_{n}$ on $T_A^n(X^*)$ is a chain map too, hence the multiplication $\star: T_A^m(X^*)\otimes T_A^n(X^*)\to T_A^{m+n}(X^*)$ is also chain map. In particular $\star: \onsf{TS}_A^m(X^*)\otimes \onsf{TS}_A^n(X^*)\to \onsf{TS}_A^{m+n}(X^*)$ is also a chain map.
\delete{The associativity of $\star$ can be verified in exactly the same way as in \cite[Chap. IV, \S 5.3, Prop.2]{Bourbaki_algebra}. Finally, using the same reasoning as in \cite[Chap. IV, \S 5.6]{Bourbaki_algebra}, we prove the statement of functoriality.}
If $\phi: X^*\to Y^*$ is a chain map, 
then $ T_A(\phi): T_A(X^*)\to T_A(Y^*)$ is also a chain map and a dg $ A$-algebra homomorphism with respect to the standard multiplication $m_T$. 
The map $T^n(\phi): T_A^n(X^*)\to T_A^n(Y^*)$ commutes with the symmetric group $ \mf S_n$ action.  Therefore, $T_A(\phi)_{Sh}:=T_A(\phi): T_A^n(X^*)_{Sh}\to T_A^n(Y^*)_{Sh}$ is a homomorphism of dg $A$-algebras. 
Clearly we have $ T_A(\phi)_{Sh} (\onsf{TS}_A(X^*)) \subseteq \onsf{TS}_A(Y^*)$. 
Let $ \onsf{TS}_A(\phi):\onsf{TS}_A(X^*)\to \onsf{TS}_A(Y^*)$ be the restriction of $T_A(\phi)_{Sh}$. 
Then the functorialities of $T_A(-)$, $T_A(-)_{Sh}$, and $\onsf{TS}_A(-)$ follow directly. 
\end{proof}

We see that $ \onsf{TS}_A^+(X^*)=\bigoplus_{n>0}\onsf{TS}_A^n(X^*)$ is a dg ideal of the dg algebra $\onsf{TS}_A(X^*)$ and is the kernel of the augmentation map $\onsf{TS}_A(X^*)\twoheadrightarrow A$, which is also a chain map. 

\begin{remark}
If instead of working in $\onsf{Cplx}(A)$ we work in $A\onsf{-mod}^{\mathbb{Z}}$, everything carries over. Forgetting the differentials, we consider $\onsf{TS}_A^n(X^*) $ as an object in the category $A\onsf{-mod}^{\mathbb{Z}}$. We will call this grading the internal dg grading in contrast to the tensor grading $n$ of elements in $ \onsf{TS}_A^n(X^*)$. Thus $\onsf{TS}_A(X^*)$ is an 
$(\bb N \times \bb Z)$-graded algebra object in $A\onsf{-mod}$ while it is an $ \bb N$-graded algebra in $ \onsf{Cplx}(A)$. For a homogeneous element $x\in \onsf{TS}_A^n(X^*)$ having degree $ (n, |x|)$, we will use $ \deg(x)=n$ for tensor degree or polynomial degree, in contrast to $|x|$ for the differential degree. Thus the differential $d: \onsf{TS}_A(X^*)\to \onsf{TS}_A(X^*)$ has degree $ (0,1)$. Taking the total degree $ \onsf{to}(x)=\deg(x)+|x|$, $\onsf{TS}_A(X^*)$ remains an algebra object in $\onsf{Cplx}(A)$. 
\end{remark}

\delete{
\begin{lemma}\label{lem:strictly_commutative}
Let $ X=\bigoplus_{i \in \bb Z} X^i$ be a graded module over $A$. If $x=v_1\otimes \cdots\otimes v_{p}\in T_A^p(X^*)$ with $v_i\in X^*$ homogeneous such that $ |x|=\sum_{i=1}^{p}|v_i|$ is odd, then
$ x\star x=\sum_{\sigma\in\onsf{Sh}(p,p)}\sigma(x\otimes x)=0$.
\end{lemma}
\begin{proof}
Let \(x=v_1\otimes\cdots\otimes v_p\in V^{\otimes p}\) be homogeneous with total degree 
\(|x|=\sum_i|v_i|\) odd. The shuffle product is
\[
x\star x \;=\; \sum_{\sigma\in\mathrm{Sh}(p,p)} \sigma(x\otimes x).
\]

Let \(c\in \mf S_{2p}\) be the block-swap permutation exchanging the first and second \(p\)-blocks 
while preserving the internal order of each block. Then the map
\[
\sigma \;\mapsto\; \sigma c
\]
is a fixed-point-free involution on $\onsf{Sh}(p,p)$. Thus the shuffle sum decomposes into pairs 
\(\{\sigma,\,\sigma c\}\). For any \(\sigma\),
\[
(\sigma c)(x\otimes x)
= \sigma\big( c(x\otimes x)\big) 
= (-1)^{\big(\sum_{\text{first block}}|v_i|\big)\big(\sum_{\text{second block}}|v_j|\big)}
   \,\sigma(x\otimes x)
= (-1)^{|x|}\,\sigma(x\otimes x).
\]

Since \(|x|\) is odd, \((-1)^{|x|}=-1\) as an integer. Therefore each pair contributes
\[
\sigma(x\otimes x) + (\sigma c)(x\otimes x)
= \big(1+(-1)\big)\,\sigma(x\otimes x)
= 0
\] 
Thus all terms cancel in pairs, and we conclude
\[
x\star x=0 \qquad \text{in } V^{\otimes 2p}.
\]
\end{proof}
}

\begin{theorem} \label{TS-scdgalgebra}
Let  $(X^*, d^*)$ be a differential complex in $ \onsf{Cplx}(A)$ such that $X^i$ is a free $A$-module for all $i$. The dg algebra $T_A(X^*)_{Sh}$ is a connected (i.e., $(T_A(X^*)_{Sh})_0=A$) strictly graded commutative $A$-algebra. In particular, $\onsf{TS}_A(X^*)$ is a connected $\bb N$-graded strictly graded commutative dg $A$-algebra.

\delete{Furthermore, the functor $\onsf{TS}_A(-): \onsf{Cplx}(A)^{free}\to \onsf{scDGA}_{\bb N}(A)$  satisfies the following:
\[
\onsf{TS}_A(X^*\oplus Y^*)\cong \onsf{TS}_A(X^*)\otimes \onsf{TS}_A(Y^*).
\]}
\end{theorem}

\delete{The commutativity \eqref{diag:commutativity} follows from direct computation as follows: Set $ X^{odd}=\bigoplus_{i}X^{2i+1}$and $ X^{ev}=\bigoplus_{i}X^{2i}$.$\bullet$  The action of $\mf S_n$ on $T_A^n(X^{ev})$ is exactly the same as that in \cite[Chap. IV, \S 5.2]{Bourbaki_algebra}. In particular, assuming each $X^i$ is a free $A$-module, then  $\onsf{TS}_A(X^{ev})$ is a strictly graded commutative dg algebra with a PD structure (\cite[Thm.3]{Roby3}) and $\onsf{TS}_A(X^{ev})=\bigotimes_{i}\onsf{TS}_A(X^{2i})$ (see \cite[Chap. IV, \S 5.6, Prop.6]{Bourbaki_algebra}) as tensor product of strictly commutative PD algebras. $\bullet$ We now consider the $X^{odd}$ case. For $x\in X^p$ with $p$ odd, we see that $ x\star x=(1+(-1)^{p^2})x\otimes x=0 $. Moreover, if $x \in X^p$ and $y \in X^q$ with $p$, $q$ odd, then we have $x \star y=- y \star x$.}

\begin{proof}
The fact that $T_A(X^*)_{Sh}$ is an $\bb N$-graded associative dg $A$-algebra has already been shown in Proposition \ref{prop:scdgal}. For given $m, n\geq 1$, define $\sigma \in \on{Sh}(m,n)$ by
\begin{align*}
&\sigma(1)=n+1, \ \sigma(2)=n+2, \ \dots \, \ , \ \sigma(m)=n+m, \\
&\sigma(m+1)=1, \ \sigma(m+2)=2, \ \dots \, \ , \ \sigma(m+n)=n.
\end{align*}
\delete{Here $\sigma$ is the block swap permutation of the proof of Lemma \ref{lem:strictly_commutative}.}
Then  $ \on{Sh}(n,m)\circ \sigma= \on{Sh}(m,n)$ and $\on{Sh}(n,m)= \on{Sh}(m,n)\circ \sigma^{-1} $.

Consider $z_1 \in T_A^m(X)$ and $z_2 \in T_A^n(X)$ of homogeneous of dg degrees $ |z_1|$ and $ |z_2|$ respectively. Then we have $ l(\sigma, (z_1\otimes z_2)) \equiv |z_1||z_2|\pmod 2$ and
\begin{align}
\sigma(z_1\otimes z_2)=(-1)^{|z_1||z_2|}(z_2\otimes z_1).
\end{align}
This can be shown by first assuming that both $z_1$ and $z_2$ are pure tensors and then one can extend linearly to both being finite sums of pure tensors. Thus 
\[
\begin{array}{rcl}
z_2 \star z_1 & = & \displaystyle\sum_{\tau\in\on{Sh}(n,m)}\tau(z_2 \otimes  z_1)=\sum_{\tau\in \on{Sh}(m,n)}\tau\sigma^{-1}(z_2 \otimes  z_1)\\
 & = & \displaystyle\sum_{\tau\in \on{Sh}(m,n)}(-1)^{|z_1|z_2|}\tau(z_1 \otimes  z_2)=(-1)^{|z_1||z_2|} (z_1 \star z_2).
\end{array}
\]
This shows that $\star$-product on $T_A(X^*)$ is graded commutative. We still need to show that $ z_1\star z_1=0$ if $ |z_1|$ is odd. When $2$ is not a zero divisor, it is a consequence of the above equality. In general, when $n=m$, the map $F:  \on{Sh}(m,m)\to \on{Sh}(m,m)$ defined by $ \tau\mapsto \tau\circ \sigma$ has no fixed point and $ F^2=\onsf{Id}$. The order two group $ \langle F\rangle$ action on $\on{Sh}(m,m)$ decomposes $\on{Sh}(m,m)=\coprod_{\tau \in (K')_{m}^2}\{\tau, \tau\circ\sigma\} $ as disjoint union of orbits. Here $(K')_{m}^2 \subseteq \on{Sh}(m,m)$ is a complete set of $F$-orbit representatives (see Lemma \ref{lem:Sh_and_c}). Then, for any $ z\in T_A^m(X^*)$ homogeneous,
\begin{align}\label{eq:x*x=0}
z\star z=\sum_{\tau\in (K')_{m}^2}
(\tau(z\otimes z)+\tau(\sigma(z\otimes z)))=\sum_{\tau\in (K')_{m}^2}(1+(-1)^{|z||z|})\tau(z\otimes z).
\end{align}
Therefore $z\star z=0$ if $ |z|$ is odd. 
\delete{
\[
\begin{array}{rcl}
z_2 \star z_1 & = & \onsf{Tr}_{\mf S_{n+m}/(\mf  S_n\times \mf  S_m)}(z_2 \otimes  z_1)\\
 & = & \onsf{Tr}_{\mf S_{n+m}/\sigma (\mf  S_m\times \mf  S_n)\sigma^{-1}}\sigma(z_1 \otimes  z_2)(-1)^{l(\sigma;(z_1,z_2))} \\
 & = &(-1)^{l(\sigma;(z_1,z_2))} \onsf{Tr}_{\mf S_{m+n}/(\mf  S_m\times \mf  S_n)}(z_1 \otimes  z_2) \text{ by \cite[Chap. IV, \S 5.1, Prop.1(i)]{Bourbaki_algebra}} \\
 & = &(-1)^{l(\sigma;(z_1,z_2))} z_1 \star z_2.
\end{array}
\]
If all the components of $z_1$ and $z_2$ are of odd degree, it follows that ${l(\sigma;(z_1,z_2))}\equiv |\on{Inv}(\sigma^{-1})| \pmod 2$.
 We can verify directly that $|\on{Inv}(\sigma^{-1})|=mn$, and so $z_2 \star z_1=(-1)^{mn}z_1 \star z_2$. 
 If $z_1=\sum (x_{1}\otimes \cdots \otimes x_{m}) \in \onsf{TS}_A^m(X^{odd})$ so $|z_1|=\sum_{i=1}^m|x_i|$ with each $x_i$ is of odd degree. Hence $|z_1| \equiv m \pmod 2$. Likewise $|z_2| \equiv n \pmod 2$. 
 It follows that $z_2 \star z_1=(-1)^{|z_1||z_2|}z_1 \star z_2$. 
 \delete{ We conclude that $\onsf{TS}_A(X^{odd})$ with the $\star$-product is the exterior algebra $\Lambda^*(X^{odd})$ of $X^{odd}$.}

\delete{
$\bullet$ We now look at the product $\onsf{TS}_A^m(X^{odd}) \otimes  \onsf{TS}_A^n(X^{ev}) \to \onsf{TS}_A^{m+n}(X^*)$. Consider $z_1 \in \onsf{TS}_A^m(X^{odd})$ and $z_2 \in \onsf{TS}_A^n(X^{ev})$. The formula $z_2 \star z_1=(-1)^{l(\sigma;(z_1,z_2))} z_1 \star z_2$ still applies. We notice that if $(i,j) \in \on{Inv}(\sigma^{-1})$, then $j$ is the index of an element of $z_2 \in \onsf{TS}_A^n(X^{ev})$, and so in the sum inside $l(\sigma;(z_1,z_2))$, we always have $|x_j|$ even. Hence $l(\sigma;(z_1,z_2))$ also even, and so $z_2 \star z_1= z_1 \star z_2$.

We again use \cite[Chap. IV, \S 5.6, Prop.6]{Bourbaki_algebra} to show that $\onsf{TS}_A(X^{odd}) \otimes  \onsf{TS}_A(X^{ev}) \cong \onsf{TS}_A(X^{odd} \oplus X^{ev}) =\onsf{TS}_A(X^*)$ using the $\star$-product. It follows that $\onsf{TS}_A(X^*)$ is a strictly graded commutative dg algebra with respect to the cohomological degree.
}}
\end{proof}

\delete{Also it follows from the construction of Andr\'e, $\onsf{TS}_A(M)$ should have a co-multiplication together with the augmentation above to make $ \onsf{TS}_A(M)$ into a $\bb N$-graded bialgebra. Thus it follows from Milner-Moore, $\onsf{TS}_A(M)$ co-enveloping algebra of dg Lie coalgebra. 

$\onsf{TS}_A(M)$ is a connected $\bb N$-graded dg-algebra in the symmetric tensor category $ \onsf{Cplx}(R)$. For any PD dg-algebra $(A, I, \gamma)$ in $\onsf{cplx}(R)$ and chain map $\phi:  M\to I$, there is a unique PD dg-algebra homomorphism $ \tilde{\phi}: \onsf{TS}_A(M)\to A$ extending $\phi$.}

 \subsection{PD algebra structure} 

 Let $(K')_{n}^p$ be the set of elements $\sigma \in \mf S_{np}$ that preserve the relative order of the terms within each of the $p$ blocks of size $n$ and such that $\sigma (n)<\sigma(2n)<\cdots <\sigma(pn)$. They form a  set of complete coset representatives in $\mf S_{pn}$ modulo the subgroup $ \mf S_p\ltimes (\mf S_n\times\cdots\times \mf S_n)$ as in \cite{Roby3}. Let us state the following well-known lemma which will be used later on.

\begin{lemma}\label{lem:Sh_and_c}
We have
\[
\on{Sh}(n,n)=(K')_{n}^2 \cup (K')_{n}^2c
\]
where $c \in \mf S_{2n}$ is defined by
\[
c:i \mapsto \left\{\begin{array}{l}
i+n \text{ if } i \leq n, \\
i-n \text{ if } i \geq n+1. 
\end{array}
\right.
\]
\end{lemma}
\delete{
\begin{proof}
By definition, $\on{Sh}(n,n)$ is the subset of $\mf S_{n+n}$ such that
\[
\sigma(1)<\sigma(2)<\dots<\sigma(n) \quad \text{and} \quad \sigma(n+1)<\sigma(n+2)<\dots<\sigma(2n).
\]
Hence $(K')_n^2$ is the subset of $\on{Sh}(n,n)$ satisfying $\sigma(n)<\sigma(2n)$. \\
Consider $\sigma \in \on{Sh}(n,n) \backslash (K')_{n}^2$. Then $\sigma(n)>\sigma(2n)$. It follows that $\sigma c(n)=\sigma(2n)<\sigma(n)=\sigma c(2n)$. So $\sigma c(n)<\sigma c(2n)$. \\
If $i < j \leq n$, then $\sigma c(i)=\sigma(i+n)<\sigma(j+n)=\sigma c(j)$ as $\sigma \in \on{Sh}(n,n)$. Hence $\sigma c(i) < \sigma c(j)$. Similarly, if $n+1 \leq i < j$, then $\sigma c(i) < \sigma c(j)$.

It follows that $\sigma c$ preserves the order in each block and  $\sigma c(n)<\sigma c(2n)$. Thus $\sigma c \in (K')_n^2$ and $\sigma  \in (K')_n^2 c$.
\end{proof}
}

\begin{theorem}\label{thm:TS_pd}
If $ V=\bigoplus_i V^i$ is a graded $A$-module, then $T_A(V)_{Sh}$ is a strictly graded commutative algebra with PD structure on the ideal $ T_A^+(V)=\bigoplus_{n>0}T_A^n(X^*)$. In particular, $ \onsf{TS}_A(V)$ is a PD subalgebra of $T_A(V)_{Sh}$ with PD ideal $\onsf{TS}_A^+(V)$.
\end{theorem}

\begin{proof}
Let $\mathcal{V}$ be the free graded monoid constructed on $V$. $\mathcal{V}$ is the set of words constructed with symbols in bijection with the homogeneous elements $V^{hom}$ of $V$. Set $\mathbb Z(\mathcal V)$ the monoid algebra over the integers and set $\mathcal V_p$ the set of words of length $p$. The permutation group $\mf S_p$ acts on $\mathbb Z \mathcal V_p$ by (cf. Eq. \eqref{action_sigma}):
\[
\sigma.(v_1 \cdots v_p)=(-1)^{l(\sigma;(v_1,\dots, v_p))}v_{\sigma^{-1}(1)} \cdots v_{\sigma^{-1}(p)}
\]
and we extend to the whole of $\mathbb Z \mathcal V_p$ by linearity. For two words $m_p \in \mathcal V_p$ and $m_q  \in \mathcal V_q$, we also define
\[
m_p \star m_q = \sum_{\sigma \in \on{Sh}(p,q)} \sigma(m_p \cdot m_q)
\]
where $\cdot$ denotes the monoidal product in $\mathbb Z(\mathcal V)$,
and extend to $\mathbb Z(\mathcal V)$ by linearity. Similarly to Eq.\eqref{eq:x*x=0}, we can show that if $z=v_1 \cdots v_p \in \mathcal V_p$ with $|z|=\sum_{i=1}^p|v_i|$ odd, then $z \star z=0$. We can also show that $\star$ is graded commutative like in the proof of Theorem \ref{TS-scdgalgebra}. It follows that $(\mathbb Z(\mathcal V), \star)$ is a strictly graded commutative algebra. Following \cite{Roby3}, we denote it by $\mathbb Z_S(\mathcal V)$ where the subscript indicates the use of the shuffle product. We will write $\mathbb Z_S^+(\mathcal V)$ for the submonoid of $\mathbb Z_S(\mathcal V)$ built on $\bigoplus_{n>0}\mathcal V_{n}$, and $Z_S^+(\mathcal V)_{ev}$ will designate the $\mathbb{Z}$-submodule $Z_S^+(\mathcal V)$ built on words of even degrees.

For $z \in \mathcal V_p$ with $|z|$ even and $n \geq 1$, we define
 \[
 \gamma_n(z)=\sum_{\sigma\in (K')_{p}^{n}}\sigma(z^n)
 \]
and it satisfies $z^{\star n}=n! \gamma_n(z)$ (see \cite[Prop.1]{Roby3}). Moreover, once can check that for any $z_i \in \mathcal V_{p_i}$, we have $(z_1+\cdots+z_k)^{\star n}=n!\sum_{h_1+\dots+h_k=n}\gamma_{h_1}(z_1) \star \cdots \star \gamma_{h_k}(z_k)$. Hence we write $\gamma_n(z_1+\dots +z_k)=\sum_{h_1+\dots+h_k=n}\gamma_{h_1}(z_1) \star \cdots \star \gamma_{h_k}(z_k)$. This defines maps $\gamma_n:Z_S^+(\mathcal V)_{ev} \to \mathbb Z_S^+(\mathcal V)_{ev}$ such that for any $z \in Z_S^+(\mathcal V)_{ev}$ we have
\begin{align}\label{eq:free_power}
z^{\star n}=n! \gamma_n(z).
\end{align}

Consider $z \in \mathcal V_{n_1}$ and $w \in \mathcal V_{n_2}$ with $|z|$, $|w|$ odd. Then 
\[
(z \star w) \star (z \star w)=\sum_{\sigma \in \on{Sh}(n_1+n_2,n_1+n_2)}\sigma((z \star w) \cdot (z \star w)).
\]
Using Lemma \ref{lem:Sh_and_c}, we see that
\[
\begin{array}{rcl}
(z \star w) \star (z \star w) & = & \displaystyle \sum_{\sigma \in (K')_{n_1+n_2}^2}\sigma((z \star w) \cdot (z \star w))+\sum_{\sigma \in (K')_{n_1+n_2}^2}\sigma c ((z \star w) \cdot (z \star w)) \\[10pt]
 & = & \displaystyle \gamma_2(z \star w)+\sum_{\sigma \in (K')_{n_1+n_2}^2}(-1)^{|z \star w||z \star w|}\sigma  ((z \star w) \cdot (z \star w)) \\[5pt]
  & = & \displaystyle (1+(-1)^{|z \star w||z \star w|})\gamma_2(z \star w) \\[5pt]
  & = & 2 \gamma_2(z \star w)
\end{array}
\]
as $|z \star w|$ is even. But we know that $\star$ is associative and strictly graded commutative, so the left hand size of the above equation is $0$. It follows that $2 \gamma_2(z \star w)=0$ in $\mathbb Z_S(\mathcal V)$, which is a free $\mathbb Z$-module, and so $ \gamma_2(z \star w)=0$. Using Eq.\eqref{eq:free_power}, one can show that if $n \geq 2$, then  $0=\gamma_2(z \star w)\gamma_{n-2}(z \star w)=\binom{n}{2}\gamma_n(z \star w)$ in $\mathbb Z_S(\mathcal V)$. Again, the freeness of $\mathbb Z_S(\mathcal V)$ over $\mathbb Z$ implies $\gamma_n(x \star y)=0$ for all $n \geq 2$. Hence Axiom \eqref{axiom_7} is satisfied for $\mathbb Z_S(\mathcal V)$. Axioms \eqref{axiom_1}-\eqref{axiom_6} are easy to verify using Eq.\eqref{eq:free_power}. It follows that $\mathbb Z_S(\mathcal V)$ is a PD algebra with PD ideal $\mathbb Z_S^+(\mathcal V)$.

Set $A \mathcal V=A \otimes_{\mathbb Z} \mathbb Z(\mathcal V)$ and $A_S \mathcal V=A \otimes_{\mathbb Z} \mathbb Z_S(\mathcal V)$. By Lemma \ref{lem:torsion_free}, the PD structure of $\mathbb Z_S(\mathcal V)$ passes to $A_S \mathcal V$. Following \cite[Section 5]{Roby3}, we see that there is a surjective algebra homomorphism $q:A \mathcal V \to T_A(V)$ such that $q:A \mathcal V_n \to T_A^n(V)$ is compatible with the action of $\mf S_n$. 

Let $T_A(V)_{Sh}=(T_A(V), \star)$  with the $\star$-product given by:
\[
x \star y = \sum_{\sigma \in \on{Sh}(p,q)} \sigma(x \otimes y)
\]
for $x \in T_A^p(V), y \in T_A^q(V)$. Therefore $q$ preserves the $\star$-product by definition. Therefore $T_A(V)_{Sh}$ is a strictly graded commutative algebra with a PD structure given by:
\[
\gamma_n(x)=\sum_{\sigma\in (K')_{p}^{n}}\sigma(x \otimes \cdots \otimes x)
 \]
for $x \in T_A^p(V)$ and $|x|$ even. Moreover the ideal of the PD structure is $T_A^+(V)$.

 We know from Eq.\eqref{shuffle-product} that if $\overline{x}\in \onsf{TS}_A^{p}(V)$ and  $\overline{y}\in \onsf{TS}_A^q(V)$, then $\overline{x} \star \overline{y}\in \onsf{TS}_A^{p+q}(V)$ is homogeneous of even degree. Moreover, note that if $\overline{x} \in \onsf{TS}_A^p(V)$ is of even degree, then  $ \overline{x}\otimes \cdots\otimes \overline{x}$ is clearly invariant under the group $ \mf S_n\ltimes (\mf S_p\times\cdots\times \mf S_p)$. Thus $\overline{\gamma_n}(\overline{x})$ is in $ \onsf{TS}_A^{np}(V)$. It follows that $\onsf{TS}_A(V)$ is a PD subalgebra of $T_A(V)_{Sh}$ with PD ideal $\onsf{TS}_A^+(V)$.
\end{proof}

\delete{
Note that for any free $A$-module $V$, $ \onsf{TS}_A(V)$ is a strictly graded commutative dg algebra based on Theorem \ref{TS-scdgalgebra}. It has a PD structure as follows. 
 
 We have seen that $ x\in \onsf{TS}_A^{n}(X)$ and  $y\in \onsf{TS}_A^m(X)$, then $x\star y\in \onsf{TS}_A^{n+m}(X)$ is homogeneous of even degree. One can now verify that if $|x|$ and $|y|$ odd, then $ (x\star y)^{(p)}=0$ for all $ p\geq 2$. {\color{red} To prove this, we first assume $n=m=1$ and one argue as in \cite{Bourbaki_commu} when $x\in X$ has even degree.} If $x, y\in X$  are of odd degree, then $x\star y=x\otimes y-y\otimes x $. Take $ p=2$ we have  
 \[
 (x\star y)\otimes (x\star y)=(x\otimes y)\otimes (x\otimes y)+(y\otimes x)\otimes (y\otimes x)-(y\otimes x)\otimes (x\otimes y)-(x\otimes y)\otimes (y\otimes x).
 \]

{\color{red} Do the case for $p=2$ and $x, y \in TS^n(X)$}
 
 Note that $(K')_2^2$ consists of the following permutations 
 \[\begin{pmatrix}
     1&2&3&4\\
     1&2&3&4
\end{pmatrix},\quad
 \sigma=\begin{pmatrix}
     1&2&3&4\\
     1&3&2&4
\end{pmatrix},\quad \tau=\begin{pmatrix}
     1&2&3&4\\
     2&3&1&4
\end{pmatrix},
 \]
 and $\ell(\sigma)=1$ and $\ell(\tau)=2$.
 A direct computation shows that 
 \[
 (x\star y)^{(2)}=(x\star y)\otimes (x\star y)+\sigma((x\star y)\otimes (x\star y))+\tau((x\star y)\otimes (x\star y))=0.
 \]
}

\begin{theorem}\label{thm_pddg}
Set $(X^*, d^*)$ a differential complex in $ \onsf{Cplx}(A)^{free}$. Then $\onsf{TS}_A(X^*)$ equipped with the shuffle product is PD dg subalgebra of $T_A(X^*)_{Sh}$.
\end{theorem}

\begin{proof}
Axioms \eqref{axiom_1}-\eqref{axiom_7} have been verified in Theorem \ref{thm:TS_pd}. It remains to check Axiom \eqref{axiom_8}.

Let $\mathcal{X}$ be the free graded monoid on the symbols $\{u_x \ | \ x \in X^{hom} \}$. Then the monoid ring $\mathbb Z(\mathcal X)$ is a dg algebra with differential $d(u_x)=u_{d(x)}$ for $x \in X^*$. Based on the proof of Theorem \ref{thm:TS_pd}, we can define the $\star$-product and see that $\mathbb Z_S(\mathcal X)$ ($=\mathbb Z(\mathcal X)$ with $\star$-product) has a PD structure and a dg structure compatible with the action of $\mf S_n$.

Consider $x \in \mathbb Z_S^+(\mathcal X)_{ev}$. Then
\[
n!\gamma_n(x)=x^{\star n}.
 \]
But we know that $\mathbb Z_S(\mathcal X)$ is a dg algebra and so 
\[
\begin{array}{rcl}
n!d(\gamma_n(x)) & = &\displaystyle d(x^{\star n})\\
& = &\displaystyle  \sum_{i=0}^{n-1} (-1)^{i|x|}x^{\star i} \star  d(x) \star x^{\star (n-1-i) } \\
& = &\displaystyle  \left(\sum_{i=0}^{n-1} x^{\star (n-1)}\right)\star  d(x) \text{ as } |x| \text{ even}   \\
& = &\displaystyle  n x^{\star (n-1)} \star d(x) \\
& = & n! \gamma_{n-1}(x) \star d(x)
\end{array}
\]
It follows that $n! \big(d(\gamma_n(x))-\gamma_{n-1}(x) \star d(x)\big)=0$ in $\mathbb Z_S(\mathcal X)$, which is free over $\mathbb Z$. Hence $d(\gamma_n(x))=\gamma_{n-1}(x) \star d(x) $ and Axiom \eqref{axiom_8} is satisfied.

The map $q$ defined in the proof of Theorem \ref{thm:TS_pd} commutes with $d$. It follows that $T_A(X^*)_{Sh}$ has a dg structure and a PD structure coming from those of $\mathbb Z_S(\mathcal X)$. But Axiom \eqref{axiom_8} is satisfied in $\mathbb Z_S(\mathcal X)$, and so it is also satisfied in $T_A(X^*)_{Sh}$.

We have already seen in Proposition \ref{prop:scdgal} and Theorem \ref{thm:TS_pd} that $\onsf{TS}_A(X^*)$ is a dg subalgebra ($d$ commutes with the action of $\mf S_n$) and a PD subalgebra of $T_A(X^*)_{Sh}$. As Axiom \eqref{axiom_8} is satisfied in $T_A(X^*)_{Sh}$, it will also be satisfied in $\onsf{TS}_A(X^*)$. Hence $\onsf{TS}_A(X^*)$ is a PD dg subalgebra of $T_A(X^*)_{Sh}$.
\end{proof}

\delete{
\subsection{Strictly graded commutative dg algebras} We now claim that $ \onsf{TS}_A(X)$ is strictly graded commutative differential graded algebra. 
In fact for any $x$ homogeneous, we have 
$x\star x=x\otimes x+(-1)^{|x|^2}x\otimes x$. 
If $x$ has odd degree, then we have $ x\star x=0$. 
In fact,   take $\tau\in \mf S_{m+n}$ by 
\[ \tau(i)= \begin{cases}
    n+i &\text{ if } i\leq m, \\
    i-m &\text{ if } i> m,
\end{cases}
\]
we have $ \tau (x_1\otimes x_2)=(-1)^{|x_1||x_2|}x_2\otimes x_1$  for any $ x_1\in T_A^m(M)$ and $x_2\in T_A^n(V)$ 
homogeneous of degree $(m, |x_1|)$ and $(n, |x_2|)$ respectively. 
Note that $ \mf S_m\times \mf S_n$ and $ \mf S_n\times \mf S_m$ are two different subgroups of $ \mf S_{m+n}$ and they are related by $\mf S_n\times \mf S_m=\tau \mf (\mf S_m\times \mf S_n)\tau^{-1} $. 
If $F\subseteq \mf S_{m+n}$ is  a complete set of coset representatives of $ \mf S_{m+n}/(\mf S_{m}\otimes \mf S_{n})$, then $ \tau F \tau ^{-1}$ is a complete set of coset representatives of $ \mf S_{m+n}/\tau(\mf S_{m}\otimes \mf S_{n})\tau^{-1}$.
Thus we have
\[
x_1\star x_2=(-1)^{|x_1||x_2|} x_2\star x_1.
\]

Let $\onsf{scDGA}_{\bb N}(R)$ be the category of $\bb N$-graded strictly graded commutative dg algebras over $ R$. Then $\onsf{scDGA}_{\bb N}(R)$ is a also symmetric monoidal category naturally. 
}

 By using Theorem \ref{thm_pddg}, we see that, as graded algebras,  
 \[
\onsf{TS}_A(X^*)=\bigotimes_{i\in \bb Z}\onsf{TS}_A(X^i[-i])
\]
where the $A$-modules $X^i$ are regarded as complexes concentrated in degree $0$. This also gives a PBW type basis for the algebra  $\onsf{TS}_A(X^*)$ if each $ X^i$ is a free $A$-module with basis $B_n$ (with an argument similar to that of \cite[Chap. IV, \S 5.5, Prop.4]{Bourbaki_algebra}). Set $\bb N_s=\bb N \coprod \mathbb{Z}/2\mathbb{Z}$ and let $ f: B=\coprod_{n}B_n\to \bb N_s$ be a function with finite support where $f(B_n) \subset \bb N$ if $n$ is even and $f(B_n) \subset \mathbb{Z}/2\mathbb{Z}$ if $n$ is odd. We use $ \bb N_s^{(B)}$ to denote the set of all such maps $f$. For each $ f \in \bb N_s^{(B)}$, set
\begin{align}\label{PBW_TS}
x^{(f)}=\prod^\star_{b\in B} b^{(f(b))}
\end{align}
where $b^{(f(b))}=\gamma_{f(b)}(b)$. Here $ \prod^\star$ means a finite $\star$-product with a fixed ordering of the set $ B$.
Note that $x^{(f)}\in \onsf{TS}_A^{\sum_{b\in B}f(b)}(X^*)$ and its dg degree is
\[|x^{(f)}|=\sum_{b\in B}f(b)|b|.
\]
In the following, we consider the differential $d$. As $T_A(X^*)_{Sh}$ is a dg algebra (cf. Proposition \ref{prop:scdgal}), it follows that for each element $f$ we have:
\begin{align}\label{eq:differential}
d(x^{(f)})=\sum_{b\in B}\left ( \prod^\star_{b'<b}(-1)^{f(b')|b'|} {b'}^{(f(b'))}\right ) \star d(b^{(f(b))}) \star \left (\prod^\star_{b'>b} {b'}^{(f(b'))}\right ).
\end{align}

Given two complexes $X^*$ and $Y^*$, we have PD dg algebra homomorphisms $\iota_{\onsf{TS}_A(X)}:\onsf{TS}_A(X) \to \onsf{TS}_A(X) \otimes \onsf{TS}_A(Y)$ and $\iota_{\onsf{TS}_A(Y)}:\onsf{TS}_A(Y) \to \onsf{TS}_A(X) \otimes \onsf{TS}_A(Y)$ (see Section \ref{sec:2}).

\begin{theorem}\label{thm_pddgfunctor}
The assignment $ \onsf{TS}_A(-): \onsf{Cplx}(A)^{free} \to \onsf{PDdgAlg}(A)$ defines a functor. Moreover, for a direct sum $X^* \oplus Y^*$ with injections $\iota_X:X^* \to X^* \oplus Y^*$ and $\iota_Y:Y^* \to X^* \oplus Y^*$, we have an isomorphism of PD dg algebras
\[
h: \onsf{TS}_A(X^*) \otimes_A \onsf{TS}_A(Y^*) \stackrel{\cong}{\to} \onsf{TS}_A(X^* \oplus Y^*)
\]
such that $h \circ \iota_{\onsf{TS}_A(X)}=\onsf{TS}_A(\iota_X)$ and $h \circ \iota_{\onsf{TS}_A(Y)}=\onsf{TS}_A(\iota_Y)$.

\delete{
where for any $m_X \in \onsf{TS}_A(X^*), m_Y \in \onsf{TS}_A(Y^*)$, we have
\[
h(m_X \otimes m_Y)=\onsf{TS}_A(\iota_X)(m_X) \star \onsf{TS}_A(\iota_Y)(m_Y).
\]
In particular, for $x \in X^*$, $y \in Y^*$,
\[
h\left(\sum_{p_X+p_Y=p}\gamma_{p_X}(x) \otimes \gamma_{p_Y}(y)\right)=\gamma_{p}(x+y).
\]
}

\delete{
{\color{red} Based on Bourbaki, it should be a morphism of algebras. But what is the product on $\onsf{TS}_A(X^*) \otimes \onsf{TS}_A(Y^*)$? Is it \\
$(x \otimes y)(x' \otimes y')=(-1)^{|x'||y|}(x \star x') \otimes (y \star y') $ ? If it is, then it works.}
}
\end{theorem}

\begin{proof}
Let $u:X^* \to Y^*$ be a chain map of complexes and $T_A(u):T_A(X^*) \to T_A(Y^*)$ the graded dg algebra homomorphism. In fact, $T_A(u)$ is also a PD dg algebra homomorphism $T_A(X^*)_{Sh} \to T_A(Y^*)_{Sh}$.  Since  $|u|=0$, we have $\ell(\sigma;(u(x_1), \cdots, u(x_n))) = \ell(\sigma;(x_1, \cdots, x_n))$. Thus, the chain map $T^n(u):T_A^n(X^*)\to T_A^n(Y^*) $ commutes with the action of $\mf S_n$.  Therefore $T_A(u)(\onsf{TS}_A(X^*)) \subset \onsf{TS}_A(Y^*)$. Moreover the restriction $\onsf{TS}_A(u):\onsf{TS}_A(X^*) \to \onsf{TS}_A(Y^*)$ of $T_A(u)$ is a homomorphism of PD dg algebras. 

The statement about direct sums and tensor products and the fact that $h$ is an isomorphism of dg algebras is done similarly to \cite[Chap. IV, \S 5.6]{Bourbaki_algebra}. To see that $h$ preserves the PD structure for even factors, consider $m_X \in \onsf{TS}_A(X^*)$, $m_Y \in \onsf{TS}_A(Y^*)$ with $|m_X|,|m_Y|$ even. Then
\[
\begin{array}{rcl}
h (\gamma_n(m_X \otimes m_Y)) & = & h(\gamma_n(m_X) \otimes m_Y^{\star n}) \quad \text{by Eq.}\eqref{pd_tensor_first} \\[5pt]
& = & \onsf{TS}_A(\iota_X)(\gamma_n(m_X)) \star \onsf{TS}_A(\iota_Y)(m_Y^{\star n}) \\[5pt]
& = & \onsf{TS}_A(\iota_X)(\gamma_n(m_X)) \star \onsf{TS}_A(\iota_Y)(m_Y)^{\star n} \text{ as } \onsf{TS}_A(\iota_Y) \text{ is a homomorphism} \\
   &&\text{ of PD dg algebras}
\end{array}
\]
and
\[
\begin{array}{rcl}
\gamma_n (h(m_X \otimes m_Y)) & = & \gamma_n(\onsf{TS}_A(\iota_X)(m_X) \star \onsf{TS}_A(\iota_Y)(m_Y)) \\[5pt]
 & = & \gamma_n(\onsf{TS}_A(\iota_Y)(m_Y) \star \onsf{TS}_A(\iota_X)(m_X)) \text{ as }\star \text{ is graded commutative}\\[5pt]
  & = & \onsf{TS}_A(\iota_Y)(m_Y)^{\star n} \star \gamma_n( \onsf{TS}_A(\iota_X)(m_X)) \text{ because of Axiom \eqref{axiom_3}}\\[5pt]
   & = &  \onsf{TS}_A(\iota_X)(\gamma_n(m_X)) \star \onsf{TS}_A(\iota_Y)(m_Y)^{\star n} \text{ as } \onsf{TS}_A(\iota_X) \text{ is a homomorphism} \\
   &&\text{ of PD dg algebras}.
\end{array}
\]
Thus $h (\gamma_n(m_X \otimes m_Y))=\gamma_n (h(m_X \otimes m_Y))$ so $h$ preserves the divided power structure. \delete{However, if either $m_X$ or $m_Y$ is of odd degree, then $\gamma_n(m_X \otimes m_Y)$ is well-defined, but $h(m_X \otimes m_Y)$ is of odd degree in $\onsf{TS}_A(X^* \oplus Y^*)$ so its divided powers are not defined. Hence only the restriction of $h$ to $\onsf{TS}_A(X^*)_{even} \otimes \onsf{TS}_A(Y^*)_{even}$ preserves the PD structure.}
\end{proof}

\subsection{Hopf algebra structures} 

\delete{
which is a coproduct of graded strictly  commutative algebra with PD structures. It is also a graded commutative Hopf algebra with PD structure. However it is not a PD Hopf algebra since comultiplication is not a Chain map in general.

It follows the same argument as in Bourbaki, $\onsf{TS}_A(-): \onsf{Cplx}(R)^{free}\to \onsf{scDGA}_{\bb N}(R)$ is  functorial. 
If $ \phi: X\to Y$ is a chain map, then $ \onsf{TS}_A(\phi): \onsf{TS}_A(X)\to \onsf{TS}_A(Y)$ is a homomorphism of DGA.
}

Consider $X^* \in \onsf{Cplx}(A)^{free}$. By using the isomorphism $h$ in Theorem \ref{thm_pddgfunctor}, the maps $+: X^*\oplus X^*\to X^* $, and $\delta: X^*\to X^*\oplus X^*$ with $\delta(x)=(x, x)$ induce homomorphisms $ \mu=\onsf{TS}_A(+) \circ h: \onsf{TS}_A(X^*)\otimes \onsf{TS}_A(X^*)\to \onsf{TS}_A(X^*) $ and $\Delta=h^{-1} \circ \onsf{TS}_A(\delta) \ :\onsf{TS}_A(X^*)\to \onsf{TS}_A(X^*)\otimes \onsf{TS}_A(X^*) $ in $ \onsf{PDdgAlg}(A)$. In fact, $\mu$ is exactly the $\star$-product in $\onsf{TS}_A(X^*)$. Indeed, for $x,x' \in \onsf{TS}_A(X^*)$, we have
\[
\begin{array}{rcl}
\mu(x \otimes x') & = & \onsf{TS}_A(+) \circ h(x \otimes x') \\
& = & \onsf{TS}_A(+)\big(\onsf{TS}_A(\iota_1)(x) \star \onsf{TS}_A(\iota_2)(x') \big) \\
 & = & \onsf{TS}_A(+)(\onsf{TS}_A(\iota_1)(x)) \star \onsf{TS}_A(+)(\onsf{TS}_A(\iota_2)(x')) \text{ as } \onsf{TS}_A(+) \text{ is a morphism} \\
  & & \text{of algebras} \\
  & = & x \star x' \text{ as } +\! \circ \iota_i=\on{id}_X \text{ for } i=1,2.
 \end{array}
 \]
As $\onsf{TS}_A(\{0\})=A$, the maps $ \{0\}\to X^*$ and $ X^*\to \{0\}$ define morphisms $u:A \to \onsf{TS}_A(X^*)$ and $\varepsilon: \onsf{TS}_A(X^*)\to A$, making $\onsf{TS}_A(X^*)$ into a commutative and cocommutative bialgebra with $u$ and $\epsilon$ the unit and counit respectively (see \cite[Chap. IV, \S 5.7]{Bourbaki_algebra}). In fact, the map $ \onsf{neg}: X^*\to X^*$ sending $ \onsf{neg}(x)=-x$ defines the antipode $S: \onsf{TS}_A(X^*)\to \onsf{TS}_A(X^*)$ making  $\onsf{TS}_A(X^*)$ into a connected commutative and cocommutative Hopf algebra. 
 The antipode property follows easily from the functoriality of $\onsf{TS}_A(-)$ and the following commutative diagrams
 \[
 \xymatrix{
 X^*\oplus X^*\oplus X^* \ar[rr]^{(+, \onsf{Id})}\ar[d]_{( \onsf{Id}, +)}&&X^*\oplus X^*\ar[d]^{+} \\
 X^*\oplus X^* \ar[rr]_{+}&& X^*}
 \quad \quad \xymatrix{
 X^*\oplus X^*\oplus X^* &&X^*\oplus X^*\ar[ll]_{(\onsf{Id},\delta)} \\
 X^*\oplus X^* \ar[u]^{(\delta, \onsf{Id})}&& \ar[ll]^{\delta}X^*\ar[u]_{\delta}}
 \]
 \[
 \xymatrix{
 X^*\ar[dr]^{\varepsilon}\ar[rr]^{\delta}\ar[d]_{\delta}&& X^*\oplus X^* \ar[d]^{(\onsf{Id}, \onsf{neg})}\\
 X^*\oplus X^*\ar[d]_{( \onsf{neg},\onsf{Id})}&\{0\}\ar[dr]^{u} &X^*\oplus X^*\ar[d]^{+}\\
 X^*\oplus X^* \ar[rr]_{+}&& X^*
 }.\]
 \delete{
 Here $ \onsf{TS}_A(\{0\})=A$, $ \varepsilon$ and $u$ correspond to counit and unit of the dg Hopf algebra $ \onsf{TS}_A(X^*)$.
}

\begin{theorem}\label{thm:Hopf}
For any $X^* \in \onsf{Cplx}(A)^{free}$, then $ \onsf{TS}_A(X^*)$ is a cocommutative dg Hopf algebra with a PD structure.
\end{theorem}

\subsection{Freeness of PD dg algebra \texorpdfstring{$\onsf{TS}_A(X^*)$}{Lg}} 
We now show that the functor $ \onsf{TS}_A(-): \onsf{Cplx}(A) \to \onsf{PDdgAlg}(A)$ is the left adjoint of the forgetful functor $ \onsf{PDdgAlg}(A)\to \onsf{Cplx}(A)$ in the sense of the following theorem:

\begin{theorem}\label{thm:freeness_TS}
Set $X^* \in \onsf{Cplx}(A)^{free}$ and $(A^*,I,\gamma) \in \onsf{PDdgAlg}(A)$.  Let $ \rho: X^*\to I$ be a morphism in $ \onsf{Cplx}(A)$, i.e., $\rho$ is a chain map. Then there is a unique 
 PD dg algebra homomorphism $ \tilde\rho: \onsf{TS}_A(X^*)\to A^*$ extending $ \rho$.  
\end{theorem}
  
\begin{proof}
 We note that $\onsf{TS}_A^1(X^*)=X^*$ is a subcomplex of $ \onsf{TS}_A(X^*)$. We fix a basis $B$ of $ X^* $ and consider the PBW basis of $\onsf{TS}_A(X^*)$ as done in Eq.\eqref{PBW_TS}.
 We now define $ \tilde\rho$ as follows: for $ f \in \bb N_s^{(B)}$
 \begin{align}
\tilde \rho (x^{(f)})=\prod_{b\in B}\gamma_{f(b)}(\rho(b)).
\end{align}
The application $\tilde\rho$ is clearly a graded PD algebra homomorphism. We only need to show that $ \tilde\rho$ is a chain map. This is clear using the fact that $ d^A$ satisfies the Leibniz rule and Eq.\eqref{eq:differential}.
The uniqueness is obvious.
\end{proof}

We remark that $\onsf{TS}_A(X^*)$ is not free as strictly commutative dg algebra in general unless $\bb Q\subseteq A$. In there are definitions of semi-free resolutions in the literature, in which each step is assumed to be free commutative algebra. In that sense, the resolution we are going to construct will not be semi-free. However, one could understand that in the PD dg sense, the construction we will get is indeed semi-free. 

\subsection{Lifting properties}

\delete{
Let $A^*$ be an object in $ \onsf{scDGA}(A)$ and $ I\subseteq A^*$ be a dg ideal. Let $ \tilde A^*$ be a PD dg algebra in $ \onsf{scDGA}(A)$ with dg ideal $ \tilde I\subseteq \tilde A^*$ and PD structure $ \gamma_*: \tilde I_{ev}\to \tilde I_{ev}$. Assume that $ \pi: \tilde A^*\to A^*$ is a surjective dg algebra homomorphism and a quasi-isomorphism with $ \pi^{-1}(I)=\tilde I$. Set $X^* \in \onsf{Cplx}(A)^{free}$. 
Then for any homomorphism $\rho: X^*\to I$ of complexes, which extends to a dg algebra homomorphism $ \hat{\rho}: \onsf{TS}_A(X^*)\to A^*$, there is a homomorphism of PD dg algebras $\tilde \phi: \onsf{TS}_A(X^*)\to \tilde A^*$ such that $\hat\rho=\pi\circ \tilde \phi$ as a dg algebra homomorphism. 

{\color{red}Problem: $\rho=\pi\circ \tilde \rho$ makes no sense because the domains are not compatible. We would need something like $f: \onsf{TS}_A(X^*)\to A^*$ and then $f=\pi \circ \tilde\rho$. But $A^*$ has no PD structure so where would $f$ come from?}

\begin{proof}
 Assuming that $X^*$ is free ($K$-projective) and is bounded above (i.e., $X^n=0$ for $n>>0$.)
 Note $ \pi$ is quasi-isomorphism, and $ \pi(\tilde I)=I. $ We note that $ H^*(\pi): H^*(\tilde A)\to H^*(A)$ is an isomorphism and the maps  $ H^*(\tilde I)\to H^*(\tilde A)$ and $ H^*( I)\to H^*( A)$ have the image isomorphic under $H^*(\pi)$.  Thus $ \pi: \tilde I\to I$ is quasi-isomorphism. The $K$-projectivity implies that there is chain map $ \phi: X^*\to \tilde I$ such that $ \rho=\pi^*\phi$. Let $ \tilde{\phi}: \onsf{TS}_A(X^*)\to \tilde A$ be the PD dg algebra homomorphism. Then $ \hat \rho=\pi\circ \tilde \phi$, which is a dg algebra homomorphism. 
\end{proof}
}

\begin{proposition}
Let $A'^*, A^*$ be two PD dg algebras with PD dg ideals $I', I$, and let $\pi: A'^* \to A^*$ be a morphism of PD dg algebras with the restriction $\pi: I' \to I$ being surjective. Then for any $X^* \in \onsf{Cplx}(A)^{free}$ and any homomorphism $\rho: X^*\to I$ of complexes, there is a homomorphism of PD dg algebras $\tilde \rho': \onsf{TS}_A(X^*)\to  A'^*$ such that $\tilde\rho=\pi\circ \tilde \rho'$. 
\end{proposition}

\begin{proof}
We have a chain map $ \rho:X^*\to  I\subseteq A^*$ and a surjection $\pi: I' \to I$. As $X^*$ is assumed to be free, there exists a lift $\rho':X^* \to I'$ such that $\rho=\pi \circ \rho'$. We then apply Theorem \ref{thm:freeness_TS} to $\rho'$ and $\rho$ to construct their respective lifts $\tilde\rho'$ and $\tilde\rho$. We obtain the following diagram:
\[
\begin{tikzpicture}[scale=0.9,  transform shape]
  \tikzset{>=stealth}
  
\node (1) at (0,4){$X^*$};
\node (2) at (0,2){$\onsf{TS}_A(X^*)$};
\node (3) at (-2,0){$A'^*$};
\node (4) at (2,0) {$A^*$};

\draw [decoration={markings,mark=at position 1 with
    {\arrow[scale=1.2,>=stealth]{>}}},postaction={decorate}] (1) --  (2) node[midway, right] {$\iota$};
\draw [decoration={markings,mark=at position 1 with
    {\arrow[scale=1.2,>=stealth]{>}}},postaction={decorate}] (1)  --  (3) node[midway, left] {$\rho'$};
\draw [decoration={markings,mark=at position 1 with
    {\arrow[scale=1.2,>=stealth]{>}}},postaction={decorate}] (1)  --  (4) node[midway, right] {$\rho$};
\draw [decoration={markings,mark=at position 1 with
    {\arrow[scale=1.2,>=stealth]{>}}},postaction={decorate}] (3)  --  (4) node[midway, above] {$\pi$};
\draw [decoration={markings,mark=at position 1 with
    {\arrow[scale=1.2,>=stealth]{>}}},postaction={decorate}] (2)  --  (3) node[midway, right] {$\tilde\rho'$};
\draw [decoration={markings,mark=at position 1 with
    {\arrow[scale=1.2,>=stealth]{>}}},postaction={decorate}] (2)  --  (4) node[midway, left] {$\tilde\rho$};
\end{tikzpicture}
\]
It follows that $(\pi \circ \tilde\rho'-\tilde\rho)\circ \iota=0$. As $\iota$ is an injection, it implies that $\pi \circ \tilde\rho'=\tilde\rho$ on the image of $\iota$, with is $\onsf{TS}_A^1(X^*)$. But as all morphisms involved are PD dg algebras homomorphisms, the equality remains valid on the PD dg subalgebra of $\onsf{TS}_A(X^*)$ generated by $\onsf{TS}_A^1(X^*)$, which is equal to $\onsf{TS}_A(X^*)$ by the PBW basis of Eq.\eqref{PBW_TS}.
\end{proof}

\delete{
\subsection{Base Change} Assume that $ \phi: A\to A'$ is a flat extension of commutative rings. Then the functor $ -\otimes_{A} A': A\Mod\to A'\Mod$
is an exact tensor functor {\color{red}(why do we need that?){\bf ZL: We don't need this. 
The isomorphism should be as a PD dg algebras. We have used this before in the proving the PD structure $\gamma_{n}(x\star y)=0$ for via first stating over $\bb Z \mc V$. 
Maybe this should be moved to earlier. So we don't need to use Roby argument, when $ X^*$ is assumed to be free. 
But when we are dealing with lift differentials, we do need exactness to make things simpler, otherwise, we have to use derived version.}} satisfying $ (X^*\otimes_{A} Y^*)\otimes_{A}A'=(X^*\otimes_{A}A')\otimes_{A'}(Y^*\otimes_{A}A')$, which induces an exact tensor functor $\onsf{Cplx}(A) \to \onsf{Cplx}(A') $.
\begin{theorem}
    If $ A\to A'$ is a flat homomorphism of commutative rings, then \\
    $ \onsf{TS}_{A'}(X^*\otimes_{A}A')=\onsf{TS}_{A}(X^*)\otimes_{A} A'$ {\color{red} AS PD DG ALGEBRAS. For the PD structure, we use the tensor product and a trivial PD structure on $A'$. For the dg structure, $A'$ is concentrated in degree $0$}.
\end{theorem}

\begin{proof}
There is an isomorphism $\psi:T_A^*(X^* \otimes_{A} A') \to T_A^*(X^*) \otimes_{A} A'$ given for any $n \in \mathbb N$ by
\[
\begin{array}{ccc}
T_A^n(X^* \otimes_{A} A') & \stackrel{\psi}{\to} & T_A^n(X^*) \otimes_{A} A' \\
(x_1 \otimes_{A} k'_1) \otimes_{A'} \cdots \otimes_{A'} (x_n \otimes_{A} k'_n) & \mapsto & (x_1 \otimes_{A} \cdots \otimes_{A} x_n) \otimes_{A} \prod_{i=1}^n k'_i
\end{array}
\]
For any $\sigma \in \mf S_n$, we have
\begin{align*}
& \psi \circ (\sigma.\left((x_1 \otimes_{A} k'_1) \otimes_{A'} \cdots \otimes_{A'} (x_n \otimes_{A} k'_n) \right) \\
&=(-1)^{l(\sigma; (x_1 \otimes k'_1,\dots,x_n \otimes k'_n))}\psi \left((x_{\sigma^{-1}(1)} \otimes_{A} k'_{\sigma^{-1}(1)}) \otimes_{A'} \cdots \otimes_{A'} (x_{\sigma^{-1}(n)} \otimes_{A} k'_{\sigma^{-1}(n)})\right) \\
&=(-1)^{l(\sigma; (x_1,\dots,x_n))}(x_{\sigma^{-1}(1)} \otimes_{A} \cdots \otimes_{A}x_{\sigma^{-1}(n)}) \otimes_{A} \prod_{i=1}^n k'_{\sigma^{-1}(i)} \text{ as }A' \text{ is concentrated in degree }0 \\
&=\sigma.(x_1 \otimes_{A} \cdots \otimes_{A}x_n) \otimes_{A} \prod_{i=1}^n k'_{i} \text{ as }A' \text{ is commutative} \\
&= (\sigma \otimes_{A} \on{id}_{A'}) \circ \psi\left((x_1 \otimes_{A} k'_1) \otimes_{A'} \cdots \otimes_{A'} (x_n \otimes_{A} k'_n)\right)
\end{align*}
Hence if $m \in \onsf{TS}_A^n(X^* \otimes_{A} A')$, then $\psi(m)=\psi \circ \sigma (m)=(\sigma \otimes_{A} \on{id}_{A'}) \circ \psi(m)$. It follows that $\psi(m) \in \onsf{TS}_A^n(X^*) \otimes A'$, and the restriction of $\psi$ gives an injective morphism
\[
\onsf{TS}_A(X^* \otimes_{A} A') \to \onsf{TS}_A(X^*) \otimes_{A} A'.
\]
We know that $\psi$ is an isomorphism, hence for any $m \in \onsf{TS}_A(X^*)$ and $k' \in A'$, we have
\[
\psi^{-1}(m \otimes_{A}k')  =  \psi^{-1} \circ (\sigma \otimes_{A} \on{id}_{A'})(m \otimes k') 
  =  \sigma \circ \psi^{-1}(m \otimes_{A}k').
\]
Hence $\psi^{-1}(m \otimes_{A}k') \in \onsf{TS}_A(X^* \otimes_{A} A')$. Therefore the above injection is also a surjection, and thus an isomorphism.
\end{proof}
}

\subsection{PD dg algebra extensions}\label{sec:pd_dg_algebra}

Let $F=V[-p]$ be a complex concentrated in degree $p$ ($p \leq 0$) for a free $A$-module $V$. Then we have $\Hom_{\onsf{Cplx}}(F, X^*)=\Hom_{A\Mod}(V, Z^p(X^*))$. Let $(A^*,d^A)$ be an object in $\onsf{scDGA}(A)$ and $ \phi: F\to A^*$ be a chain map. We define a map $d^\phi: \onsf{TS}_A(F[1])\otimes A^*\to \onsf{TS}_A(F[1])\otimes A^*$ as follows:
\begin{enumerate}
    \item $d^\phi(x \otimes 1)=1 \otimes \phi(x)$ for all $x \in \onsf{TS}_A^1(F[1])$;
    \item $d^\phi(1 \otimes a)=1 \otimes d^A(a)$ for all $a \in A^*$;
    \item $d^\phi(x^{(f)} \otimes a)=d(x^{(f)})\otimes a + (-1)^{|x^{(f)}|}x^{(f)} \otimes d^A(a)$ for all $x^{(f)} \in \onsf{TS}_A(F[1])$ and $a \in A^*$;
\end{enumerate}
where
\begin{align*}
d(x^{(f)}) \otimes a=\sum_{b\in B}\left ( \prod^\star_{b'<b}(-1)^{f(b')|b'|} {b'}^{(f(b'))}\right ) \star b^{(f(b)-1)} \star \left (\prod^\star_{b'>b} {b'}^{(f(b'))}\right) \otimes \phi(b)a .
\end{align*}
Then $d^\phi$ is derivation extending $d^A$ such that $\onsf{TS}_A(F[1])\otimes A^*$ is a dg algebra with derivation $d^\phi$ and $ A^*\to \onsf{TS}_A(F[1])\otimes A^*$ is a homomorphism of dg algebras.

\begin{theorem}\label{thm:Tate_one_step}
Assume that $A^*$ is a PD dg cocommutative Hopf $A$-algebra with $A^{>0}=\{0\}$ and $A^0=A$. Then the pair $(\onsf{TS}_A(F[1])\otimes A^*, d^\phi)$ is PD dg cocommutative Hopf $A$-algebra, and it satisfies $Z^{p}(\onsf{TS}_A(F[1])\otimes A^*)=Z^{p}(A^*)$ and $B^{p}(\onsf{TS}_A(F[1])\otimes A^*)=B^{p}(A^*)+\phi(V)$. In particular, $H^n(\onsf{TS}_A(F[1])\otimes A^*)=H^n(A^*)$ for all $ n>p$ and $H^p(\onsf{TS}_A(F[1])\otimes A^*)=H^p(A^*)/\overline{\phi(V)}$. Finally, the embedding $A^* \to \onsf{TS}_A(F[1])\otimes A^*$ is a PD dg Hopf algebra homomorphism.
\end{theorem}

\begin{remark}
The embedding $\onsf{TS}_A(F[1]) \to \onsf{TS}_A(F[1])\otimes A^*$ is PD Hopf algebra homomorphism but not a dg homomorphism due to the way $d^\phi$ is defined.
\end{remark}

\delete{
{\color{red} We need $\onsf{TS}_A(F[1])$ to be only in negative degrees and $A^{>0}=\{0\}$ so that in $\onsf{TS}_A(F[1])\otimes A^*$, the part with $n>p$ is not impacted by the tensor product, and so the cohomology in unchanged. Hence we need $p-1 < 0$. If $A^*$ had some positive part, then for example: $p=-1$, $F[1]$ concentrated in degree $-2$. So $F[1] \otimes A^2$ is in degree $0>p$, hence the complex $\onsf{TS}_A(F[1])\otimes A^*$ in degree $0$ is different from $A^*$, and their cohomology will likely be different too.

We need $A^0=A$ because in $B^p$ we get $1 \otimes \phi(V)A^0$.
negatively graded connected}
}

\begin{proof}
We have seen that $\onsf{TS}_A(F[1])\otimes A^*$ is a PD algebra (cf. Theorem \ref{thm:TS_pd} and Section \ref{sec:pd_rings}). We also know from Theorem \ref{thm:Hopf} that $\onsf{TS}_A(F[1])$ has a cocommutative Hopf algebra structure. It follows that $\onsf{TS}_A(F[1])\otimes A^*$ has a natural cocommutative Hopf algebra structure. Then we also see that $(\onsf{TS}_A(F[1])\otimes A^*, d^\phi)$ is a dg algebra because of the dg structure on the tensor product and the definition of $d^\phi$. It remains to check Axiom \eqref{axiom_8}.

Let $x \in \onsf{TS}_A^+(F[1])_{ev}$ and $a \in A_{ev}$. Then
\[
\begin{array}{rcl}
d^\phi(\gamma_n(x \otimes a)) &=&\displaystyle d^\phi(\gamma_n(x) \otimes a^n)  \\
&=&\displaystyle d(\gamma_n(x)) \otimes a^n+\gamma_n(x) \otimes d^A(a^n) \\
&=&\displaystyle d(x) \gamma_{n-1}(x) \otimes a^n+n\gamma_n(x) \otimes d^A(a)a^{n-1}.
\end{array}
\]
On the other hand,
\[
\begin{array}{rcl}
d^\phi(x \otimes a)\gamma_{n-1}(x \otimes a) &=&\left(d(x) \otimes a+x \otimes d^A(a)\right) \left(\gamma_{n-1}(x) \otimes a^{n-1}\right) \\
&=&d(x)\gamma_{n-1}(x) \otimes a^n+x\gamma_{n-1}(x) \otimes d^A(a)a^{n-1}.
\end{array}
\]

We know that $\onsf{TS}_A(F[1])$ is a PD algebra, and so it satisfies Axioms \eqref{axiom_1} and \eqref{axiom_2}. But then for any element in $x \in \onsf{TS}_A^+(F[1])_{ev}$, we have
\begin{align}\label{eq:new_equation}
x \gamma_{n-1}(x)=\gamma_1(x)\gamma_{n-1}(x)=\binom{n}{1}\gamma_n(x)=n\gamma_n(x).
\end{align}
\delete{
Let $\mathcal{F}$ be the free graded monoid constructed on $F[1]$. We know from Eq.\eqref{eq:free_power} that $z^{\star n}=n! \gamma_n(z)$ for any $z \in \mathbb Z_S^+(\mathcal F)_{ev}$. But then we have $z^{\star n}=(n-1)!z \gamma_{n-1}(z)=(n-1)!n\gamma_n(z)$. It follows that
\[
(n-1)!(z \gamma_{n-1}(z)-n\gamma_n(z))=0
\]
in $\mathbb Z_S(\mathcal F)$. However, $\mathbb Z_S(\mathcal F)$ is torsion free over $\mathbb Z$, which implies
\begin{align}\label{eq:new_equation}
z \gamma_{n-1}(z)=n\gamma_n(z).
\end{align}
Moreover, we know from the proof of Theorem \ref{thm:TS_pd} that the PD structure of $\mathbb Z_S(\mathcal F)$ passes to $\onsf{TS}_A(F[1])$. Hence Eq.\eqref{eq:new_equation} is also valid for $z \in \onsf{TS}_A^+(F[1])_{ev}$.}
It follows that $d^\phi(\gamma_n(x \otimes a)) =d^\phi(x \otimes a)\gamma_{n-1}(x \otimes a)$.

Now let $x \in \onsf{TS}_A(F[1])_{ev}$ and $a \in I_{ev}$. Then
\[
\begin{array}{rcl}
d^\phi(\gamma_n(x \otimes a)) &=&\displaystyle d^\phi(x^{\star n} \otimes \gamma_n(a))  \\
&=&\displaystyle d(x^{\star n}) \otimes \gamma_n(a)+x^{\star n} \otimes d^A(\gamma_n(a)) \\
&=&\displaystyle d(x)x^{\star n-1} \otimes n \gamma_n(a)+x^{\star n} \otimes d^A(a)\gamma_{n-1}(a).
\end{array}
\]
On the other hand,
\[
\begin{array}{rcl}
d^\phi(x \otimes a)\gamma_{n-1}(x \otimes a) &=&\left(d(x) \otimes a+x \otimes d^A(a)\right) \left(x^{\star n-1} \otimes \gamma_{n-1}(a)\right) \\
&=&d(x)x^{\star n-1} \otimes a \gamma_{n-1}(a)+x^{\star n} \otimes d^A(a)\gamma_{n-1}(a).
\end{array}
\]
As $A^*$ is a PD dg algebra, Eq.\eqref{eq:new_equation} is also valid in $I_{ev}$. It follows that $d^\phi(\gamma_n(x \otimes a)) =d^\phi(x \otimes a)\gamma_{n-1}(x \otimes a)$.

If $x$ and $a$ are odd and $n \geq 3$, then Axiom \eqref{axiom_8} is obvious.
If $n=2$, we have $d^\phi(\gamma_2(x \otimes a))=0$ and
\[d^\phi(x \otimes a)\gamma_{1}(x \otimes a)=d(x)x \otimes a^2+x^{\star 2} \otimes d^A(a)a=0
\]
because of the strict graded commutativity of both $A^*$ and $\onsf{TS}_A(F[1])$.

Finally, using Axiom \eqref{axiom_4} and the above computations, we show that Axiom \eqref{axiom_8} is satisfied for any element in $\onsf{TS}_A^+(F[1])_{ev} \otimes A_{ev} +\onsf{TS}_A(F[1])_{ev} \otimes I_{ev}$. Therefore $(\onsf{TS}_A(F[1])\otimes A^*, d^\phi)$ is PD dg algebra.

We know that $F[1]$ is concentrated in degree $p-1 < 0$ and $A^*$ only has non-positive degrees. Hence $\onsf{TS}_A(F[1])\otimes A^*$ only has non-positive degrees. It follows that for any $n \leq 0$,  
\[
\begin{array}{rcl}
(\onsf{TS}_A(F[1])\otimes A^*)^n & = &\displaystyle \sum_{n_1+n_2=n}\onsf{TS}_A(F[1])^{n_1}\otimes A^{n_2} \\
& = &\displaystyle \sum_{n_2=n}^0\onsf{TS}_A(F[1])^{n-n_2}\otimes A^{n_2}
\end{array}
\]
But we also know that the homogeneous components of $\onsf{TS}_A(F[1])$ are in degree $k(p-1)$ with $k \in \bb N$. Hence if $n \geq p$, then $0 \geq n-n_2 \geq n \geq p> k(p-1)$ for all $k \geq 1$ as $p-1<0$. It follows that if $n \geq p$, then
\[
(\onsf{TS}_A(F[1])\otimes A^*)^n = A^{n},
\]
implying that for all $n>p$, we have
\[
H^n(\onsf{TS}_A(F[1])\otimes A^*)=H^n(A^*).
\]

As $(\onsf{TS}_A(F[1])\otimes A^*)^p = A^{p}$, it is natural that $Z^{p}(\onsf{TS}_A(F[1])\otimes A^*)=Z^{p}(A^*)$. We have $(\onsf{TS}_A(F[1])\otimes A^*)^{p-1}=F[1] \otimes A + A \otimes A^{p-1}$. If $x \in F[1]$ and $a \in A^{p-1}$, then $d^\phi(x \otimes 1)=1 \otimes \phi(x)$ and $d^\phi(1 \otimes a)=1 \otimes d^A(a)$. Hence $B^{p}(\onsf{TS}_A(F[1])\otimes A^*)=B^{p}(A^*)+\phi(V)$. It then follows that $H^p(\onsf{TS}_A(F[1])\otimes A^*)=H^p(A^*)/\overline{\phi(V)}$ where $\overline{\phi(V)}$ is the image of $\phi(V)$ modulo $B^p(A^*)$.

The last statement is a direct consequence of the definition of $d^\phi$ and the definition of the PD structure on the tensor product (see Section \ref{sec:pd_rings}).
\end{proof}

\begin{theorem}\label{thm:Tate}
Let $A$ be any commutative ring and $I \subset A$ an ideal. There exists a PD dg  algebra with a cocommutative Hopf structure $P^*=\bigoplus_{i=0}^\infty P^{-i}$ satisfying the following conditions:
\begin{itemize}
 \item $\mc I=\bigoplus_{i>0}P^{-i}$ is the PD dg ideal;
 \item $P^*$ is K-projective, i.e., $\mc Hom_A^* (P^*,-)$ sends acyclic complexes to acyclic complexes (see \cite{Spaltenstein});
 \item There is a quasi-isomorphism $P^* \to A/I$ of PD dg $A$-algebras, where $A/I$ is seen as a PD dg $A$-algebra with PD dg ideal $\{0\}$.
\end{itemize}
\end{theorem}

\begin{proof}
We set $P_0=A$ concentrated in degree $0$, $F_1$ free $A$-module with surjective $A$-module homomorphism $F_1 \to I$. Thus we have $\phi_1:F_1 \to P_0$ is a chain map. Then define $P_1=\onsf{TS}_A(F_1[1])\otimes_A P_0$ using Theorem \ref{thm:Tate_one_step}. It is a PD dg cocommutative Hopf algebra with PD dg ideal $\mathcal I_1=\onsf{TS}_A^+(F_1[1])\otimes_A P_0=P_1^-$ with the properties:  $H^n(P_1)=H^n(P_0)$ for all $ n>0$ and $H^0(P_1)=H^0(P_0)/\overline{\phi_1(F_1)}=A/I$. 

Assume that $P_n$ and $\mathcal I_n$ have been constructed, with the property: $P_n \to A/I$ such that
\[
H^{-i}(P_n)=\begin{cases}
    A/I & \text{ if }i=0, \\
    0 & \text{ if } 0<i<n.
\end{cases}
\]
and
\[
\mathcal I_n=P_n^-.
\]
Then let $F_{n+1}$ be a free $A$-module with surjective $A$-module homomorphism $F_{n+1} \to Z^{-n}(P_n)$. Thus $\phi_{n+1}:F_{n+1}[n+1] \to P_n$ is a chain map. Then define
\[
P_{n+1}=\onsf{TS}_A(F_{n+1}[n+1])\otimes_A P_n
\]
using Theorem \ref{thm:Tate_one_step}. We have
\[
H^{-i}(P_{n+1})=\begin{cases}
    A/I & \text{ if }i=0, \\
    0 & \text{ if } 0<i<n+1.
\end{cases}
\]
and
\[
\mathcal I_{n+1}=\onsf{TS}_A^+(F_{n+1}[n+1])\otimes_R P_n+\onsf{TS}_A(F_{n+1}[n+1]) \otimes_R P_n^-=P_{n+1}^-.
\]
There is a PD dg algebra homomorphism $P_n \to P_{n+1}$ and it commutes with the maps $P_n \to A/I$ and $P_{n+1} \to A/I$.

Finally, take $P^*=\varinjlim_{n}P_n$ and $\mathcal I=\varinjlim_{n}\mathcal I_n=P^-$. We get a quasi-isomorphism $P^* \to A/I$. Every cohomological degree of $P^*$ is projective and $P^*$ has no positive cohomological degree, so it is K-projective.
\end{proof}

\begin{remark}
Note that in Theorem \ref{thm:Tate}, the dg structure and the Hopf structure both exist, but they are not compatible with one another. More particularly, the comultiplication $\Delta:P^* \to P^* \otimes P^*$ is not a chain map. Indeed, based on the construction, we know that elements in $F_1$ are primitive. Consider $T \in F_1$ with $dT=\phi(T)=a \in A$. It follows that 
\[
d\Delta(T)=d(T \otimes 1 +1 \otimes T)=a \otimes 1+1 \otimes a=2a(1 \otimes 1)=2\Delta(a)=2\Delta(d(T))\]
So $\Delta$ will never be a chain map.
\end{remark}

\begin{theorem}\label{thm:3.15}
Let $A$ be any $\msf k$-algebra, $I \subset A$ an ideal, $Q^*$ a PD dg algebra with PD ideal $ J_{Q}^*=\oplus_{n>0}Q^{-n}\subseteq Q^*$ and satisfying
\begin{itemize}
    \item [(i)] $Q^n=0$ for all $ n>0$;
    \item [(ii)]    $p^*: Q^*\to A/I$ is homomorphism of PD dg $A$-algebras and a quasi-isomorphism as complex of $A$-modules (in particular, $p^*$ is componentwise surjective).
\end{itemize}
Let $P^*$ be the PD dg algebra constructed in Theorem \ref{thm:Tate}. Then there exists a PD dg algebra homomorphism $\theta:P^* \to Q^*$ making the following triangle commute:
\[
\begin{tikzpicture}[scale=1,  transform shape]
  \tikzset{>=stealth}
  
\node (1) at (0,0){$P^*$};
\node (2) at (3,0){$Q^*$};
\node (3) at (1.5,-1.5){$A/I$};

\draw [decoration={markings,mark=at position 1 with
    {\arrow[scale=1.2,>=stealth]{>}}},postaction={decorate}] (1) --  (2) node[above] {};
\draw [decoration={markings,mark=at position 1 with
    {\arrow[scale=1.2,>=stealth]{>}}},postaction={decorate}] (2)  --  (3) node[midway, left] {};
\draw [decoration={markings,mark=at position 1 with
    {\arrow[scale=1.2,>=stealth]{>}}},postaction={decorate}] (1)  --  (3) node[midway, right] {};
\end{tikzpicture}
\]
\end{theorem}

\begin{proof} We will proceed by induction using the PD dg algebras $P^*_n$ constructed in the proof of Theorem \ref{thm:Tate} by constructing PD dg algebra homomorphisms $ \theta_{n}: P_n^*\to Q^*$ making the following diagram of PD dg $A$-algebras \[
\begin{tikzpicture}[scale=1,  transform shape]
  \tikzset{>=stealth}
  
\node (1) at (0,0){$P_n^*$};
\node (2) at (3,0){$Q^*$};
\node (3) at (1.5,-1.5){$A/I$};

\draw [decoration={markings,mark=at position 1 with
    {\arrow[scale=1.2,>=stealth]{>}}},postaction={decorate}] (1) --  (2) node[midway, right] {};
\draw [decoration={markings,mark=at position 1 with
    {\arrow[scale=1.2,>=stealth]{>}}},postaction={decorate}] (2)  --  (3) node[midway, left] {};
\draw [decoration={markings,mark=at position 1 with
    {\arrow[scale=1.2,>=stealth]{>}}},postaction={decorate}] (1)  --  (3) node[midway, right] {};
\end{tikzpicture}
\]
commute and $ \theta_{n+1}|_{P_n^*}=\theta_n$.

For $ n=0$, we have $P_0^*=A$ and define $\theta_0: P^*_0 \to Q^*$ as the homomorphism of the $A$-algebra structure. Since $Q^{1}=0$, then $ \theta_0$ is a chain map, and it also preserves the algebra structure as it comes from the structure of $Q^*$. Hence $\theta_0$ is a dg algebra homomorphism.

We know that $\epsilon_Q:Q^0 \to A/I$ is an $A$-module homomorphism and we can thus check that $\epsilon_Q \circ \theta_0=\epsilon_Q$.

Let $\phi_1: F_1\to I$ be the surjective $A$-module homomorphism introduced in Theorem \ref{thm:Tate}. We thus have
\[
F_1 \overset{\phi_1}{\longrightarrow} I \overset{\theta_0}{\longrightarrow} h^*_0(I) \subset \on{Ker}(\epsilon_Q)=\on{Im}(d_Q^{-1})
\]
as $Q^*$ is an exact complex. We can thus lift $\theta_0(I)$ to an $A$-module $\widetilde{F_1} \subset Q^{-1}$. We naturally get an $A$-module homomorphism $\widetilde{
\phi_1}:F_1 \to \widetilde{F_1}$ through the lift of the image of a basis of the free $A$-module $F_1$. By shifting, it can be seen as an $A$-module homomorphism $F_1[1] \to Q^{-1}$. Then using Eq.\eqref{PBW_TS}, we obtain a PD dg algebra homomorphism $\psi_1:\onsf{TS}(F_1[1]) \to Q^*$, and finally we set
\[
\begin{array}{rccc}
   P_1^*=& \onsf{TS}(F_1[1]) \otimes_A P_0^*  & \overset{\theta_1}{\longrightarrow} & Q^* \\
    &x \otimes p & \longmapsto & \psi_1(x)\theta_0(p). 
\end{array}
\]
One can verify through explicit computations that $\theta_1$ is indeed a PD dg algebra homomorphism. Moreover, we have $\theta_{1}|_{P_0^*}=\theta_0$ so the triangle commutes.

Assume we have a PD dg algebra homomorphism $\theta_{n}:P^*_{n} \to Q^*$ making the triangle commute. We want to extend $\theta_{n}$ to $P_{n+1}=\onsf{TS}(F_{n+1}[n+1]) \otimes_A P_{n}$. As done for the first step, consider $\phi_{n+1}: F_{n+1}[n]\to Z^{-n}(P_n^*)$ as in Theorem \ref{thm:Tate}. We thus have
\[
F_{n+1}[n] \overset{\phi_{n+1}}{\longrightarrow} Z^{-n}(P_n^*) \overset{\theta_n}{\longrightarrow} \theta_n(Z^{-n}(P_n^*)) \subset \on{Ker}(d^{-n}_Q)=\on{Im}(d_Q^{-(n+1)}).
\]
As before, we then define an $A$-module homomorphism $F_{n+1}[n+1] \to Q^{-(n+1)}$, which then leads to a PD dg algebra homomorphism $\psi_{n+1}:\onsf{TS}(F_{n+1}[n+1]) \to Q^*$. We finally set
\[
\begin{array}{rccc}
   P_{n+1}^*=& \onsf{TS}(F_{n+1}[n+1]) \otimes_A P_n^*  & \overset{\theta_{n+1}}{\longrightarrow} & Q^* \\
    &x \otimes p & \longmapsto & \psi_{n+1}(x)\theta_n(p). 
\end{array}
\]
We verify explicitly that $\theta_{n+1}$ is a PD dg algebra homomorphism. As $P^*_{n+1}$ and $P^*_0$ differ only in degree strictly positive, the triangle
\[
\begin{tikzpicture}[scale=1,  transform shape]
  \tikzset{>=stealth}
  
\node (1) at (0,0){$P_{n+1}^*$};
\node (2) at (3,0){$Q^*$};
\node (3) at (1.5,-1.5){$A/I$};

\draw [decoration={markings,mark=at position 1 with
    {\arrow[scale=1.2,>=stealth]{>}}},postaction={decorate}] (1) --  (2) node[midway, right] {};
\draw [decoration={markings,mark=at position 1 with
    {\arrow[scale=1.2,>=stealth]{>}}},postaction={decorate}] (2)  --  (3) node[midway, left] {};
\draw [decoration={markings,mark=at position 1 with
    {\arrow[scale=1.2,>=stealth]{>}}},postaction={decorate}] (1)  --  (3) node[midway, right] {};
\end{tikzpicture}
\]
commutes. We conclude the proof by taking $P^*=\varinjlim_{n}P^*_n$.
\end{proof}

\begin{corollary}\label{cor:3.16} 
Let $P^*$ and $Q^*$ be  two resolutions of $A/I$ obtained using the construction of Theorem \ref{thm:Tate}. Let $\theta: P^* \to Q^*$ and $\tau:Q^* \to P^*$ be the two PD dg algebra homomorphisms constructed from Theorem \ref{thm:3.15}. Then as homomorphisms of $A$-modules, $\theta \circ \tau$ and $\tau \circ \theta$ are homotopic to the identity, i.e.,
\[
\tau \circ \theta \simeq \on{id}_{P^*} \quad \text{and} \quad \theta \circ \tau \simeq \on{id}_{Q^*}.
\]
\end{corollary}

\begin{proof}
 To fix notation, we write $\epsilon_P:P^0 \to A/I$ and $\epsilon_Q:Q^0 \to A/I$ for the augmentation maps. Using the commutation of the diagram in Theorem \ref{thm:3.15}, we see that $\epsilon_P \circ \tau^0 \circ \theta^0=\epsilon_Q \circ \theta^0=\epsilon_P$, and so we obtain the following commutative diagram
\[
\begin{tikzpicture}[scale=1,  transform shape]
  \tikzset{>=stealth}
  
\node (1) at (0,0){$P^0$};
\node (2) at (3,0){$P^0$};
\node (3) at (1.5,-1.5){$A/I$};

\draw [decoration={markings,mark=at position 1 with
    {\arrow[scale=1.2,>=stealth]{>}}},postaction={decorate}] (1) --  (2) node[above, midway] {$\tau^0 \circ \theta^0$};
\draw [decoration={markings,mark=at position 1 with
    {\arrow[scale=1.2,>=stealth]{>}}},postaction={decorate}] (2)  --  (3) node[midway, right] {$\epsilon_P$};
\draw [decoration={markings,mark=at position 1 with
    {\arrow[scale=1.2,>=stealth]{>}}},postaction={decorate}] (1)  --  (3) node[midway, left] {$\epsilon_P$};
\end{tikzpicture}
\]
We thus have a chain map $\tau \circ \theta :P^* \to P^*$ such that $\epsilon_P \circ (\tau \circ \theta)=\epsilon_P$. We can then set the chain map $h^*=\tau \circ \theta -\on{id}_{P^*}$ and we have $\epsilon_P \circ h^0=0$. Through a standard induction, one can verify that $h^*$ is null homotopic, and so $\tau \circ \theta \simeq \on{id}_{P^*}$. We prove similarly that $\theta \circ \tau \simeq \on{id}_{Q^*}$.
   
\end{proof}

Let $\theta: P^* \to Q^*$ and $\tau: Q^* \to P^*$ be PD dg $A$-algebra homomorphisms such that $\theta \circ \tau \simeq \on{id}_{Q^*}$ and $\tau \circ \theta \simeq \on{id}_{P^*}$. Define $\alpha: \mc Hom_A^*(P^*, P^*) \to \mc Hom_A^*(Q^*, Q^*)$ by $\alpha(D)=\theta \circ D \circ \tau$. We note that $ \alpha$ is chain map. But $\alpha$ does not preserve the composition multiplication unless $ \theta$ and $ \tau $ are inverse to each other on the chain level.  However, we have the following: 

\begin{corollary}\label{cor:3.17} $\alpha$  induces an $A$-algebra isomorphism
\[
H^*(\alpha):H^*(\mc Hom_A^*(P^*, P^*)) \to H^*(\mc Hom_A^*(Q^*, Q^*)).
\]
with inverse $H^*(\alpha')$ with $\alpha'(D')=\tau \circ D' \circ \theta$.
\end{corollary}

\begin{proof}
Set $P^*$, $Q^*$, $\alpha$ and $\alpha'$ as in the statement of the corollary. We know that $\psi \circ \varphi$ is a chain map so it is a cocycle of $\mc Hom_A^*(P^*, P^*)$. Then, we know that there exists $h:P^* \to P^*$ of degree $1$ such that $\tau \circ \theta =\on{id}_{P^*}+h \circ d_{P^*}+d_{P^*} \circ h=\on{id}_{P^*}+d_{\on{End}_{P^*}}(h)$. Hence $[\tau \circ \theta ]=[\on{id}_{P^*}]$, where $[\cdot]$ denote the class in $H^*(\mc Hom_A^*(P^*, P^*))$. Likewise, $[\theta \circ \tau ]=[\on{id}_{Q^*}]$.

We can see that $\alpha$ and $\alpha'$ are chain maps. Moreover, $H^*(\alpha)$ is an isomorphism of $A$-modules with inverse $H^*(\alpha')$. Indeed, for any $D \in Z^*(\mc Hom_A^*(P^*, P^*))$, we have $H^*(\alpha' \circ \alpha)([D])=[\tau \circ \theta \circ D \circ \tau \circ \theta ]=[\tau \circ \theta ] [D] [\tau \circ \theta ]=[D]$ as $[\tau \circ \theta ]=[\on{id}_{P^*}]$. Similarly, we obtain $H^*(\alpha \circ \alpha')([D'])=[D']$ for any $D' \in Z^*(\mc Hom_A^*(Q^*, Q^*))$.

It remains to verify that $H^*(\alpha)$ is an algebra morphism. Consider $D_1, D_2 \in Z^*(\mc Hom_A^*(P^*, P^*))$. Then we can verify that
\[
\alpha(D_1) \circ \alpha(D_2)-\alpha(D_1 \circ D_2)=(-1)^{|D_1]}d_{\mc Hom_A^*(P^*, P^*)}(\varphi \circ D_1 \circ h \circ D_2 \circ \psi).
\]
Hence $\alpha$ induces an algebra isomorphism $H^*(\alpha) : H^*(\mc Hom_A^*(P^*, P^*)) \to H^*(\mc Hom_A^*(Q^*, Q^*))$ with inverse $H^*(\alpha')$.
\end{proof}

\delete{\begin{remark}
    {\color{red} Model category structure TO DO}
\end{remark}
}

\delete{
We will need the following lemma section 4. 
\begin{lemma} \label{lem:homotopy-identity}
If $ (X^*, d_X)$ is a complex and $\sigma: (X^*, d_X)\to (X^*, d_X)$ is a chain map which is homotopy equivalent to identity, i.e., there is an $h\in \mc Hom^{-1}(X^*, X^*)$ such that $ \sigma=\on{id}_X+d_X\circ h+h\circ d_X$, then the chain map $ \sigma\circ (-):\mc Hom^{*}(X^*, X^*) \to \mc Hom^{*}(X^*, X^*)$ induces 
\[H^*(\sigma\circ (-))=\on{id}: H^*(\mc Hom^{*}(X^*, X^*))\to H^*(\mc Hom^{*}(X^*, X^*)).\]
\end{lemma}
\begin{proof}
    For $D\in Z^*(\mc Hom^{*}(X^*, X^*))$ homogeneous, we have $d_{X}\circ D=(-1)^{|D|}D\circ d_{X}$. Thus
    \[
    \sigma\circ D=D+(d_X\circ h+h\circ d_X)\circ D=D+d_X\circ(h\circ D)-(-1)^{|D|-1}(h\circ D)\circ d_X=D+d_{\mc Hom^*}(h\circ D).
    \]
    Hence $[\sigma\circ D]=[D]$ in $H^*(\mc Hom^{*}(X^*, X^*)) $.
\end{proof}
}

\section{Derivations of PD dg algebras}\label{sec:4}

Let $A$ be a commutative ring. A quick consequence of $P^*$ in Theorem \ref{thm:Tate} being K-projective is that the functor $\mc Hom_A^*(P^*, -)$ sends quasi-isomorphisms to quasi-isomorphisms. To see this, consider a quasi-isomorphism $f: M^*\to N^*$. We know that $f$ is a quasi-isomorphism if and only if $ \onsf{Cone}(f)$ is acyclic \cite[Cor.1.5.4]{Weibel}. It follows that $\mc Hom_A^*(P^*,\onsf{Cone}(f))$ is also acyclic. But one can show that $\mc Hom_A^*(P^*,\onsf{Cone}(f)) \cong \onsf{Cone}(\mc Hom_A^*(P^*,f))$ (\cite[p.271]{Kashiwara-Schapira}), implying that $\mc Hom_A^*(P^*,f)$ is a quasi-isomorphism (using  \cite[Cor.1.5.4]{Weibel} once more). Finally, it can be shown that any complex of projective $A$-modules bounded above is K-projective (cf. \cite[Lemma 13.31.4]{Stacks-Project}).

Set $I \subset A$ an ideal and consider the $A$-algebra $A/I$. In this section, we will rely on the $A$-algebra structure of $A/I$ in order to make use of Theorem \ref{thm:Tate}. Moreover, this structure naturally makes $A/I$ into an $A$-module, and we will be able to consider the Yoneda algebra $\on{Ext}_A^*(A/I,A/I)$. Let $ \epsilon: P^*\to A/I$ be an $A$-projective resolution. Given an element $\widetilde{\varphi} \in \mc Hom_A^{n}(P^*,P^*)$, we write $\widetilde{\varphi}_{-k} :P^{-k} \to P^{-k+n}$ for each homogeneous component. We say that $\widetilde{\varphi}$ is a cocycle if $d(\widetilde{\varphi})=0$, i.e, if $d \circ \widetilde{\varphi}_{-k-1} -(-1)^{n}\widetilde{\varphi}_{-k} \circ d=0$ for all $k$. Notice that $|d(\widetilde{\varphi})|=|\widetilde{\varphi}|+1$, hence $\mc Hom_A^{*}(P^*,P^*)$ is a cochain complex.

Set $\varphi:P^{-n} \to A/I$ a morphism of $A$-modules. We say that a morphism of complexes $\widetilde{\varphi}:P^* \to P^*$ is a lift of $\varphi$ if the following conditions are satisfied:
\begin{itemize}[nosep]
    \item $|\widetilde{\varphi}|=n$;
    \item $\widetilde{\varphi}$ is a cocycle;
    \item $\epsilon \circ \widetilde{\varphi}_{-n}=\varphi$.
\end{itemize}
This can be summarized in the following commutative diagram:

 \begin{center}
  \begin{tikzpicture}[scale=0.9,  transform shape]
  \tikzset{>=stealth}
  
\node (1) at ( 0,0){$P^{-n-2}$};
\node (2) at ( 3,0){$P^{-n-1}$};
\node (3) at ( 6,0){$P^{-n}$};
\node (4) at ( 0,-3) {$P^{-2}$};
\node (5) at ( 3,-3){$P^{-1}$};
\node (6) at ( 6,-3){$P^0$};
\node (7) at ( 9,-3){$A/I$};

\node (8) at ( -3,0){$\dots$};
\node (9) at ( -3,-3){$\dots$};

\draw [decoration={markings,mark=at position 1 with
    {\arrow[scale=1.2,>=stealth]{>}}},postaction={decorate}] (1) --  (2) node[midway, above] {$d$};
\draw [decoration={markings,mark=at position 1 with
    {\arrow[scale=1.2,>=stealth]{>}}},postaction={decorate}] (2) --  (3) node[midway, above] {$d$};
\draw [decoration={markings,mark=at position 1 with
    {\arrow[scale=1.2,>=stealth]{>}}},postaction={decorate}] (4) --  (5) node[midway, below] {$(-1)^nd$};
\draw [decoration={markings,mark=at position 1 with
    {\arrow[scale=1.2,>=stealth]{>}}},postaction={decorate}] (5) --  (6) node[midway, below] {$(-1)^nd$};
\draw [decoration={markings,mark=at position 1 with
    {\arrow[scale=1.2,>=stealth]{>}}},postaction={decorate}] (6) --  (7) node[midway, below] {$\epsilon$};

\draw [decoration={markings,mark=at position 1 with
    {\arrow[scale=1.2,>=stealth]{>}}},postaction={decorate}] (1) --  (4) node[midway, left] {$\widetilde{\varphi}_{-n-2}$};
\draw [decoration={markings,mark=at position 1 with
    {\arrow[scale=1.2,>=stealth]{>}}},postaction={decorate}] (2) --  (5) node[midway, left] {$\widetilde{\varphi}_{-n-1}$};
\draw [decoration={markings,mark=at position 1 with
    {\arrow[scale=1.2,>=stealth]{>}}},postaction={decorate}] (3) --  (6) node[midway, left] {$\widetilde{\varphi}_{-n}$};
\draw [decoration={markings,mark=at position 1 with
    {\arrow[scale=1.2,>=stealth]{>}}},postaction={decorate}] (3) --  (7) node[midway, right] {$\varphi$};

\draw [decoration={markings,mark=at position 1 with
    {\arrow[scale=1.2,>=stealth]{>}}},postaction={decorate}] (8) --  (1) node[midway, above] {};
\draw [decoration={markings,mark=at position 1 with
    {\arrow[scale=1.2,>=stealth]{>}}},postaction={decorate}] (9) --  (4) node[midway, above] {};
\end{tikzpicture}
 \end{center}

The Yoneda algebra $\on{Ext}_A^*(A/I,A/I)$ is equipped with two distinct products, which we differentiate here to avoid sign problems. To define them, take $\varphi:P^{-n} \to A/I$ and $\psi:P^{-m} \to A/I$ be representatives of elements in $\on{Ext}_A^*(A/I,A/I)$. Using the lifts defined above, the composition product is given as $\psi \circ \varphi=\epsilon \circ \widetilde{\psi} \circ \widetilde{\varphi}$. To compute the other product, called Yoneda product (see \cite{CTVZ}), one needs a different lift from the one defined above. For a map $\varphi:P^{-n} \to A/I$, we define $\varphi_{-n-i} :P_{-n-i} \to P_{-i}$ ($i \geq 0$) as we did $\widetilde{\varphi}_{-n-i}$ but the bottom row of the diagram is given by $d$ and not $(-1)^nd$. The Yoneda product is then $\psi \cdot \varphi=\psi \circ \varphi_{-m-n}$.
 
\begin{lemma}\label{lem:Yoneda_prod}
The Yoneda product and the composition products differ by a sign change. More precisely, for any $\varphi, \psi \in \on{Ext}_A^*(A/I,A/I)$, we have
\[
\psi \cdot \varphi=(-1)^{|\varphi||\psi|}\psi \circ \varphi.
\]
\end{lemma}

\begin{proof}
Set $\varphi:P^{-n} \to A/I$. Using the notations introduced above, we have $\widetilde{\varphi}_{-n-i}=(-1)^{ni}\varphi_{-n-i}$ up to boundaries.

Consider $\varphi:P^{-n} \to A/I$ and $\psi:P^{-m} \to A/I$. Then, by definition, $\psi \cdot \varphi=\psi \circ \varphi_{-m-n}:P^{-m-n} \to A/I$.
But $\psi \circ \varphi_{-m-n}=\epsilon \circ \widetilde{\psi}_{-m} \circ \varphi_{-m-n}=(-1)^{mn}\epsilon \circ \widetilde{\psi}_{-m} \circ \widetilde{\varphi}_{-m-n}$ as $\on{Ker}d=\on{Im}\epsilon$. We see that $\epsilon \circ \widetilde{\psi}_{-m} \circ \widetilde{\varphi}_{-m-n}$ is the component $P^{-m-n} \to A/I$ of the map of complexes $\epsilon \circ \widetilde{\psi} \circ \widetilde{\varphi}$. But as $A/I$ is concentrated in degree $0$, all the other components of this map are zero. It follows that $\psi \cdot \varphi$ is the only (possibly) non-zero component of the map of complexes $(-1)^{mn}\epsilon \circ \widetilde{\psi} \circ \widetilde{\varphi}$.
\end{proof}

\begin{remark}
In the remainder of the paper, when speaking about the product in $\on{Ext}_A^*(A/I,A/I)$, if not explicitly stated, we will always consider the Yoneda product.    
\end{remark}

The $A$-projective resolution $ \epsilon: P^*\to A/I$ is a quasi-isomorphism of complexes of $ A$-modules with $ A/I$ regarded as a complex concentrated in degree $0$. 
Moreover, $ \mc Hom_A^*(P^*, P^*)$ is a dg algebra over $A$. 
Post-composition with $ \epsilon$ defines a chain map $\mc Hom_A^*(P^*,\epsilon): \mc Hom_A^*(P^*, P^*)\to \mc Hom_A^*(P^*, A/I) $, which is surjective because of the projectivity of $P^*$. As each summand of $P^*$ is projective and they are bounded above, it follows that $P^*$ is $K$-projective. Then, following the discussion at the beginning of the section, $\mc Hom_A^*(P^*,\epsilon)$ is a quasi-isomorphism.
The natural dg algebra structure on $\mc Hom_A^*(P^*,P^*)$ defines the algebra structure on $\onsf{Ext}^*_A(A/I,A/I) $:
\delete{, which is exactly the Yoneda algebra structure (cf. \cite[Thm.4.2.2]{CTVZ}), as seen below.}

\begin{proposition}\label{prop:comult_Yoneda}
The map $\mc Hom_A^*(P^*,\epsilon)$ induces an isomorphism of graded algebras
\[
H^*(\mc Hom_A^*(P^*, P^*))\stackrel{\cong}{\to} H^*(\mc Hom_A^*(P^*, A/I))=\onsf{Ext}^*_A(A/I,A/I)
\]
for the composition product structure $\onsf{Ext}^*_A(A/I,A/I)$ (see Lemma \ref{lem:Yoneda_prod}). Moreover, since $P^*$ is a Hopf algebra and $A/I$ is an algebra, there is a convolution algebra structure on $\mc Hom_A^*(P^*, A/I)$. \delete{which is the dual of the comultiplication on $P^*$, which is the Yoneda product.
}
\end{proposition}

\begin{remark}\label{rem:conv_not_Yoneda}
    The convolution algebra structure in the above proposition is different from the Yoneda product and composition product on $\onsf{Ext}^*_A(A/I,A/I)$. Indeed, as the comultiplication on $P^*$ is cocommutative (Theorem \ref{thm:Tate}), it follows that the convolution product, being the dual of the comultiplication, is graded commutative. However, the Yoneda product (and the composition product) is not necessarily graded commutative. For example, consider $A=\mathbb{C}[x,y]/(xy)$ and $I=(x,y)$. Then $\onsf{Ext}^*_A(A/I,A/I)$ is generated by two generators of degree $1$, written $\alpha_x$ and $\alpha_y$, and by one generator of degree $2$, written $\beta$. The graded commutator of the Yoneda product is then given by $[\alpha_x,\alpha_y]=\beta$, $[\alpha_x,\alpha_x]=[\alpha_y,\alpha_y]=0$, and $\beta$ is central. It follows that the Yoneda product is not graded commutative while the convolution product is. We will not consider the convolution product in the remainder of this paper.
\end{remark}

\delete{
$\alpha \cup \beta=m_M \circ (\alpha \otimes \beta) \circ \Delta_P $.

$\alpha \cup \beta(x)=m_M \circ (\alpha \otimes \beta) \circ \Delta_P(x)=m_M \circ (\alpha \otimes \beta) \circ \sum x_1 \otimes x_2= \sum \alpha(x_1) \beta(x_2)$.

$\beta \cup \alpha(x)= \sum \beta(x_1) \alpha(x_2)$.

But $\tau \Delta=\Delta$, then $\sum x_1 \otimes x_2=\sum(-1)^{|x_1||x_2|} x_2 \otimes x_1$. So $\alpha \cup \beta(x)=\sum \alpha(x_1) \beta(x_2)=\sum(-1)^{|x_1||x_2|}\alpha(x_2)\beta(x_1)=\sum(-1)^{|\alpha||\beta|}\alpha(x_2)\beta(x_1)=(-1)^{|\alpha||\beta|}\sum\alpha(x_2)\beta(x_1)=(-1)^{|\alpha||\beta|}\beta \cup \alpha(x)$.
}

\delete{
Use the map $P^* \stackrel{\Delta}{\to} P^* \otimes P^* \stackrel{\on{id} \otimes \beta}{\to} P^* \otimes M$ corresponding to the lift $\tilde{\beta}$ of $\beta$. Then the composition $\alpha \circ \tilde{\beta}$ is $P^* \stackrel{\Delta}{\to} P^* \otimes P^* \stackrel{\alpha \otimes \beta}{\to} M \otimes M \stackrel{m_M}{\to} M$. However, we know that $\alpha \cdot \beta=\alpha \circ \tilde{\beta}$ is the Yoneda product, and the cup product is $\alpha \cup \beta= m_M \circ (\alpha \otimes \beta) \circ \Delta$. So we see they are the same, i.e., $\alpha \cdot \beta=\alpha \cup \beta$. 
}

We now consider $P^*$ the resolution of $A/I$ obtained in Theorem \ref{thm:Tate}. It is a PD dg algebra. We define the set of PD dg derivations on $P^*$:
\[
\begin{array}{rl}
\mc Der_A^{*,\on{pd}}(P^*,P^*)=\displaystyle\bigoplus_{n \in \mathbb Z} & \{  \varphi \in \mc Hom_A^n(P^*,P^*) \ | \ \varphi(ab)=\varphi(a)b+(-1)^{|a|n}a\varphi(b) \\ 
& \quad \text{and } \varphi(a^{(k)})=\varphi(a)a^{(k-1)} \text{ if } |a| \text{ even} \}
\end{array}
\]
where we write $a^{(k)}=\gamma_k(a)$ to lighten notation. It is a routine to verify that $\mc Der_A^{*,\on{pd}}(P^*,P^*)$ is a subcomplex of $\mc Hom_A^*(P^*,P^*)$.

\begin{proposition}\label{prop:dg_Lie_algebra}
The cochain complex $\mc Der_A^{*,\on{pd}}(P^*,P^*)$ is a dg Lie algebra with the bracket given by the graded commutator in $\mc Hom_A^{*}(P^*,P^*)$, i.e., $[\phi, \psi]=\phi\circ \psi-(-1)^{|\phi||\psi|}\psi\circ \phi$ for homogeneous $ \phi, \psi \in \mc Hom_A^{*}(P^*,P^*)$.
\end{proposition}

\begin{proof}
 The Jacobi identity and the skew-symmetry of the bracket are direct consequences of the use of the graded commutator. We first show that $\mc Der_A^{*,\on{pd}}(P^*,P^*)$ is closed under the bracket.  For any $a,b \in P^*$ with $a$ homogeneous, we have
\[
[\varphi, \psi](ab)=[\varphi, \psi](a)b+(-1)^{(|\varphi|+|\psi|)|a|}a[\varphi, \psi](b).
\]
However, we know that for any homogeneous $x \in P^*$, $|[\varphi, \psi](x)|=|\varphi|+|\psi|+|x|$, and so $|[\varphi, \psi]|=|\varphi|+|\psi|$. It follows that $[\varphi, \psi](ab)=[\varphi, \psi](a)b+(-1)^{|[\varphi, \psi]||a|}a[\varphi, \psi](b)$ and $[\varphi, \psi]$ is a derivation.

Using the fact that $\varphi$ and $\psi$ are PD dg derivations and that $P^*$ is graded commutative, we can show that $[\varphi, \psi](a^{(k)})=a^{(k-1)}[\varphi, \psi](a)$ if $|a|$ is even. Hence $[\varphi, \psi]$ preserves the PD structure.

It remains to check that there exists a differential $d:\mc Der_A^{*,\on{pd}}(P^*,P^*) \to \mc Der_A^{*,\on{pd}}(P^*,P^*)$ compatible with the bracket. This differential is induced from the one on $\mc Hom_A^*(P^*,P^*)$ and reads $d(\varphi)(a)=d(\varphi(a))-(-1)^{|\varphi|}\varphi(d(a))$. We then check explicitly that:
\begin{itemize}
    \item $d(\varphi)(ab)=d(\varphi)(a)b+(-1)^{|d(\varphi)||a|}ad(\varphi)(b)$ for any $a, b \in P^*$ with $a$ homogeneous;
    \item $d(\varphi)(a^{(k)})=a^{(k-1)}d(\varphi)(a)$ if $|a|$ is even;
    \item $d([\varphi, \psi])=[d(\varphi), \psi]+(-1)^{|\varphi|}[\varphi, d(\psi)]$.
\end{itemize}
\end{proof}

An immediate consequence of the previous proposition is the following corollary:

\begin{corollary}\label{cor:der_Lie}
The complex $H^*(\mc Der_A^{*,\on{pd}}(P^*,P^*))$ is a graded Lie algebra.
\end{corollary}

\delete{
\begin{proof}
Consider $\varphi, \psi \in Z^*(\mc Der_A^{*,\on{pd}}(P^*,P^*))$ and $f, g \in \mc Der_A^{*,\on{pd}}(P^*,P^*)$. Then we have
\[
[\varphi+d(f),\psi+d(g)]=[\varphi,\psi]+[\varphi,d(g)]+[d(f),\psi]+[d(f),d(g)].
\]
But we know that $d([\varphi,g])=[d(\varphi),g]+(-1)^{|\varphi|}[\varphi,d(g)]=(-1)^{|\varphi|}[\varphi,d(g)]$ as $d(\varphi)=0$. Thus $[\varphi,d(g)]=(-1)^{|\varphi|}d([\varphi,g]) \in B^*(\mc Der_A^{*,\on{pd}}(P^*,P^*))$. Similarly, we show that $[d(f),\psi],[d(f),d(g)] \in B^*(\mc Der_A^{*,\on{pd}}(P^*,P^*))$. Hence
\[
[\varphi+d(f),\psi+d(g)]=[\varphi,\psi] \mod B^*(\mc Der_A^{*,\on{pd}}(P^*,P^*)).
\]
Moreover, for any $\varphi, \psi \in Z^*(\mc Der_A^{*,\on{pd}}(P^*,P^*))$, we have $d([\varphi, \psi])=0$ so $[\varphi, \psi]$ is a cocycle. It follows that we can define a bracket on $H^*(\mc Der_A^{*,\on{pd}}(P^*,P^*))$. For an element $\varphi \in Z^*(\mc Der_A^{*,\on{pd}}(P^*,P^*))$, we write $\overline{\varphi}$ for its class in $H^*(\mc Der_A^{*,\on{pd}}(P^*,P^*))$ and set
\[
[\overline{\varphi},\overline{\psi}]=\overline{[\varphi,\psi]}.
\]
The fact that the above bracket is a graded Lie bracket is then a direct consequence of Proposition \ref{prop:dg_Lie_algebra}.
\end{proof}
}

Consider two  PD dg $A$-algebras $P^*$ and $Q^*$ equipped with a PD dg algebra morphism $\theta:P^* \to Q^*$. Then $\mc Hom_A^{*}(P^*,Q^*)$ is a cochain complex with differential $d(f)=d_Q\circ f-(-1)^{|f|}f\circ d_P$. In fact $\mc Hom_A^{*}(P^*,Q^*)$ is a right dg $\mc Hom_A^{*}(P^*,P^*)$-module and left dg  $\mc Hom_A^{*}(Q^*,Q^*)$-module with the actions defined by
\begin{align}\label{eq:action_dg}
 f\cdot \varphi=f \circ \varphi \quad \text{ and } \quad 
 \psi\cdot f=\psi\circ f 
\end{align}
for any $\varphi \in \mc Hom_A^{|\varphi|}(P^*,P^*)$, $\psi \in \mc Hom_A^{|\psi|}(Q^*,Q^*)$ and $f \in \mc Hom_A^{|f|}(P^*,Q^*)$. 

Let $ \theta_*:\mc Hom_A^{*}(P^*,P^*)\to \mc Hom_A^{*}(P^*,Q^*) $ and $  \theta^*:\mc Hom_A^{*}(Q^*,Q^*)\to \mc Hom_A^{*}(P^*,Q^*) $ be the maps defined by 
\[
\theta_*(\varphi)=\theta\circ \varphi \quad \text{and } \quad \theta^*(\psi)=\psi\circ \theta
\]
for all $ \varphi \in \mc Hom_A^{*}(P^*,P^*)$ and $\psi\in \mc Hom_A^{*}(Q^*,Q^*)$.
Then both $\theta^*$ and $ \theta_*$ are chain maps. 

We can then define
\[
\begin{array}{rl}
\mc Der_{A, \theta}^{*,\on{pd}}(P^*,Q^*)=\displaystyle\bigoplus_{n \in \mathbb Z} & \{  f \in \mc Hom_A^n(P^*,Q^*) \ | \ f(ab)=f(a)\theta(b)+(-1)^{|a|n}\theta(a)f(b) \\ 
& \quad \text{and } f(a^{(k)})=f(a)\theta(a^{(k-1)}) \text{ if } |a| \text{ even} \}.
\end{array}
\]
It can be similarly verified that $\mc Der_{A,\theta}^{*,\on{pd}}(P^*,Q^*)$ is a subcomplex of $\mc Hom_A^*(P^*,Q^*)$ and  that 
\begin{align}\label{eq:derivation-transfer}
\theta_*(\mc Der_A^{*,\on{pd}}(P^*,P^*))\subseteq \mc Der_{A,\theta}^{*,\on{pd}}(P^*,Q^*)   \ \text{ and } \  \theta^*(\mc Der_A^{*,\on{pd}}(Q^*,Q^*))\subseteq \mc Der_{A,\theta}^{*,\on{pd}}(P^*,Q^*) 
 \end{align}

\begin{proposition}\label{prop:der_hom}
Let $P^*$, $Q^*$ and $\theta$ be as above. Then   \[
\mc Der_{A,\theta}^{*,\on{pd}}(P^*,Q^*)=\mc Hom_A^*(\bigoplus_{n \geq 0} F_{n+1}[n+1],Q^*)
\]
as graded $A$-modules. 
\end{proposition}

\begin{proof}
\delete{
By Theorem \ref{thm:Tate}, an element $\widetilde{\varphi} \in \mc Der_{A,\theta}^{*,\on{pd}}(P^*,Q^*)$ is defined by its values on $F_{n+1}[n+1]$ for all $n \geq 0$ \textcolor{red}{How does the maps dependent on $\theta$}. Hence it defines an element in $\mc Hom_A^*(\bigoplus_{n \geq 0} F_{n+1}[n+1],Q^*)$. The converse is treated similarly.
}

By Theorem \ref{thm:Tate}, each element $\widetilde{\varphi} \in \mc Der_{A,\theta}^{*,\on{pd}}(P^*,Q^*)$ is uniquely determined by its values on $F_{n+1}[n+1]$ for all $n \geq 0$, and they define an element in $\mc Hom_A^*(\bigoplus_{n \geq 0} F_{n+1}[n+1],Q^*)$. Clearly, each element in $\mc Hom_A^*(\bigoplus_{n \geq 0} F_{n+1}[n+1],Q^*)$ uniquely extends to a derivation in $\mc Der_{A,\theta}^{*,\on{pd}}(P^*,Q^*)$.
\end{proof}

\delete{
We know that $P^*  \to A/I$ is a quasi-isomorphism, where $A/I$ is seen as a PD dg algebra concentrated in degree $0$. Hence, with the same reasoning as the beginning of the section, the map $\mc Hom_A^*(\bigoplus_{n \geq 0} F_{n+1}[n+1],P^*) \to \mc Hom_A^*(\bigoplus_{n \geq 0} F_{n+1}[n+1],A/I)$ is also a {\color{red}quasi-isomorphism (PROBLEM: what is the differential on $\bigoplus_{n \geq 0} F_{n+1}[n+1]$?)}. By making use of Proposition \ref{prop:der_hom}, it follows that
\[
H^{*}(\mc Der_A^{*,\on{pd}}(P^*,P^*)) \cong H^{*}(\mc Der_A^{*,\on{pd}}(P^*,A/I)).
\]
}

\begin{conjecture}\label{conj:3.20}
Let $Q^*$ be a PD dg algebra constructed as in Theorem \ref{thm:Tate} and let $\sigma:Q^* \to Q^*$ be a PD dg algebra homomorphism that is homotopic to the identity. Then the following commutative diagram
 \begin{center}
  \begin{tikzpicture}[scale=0.9,  transform shape]
  \tikzset{>=stealth}

\node (3) at (0,-4){$\mc Der_{A,\sigma}^{*,\on{pd}}(Q^*,Q^*)$};

\node (6) at (4,-4){$\mc Hom_A^{*}(Q^*,Q^*)$};

\node (10) at (0,-6){$\mc Der_{A}^{*,\on{pd}}(Q^*,Q^*)$};
\node (11) at (4,-6){$\mc Hom_A^{*}(Q^*,Q^*)$};

\draw[{Hooks[right]}-Stealth] (3) -- (6) node[midway, above] {$\iota_\sigma$};
\draw[{Hooks[right]}-Stealth] (10) -- (11) node[midway, below] {$\iota$};

\draw [decoration={markings,mark=at position 1 with
    {\arrow[scale=1.2,>=stealth]{>}}},postaction={decorate}] (10) --  (3) node[midway, left] {$\sigma \circ (-)$};

\draw [decoration={markings,mark=at position 1 with
    {\arrow[scale=1.2,>=stealth]{>}}},postaction={decorate}] (11) --  (6) node[midway, right] {$\sigma \circ (-)$};

\end{tikzpicture}
 \end{center}
induces an equality of $A$-modules
\[
H^*(\iota_\sigma)(H^*(\mc Der_{A,\sigma}^{*,\on{pd}}(Q^*,Q^*))) = H^*(\iota)(H^*(\mc Der_{A}^{*,\on{pd}}(Q^*,Q^*)))
\]
inside of $H^*(\mc Hom_A^*(Q^*,Q^*))$.
\end{conjecture}

The homomorphism $\sigma \circ (-):\mc Hom_A^{*}(Q^*,Q^*) \to \mc Hom_A^{*}(Q^*,Q^*)$ induces the identity on $H^*(\mc Hom_A^{*}(Q^*,Q^*))$ since $\sigma - \on{id}_{Q^*}=d_{\mc Hom}(h)$ for some $h \in \mc Hom_A^{-1}(Q^*,Q^*)$.

The composition of chain maps:
\[
\mc Der_A^{*,\on{pd}}(P^*,P^*) \hookrightarrow \mc Hom_A^{*}(P^*,P^*) \to \mc Hom^{*}(P^*,A/I)
\]
induces an $A$-linear map of graded Lie algebras
\begin{align}\label{eq:morphism}
\Psi_P: H^*(\mc Der_A^{*,\on{pd}}(P^*,P^*)) \to \on{Ext}_A^*(A/I,A/I)
\end{align}
where the Lie bracket is the graded commutator.

Set $\varphi \in Z^0(\mc Der_A^{*,\on{pd}}(P^*,P^*))$, then $\varphi_0:P^0 \to P^0$ is an $A$-module morphism defined by $\varphi_0(1)$ as $P^0=A$. But $\varphi_0(1)=0$ due to the Leibniz rule, and so $\varphi_0=0$. As $P^*$ is exact and $\varphi$ is a chain map, one can construct by induction on $n$ an $A$-module homomorphism $h_n:P^{n} \to P^{n-1}$ such that $\varphi_n=d_{P^*} \circ h_n+h_{n-1} \circ d_{P^*}=d_{\mc Hom}(h)_n$. Hence the cohomology class of $\varphi$ is zero in $H^0(\mc Hom_A^{*,\on{pd}}(P^*,P^*))$. Combining that with the result of Lemma \ref{lem:homotopy} with $X^*=P^*$, we get that

\delete{
Set $\varphi \in Z^0(\mc Der_A^{*,\on{pd}}(P^*,A/I))$, then $\varphi_0:P^0 \to A/I$ is an $A$-module morphism defined by $\varphi_0(1)$ as $P^0=A$. But $\varphi_0(1)=0$ due to the Leibniz rule, and so $\varphi_0=0$. 
It follows that $H^0(\mc Der_A^{*,\on{pd}}(P^*,A/I))=0$. Combining that with the result of Lemma \ref{lem:homotopy} with $X^*=P^*$, we get that
}

\begin{align}\label{eq:neg_hom}
\Psi_P(H^{\leq 0}(\mc Der_A^{*,\on{pd}}(P^*,P^*)))=0.
\end{align}

\begin{definition}\label{def:restr_Lie}
Let $A$ be a commutative ring. A \textbf{restricted graded Lie algebra} $\mathfrak{g}$ over $A$ is a strictly positively graded $A$-module $\bigoplus_{n \geq 1}\mathfrak{g}_n$ equipped with a bilinear pairing
\[
[\cdot, \cdot]: \mathfrak{g}_m \times \mathfrak{g}_n \longrightarrow \mathfrak{g}_{m+n}
\]
and a quadratic operator
\[
q:\mathfrak{g}_{2n+1} \longrightarrow \mathfrak{g}_{4n+2}
\]
satisfying the following properties: for any $a \in \mathfrak{g}_{|a|}$, $b \in \mathfrak{g}_{|b|}$ and $c \in \mathfrak{g}_{|c|}$, then
\begin{itemize}
    \item[{\crtcrossreflabel{(1)}[def:Lie:cond1]}] $[a, b]=-(-1)^{|a||b|}[b, a]$; 
    \item[{\crtcrossreflabel{(1$\frac{1}{2}$)}[def:Lie:cond11/2]}] $[a,a]=0$ for $|a|$ even;
    \item[{\crtcrossreflabel{(2)}[def:Lie:cond2]}] $(-1)^{|a||c|}[a,[b,c]]+(-1)^{|b||a|}[b,[c,a]]+(-1)^{|c||b|}[c,[a,b]]=0$;
    \item[{\crtcrossreflabel{(2$\frac{1}{3}$)}[def:Lie:cond21/3]}] $[a,[a,a]]=0$ for $|a|$ odd;
    \item[{\crtcrossreflabel{(3)}[def:Lie:cond3]}] $q(\lambda a)=\lambda^2 q(a)$ for $\lambda \in A$ and $|a|$ odd;
    \item[{\crtcrossreflabel{(4)}[def:Lie:cond4]}] $[a, b]=q(a+b)-q(a)-q(b)$ for $|a|=|b|$ odd;
    \item[{\crtcrossreflabel{(5)}[def:Lie:cond5]}] $[a,[a,b]]=[q(a),b]$ for $|a|$ odd.
\end{itemize}
\end{definition}

\begin{remark}
    \begin{itemize}[wide]
        \item Notice that when $2$ is invertible in $A$, then condition \ref{def:Lie:cond11/2} is a direct consequence of \ref{def:Lie:cond1}. Likewise, when $3$  is invertible in $A$, then condition \ref{def:Lie:cond21/3} is a direct consequence of \ref{def:Lie:cond2}. This explain the choice of notation. 
        \item If $2$ is invertible in  $A$, then \ref{def:Lie:cond3} and \ref{def:Lie:cond4} imply that $q(a)=\frac{1}{2}[a,a]$ for any $|a|$ odd, and all relations except \ref{def:Lie:cond21/3} are consequences of \ref{def:Lie:cond1} and \ref{def:Lie:cond2}. Hence if  $2$ and $3$ are invertible in  $A$, then Definition \ref{def:restr_Lie} is the usual definition of a graded Lie algebra.
        \item The above definition of a restricted graded Lie algebra was introduced in \cite{Avramov} as a graded Lie algebra over a field. It also appears in \cite{Sjodin80} as an adjusted Lie algebra and in \cite{Milnor-Moore} as a restricted Lie algebra following Jacobson (\cite{Jacobson}). We choose the terminology ``restricted graded algebra" to put the emphasis both the graded structure and the additional map $q$.
    \end{itemize}
   
\end{remark}

\begin{proposition}\label{prop:homotopy_restricted}
With $P^*$, $A$ and $M$ as above, the cohomology space $H^{*}(\mc Der_A^{*,\on{pd}}(P^*,P^*))$ is a restricted graded Lie algebra.    
\end{proposition}

\begin{proof}
Use Corollary \ref{cor:der_Lie} for the graded Lie bracket, Eq.\eqref{eq:neg_hom} for the grading, and Appendix \ref{appendix:sec:square_pd_der} for the quadratic map.
\end{proof}

We now introduce the following definition:

\begin{definition}\label{def:homotopy_Lie}
Let $A$ and $I$ be as above and $P^*$ be the PD dg algebra obtained in Theorem \ref{thm:Tate}. The \textbf{homotopy Lie algebra} of the pair $(A,I)$, written $\pi^*(A,I)$, is the graded Lie algebra given by
\[
\pi^*(A,I)=\Psi_P(H^{*}(\mc Der_A^{*,\on{pd}}(P^*,P^*))) \subset \on{Ext}_A^*(A/I,A/I).
\]
\end{definition}

\delete{
\begin{remark}
We expect this $\pi^*(A,I)$ to be independent of the construction of $P^*$. {\color{red} TO DO}.  
\end{remark}
}

\begin{theorem}
Let $P^*$ and $Q^*$ be two PD dg algebras obtained through the construction of Theorem \ref{thm:Tate}. Then if Conjecture \ref{conj:3.20} is true, then
    \[
    \on{Im}\Psi_P=\on{Im}\Psi_Q.
    \]
\end{theorem}

\begin{proof}
Consider $\theta: P^* \to Q^*$ and $\tau :Q^* \to P^*$ the PD dg algebras homomorphisms obtained in Theorem \ref{thm:3.15} and Corollary \ref{cor:3.16}. Then we have following commutative diagram of complexes
 \begin{center}
  \begin{tikzpicture}[scale=0.9,  transform shape]
  \tikzset{>=stealth}
  
\node (1) at ( 0,0){$\mc Der_A^{*,\on{pd}}(P^*,P^*)$};
\node (2) at ( 0,-2){$\mc Der_{A,\tau}^{*,\on{pd}}(Q^*,P^*)$};
\node (3) at (0,-4){$\mc Der_{A,\theta\circ\tau}^{*,\on{pd}}(Q^*,Q^*)$};
\node (4) at (4,0) {$\mc Hom_A^{*}(P^*,P^*)$};
\node (5) at (4,-2){$\mc Hom_A^{*}(Q^*,P^*)$};
\node (6) at (4,-4){$\mc Hom_A^{*}(Q^*,Q^*)$};

\node (7) at (9,0){$\mc Hom_A^{*}(P^*,A/I)$};
\node (8) at (9,-2){$\mc Hom_A^{*}(Q^*,A/I)$};
\node (9) at (9,-4){$\mc Hom_A^{*}(Q^*,A/I)$};
\node (10) at (0,-6){$\mc Der_{A}^{*,\on{pd}}(Q^*,Q^*)$};
\node (11) at (4,-6){$\mc Hom_A^{*}(Q^*,Q^*)$};

\draw[{Hooks[right]}-Stealth] (1) -- (4);
\draw[{Hooks[right]}-Stealth] (2) -- (5);
\draw[{Hooks[right]}-Stealth] (3) -- (6);
\draw[{Hooks[right]}-Stealth] (10) -- (11);
\draw [decoration={markings,mark=at position 1 with
    {\arrow[scale=1.2,>=stealth]{>}}},postaction={decorate}] (1) --  (2) node[midway, left] {$(-)\circ \tau$};
\draw [decoration={markings,mark=at position 1 with
    {\arrow[scale=1.2,>=stealth]{>}}},postaction={decorate}] (2) --  (3) node[midway, left] {$\theta \circ (-)$};
\draw [decoration={markings,mark=at position 1 with
    {\arrow[scale=1.2,>=stealth]{>}}},postaction={decorate}] (10) --  (3) node[midway, left] {$\theta\circ\tau \circ (-)$};
\draw [decoration={markings,mark=at position 1 with
    {\arrow[scale=1.2,>=stealth]{>}}},postaction={decorate}] (4) --  (5) node[midway, right] {$(-)\circ \tau$};
\draw [decoration={markings,mark=at position 1 with
    {\arrow[scale=1.2,>=stealth]{>}}},postaction={decorate}] (5) --  (6) node[midway, right] {$\theta \circ (-)$};
\draw [decoration={markings,mark=at position 1 with
    {\arrow[scale=1.2,>=stealth]{>}}},postaction={decorate}] (11) --  (6) node[midway, right] {$\theta\circ \tau \circ (-)$};

    \draw [decoration={markings,mark=at position 1 with
    {\arrow[scale=1.2,>=stealth]{>}}},postaction={decorate}] (4) --  (7) node[midway, above] {$\epsilon_P \circ (-)$};
\draw [decoration={markings,mark=at position 1 with
    {\arrow[scale=1.2,>=stealth]{>}}},postaction={decorate}] (5) --  (8) node[midway, above] {$\epsilon_P \circ (-)$};
\draw [decoration={markings,mark=at position 1 with
    {\arrow[scale=1.2,>=stealth]{>}}},postaction={decorate}] (6) --  (9) node[midway, below] {$\epsilon_Q \circ (-)$};
\draw [decoration={markings,mark=at position 1 with
    {\arrow[scale=1.2,>=stealth]{>}}},postaction={decorate}] (7) --  (8) node[midway, right] {$(-) \circ \tau$};
\draw [decoration={markings,mark=at position 1 with
    {\arrow[scale=1.2,>=stealth]{>}}},postaction={decorate}] (8) --  (9) node[midway, right] {$=$};
\end{tikzpicture}
 \end{center}

The top right square obviously commutes, and the bottom right triangle commutes because of Theorem \ref{thm:3.15}.  The left two squares commute following \eqref{eq:derivation-transfer}. Using Conjecture \ref{conj:3.20}, we have that the image of $H^*(\mc Der_{A}^{*,\on{pd}}(Q^*,Q^*))$ in $  H^*(\mc Hom_A^*(Q^*,Q^*))$ is equal to that
of $H^*(\mc Der_{A, \theta\circ \tau}^{*,\on{pd}}(Q^*,Q^*))$. We note that composition of the top row is $ \Psi_P$ when they are all mapped into $ \on{Ext}_A^*(A/I,A/I)$ after taking the cohomology. Thus we have $\on{Im}(\Psi_{P})\subseteq \on{Im}(\Psi_{Q}) $. One can use the other direction of the commutative diagram to get $\on{Im}(\Psi_{Q})\subseteq \on{Im}(\Psi_{P}) $ using the fact that $ \tau\circ \theta$ is homotopy equivalent to identity of $P^*$.

\delete{
 \begin{center}
  \begin{tikzpicture}[scale=0.9,  transform shape]
  \tikzset{>=stealth}
  
\node (1) at ( 0,0){$\mc Der_{A, \tau\circ\theta}^{*,\on{pd}}(P^*,P^*)$};
\node (2) at ( 0,-2){$\mc Der_{A,\theta}^{*,\on{pd}}(P^*,Q^*)$};
\node (3) at (0,-4){$\mc Der_A^{*,\on{pd}}(Q^*,Q^*)$};
\node (4) at (4,0) {$\mc Hom_A^{*}(P^*,P^*)$};
\node (5) at (4,-2){$\mc Hom_A^{*}(P^*,Q^*)$};
\node (6) at (4,-4){$\mc Hom_A^{*}(Q^*,Q^*)$};

\node (7) at (9,0){$\mc Hom_A^{*}(P^*,A/I)$};
\node (8) at (9,-2){$\mc Hom_A^{*}(P^*,A/I)$};
\node (9) at (9,-4){$\mc Hom_A^{*}(Q^*,A/I)$};

\draw[{Hooks[right]}-Stealth] (1) -- (4);
\draw[{Hooks[right]}-Stealth] (2) -- (5);
\draw[{Hooks[right]}-Stealth] (3) -- (6);

\draw [decoration={markings,mark=at position 1 with
    {\arrow[scale=1.2,>=stealth]{>}}},postaction={decorate}] (2) --  (1) node[midway, left] {$ \tau\circ (-)$};
\draw [decoration={markings,mark=at position 1 with
    {\arrow[scale=1.2,>=stealth]{>}}},postaction={decorate}] (3) --  (2) node[midway, left] {$(-)\circ\theta $};

    \draw [decoration={markings,mark=at position 1 with
    {\arrow[scale=1.2,>=stealth]{>}}},postaction={decorate}] (5) --  (4) node[midway, right] {$ \tau\circ (-)$};
\draw [decoration={markings,mark=at position 1 with
    {\arrow[scale=1.2,>=stealth]{>}}},postaction={decorate}] (6) --  (5) node[midway, right] {$(-)\circ\theta$};

    \draw [decoration={markings,mark=at position 1 with
    {\arrow[scale=1.2,>=stealth]{>}}},postaction={decorate}] (4) --  (7) node[midway, above] {$\epsilon_P \circ (-)$};
\draw [decoration={markings,mark=at position 1 with
    {\arrow[scale=1.2,>=stealth]{>}}},postaction={decorate}] (5) --  (8) node[midway, above] {$\epsilon_P \circ (-)$};

\draw [decoration={markings,mark=at position 1 with
    {\arrow[scale=1.2,>=stealth]{>}}},postaction={decorate}] (8) --  (7) node[midway, right] {$=$};
\draw [decoration={markings,mark=at position 1 with
    {\arrow[scale=1.2,>=stealth]{>}}},postaction={decorate}] (9) --  (8) node[midway, right] {$(-)\circ \theta$};

\draw [decoration={markings,mark=at position 1 with
    {\arrow[scale=1.2,>=stealth]{>}}},postaction={decorate}] (6) --  (9) node[midway, below] {$\epsilon_Q \circ (-)$};
\end{tikzpicture}
 \end{center}

    {\color{red}TO DO}
    }
\end{proof}

We will thus omit the subscript and write $\Psi$ in the rest of this paper.

\begin{remark}
We expect that the graded Lie algebra $H^*(\mc Der^{*,\on{pd}}_{A} (P^*, P^*))$ is independent of the construction of $P^*$. This would be a homotopy theory question with respect to a model category structure on the category of PD dg $A$-algebras. Even though $P^*$ is a cofibrant replacement of $A/I$, the literature still could not automatically let us conclude. Hopefully with the work of \cite{Richter}, one can pass to the simplicial world to draw this conclusion.         
\end{remark}

\delete{
Set $\varphi \in \mc Der_A^{n,\on{pd}}(P^*,P^*)$ with $n<0$. Then $\varphi(P^*) \subset \bigoplus_{k \geq 1}P^{-k}$ and so $\epsilon \circ \varphi=0$. Hence 
\[
H^n(\mc Der_A^{*,\on{pd}}(P^*,P^*))=0 \text{ for all } n <0.
\]
}

A complex $X^*$ of $A$-modules is called $I$-minimal if $d(X^*) \subseteq I X^*$. In this case, we have $H^*(\mc Hom_A(X^*,A/I))=\mc Hom_A(X^*,A/I)$ and $H^*(X^* \otimes A/I)=X^* \otimes A/I$.

We now assume that $P^*$. We also assume that $F_n$ is a free $A$-module of finite rank for each $n \geq 0$. It follows that as a PD dg algebra, $P^*$ might have infinitely many generators, but for each degree, there will be only finitely many of them. 

Following the proof of Theorem \ref{thm:Tate}, let $\mathcal B(-(n+1))$ be an $A$-basis of $F_{n+1}[n+1]$ and set $\mathcal B(0)=0$. We also set an order on each $\mathcal B(-n)$. It follows that $\mathcal B=\bigcup_{n \geq 0} \mathcal B(-n)$ is a set of generators of the PD dg algebra $P^*$. Finally we set an order on $\mathcal B$ such that $b'<b$ for any $b \in \mathcal B(-n)$ and $b' \in \mathcal B(-(n+1))$. 

A basis of $P^*$ as an $A$-module is given through
\begin{align}\label{eq:basis_P}
\mathcal{X}=\{f : \mathcal{B} \to \bb N \ | \ f \text{ has finite support and }f(b)=0, 1 \text{ if } |b| \text{ odd for }b \in \mathcal{B} \}.
\end{align}
For $f \in \mathcal{X}$, we write $\on{deg}f=\sum_{b \in \mathcal{B}}f(b)|b|$ for the degree of $f$. We also write $\on{poldeg}b^{(f)}=\sum_{b \in \mathcal B} f(b)$ for its polynomial degree. There is a decomposition $\mathcal{X}=\bigcup_{n \in \bb N}\mathcal{X}(-n)$ where $\mathcal{X}(-n)=\{f \in \mathcal{X} \ | \ \on{deg}f=-n\}$. 
The basis element of $P^*$ corresponding to $f \in \mathcal{X}$ will be written $b^{(f)}=\prod_{b \in \mathcal B} b^{(f(b))}$ where the $b$'s are in descending order and $b^{(f(b))}=\gamma_{f(b)}(b)$. Set $b^{(f)} \in \mathcal X$ and define the $A$-linear map $(b^{(f)})^\vee:P^{\on{deg} f} \to A/I$ that sends all elements of $\mathcal{X}(\on{deg}f)$ to zero except for $b^{(f)}$ which is sent to $1$ ($A/I$ is an algebra so it has a unit). 

Similarly, we write $(b^\vee)^{f}=\prod_{b \in \mathcal B} (b^\vee)^{f(b)}$ for the ordered product in $\on{Ext}_A^*(A/I,A/I)$ (notice that $(b^\vee)^{f}$ uses the product without divided powers).

\begin{theorem}\label{thm:lift_pd}
With $P^*$ and $A$ as above, let $b \in \mathcal B$  be a generator of $P^*$ as a PD dg algebra. Then there exists a cocycle $\widetilde{\varphi} \in Z^{-|b|}(\mc Der_A^{*,\on{pd}}(P^*,P^*))$ such that $\epsilon \circ \widetilde{\varphi}_{| P^{|b|}} =b^\vee$.
\end{theorem}

\begin{proof}
Set $s=|b|<0$ and define $\widetilde{\varphi}_n=0$ for $n >s$ and $\widetilde{\varphi}_s:P^s \to P^0$ sending $b$ to $1$ and any other element of $\mathcal{X}(s)$ to zero. So $\epsilon \circ \widetilde{\varphi}_s=b^\vee$.

We know that $P^*$ is an $I$-minimal resolution so
\[
\epsilon \circ \widetilde{\varphi}_s \circ d(P^{s-1})=b^\vee(d(P^{s-1})) \subset b^\vee(I P^s)=0.
\]
Hence $\on{Im}(\widetilde{\varphi}_s \circ d) \subset \on{Ker}\epsilon = \on{Im}((-1)^sd)$, and so we can construct an $A$-linear map $\widetilde{\varphi}_{s-1}:P^{s-1} \to P^{-1}$ making the diagram commute as follows.

 The PD generators of $P^*$ of degree $-1$ will be written as $T_i$. Any  $b^{(f)} \in \mathcal{X}(s-1)$ such that $b$ does not appear in $d(b^{(f)})$ is sent to zero by $\widetilde{\varphi}_{s-1}$. But $b$ appears in $d(b^{(f)})$ only if $b^{(f)}=T_ib$ or $b^{(f)}=b' \in \mathcal B(s-1)$. In the first case, $\widetilde{\varphi}_{s}d(T_ib)=d(T_i)=(-1)^sd((-1)^sT_i)$ so we set $\widetilde{\varphi}_{s-1}(T_ib)=(-1)^sT_i$. In the second case, there exists $x \in I$ such that $xb$ appears in $d(b')$. Thus $\widetilde{\varphi}_{s}d(b')=\widetilde{\varphi}_{s}(xb)=x \in I=\on{Ker}\epsilon=\on{Im}(-1)^sd$. Hence there exists $y \in P^{-1}$ such that $(-1)^s dy=x$, and we set $\widetilde{\varphi}_{s-1}(b')=y$.

We know check that $\widetilde{\varphi}_{s-1}$ satisfies the PD relations. The elements in $\mathcal{X}(s-1)$ are of the following forms:
\begin{enumerate}[nosep]
    \item elements of $\mathcal B(s-1)$;
    \item monomials $b^{(f)}$ with $\on{poldeg}f \geq 2$ and with $f(b)=1$, i.e., $b^{(f)}=T_ib$;
    \item monomials $b^{(f)}$ with $\on{poldeg}f \geq 2$ and with $f(b)=0$.
\end{enumerate}
In case (1), there is nothing to verify. In case (2), we have $\widetilde{\varphi}_{s-1}(T_ib)=(-1)^sT_i=\widetilde{\varphi}_{-1}(T_i)b+(-1)^{s|T_i|}T_i\widetilde{\varphi}_{s}(b)$. Lastly, in case (3), all variables in $b^{(f)}$ have degree $\geq s$ and $f(b)=0$, so $b$ does not appear in $d(f^{(b)})$. As explained in the previous paragraph, we set $\widetilde{\varphi}_{s-1}(b^{(f)})=0$, and it follows that $\widetilde{\varphi}_{s-1}$ satisfies the PD relation on this monomial.

We now assume that we have lifted $\widetilde{\varphi}_{s}$ up to $\widetilde{\varphi}_{N}$ and each $\widetilde{\varphi}_{k}$ ($N \leq k \leq s$) satisfies the PD relations. Consider an element in $\mathcal X(N-1)$. There are three cases:

$\bullet$ Case 1. The element is of the form $b^{(f_1)}b^{(f_2)}$ such that there exist $b_1$ and $b_2$ satisfying $f_1(b_1) \neq 0$ and $f_2(b_1)=0$, and satisfying $f_1(b_2) = 0$ and $f_2(b_2) \neq 0$. We write $Y=b^{(f_1)}$ and $Z=b^{(f_2)}$. Then
\begin{align*}
\widetilde{\varphi}_{N}d(YZ) & =  \widetilde{\varphi}_{N}\left(d(Y)Z+(-1)^{|Y|}Yd(Z)\right) \\[5pt]
 & =  \widetilde{\varphi}_{|Y|+1}(d(Y))Z+(-1)^{s(|Y|+1)}d(Y)\widetilde{\varphi}_{|Z|}(Z) \\[5pt]
 & \quad  +(-1)^{|Y|}\left(\widetilde{\varphi}_{|Y|}(Y)d(Z)+(-1)^{s|Y|}Y\widetilde{\varphi}_{|Z|+1}(d(Z))\right) \\[5pt]
 & =  (-1)^sd(\widetilde{\varphi}_{|Y|}(Y))Z+(-1)^{|Y|}\widetilde{\varphi}_{|Y|}(Y)d(Z) \\[5pt]
& \quad  +(-1)^{s(|Y|+1)}d(Y)\widetilde{\varphi}_{|Z|}(Z)+(-1)^{|Y|(s+1)}(-1)^sY d( \widetilde{\varphi}_{|Z|}(Z)) \\[5pt]
&=(-1)^sd\left(\widetilde{\varphi}_{|Y|}(Y)Z+(-1)^{s|Y|}Y\widetilde{\varphi}_{|Z|}(Z)\right).
\end{align*}
Hence we set $\widetilde{\varphi}_{N-1}(b^{(f_1)}b^{(f_2)})=\widetilde{\varphi}_{|b^{(f_1)}|}(b^{(f_1)})b^{(f_2)}+(-1)^{s|b^{(f_1)}|}b^{(f_1)}\widetilde{\varphi}_{|b^{(f_2)}|}(b^{(f_2)})$. As $\widetilde{\varphi}_{|b^{(f_1)}|}$ and $\widetilde{\varphi}_{|b^{(f_2)}|}$ satisfy the PD relations, then so does $\widetilde{\varphi}_{N-1}$ on $b^{(f_1)}b^{(f_2)}$.

$\bullet$ Case 2. The element is of the form $b'^{(k)}$ where $b' \in \mathcal B(-2n)$ for some $n>0$. Using the fact that $\widetilde{\varphi}_{N}$ satisfies the PD relations and that $|b'|$ is even, we can show that
\[
\begin{array}{rcl}
\widetilde{\varphi}_{N}d(b'^{(k)}) & = & (-1)^sd\left(\widetilde{\varphi}_{|b'|}(b')b'^{(k-1)}\right)
\end{array}
\]
so we can set $\widetilde{\varphi}_{N-1}(b'^{(k)})=\widetilde{\varphi}_{|b'|}(b')b'^{(k-1)}$ and it satisfies the PD relations on this element.

$\bullet$ Case 3. The element is a variable $b' \in \mathcal B(N-1)$. Then $(-1)^sd \circ \widetilde{\varphi}_{N} \circ d(b')=\widetilde{\varphi}_{N+1} \circ d \circ d(b')=0 $. So $\widetilde{\varphi}_{N} d(b') \in \on{Ker}(-1)^sd=\on{Im}(-1)^sd$. Hence there exists an element $Z \in P^{N-1-s}$ such that $(-1)^sd(Z)=\widetilde{\varphi}_{N}d(b')$. We can set $\widetilde{\varphi}_{N-1}(b')=Z$.

We have thus defined an $A$-linear map $\widetilde{\varphi}_{N-1}:P^{N-1}\to P^{N-1-s}$ that satisfies the PD relations on every element of $\mathcal X(N-1)$. By induction, it follows that the map $\widetilde{\varphi}:P^* \to P^*$ is an element of $Z^{-|b|}(\mc Der_A^{*,\on{pd}}(P^*,P^*))$.
\end{proof}

\begin{example}\label{ex:not_minimal}
The fact that the resolution $P^*$ is $I$-minimal is not automatic. For example, take $A=\mathbb{C}[x,y]/(c_1,c_2)$ where $c_1=x^3,c_2=x^2(1+y)$. Set $\mathfrak{m}=(x,y) \subset \mathbb{C}[x,y]$ and $\mathfrak{m}_x$ its image in $A$. Notice that $A$ is not a complete intersection as $xc_2=(1+y)c_1$ (see Section \ref{sec:Yoneda_CI} for the definition of complete intersection). 
The free variables in degrees $-1$ and $-2$ of the partial resolution $P_2^*$ are given by  $dT_x=x,dT_y=y,dS_1=x^2T_x,dS_2=(1+y)xT_x$. 
But we see that $d\big((1+y)S_1\big)=(1+y)x^2T_x=0$ in $P_2^{-1}$, 
so $(1+y)S_1 \in \on{Ker}d$ and obviously $(1+y)S_1 \notin \mathfrak{m}_xP_2^{-2}=\mathfrak{m}_xP^{-2}$. 
But then it means that $(1+y)S_1 \in d(P^{-3})$ as $P^*$ is exact, so we cannot have $d(P^{-3}) \subset \mathfrak{m}_x P^{-2}$. Hence the resolution $P^*$ is not $I$-minimal.
\end{example}

We cannot lift morphisms corresponding to elements of the form $(b^{(f)})^\vee$ if there exist $b_1 \neq b_2$ in $\mathcal B$ with $f(b_1) \neq 0 \neq f(b_2)$, as they would not satisfy the PD relations. For example, consider $b^{(k)}$ for an even variable $b$ and $k \geq 2$. We want to lift $(b^{(k)})^\vee$. As in Theorem \ref{thm:lift_pd}, we define $\widetilde{\varphi}_n=0$ for $n >|b^{(k)}|$ and $\widetilde{\varphi}_{|b^{(k)}|}:P^{|b^{(k)}|} \to P^0$ where $\widetilde{\varphi}_{|b^{(k)}|}(b^{(k)})=1$. But we want $\widetilde{\varphi}_{|b^{(k)}|}$ to satisfy the PD relations, so $1=\widetilde{\varphi}_{|b^{(k)}|}(b^{(k)})=\widetilde{\varphi}_{|b|}(b)b^{(k-1)}=0$ as $\widetilde{\varphi}_{|b|}=0$, which is a contradiction. Likewise, if $Z=b^{(f_1)}b^{(f_2)}$ as in the previous proof. Then $\widetilde{\varphi}_{|b^{(f_1)}b^{(f_2)}|}(b^{(f_1)}b^{(f_2)})=1$, $\widetilde{\varphi}_{|b^{(f_1)}|}(b^{(f_1)})=0$, and $\widetilde{\varphi}_{|b^{(f_2)}|}(b^{(f_2)})=0$. Hence $1=\widetilde{\varphi}_{|b^{(f_1)}b^{(f_2)}|}(b^{(f_1)}b^{(f_2)})=\widetilde{\varphi}_{|b^{(f_1)}|}(b^{(f_1)})b^{(f_2)}+(-1)^{|b^{(f_1)}b^{(f_2)}||b^{(f_1)}|}b^{(f_1)}\widetilde{\varphi}_{|b^{(f_2)}|}(b^{(f_2)})=0$, which is also a contradiction. \delete{It follows that an element in $\on{Ext}_A^*(A/I,A/I)$ corresponding to monomial $f \in \mathcal X$ with $\on{poldeg}f \geq 2$ cannot be reached by the morphism \eqref{eq:morphism}.}

\begin{theorem}\label{thm:lin_comb}
Let $A$ and $P^*$ be as above. Then any element in $H^{n}(\mc Der_A^{*,\on{pd}}(P^*,P^*))$ with $n >0$ is the image of an $A$-linear combination of lifts of $b^\vee$ with $b \in \mathcal B(-n)$.
\end{theorem}

\begin{proof}
Consider $\varphi \in Z^{n}(\mc Der_A^{*,\on{pd}}(P^*,P^*))$ with $n>0$. Then we have $\varphi_{-n}:P^{-n} \to P^0$ and as $\varphi$ is a PD derivation, we see that  $\varphi_{-n}(b^{(f)})=0$ if $\on{poldeg}b^{(f)} \geq 2$, and so $\varphi_{-n}=\sum_{b \in \mathcal B(-n)}\varphi(b)\widetilde{\varphi^b}_{|b|}$, where $\varphi(b) \in P^0=A$ and $\widetilde{\varphi^b}_{|b|}:P^{|b|} \to P^0$ is the map sending $b$ to $1$ and any other element of $\mathcal X(-n)$ to zero. So $\epsilon \circ \widetilde{\varphi^b}_{|b|}=b^\vee$. We have seen in Theorem \ref{thm:lift_pd} that $b^\vee$ can be lifted to $\widetilde{\varphi^b} \in Z^{n}(\mc Der_A^{*,\on{pd}}(P^*,P^*))$. It follows that $\sum_{b \in \mathcal B(-n)}\varphi(b)\widetilde{\varphi^b}$ is a lift of $\varphi_{-n}$. But we also know that $\varphi$ is a cocycle so it is also a lift of $\varphi_{-n}$. It follows that in $H^{n}(\mc Der_A^{*,\on{pd}}(P^*,P^*))$, we have
\[
\overline{\varphi}=\sum_{b \in \mathcal B(-n)}\varphi(b)\overline{\widetilde{\varphi^b}}
\]
where $\overline{\varphi}$ indicate the cohomology class of $\varphi$.
\end{proof}

Based on Theorem \ref{thm:lin_comb} and Eq.\eqref{eq:neg_hom}, we have the following situation:
\[
\begin{array}{cccc}
\Psi: & H^*(\mc Der_A^{*,\on{pd}}(P^*,P^*)) & \to&  H^*(\mc Hom_A^{*}(P^*,P^*)) \\
&\{\text{degree }\leq 0\} & \mapsto  & 0 \\
&\text{degree }n \geq 1 & \mapsto &  A/I\text{-linear span of }\{b^\vee \ | \ b \in \mathcal B(-n)\} \subset H^n(\mc Hom_A^{*}(P^*,P^*)).
\end{array}
\]
where $H^{<0}(\mc Hom_A^{*}(P^*,P^*))=0$ by Lemma \ref{lem:homotopy}.

We have seen in Theorem \ref{thm:lin_comb} that $\Psi$ is not surjective as products of $b^\vee$ cannot be reached. Moreover, $\Psi$ might not be injective. Indeed, using Theorem \ref{thm:lin_comb}, if we write $\widetilde{\varphi^{b}}$ for the lift of $b^\vee$ and set $a \in A$, then we see that $\Psi(a \widetilde{\varphi^{b}})=\overline{a}b^\vee$. Hence if $\overline{a}=\overline{a'}$ with $a \neq a'$, then $\Psi(a \widetilde{\varphi^{b}})=\Psi(a' \widetilde{\varphi^{b}})$ but $a \widetilde{\varphi^{b}} \neq a' \widetilde{\varphi^{b}}$. 

It was stated in Theorem \ref{prop:homotopy_restricted} that $H^*(\mc Der_A^{*,\on{pd}}(P^*,P^*))$ is a restricted graded Lie algebra with $q(D)=D \circ D$. However, we see from Lemma \ref{lem:Yoneda_prod} that $\Psi(q(D))=\Psi(D) \circ \Psi(D)=-\Psi(D) \cdot \Psi(D)$ as $|D|$ odd. On $\on{Im}(\Psi)$ we define two quadratic maps $q_{yon}:\Psi(D) \mapsto \Psi(D) \cdot \Psi(D)$ and $q_{comp}:\Psi(D) \mapsto \Psi(D) \circ \Psi(D)$. We also write $[\cdot,\cdot]_{yon}$ and $[\cdot,\cdot]_{comp}$ for the super commutator with respect to the Yoneda and composition products, respectively. Then we can verify that $([\cdot,\cdot]_{yon}, q_{yon})$ and $([\cdot,\cdot]_{comp}, q_{comp})$ provide $\pi^*(A,I)$ with two different restricted graded Lie algebra structures. In what follows, we will only consider the structure coming from the Yoneda product and write only $([\cdot,\cdot], q)$ for the Lie bracket and the quadratic map.

\delete{
We set $\pi^*(A,I)=H^*(\mc Der_A^{*,\on{pd}}(P^*,P^*)) \subset \on{Ext}_A^*(A/I,A/I)$, and see that
\begin{align}\label{eq:Im_Psi}
\on{Im}\Psi =\mathcal L.
\end{align}
We can now use the morphism $\Psi$ to provide a description of the Yoneda algebra.

{\color{red}
$\on{Im}\Psi$ is a graded Lie algebra
}

\begin{proposition}\label{prop:L_Lie}
Let $A$ and $P^*$ be as above. Then $\mathcal{L}$ is a graded Lie algebra.
\end{proposition}

\begin{proof}
Consider $\alpha,\beta \in \on{Ext}_A^*(A/I,A/I)$ duals of PD variables in $P^*$. We know from Theorem \ref{thm:lift_pd} that there exist $\widetilde{\varphi}, \widetilde{\psi} \in Z^*(\mc Der_A^{*,\on{pd}}(P^*,P^*))$ such that $\alpha=\epsilon \circ \widetilde{\varphi}$ and $\beta=\epsilon \circ \widetilde{\psi}$. We have seen in Lemma \ref{lem:Yoneda_prod} that the Yoneda product satisfies $\alpha\beta=(-1)^{|\alpha||\beta|}\epsilon \circ \widetilde{\varphi} \circ \widetilde{\psi}$ (we do not write the $\cdot$ of the Yoneda product to lighten notations). Consider the graded commutator
\[
[\alpha, \beta]=\alpha\beta-(-1)^{|\alpha||\beta|}\beta\alpha.
\]
Then $[\alpha, \beta]=(-1)^{|\alpha||\beta|}\epsilon \circ [\widetilde{\varphi}, \widetilde{\psi}]$. We know that $[\widetilde{\varphi}, \widetilde{\psi}] \in \mc Der_A^{|\alpha|+|\beta|,\on{pd}}(P^*,P^*)$ by Proposition \ref{prop:dg_Lie_algebra}. It follows that $[\widetilde{\varphi}, \widetilde{\psi}]_{-n}=0$ for $n<|\alpha|+|\beta|$, and that  $[\widetilde{\varphi}, \widetilde{\psi}]_{-|\alpha|-|\beta|}:P^{-|\alpha|-|\beta|} \to P^0$. But as $[\widetilde{\varphi}, \widetilde{\psi}]_{-|\alpha|-|\beta|}$ needs to satisfy the PD relations, it follows that $[\widetilde{\varphi}, \widetilde{\psi}]_{-|\alpha|-|\beta|}=\displaystyle \sum_{X \text{ variable} \atop |X|=-|\alpha|-|\beta|}[\widetilde{\varphi}, \widetilde{\psi}](X)\widetilde{\chi}^X_{|X|}$ as in the proof of Theorem \ref{thm:lin_comb}, and that $\epsilon \circ \widetilde{\chi}^X=X^\vee$. It follows that $[\alpha, \beta]=\displaystyle(-1)^{|\alpha||\beta|}\sum_{X \text{ variable} \atop |X|=-|\alpha|-|\beta|}[\widetilde{\varphi}, \widetilde{\psi}](X)X^\vee \in \mathcal{L}$. It is then clear that we have a bilinear bracket
\[
[\cdot,\cdot]: \mathcal{L} \otimes \mathcal{L} \to \mathcal{L}.
\]

We have seen in Proposition \ref{prop:dg_Lie_algebra} that $[\cdot,\cdot]$ is a Lie bracket for the (dg) graded Lie algebra $\mc Der_A^{*,\on{pd}}(P^*,P^*)$. Hence it is also a Lie bracket for the graded structure on $\mathcal{L}$.
\end{proof}
}

\begin{proposition}\label{prop:decomp_dual}
Set $A$, $P^*$ as above, and set $b^{(f)} \in \mathcal{X}$. Then  $(b^{(f)})^\vee \in \on{Ext}_A^*(A/I,A/I)$ can be expressed as a finite $A/I$-linear combination of elements of the form $(b^\vee)^{h}$ with $h \in \mathcal{X}(\on{deg}f)$. In particular, $\{b^\vee \ | \  b \in \mathcal B\}$ is a set of generators of $\on{Ext}_A^*(A/I,A/I)$ as an $A/I$-algebra.
\end{proposition}

\begin{proof}
Set $\mathcal B(b^{(f)})=\{b \in \mathcal B \ | \ f(b) \neq 0\}$. We know that $\mathcal B(b^{(f)})$ is a ordered set coming from the order on $\mathcal B$ and it is finite as $f$ has finite support. Set $b'=\on{min} \mathcal B(b^{(f)})$. Hence $b'$ is the right-most variable of the ordered monomial $b^{(f)}$. We lift $b'^\vee$ and call $\widetilde{\varphi}'$ the resulting PD dg derivation. 

For any $b'' \in \mathcal{B}$, define $\delta_{b''} : \mathcal{B} \to \bb N$ such that $\delta_{b''}(b)=\delta_{b'', b}$ (Kronecker coefficient). Consider $\widetilde{\varphi}'_{|b^{(f)}|} :P^{|b^{(f)}|} \to P^{|b^{(f-\delta_{b'})}|}$ and set $b^{(g)}$ an element in $\mathcal X(b^{(f)})$, i.e., $|b^{(g)}|=|b^{(f)}|$. Then for $b^{(f-\delta_{b'})}$ to appear in $\widetilde{\varphi}'_{|b^{(f)}|}(b^{(g)})$, there are two possibilities:
\begin{enumerate}
    \item $b^{(g)}=b^{(f)}$;
    \item $g(b) \leq f(b)$ for all $|b| \leq |b'|$ and there exists a unique $b''$ such that $|b''|>|b'|$ and $g(b'')=1$. As $|b^{(g)}|=|b^{(f)}|$, this is only possible if $\on{poldeg}b^{(g)}<\on{poldeg}b^{(f)}$.
\end{enumerate}
It follows that
\[
(b^{(f-\delta_{b'})})^\vee \circ \widetilde{\varphi}'_{|b^{(f)}|}=(b^{(f)})^\vee+\sum_{\substack{|b^{(g)}|=|b^{(f)}| \\ \on{poldeg}b^{(g)}<\on{poldeg}b^{(f)}}} \overline{a_g}(b^{(g)})^\vee
\]
where $\overline{a_g} \in A/I$. Then, using Lemma \ref{lem:Yoneda_prod}, we can rewrite the above equation into
\begin{align}\label{eq:decomp}
(b^{(f)})^\vee=(-1)^{|b'||b^{(f-\delta_{b'})}|}(b^{(f-\delta_{b'})})^\vee \cdot b'^\vee-\sum_{\substack{|b^{(g)}|=|b^{(f)}| \\ \on{poldeg}b^{(g)}<\on{poldeg}b^{(f)}}} \overline{a_g} (b^{(g)})^\vee.
\end{align}
But we know that $\on{poldeg}b^{(f-\delta_{b'})}<\on{poldeg}b^{(f)}$, and so Eq.\eqref{eq:decomp} states that $(b^{(f)})^\vee$ is an $A/I$-linear combination of Yoneda products of $(b^{(h)})^\vee$ with $\on{poldeg}b^{(h)}<\on{poldeg}b^{(f)}$. By doing an induction on this polynomial degree, we see that $(b^{(f)})^\vee$ is an $A/I$-linear combination of unordered products of elements in $\{b^\vee \ | \ b \in \mathcal B\}$. However, we know that $\on{Im}\Psi$ is a graded Lie algebra whose Lie bracket is the graded commutator in $\on{Ext}_A^*(A/I,A/I)$. Hence we can reorder the monomials so that $(b^{(f)})^\vee$ is an $A/I$-linear combination of elements in $\{(b^\vee)^h \ | \ h \in \mathcal X(\on{deg}f)\}$. 
\end{proof}

\delete{
{\color{red}
Consider a commutative ring $A$ and $A$ an $\bb N$-graded connected Hopf $A$-algebra. We also assume that $A^n$ is a free $A$-module for all $n$. Then based on \cite{Andre} where we replace the field $K$ by the ring $A$, we have $A \cong U^c(Q(A))$ as coalgebras (\cite[Thm.17]{Andre}), where $Q(A)=A^+/J(A)$ (see \cite{Andre} for notations) is the set of PD indecomposable elements of $A$. If in addition $A^n$ finitely generated projective over $A$, then $A' \cong U(Q(A)')=U(\onsf{Prim}(A'))$ as algebras (cf. \cite[Prop.3.10]{Gulliksen-Levin}), where $(-)'$ denotes the restricted dual. For a Noetherian ring $R$ with ideal $I$, we set $A=R/I$ and $A=P^* \otimes_R A$ where $P^*$ is the $I$-minimal PD dg resolution of the $R$-module $A$ (cf. Theorem \ref{thm:Tate}). Then $A=\on{Tor}^R_*(A,A)$ is a Hopf algebra and the homotopy Lie algebra $\pi^*(R, A)$ is defined as
\[
\pi^*(R, A)=\onsf{Prim}(A')=\onsf{Prim}(\onsf{Ext}_R^*(A, A)),
\]
where the Lie bracket is the graded commutator inside of $\on{Ext}_R^*(A,A)$.}
}

We give below the last theorem of this section, providing a connection between the homotopy Lie algebra of $A$ and the PD dg derivations of the resolution $P^*$.

\begin{theorem}\label{thm:image_homotopy}
Let $A$ be a Noetherian ring, $I \subset A$ an ideal, and $P^*$ is the $I$-minimal PD dg resolution of the $A$-module $A/I$ with $P^n$ a free $A$-module for all $n$. Then the Yoneda algebra is equipped with a comultiplication satisfying
\begin{align}\label{eq:inclusion}
\pi^*(A,I) \subseteq \on{Prim}(\on{Ext}_A^*(A/I,A/I)),
\end{align}
where $\pi^*(A,I)$ is a restricted graded Lie algebra with an ordered basis $\{b^\vee \ | \ b \in \mathcal B\}$.  In particular $\on{Ext}_A^*(A/I,A/I)$ is generated as an $A/I$-algebra by $\pi^*(A,I)$. Moreover, we have
\begin{enumerate}
    \item If $\bb Q \subset A$, then Eq.\eqref{eq:inclusion} in an equality.
    \item If $\bb Q \not\subset A$, then Eq.\eqref{eq:inclusion} is a strict inclusion.
\end{enumerate}
\end{theorem}

\begin{proof}
We know (\cite[Ch. 8, Sec. 1]{MacLane})
that $\on{Tor}^A_*(A/I,A/I)=H^*(P^* \otimes_A A/I)$ is equipped with a product
\[
\begin{array}{cccc}
\varpi: & H^*(P^* \otimes_A A/I) \otimes_{A/I} H^*(P^* \otimes_A A/I) & \to & H^*(P^* \otimes_A A/I) \\
& \overline{u} \otimes \overline{v} & \mapsto & \overline{uv}.
\end{array}
\]
But as $P^*$ is assumed to be $I$-minimal, we have $H^*(P^* \otimes_A A/I)=P^* \otimes_A A/I$ and the above product is then the regular product of the algebra $P^* \otimes_A A/I$. The above product is a morphism of $A/I$-modules and for all $n \geq 0$, its dual gives a map
\[
\Delta_n: \mc Hom_{A/I}(P^{-n} \otimes_A A/I,A/I) \to \mc Hom_{A/I}\left(\bigoplus_{i+j=n}(P^{-i} \otimes_A A/I) \otimes_{A/I} (P^{-j} \otimes_A A/I),A/I\right).
\]
However, we have $\mc Hom_{A/I}(P^{-n} \otimes_A A/I,A/I)=\mc Hom_A(P^{-n},A/I)$ by base change. Moreover, we can rewrite $(P^{-i} \otimes_A A/I) \otimes_{A/I} (P^{-j} \otimes_A A/I)=P^{-i} \otimes_A P^{-j} \otimes_A A/I$. The right hand side becomes $\bigoplus_{i+j=n}\mc Hom_A(P^{-i} \otimes_A P^{-j},A/I)$ using the commutation of $\mc Hom$ with finite direct sums and the base change. Then, we know that $P^{-i}$ and $P^{-j}$ are free $A$-modules of finite rank due to $A$ being Noetherian, so we can split the tensor product. Finally, using the fact that $\on{Ext}_A^n(A/I,A/I)=\mc Hom_A(P^{-n},A/I)$ for all $n \geq 0$ due to $P^*$ being $I$-minimal, we can rewrite $\Delta_n$ as
\[
\Delta_n: \on{Ext}_A^n(A/I,A/I) \to \bigoplus_{i+j=n}\on{Ext}_A^i(A/I,A/I) \otimes_A \on{Ext}_A^j(A/I,A/I).
\]
We thus have a comultiplication  $\Delta:\on{Ext}_A^*(A/I,A/I) \to \on{Ext}_A^*(A/I,A/I) \otimes_A \on{Ext}_A^*(A/I,A/I)$. 

\delete{Moreover we know from Proposition \ref{prop:comult_Yoneda} the dual of the comultiplication in $P^*$ gives the Yoneda product. But we also know that the product in $P^*$ is a coalgebra morphism as $P^*$ is a bialgebra. Hence the comultiplication in $\on{Ext}_A^*(M,M)$ becomes an algebra morphism, where the algebra structure is given by the Yoneda product. We want to determine the primitive elements of $\Delta$.}

Let $b \in \mathcal B$ and consider $b^\vee \in \on{Ext}_A^{-|b|}(A/I,A/I)$. Then by definition $\Delta(b^\vee)=b^\vee \circ \varpi$. As $b$ is a generator of $P^*$ as a PD algebra and $\varpi$ is the product in $P^* \otimes_AA/I$, the only possibility for $b$ to appear as an element in $\on{Im}\varpi$ is as $b=\varpi(b \otimes 1)=\varpi(1 \otimes b)$. Hence $\Delta(b^\vee)=b^\vee \otimes_A 1 +1 \otimes_A b^\vee$ and $b^\vee \in \on{Prim}(\on{Ext}_A^{*}(A/I,A/I))$. As $\Delta$ is linear, it follows that $\pi^*(A,I) \subseteq \on{Prim}(\on{Ext}_A^*(A/I,A/I))$.

We have assumed $P^*$ $I$-minimal, and so $\on{Ext}^*_A(A/I,A/I)$ is spanned by elements $(b^{(f)})^\vee$ with $f \in \mathcal{X}$ (cf. Eq.\eqref{eq:basis_P}). We have seen above that if $\on{poldeg}b^{(f)}=1$, then $(b^{(f)})^\vee$ is primitive. Now assume that $\on{poldeg}b^{(f)} \geq 2$. We have two cases:\\
\underline{Case 1:} there exist $b \neq b' \in \mathcal{B}$ such that $f(b) \neq 0 \neq f(b')$. Hence we can decompose $b^{(f)}=b^{(g)}b^{(h)}$ with $g,h \in \mathcal{X}$ having disjoint supports. But then we have $\Delta((b^{(f)})^\vee)(b^{(g)} \otimes_A b^{(h)})=(b^{(f)})^\vee \circ \varpi(b^{(g)} \otimes_A b^{(h)})=(b^{(f)})^\vee(b^{(f)})=1$, hence $\Delta((b^{(f)})^\vee)$ contains $(b^{(g)})^\vee \otimes_A (b^{(h)})^\vee$ so $(b^{(f)})^\vee$ cannot be primitive.

\noindent \underline{Case 2:} we have $b^{(f)}=b^{(k)}$ for some $b \in \mathcal{B}$ with $|b|$ even and $k \geq 2$. We need to consider two subcases:
\begin{enumerate}
    \item Assume $\bb Q \subset A$. Then we have $kb^{(k)}=bb^{(k-1)}$, and so $\Delta((b^{(k)})^\vee)(b \otimes_A b^{(k-1)})=(b^{(k)})^\vee(kb^{(k)})=k$. Hence we see that $\Delta((b^{(k)})^\vee)$ contains $k (b^\vee \otimes_A (b^{(k-1)})^\vee)$, implying that $(b^{(k)})^\vee$ is not primitive.
    \item Assume $\bb Q \not\subset A$. Then there exists a prime $p \in \bb N$ such that for any $1 \leq m \leq p-1$, we have $b^{(m)}b^{(p-m)}=\binom{p}{m}b^{(p)}=0$ as $p$ divides $\binom{p}{m}$. Hence $b^{(p)}$ cannot be obtained using elements of lower degrees. It follows that $(b^{(p)})^\vee$ is primitive.
\end{enumerate}

In the case where $\bb Q \subset A$, we have seen that $(b^{(f)})^\vee$ is primitive if and only if $b^{(f)} \in \mathcal{B}$. As the comultiplication is linear, it follows that $\on{Prim}(\on{Ext}_A^*(A/I,A/I))=\pi^*(A,I)$. Moreover, in the case where $\bb Q \not\subset A$, there exists a prime $p \in \mathbb{N}$ such that $(b^{(p)})^\vee \in \on{Prim}(\on{Ext}_A^*(A/I,A/I))$ but we know that $(b^{(p)})^\vee \notin \pi^*(A,I)$. As a consequence, we have $\pi^*(A,I) \subsetneq \on{Prim}(\on{Ext}_A^*(A/I,A/I))$. The remaining statements follow from Proposition \ref{prop:decomp_dual}.

\delete{
We have seen in Proposition \ref{prop:decomp_dual} that $\on{Ext}_A^{*}(A/I,A/I)$ is generated by $\{b^\vee \ | \ b \in \mathcal B\}$. So if $b, b' \in \mathcal B$, we have $\Delta(b^\vee b'^\vee)=\Delta(b^\vee)\Delta(b'^\vee)=(b^\vee \otimes_A 1 +1 \otimes_A b^\vee) (b'^\vee \otimes_A 1 +1 \otimes_A b'^\vee)=b^\vee b'^\vee \otimes_A 1 +1 \otimes_A b^\vee b'^\vee+(-1)^{|b||b'|}b'^\vee \otimes_A b^\vee+b^\vee \otimes_A b'^\vee$, which is not primitive. By doing an induction on $\on{poldeg}b^{(f)}$, we see that $(b^\vee)^{f}$ is primitive if and only if $\on{poldeg}b^{(f)}=1$, i.e., $(b^\vee)^{f} \in \{b^\vee \ | \ b\in \mathcal B\}$. It follows from the discussion before Proposition \ref{prop:decomp_dual} that $\on{Prim}(\on{Ext}_A^{*}(A/I,A/I))=\pi^*(A,I)$ as $A/I$-modules. We can thus equip $\on{Prim}(\on{Ext}_A^{*}(A/I,A/I))$ with a restricted graded Lie algebra structure. The rest of the proof follows from Proposition \ref{prop:decomp_dual}.}
\end{proof}

\delete{
If $ A$ is a commutative ring and $ A/I$ is an $A$-algebra. 
Suppose that $ P_*$ has a dg algebra structure and the resolution $ \epsilon:P_*\to A/I$ is a dg algebra homomorphism, 
then the algebra $\mc Hom_A^*(P_*,P_*)$ has subcomplex of vector fields 
\[\mc Der_A^*(P_*, P_*)= \displaystyle\bigoplus_n\{\varphi \in Hom_A^n(P_*,P_*) \ | \ \varphi(ab)=\varphi(a)b+(-1)^{|a||\varphi|}a\varphi(b) \}
\]
which is a dg Lie subalgebra of $\mc Hom_A^*(P_*,P_*)$ {\color{red}(checked)}. The cohomology $ H^*(\mc Der_A^*(P_*, P_*))$ has a graded Lie algebra structure {\color{red}(checkedBerthelot Ogus)} and the induced map $H^*(\mc Der_A^*(P_*, P_*))\to H^*(\mc Hom_A^*(P_*, P_*))=\onsf{Ext}_A^*(A/I,A/I) $, which is a graded Lie subalgebra. 
We {\color{red} expect} that this Lie subalgebra generates $\onsf{Ext}^*_A(A/I,A/I)$ as a graded associative algebra and this Lie algebra is independent of the choice of the dg algebra resolution $P_*$ of $A/I$ {\color{red}(CHECK for the complete intersection)}.

And also there are $A$-subcomplexes $ \mc D^{n,*}_A(P_*, P_*)\subseteq \mc Hom_A^*(P_*,P_*)$ for $ n\geq 0$ (differentials).
They satisfy

\[
\mc D^{n,*}_A(P_*, P_*)\circ \mc D^{m,*}_A(P_*, P_*)\subseteq \mc D^{n+m,*}_A(P_*, P_*)
\]
making $\mc Hom_A^*(P_*,P_*)$ a filtered dg algebra {\color{red}(Not really as the union is not the whole space, but a subalgebra)}. The dg subalgebra $\mc D_A^*(P_*,P_*)=\bigcup_n \mc D^{n,*}_A(P_*, P_*)$ is called the algebra of differential operators of $ P_*$. 
Thus we have the homomorphism $\pi$
\[
H^*(\pi): H^*(\mc D_A^*(P_*,P_*))\to \onsf{Ext}_A^*(A/I,A/I).
\]
{\bf Conjecture} is that $ H^*(\pi)$ is surjective. {\color{red}(CHECK for the complete intersection)}

It should be standard argument to show that 
\begin{align}
    \label{almost-commutative}
[\mc D^{n,*}_A(P_*, P_*), \mc D^{m,*}_A(P_*, P_*)]_{gr}\subseteq \mc D^{n+m-1,*}_A(P_*, P_*).
\end{align}
and the associative graded dg algebra 
\[\onsf{gr}(\mc D_A^*(P_*,P_*)) =\bigoplus_{n\geq 0}\mc D^{n,*}_A(P_*, P_*)/\mc D^{n-1,*}_A(P_*, P_*)\]
is a strictly graded commutative associative algebra and generated by $\mc Der_A^*(P_*,P_*)=\onsf{gr}_1(\mc D_A^*(P_*,P_*)) $. In fact $ \mc D^*_A(P_*,P_*)$ should be generated by $\mc Der_A^*(P_*,P_*) $ over $ \mc Hom_A^0(A/I,A/I)$.

The commutative dg $A$-algebra $\onsf{gr}(\mc D_A^*
(P_*, P_*)) $ has a Poisson structure by \eqref{almost-commutative}.

It should be a standard argument that there is a surjective filtered dg algebra homomorphism $ U_A(\mc Der_A^*(P_*,P_*))\to \mc D^*_A(P_*,P_*)$.
  If the conjecture holds, then $\onsf{Ext}_A^*(A/I,A/I) $ should have an associated filtered algebra structure (the image of $H^*(\mc D_A^{\cdot ,*}(P_*, P_*))$ in $ \onsf{Ext}^{*}_A(\onsf{k}_x, \onsf{k}_x)$) such that the associated graded algebra is graded commutative with a Poisson structure. 
  Thus there is a homomorphism of graded dg algebras $ \onsf{gr}(U_A(\mc Der_A^*(P_*,P_*)))\to \onsf{gr}\mc D_A^{\cdot ,*}(P_*, P_*)$ inducing 
  \[
  H^*(\onsf{gr}(U_A(\mc Der_A^*(P_*,P_*))))\to H^*(\onsf{gr}\mc D_A^{\cdot ,*}(P_*, P_*))\to \onsf{gr} \onsf{Ext}^*_A(A, A)  \]
  where $\onsf{gr}(U_A(\mc Der_A^*(P_*,P_*))))=\onsf{Sym}^*(\mc Der_A^*(P_*,P_*))$.  

We note that $\mc Hom^*(P_*, A/I)$ module for the dg Lie algebra
$\mc Der^{*}(P_*,P_*)$.
  
  When $ A=R(V)$ with $V$ being a vertex algebra over the commutative ring $\onsf{k}$, for $ x: A\to A_*$, we have a $\onsf{Ext}^*_A(A, A)$ that has an associated graded commutative algebra $\onsf{gr}_*\onsf{Ext}^*_A(A_x, A_x)$. We now defined the superscheme $ \onsf{spec}(\onsf{gr}_*\onsf{Ext}^*_A(A_x, A_x))$ as the cohomological variety of a vertex algebra at $ x.$ when of $ A=C_2(V)$ (with $V$ being a vertex algebra over commutative ring $ R$) and point $ x: A\to R$.)

 Compare the algebra $\onsf{gr}_*\onsf{Ext}^*_A(A_x, A_x)$ with  $\onsf{Ext}^*_A(A_x, A_x)^{ab}$. If $ x$ is a smooth point they should be the same. What about the complete interest case when $ \onsf{Ext}^*_A(A_x, A_x)$ is graded commutative (and the Lie algebra $ H^*(\mc Der_A^*(P_*,P_*))$ is graded commutative)? Here $\onsf{Ext}^*_A(A_x, A_x)^{ab}$ should be in the graded commutative sense. What happens if $P_*$ has a PD structure? 
 
{\color{red} do some explicit examples in the complete intersection case}

 if $ A/I$ is an $A$-module and $ A/I_x=A/I\otimes _A\onsf{k}_x$ is $A$-module. 

 Consider $\onsf{Ext}^*_A(\onsf{k}_x,A/I)$ is right  module for  $\onsf{Ext}^*_A(\onsf{k}_x,\onsf{k}_x)$. Also  $ \mc Hom^*(P_*, A/I)$ is an right dg-module of the dg algebra $\mc Hom^*(P_*, P_*)$ by precomposition. Define a filtered $ \mc D^{\cdot,*}(P_*,P_*)$-module structure on $\mc Hom^*(P_*, A/I) $ (there are many ways to define such filtered structure (Bernstein filtration). Thus there is an associated grade complex $ \onsf{gr}\mc Hom^*(P_*, A/I) $ which is graded dg module for $\onsf{gr}(\mc D^{\cdot, *}(P_*, P_*)$. Taking cohomology $ H^*(\onsf{gr}\mc Hom^*(P_*, A/I))$, which is a graded module of $\onsf{gr}\onsf{Ext}^*_A(\onsf{k}_x,\onsf{k}_x)$. $\onsf{Ann}_{\onsf{gr}\onsf{Ext}^*_A(\onsf{k}_x,\onsf{k}_x)}(H^*(\onsf{gr}\mc Hom^*(P_*, A/I))$ defines a closed graded subscheme  in $
 \onsf{spec}(\onsf{gr}\onsf{Ext}^*_A(\onsf{k}_x,\onsf{k}_x))$ as 
 \[\onsf{spec}(\onsf{gr}\onsf{Ext}^*_A(\onsf{k}_x,\onsf{k}_x)/\onsf{Ann}_{\onsf{gr}\onsf{Ext}^*_A(\onsf{k}_x,\onsf{k}_x)}(H^*(\onsf{gr}\mc Hom^*(P_*, A/I)).\]
 Thus it will be called homotopy associated variety of $A/I$ at $x$.

 Recall that we plan to define cohomology support variety of $A$-module $A/I$ to be the quotient $$ \onsf{Ext}\to \onsf{Ext}/\onsf{Ann}_{\onsf{Ext}}(\onsf{Ext(\onsf{k}_x, A/I)},$$ then taking maximal commutative quotient. Thus there is a homomorphism $\onsf{Ext}^{ab}\to (\onsf{Ext}/\onsf{Ann}_{\onsf{Ext}}(\onsf{Ext(\onsf{k}_x, A/I)}))^{ab}$. 
 So the cohomological support scheme of $A/I$ is the morphism 
 \[
 \onsf{spec}(\onsf{Ext}/\onsf{Ann}_{\onsf{Ext}}(\onsf{Ext(\onsf{k}_x, A/I)})^{ab}) \to \onsf{spec}(\onsf{Ext}^{ab})
 \]
}
\delete{
\section{Take finitely generated algebras }

\section{Koszul-Tate Resolutions for arbitrary commutative rings}
Let $ R$ be a commutative ring and $ I\subseteq R$ be an ideal. The Koszul-Tate resolution is a strictly graded commutative differential $R$- algebra $R\langle T\rangle=\onsf{TS}_R(T)$ where $T$ is a free graded $R$-module and $\onsf{TS}_R(T)$ is the strictly commutative  differential graded $R$-algebra such that $H_*(TS(T))=R/I$ concentrated in degree 1. In particular $B_i(\onsf{TS}_A(T))\subseteq I\onsf{TS}_A(T))_i$. 

\subsection{Symmetric tensors of differential complexes} Let $\onsf{Cplx}(R)$ be the symmetric differential chain complexes of $R$-modules with differential $d$ of degree $-1$. Here the tensor product of two complexes $(X_*, d_*)\otimes_R(Y_*, d_*)$ is defined by 
\[ (X_*, d_*)\otimes_R(Y_*, d_*)_n=\oplus _{i+j=n}X_i\otimes_{R}Y_j
\]
with differential $d^{X\otimes Y}(x\otimes y)=d(x)\otimes_R y+(-1)^{|x|}x\otimes d(y)$ for all homogeneous $x\in X$ and $ y\in Y$. 

The braiding $b_{X, Y}: X\otimes_R Y\to Y\otimes_R X$ is defined by 
\[ b_{X,Y}(x\otimes y)=(-1)^{|x||y|}y\otimes x
\]
for all homogeneous $x\in X$ and $y\in Y$.  
$b_{X,y}$ is clearly a chain map of chain complexes.
An associative algebra $(A, m, 1)$ (with multiplication $m$ of homogeneous of degree 0) in $ \onsf{Cplx}(R)$ is called strictly graded commutative if  
\[\xymatrix{A\otimes A\ar[dr]_m\ar[rr]^{b_{A, A}}&& A\otimes A\ar[dl]^{m}\\
&A&
}
\] 
and $m(x, x)=0$ for all homogeneous $ x\in X$ of odd degree.   Let $ \onsf{scDGA}(R)$ denote the category of strictly graded commutative dg algebras in $ \onsf{Cplx}(R)$. If $ (A, d)$ and $ (B, d_)$ are two strictly graded commutative algebras, then $ A\otimes_R B$ is also strictly graded commutative dg algebra. We recall the multiplication $m_{A\otimes_R B}$ is defined as the composition of the map:
\[ (A\otimes_R B)\otimes_R(A\otimes_R B)\stackrel{1\otimes b_{B, A}\otimes 1}{\longrightarrow}(A\otimes_R A)\otimes_R(B\otimes_R B)\stackrel{m_A\otimes m_B}{\longrightarrow} A\otimes B.\]

 Let $ (X_*, d_*)$ be an object in $\onsf{Cplx}(R)$, then $T_A^n(X)=X\otimes_R\cdots\otimes_R X$ be the differential complex. The symmetric group $ \mf S_n$ acts on $T_A^n(X)$ as automorphism of chain complexes as follows. 
\[ \sigma(x_1\otimes\cdots\otimes x_n)=(-1)^{\ell(\sigma; (x_1, \cdots, x_n)}(x_{\sigma^{-1}(1)}\otimes \cdots \otimes x_{\sigma^{-1}(n)}).
\]
Here $\ell(\sigma;(x_1, \cdots, x_n))=\sum_{(i, j)\in\onsf{Inv}(\sigma^{-1})}|x_i||x_j|$ and $ \onsf{Inv}(\sigma)=\{ (i< j) \;|\; \sigma(i)> \sigma (j)\}$ is the set of inversions of $\sigma$. 

Let $\onsf{TS}_A^n(X)=T_A^n(X)^{\mf S_n}$ the $R$-submodule of $\mf S_n$-fixed points. It is straight forward to verify that $\onsf{TS}_A^n(X)$
 is closed under the differentials since $ \mf S_n$-action commutes with the differentials since the action of $\sigma\in \mf S_n$ on $ T_A^n(X)$ are chain maps. One can check, for any sequence of homogeneous elements $(x_1\cdots, x_n)$ in $X$ and $ \sigma, \tau\in \mf S_n$
 \[\ell(\tau\sigma; (x_1, \cdots, x_n))\equiv \ell(\sigma; (x_1, \cdots, x_n)+\ell(\tau; (x_{\sigma^{-1}(1)}, \cdots, x_{\sigma^{-1}(n)}) \mod 2.
 \]
 
We now define the multiplication 
\[ \onsf{TS}_A^m(X)\otimes_R \onsf{TS}_A^n(X)\to \onsf{TS}_A^{n+m}(X)
\]
by using the trace map $\onsf{tr}_{\mf S_{n+m}/(\mf  S_m\times \mf  S_n)}$ (see Bourbaki Algebra Ch IV, 5) defined by

\[ f\star g=\sum_{\sigma\in {\mf S_{n+m}/(\mf  S_m\times \mf  S_n)}}\sigma (f\otimes_R g)
\]
for all any coset representatives of $\sigma\in {\mf S_{n+m}/(\mf  S_m\times \mf  S_n)}$. One can verify directly that, 
for any $ (\sigma_m, \sigma_n) \in \mf S_m\times \mf S_n\subseteq \mf S_{m+n}$, 
\[\sigma \circ(\sigma_m, \sigma_n) (f\otimes_R g)=\sigma(f\otimes_r g).
\]

There is a canonical choice of minimal coset representatives $ \sigma\in \mf S_{m+n}/{(\mf  S_m\times \mf  S_n)}$ such that
\[ \sigma(1)<\cdots<\sigma(m)\; \text{ and } \; \sigma(m+1)<\cdots<\sigma(m+n).
\]

Then $\onsf{TS}_A(X)=\oplus _{n=0}^{\infty}\onsf{TS}_A^n(X)$ is an $ \bb N$-graded differential algebra under the multiplication $ \star$. This is a generalization of the shuffle product. Let $ \onsf{TS}_A^+(A/I)=\oplus_{n>0}\onsf{TS}_A^n(A/I)$ is an dg-ideal of the dg algebra.  

We now claim that $ \onsf{TS}_A(X)$ is strictly graded commutative differential graded algebra.  In fact for any $x$ homogeneous, we have 
$x\star x=x\otimes x+(-1)^{|x|^2}x\otimes x$. If $x$ has odd degree, then we have $ x\star x=0$. 

\begin{theorem} For differential complex $(X_*, d_*)$ in $ \onsf{cplx}(R)$, $ \onsf{TS}_R(X)$ is a strictly graded commutative dg algebra (forgetting the $\bb N$-graded structure. Furthermore, $\onsf{TS}_R: \onsf{Cplx}(R)\to \onsf{scDGA}(R)$ is a functor that satisfies the following
$\onsf{TS}_R(M\oplus M')\cong \onsf{TS}_A(M)\otimes_R\onsf{TS}_R(M')$
\end{theorem}

If follows the same argument as in Bourbaki, $\onsf{TS}_R(-): \onsf{Cplx}(R)\to \onsf{scGDA}(R)$ is  functorial. If $ \phi: X\to Y$ is Chain map, then $ \onsf{TS}_A(\phi): \onsf{TS}_A(X)\to \onsf{TS}_A(Y)$ is a homomorphism of DGA.  But this is not the left adjoint of the forgetful functor $ \onsf{scDGA}(R)\to \onsf{Cplx}(R)$.  The correct left adjoint to this forgetful functor the functor $ \onsf{Sym}_R(-): \onsf{Cplx}(R)\to \onsf{scDGA}(R)$ which defined as the coinvariant of $\mf S_n$ for the $T_A^n(M)$. 

We note that there is a natural chain map $ X\to \onsf{TS}_A^1(X)\subseteq \onsf{TS}_A(X)$. 

If $ X$ is free over $ R$, then $\onsf{TS}_A(X)$ has a divided power structure. 

(Stack project discusses only the divided power structure on the event degree part). Thus $\onsf{TS}_A(X)$ is universal strictly graded commutative dg algebra with a power structure.  The functor $M\mapsto \onsf{TS}_A(M) $ is determined by an operad $\scr P$ defined by $ \scr P_n(M)=\Hom_{R\mf S_n}(R, M^{\otimes n})$.  

\subsection{Divided power structure and $D$-module structure of $G_a$}

Let $(A, d) $ strictly graded commutative DG algebra over $R$. A PD -structure on $A$ is $(A, I, \gamma)$ with $I \subseteq A$ being a dg ideal and $\gamma_n: I\to A$ being a sequences of maps satisfying the following conditions
\begin{enumerate}
\item $\gamma_1(x)=x$ and $\gamma_0(x)=1$ for all $x \in I$;
\item $\gamma_n(x)\gamma_m(x)=\binom{n+m}{m} \gamma_{n+m}(x)$;
\item $\gamma_n(ax)=a^n\gamma_n(x)$ for all $ a\in A$ and $x\in I$;
\item $\gamma_n(x+y)=\sum_{i=0}^n\gamma_i(x)\gamma_{n-i}(y)$ for all $ x, y\in I$;
\item $\gamma_n(\gamma_m(x))=\frac{(mn)!}{n! (m!)^n}\gamma_{nm}(x)$;
\item $ \gamma_n(I_r)\subseteq A_{rn}$;
\item $\gamma_n(xy)=0$ for all $ x, y\in I$ homogeneous of odd degrees and $ n\geq 2$;
\item $d(\gamma_n(x))=\gamma_{n-1}(x)d(x)$.
\end{enumerate}
The condition (7) on the odd degree was imposed by Andr\'e in \cite{Andre}. We will simply write $x^{(n)}=\gamma_n(x)$. We recall that the PD structure heavily dependent on the ideal $I$. Given and  graded commutative dg ring, we take $I=0$ and thus have a PD structure. 

Let $(A, I, \gamma)$ be a dg PD ring. Let $(M, d)$ be a dg-module over $(A, d)$. A PD module structure is consists of a subcomplex $J\subseteq M$ such that $IJ\subseteq J$ and $\gamma_n : J\to J $ for all $ n\geq 1$ such that 
\begin{itemize}
\item $\gamma_1(v)=x$  for all $x \in J$;
\item $\gamma_n(ax)=a^n\gamma_n(x)$ for all $ a\in A$ and $x\in J$;
\item $\gamma_n(x+y)=\sum_{i=0}^n\gamma_i(x)\gamma_{n-i}(y)$ for all $ x, y\in J$;
\item $\gamma_n(\gamma_m(x))=\frac{(mn)!}{n! (m!)^n}\gamma_{nm}(x)$;
\item $ \gamma_n(J_r)\subseteq M_{rn}$;
\item $d(\gamma_n(x)=\gamma_{n-1}(x)d(x)$???.
\end{itemize}
Let $ G_a$ be the $R$-group scheme and $\onsf{Dist}$ be the algebra of distributions of $G_a$. We can regard $\onsf{Dist}$ be the differential complex concentrated in even degrees and thus have zero differential. 

Recall that $\onsf{Dist}=R\langle \mc D\rangle$ has a $R$-basis $\{\mc D^{(n)}\;|\; n\geq 0\}$ with $\mc D^{(0)}=1$ and  has an ideal $\onsf{Dist}^+$. Thus $(\onsf{Dist}, \onsf{Dist}^+, \gamma)$  has a PD structure defined by 
\[
\gamma_n(\mc D^{(d)})=\frac{(dn)!}{n! (d!)^n}\mc D^{(nd)}\frac{(dn)!}{n! (d!)^n}\gamma_{dn}(\mc D^{(1)}).
\]
Furthermore, $\onsf{Dist}$ is a Hopf algebra with comultiplication $\Delta: \onsf{Dist}\to \onsf{Dist}\otimes_R \onsf{Dist}$, counit $\varepsilon: \onsf{Dist}\to R$, and antipode $ S: \onsf{Dist}\to \onsf{Dist}$ defined by 
\begin{align}
\Delta(\mc D^{(n)})&=\sum_{i=0}^n \mc D^{(i)}\otimes_R \mc D^{(n-i)};\\
S(\mc D^{(n)})&=(-1)^n\mc D^{(n)};\\
\varepsilon(\mc D^{(n)})&=\delta_{0, n}.
\end{align}

In fact this is just one special case of the following. 
\begin{theorem} 
Let $ R$ be any commutative ring and $ M$ is any object in $\onsf{Cplx}(R)$, then $(\onsf{TS}_A(M), \onsf{TS}_A^+(M), \gamma)$ is a PD structure with $\gamma_n$ defined as follows. If $x\in \onsf{TS}_A^m(M)$ is homogeneous, then 
\[ \gamma_n(x)=x\otimes\cdots \otimes x
\]
\end{theorem}

Let $(A, I, \delta)$ be a dg  commutative dg ring a PD structure. We assume that $A$ is also $\onsf{Dist}$-algebra if $ A$ is a 
$\onsf{Dist}$-module and the multiplication $A\otimes_R A\to A$ is a $\onsf{Dist}$-module homomorphism and the differential $d: A\to A$ is also a $\onsf{Dist}$-module homomorphism. Thus the homology $H(A, d)$ is also a $\onsf{Dist}$-module homomorphism. 
 
 We say the PD structure $(A, I, \gamma)$ and the PD structure $(\onsf{Dist}, \onsf{Dist}^,\gamma)$ are compatible if
 $\onsf{Dist}I\subseteq I$ and ???

A $\onsf{Dist}$-module $X$ in $ \onsf{Cplx}(R)$ is such that the action $\onsf{Dist}\otimes_R X\to X$ is a differential chain map. 
If $ X$ and $ Y$ are $\onsf{Dist}$-modules, so is $X\otimes_R Y$ with the diagonal action.  If $ A$ is DGA in $ \onsf{Cplx}(R) $ and $ A$ is also a $\onsf{Dist}$-algebra in the sense that the algebra structure $ A\otimes A\to A$ is also a $\onsf{Dist}$-module homomorphism.

Recall that from Stack project Ch23

\begin{definition} 
A strictly graded commutative dg algebra $A$ is called resolving algebra if there is differential complex $X$ in $\in\onsf{cplx}(R)$ such that $ A \cong \onsf{TS}_A(X)$ and that each $ X_n$ is a free $R$-module. 
\end{definition}

The following lemma is from Stack-Project. 
\begin{lemma} Let $S\to R$ be a homomorphism of commutative rings. There exists a factorization

\[ S\to A\to R\]
with the following properties:

$(A, d,\gamma)$ is as in Definition 23.6.5,

$A\to R$ is a quasi-isomorphism (if we endow S with the zero differential),

$A_0=S[xj : j\in J]\to R$ is any surjection of a polynomial ring onto $R$, and

$A=\onsf{TS}_A(X)$  is a resolving $S$-algebra.

The last condition means that $A$ is constructed out of $A_0$ by successively adjoining a set of variables $T$ in each degree $>0$. Moreover, if $S$ is Noetherian and $S\to R$ is of finite type, then $A$ can be taken to have only finitely many generators in each degree.
\end{lemma}

The universal property of the resolving algebra $A$. If $(C, d, \delta)$ is another differential graded $S$-algebra with a divided power structure, then with a subjective homomorphism of $ C\to R$, and there are maps from $X\to C$ as a differential complex, then there is a unique homomorphism $(A, d, \gamma)\to (C, d, \delta)$ of differential graded algebra that factorize to homomorphism $A\to R$.

\section{dg version of operator product expansion over PD rings}

Let $\mf g$ be a dg Lie algebra in the category $ \onsf{Dist}_R$. A homogeneous element $K$ in $\mf g$ is called central if $ [K, a]=0$ for all $a\in \mf g$. 

We assume that $\mf g$ has a filtration of subcomplex 
\begin{align}
\cdots\supseteq \mf g_{(-1)}\supseteq \mf g_{(0)}\supseteq \mf g_{(1)}\cdots 
\end{align}
with $ \cup_i \mf g_{(i)}=\mf g$, $\cap_i \mf g_{(i)}=0$,  $[\mf g_{(i)}, \mf g_{(j)}]\subseteq \mf g_{(i+j)}$.
 
 Let $U(\mf g)$ be the universal enveloping algebra of $\mf g$ in the category of $\onsf{Cplx}(R)$. Then $U(\mf g)$ has a filtrations of left ideals $U(\mf g)\mf g_{(N)}\supseteq U(\mf g)\mf g_{(N+1)}$ and thus defines a linear topology on $U(\mf g)$. For any let $U(\mf g)$-module $ (M, d)$, an element $v$ is called smooth if $\mf g_{(N)}v=0$ for some $N$. Clearly the set of smooth vectors of $M$ is an $\mf g$-submodule. In fact, if $\mf g_{(N)}v=0$, for any $ a\in \mf g_{(n)}$, take $l\geq \max\{N, N-n\}$, then $\mf g_{(l)}av\subseteq a\mf g_{(l)}v+\mf g_{(l+n)}v=0$. We denote by $ M^{(sm)}$ the set of all smooth vectors. We call $M$-smooth if $M^{(sm)}=M$. The category of smooth modules are exactly the category of smooth modules of the completion $ \hat U(\mf g)=\varprojlim_{N}U(\mf g)/U(\mf g)\mf g_{(N)}$, which is a dg-algebra. 
 
 We consider the $ \hat U(\mf g)$-valued field as a homogeneous element
 \[
 a(z)=\sum_{n\in \bb Z}a_{(n)}z^{-n-1} \in \Hom^*_{R}(R[t, t^{-1}], \hat U(\mf gL))
 \]
 where $a_{(n)}=a(z)(t^{n})$ such that for each $N$, $ a_{(n)}\in \hat U(\mf g)\mf g_{(N)}$ for all $ n>>0$.
 
 $\onsf{Fie}(\hat U(\mf g))$ be the set of all fields. 
 
More generally, assume that $(A, d)$ is a dg-algebra with a filtration of left subcomplexes 
\begin{align}
\cdots\supseteq A_{(-1)}\supseteq A_{(0)}\supseteq A_{(1)}\cdots 
\end{align}
with $ \cup_i A_{(i)}=A$, $ \cap_i A_{(i)}=0$. We further assume that for any $ a\in A$, any $n$,   $A_{(n)}a\subseteq A_{(l)} $ for some $l$ (depending on $a$ and $n$).  
 
Let $\hat A=\varprojlim_n A/A_{(n)}$. The completion should be taking degree wise.  This is an  $A$-modules. In fact, it is  s dg-algebra since the map $A\to \Hom^*_R(\hat A, \hat A)$ defined by $a: \hat A\to \hat A$ and $ d(a \hat a)=d(a)\hat a+(-1)^{|a| |\hat a|}ad(\hat a)$. { \color{red} Do I need the completion?}

We can define the smooth vectors in an $A$-module (with respect to filtered structure) and smooth modules. We note that there is chain map 
\[
R[t, t^{-1}]\to R[t, t^{-1}]\hat \otimes_R R[t, t^{-1}]
\]
defined by $\Delta(t^n)=\sum_{i\in \bb Z}(t^i\otimes_{R} t^{n+1-i}$. Here we take 
\[
R[t, t^{-1}]\hat \otimes_R R[t, t^{-1}]=\varprojlim_N (R[t, t^{-1}]\otimes R[t, t^{-1}]/t^nR[t])
\]

We define $\onsf{Fie}(A)$ to be a homogenous element  $a\in Z\Hom_R^*(R[t, t^{-1}], A)$ satisfying the following condition. For each $N\in \bb Z$, one has $a(t^{n})\in A_{(N)}$ for $n>> 0$. 

Given two fields  $a, b\in \onsf{Fie}(A)$ then $a\star b=m(a\otimes)\circ \Delta$ is not defined in general. Thus we define the normal order

We now assume the $A$ is an $ \onsf{Dist}$ Lie algebra, thus $\onsf{Dist}$ acts on $ \hat U(A)$ continuously. $R[t, t^{-1}]$ is also a $\onsf{Dist}$-module. Hence $\onsf{Dist}$ acts on $ \onsf{Fie}(\hat U(A))$. 

The complex $R[t, t^{-1}]$ has a natural decomposition  Rota-Baxter decomposition $R[t]\oplus t^{-1}R[t^{-1}]$. Thus every field $a(z)$ has a natural decomposition $ a(z)=a(z)_++a(z)_-$. Then the norm order is defined by 
\begin{align}
:a(z)b(z):=a(z)_b(z)+b(z)a(z)_+
\end{align}
}

\delete{
\section{The Ext algebras of a ring and of its localisation}
Let $R$ be a commutative Noetherian ring. If $\mf{p}\in \spec(R)$ is a prime ideal $R$, we denote by $R_{\msf p}$ the local ring of $R$ at $\mf p$ and $\mf m_{\mf p}$ the maximal ideal of $R_{\mf p}$ and $ A_{\mf p}=R_{\msf p}/\mf m_{\mf p}$ the residue field.  We want to compare the Ext algebra of $R$ for the $R$-module $R/\mf{p}$ with that of the localisation of $R_{\mf p}$ at $\mf{m}$. We will need the following proposition:

\begin{proposition}\emph{\textbf{(\cite[3.3.10]{Weibel}).}}\label{PropWeibel}
Let $M$ be a finitely generated module over a commutative Noetherian ring $R$. Then for every multiplicative set $S$, all modules $N$, and all $n \in \bb N$, we have
\begin{align*}
S^{-1}\Ext_R^n(M,N) \cong \Ext_{S^{-1}R}^n(S^{-1}M,S^{-1}N).
\end{align*}
\end{proposition}

Using the above proposition, it is then straightforward to prove the following result:

\begin{theorem}\label{ExtofLocalisation}
Let $R$ be a commutative Noetherian ring, and take $\mf{p} \in \on{Spec}(R)$. Denote by $\msf{k}_{\mf{p}}$ the field of fractions of $R/\mf{p}$, and by $\epsilon$ the composition $R \longrightarrow R/\mf{p} \myhookrightarrow \msf{k}_{\mf{p}}$. Finally let $R_{\mf{p}}$ be the localisation of $R$ at $\mf{p}$. Then we have an isomorphism:
\begin{align*}
\Ext_R^*(\msf{k}_{\mf{p}},\msf{k}_{\mf{p}}) \cong \Ext_{R_{\mf{p}}}^*(\msf{k}_{\mf{p}},\msf{k}_{\mf{p}}).
\end{align*}
where $\msf{k}_{\mf{p}}$ is seen as an $R$-module through $\epsilon$.
\end{theorem}

\delete{
\begin{proof}
Set $S=R\backslash \mf{p}$. We see that every element of $\epsilon(S)$ is non-zero in $\msf{k}_{\mf{p}}$, so every element of $\epsilon(S)$ is invertible in $\msf{k}_{\mf{p}}$. The elements of $S^{-1}\msf{k}_{\mf{p}}$ are of the form $\frac{k}{s}$ where $s \in S$ and $r \in \msf{k}_{\mf{p}}$. But we saw that $\epsilon(s)$ is invertible so $\frac{k}{s}=\frac{\epsilon(s)^{-1}k}{1}$. Therefore there is an isomorphism $S^{-1}\msf{k}_{\mf{p}} \cong \msf{k}_{\mf{p}}$. Furthermore, there exists $\epsilon_{\mf{p}}:R_{\mf{p}} \longrightarrow S^{-1}\msf{k}_{\mf{p}} \cong \msf{k}_{\mf{p}}$ such that we have a factorisation $\epsilon=\epsilon_{\mf{p}} \circ \pi$, where $\pi: R  \longrightarrow  R_{\mf{p}}$ sends $r$ to $\frac{r}{1}$. 

By seeing $\msf{k}_{\mf{p}}$ as an $R$-module (respectively $R_{\mf{p}}$-module) through $\epsilon$ (respectively $\epsilon_{\mf{p}}$), Proposition \ref{PropWeibel} and the previous paragraph give:
\begin{align*}
S^{-1}\Ext_R^n(\msf{k}_{\mf{p}},\msf{k}_{\mf{p}}) \cong \Ext_{R_{\mf{p}}}^n(\msf{k}_{\mf{p}},\msf{k}_{\mf{p}}),
\end{align*}
for all $n \in \bb N$. The morphism $\pi$ induces an algebra morphism $\pi_{\mf{p}}: \Ext_{R_{\mf{p}}}^*(\msf{k}_{\mf{p}},\msf{k}_{\mf{p}}) \longrightarrow \Ext_R^*(\msf{k}_{\mf{p}},\msf{k}_{\mf{p}})$, and there is an algebra morphism $\iota_{\mf{p}}: \Ext_{R}^*(\msf{k}_{\mf{p}},\msf{k}_{\mf{p}})  \longrightarrow  S^{-1}\Ext_{R}^*(\msf{k}_{\mf{p}},\msf{k}_{\mf{p}})$ sending $\alpha$ to $1 \otimes \alpha$. We are interested in the composition:
\begin{align*}
\Ext_{R}^*(\msf{k}_{\mf{p}},\msf{k}_{\mf{p}}) \stackrel{\iota_{\mf{p}}}{\longrightarrow} S^{-1}\Ext_{R}^*(\msf{k}_{\mf{p}},\msf{k}_{\mf{p}}) \cong \Ext_{R_{\mf{p}}}^*(\msf{k}_{\mf{p}},\msf{k}_{\mf{p}}) \stackrel{\pi_{\mf{p}}}{\longrightarrow} \Ext_{R}^*(\msf{k}_{\mf{p}},\msf{k}_{\mf{p}}).
\end{align*}

Let $X=0 \longrightarrow \msf{k}_{\mf{p}} \longrightarrow X_n \longrightarrow \dots \longrightarrow X_1 \longrightarrow \msf{k}_{\mf{p}} \longrightarrow 0$ be an $n$-fold exact sequence of $R$-modules. Then $S^{-1}X=0 \longrightarrow \msf{k}_{\mf{p}} \longrightarrow S^{-1}X_n \longrightarrow \dots \longrightarrow S^{-1}X_1 \longrightarrow \msf{k}_{\mf{p}} \longrightarrow 0$ is an $n$-fold exact sequence of $R_{\mf{p}}$-modules. The morphism $\pi_{\mf{p}}$ then sends it to:
\[
0 \longrightarrow \msf{k}_{\mf{p}} \longrightarrow R_{\mf{p}} \otimes_R X_n \longrightarrow \dots \longrightarrow R_{\mf{p}} \otimes_R X_1 \longrightarrow \msf{k}_{\mf{p}} \longrightarrow 0
\]
as $R$-modules. However, we have a natural morphism of exact sequences of $R$-modules:
\begin{center}
  \begin{tikzpicture}[scale=0.9,  transform shape]
  \tikzset{>=stealth}
  
\node (1) at (-6,0){$0$};
\node (2) at (-4,0){$\msf{k}_{\mf{p}}$};
\node (3) at (-2,0){$X_n$};
\node (4) at (0,0){$\dots \dots$};
\node (5) at (2,0){$X_1$};
\node (6) at (4,0){$\msf{k}_{\mf{p}}$};
\node (7) at (6,0){$0$};

\node (8) at (-6,-2){$0$};
\node (9) at (-4,-2){$\msf{k}_{\mf{p}}$};
\node (10) at (-2,-2){$R_{\mf{p}} \otimes_R X_n$};
\node (11) at (0,-2){$\dots \dots$};
\node (12) at (2,-2){$R_{\mf{p}} \otimes_R X_1$};
\node (13) at (4,-2){$\msf{k}_{\mf{p}}$};
\node (14) at (6,-2){$0$};

\node (15) at (-6.8,0){$X=$};
\node (16) at (-7.5,-2){$\pi_{\mf{p}} \circ \iota_{\mf{p}}(X)=$};

\draw [->] (1) --  (2);
\draw [->] (2) --  (3);
\draw [->] (3) --  (4);
\draw [->] (4) --  (5);
\draw [->] (5) --  (6);
\draw [->] (6) --  (7);

\draw [->] (8) --  (9);
\draw [->] (9) --  (10);
\draw [->] (10) --  (11);
\draw [->] (11) --  (12);
\draw [->] (12) --  (13);
\draw [->] (13) --  (14);

\draw [->] (2) --  (9) node[midway, left] {$\on{id}$};
\draw [->] (3) --  (10);
\draw [->] (5) --  (12);
\draw [->] (6) --  (13) node[midway, right] {$\on{id}$};
\end{tikzpicture}
 \end{center}
Therefore the two sequences are equivalent $n$-fold exact sequences of $R$-modules (cf. \cite[Proposition VII.3.1]{Mitchell}). It follows that $\pi_{\mf{p}} \circ \iota_{\mf{p}}(X)=X$, and so $\iota_{\mf{p}}$ is injective.

Now let $Y=0 \longrightarrow \msf{k}_{\mf{p}} \longrightarrow Y_n \longrightarrow \dots \longrightarrow Y_1 \longrightarrow \msf{k}_{\mf{p}} \longrightarrow 0$ be an $n$-fold exact sequence of $R_{\mf{p}}$-modules. Each module can be seen as an $R$-module through the morphism $\pi$. Furthermore, it is known that we have isomorphisms of $R_{\mf{p}}$-modules (cf. \cite[Ch. II, \S 2 Proposition 3]{Bourbaki_commu}):
\[
S^{-1}Y_n \cong Y_n, \quad \dots \dots  \quad ,  S^{-1}Y_1 \cong Y_1, \quad S^{-1}\msf{k}_{\mf{p}} \cong \msf{k}_{\mf{p}}.
\]
Hence by seeing $Y$ as a sequence of $R$-modules, we get $\iota_{\mf{p}}(Y)=Y$, and so $\iota_{\mf{p}}$ is surjective. It follows that $\iota_{\mf{p}}$ is an isomorphism of algebras, proving the theorem.
\end{proof}
}

This theorem tells us that when considering the Yoneda algebra of $R$ for a fraction field $\msf{k}_{\mf{p}}$, we can assume that $R$ is local. Furthermore, this theorem can be used to decompose the Yoneda algebra of the tensor product of commutative rings.

\begin{proposition}
Let $A_1, A_2$ be two finitely generated algebras over an algebraically closed field $\msf{k}$. Let $\mf{m}_1$ (respectively $\mf{m}_2$) be a maximal ideal of $A_1$ (respectively $A_2$). Consider $\msf{k}$ as an $A_1$-module (respectively $A_2$-module) through by $\mf{m}_1$ (respectively $\mf{m}_2$). The space $\msf{k}=A_1/\mf{m}_1 \otimes_{\msf{k}} A_2/\mf{m}_2$ is an $A_1\otimes_{\msf{k}}  A_2$-module, and there exists an algebra isomorphism:
\[
\Ext_{A_1}^*(\msf{k},\msf{k}) \otimes_{\msf{k}}  \Ext_{A_2}^*(\msf{k},\msf{k}) \cong \Ext_{A_1 \otimes_{\msf{k}}  A_2}^*(\msf{k},\msf{k}).
\] 
\end{proposition}

\begin{proof}
We first consider $\msf{k}$ as an $(A_1)_{\mf{m}_1}$-module and an $(A_2)_{\mf{m}_2}$-module. Using K\"unneth tensor product formula, one can directly determine that the tensor product of the $I$-minimal projective resolutions of $\msf{k}$ as $(A_i)_{\mf{m}_i}$-module ($i=1,2$) is an $I$-minimal projective resolution of $\msf{k}$ as an $(A_1)_{\mf{m}_1} \otimes (A_2)_{\mf{m}_2}$-modules. The wedge product (cf. \cite[\textrm{VIII}.4]{MacLane}) is given by the following morphism: 
\[
\psi: \Ext_{(A_1)_{\mf{m}_1}}^*(\msf{k},\msf{k}) \otimes_{\msf{k}}  \Ext_{(A_2)_{\mf{m}_2}}^*(\msf{k},\msf{k}) \longrightarrow \Ext_{(A_1)_{\mf{m}_1} \otimes_{\msf{k}}  (A_2)_{\mf{m}_2}}^*(\msf{k},\msf{k}).
\]
The rings $(A_1)_{\mf{m}_1}$ and $(A_2)_{\mf{m}_2}$ are local, so there exists an $I$-minimal free resolution $X_1$ (respectively $X_2$) of $\msf{k}$ as $(A_1)_{\mf{m}_1}$-module (respectively as an $(A_2)_{\mf{m}_2}$-module) (cf. \cite[Chapter 19]{Eisenbud}). We then use the facts that $X_1 \otimes X_2$ is an $I$-minimal projective resolution of $\msf{k}$ as $(A_1)_{\mf{m}_1} \otimes (A_2)_{\mf{m}_2}$-module, and that each space $X_j^n$, $j=1,2$, has a basis to compute $\psi$ and see that it is an isomorphism of algebras. 
Then by applying Theorem \ref{ExtofLocalisation}, we obtain the following isomorphism of algebras:
\[
\Ext_{A_1}^*(\msf{k},\msf{k}) \otimes_\msf{k} \Ext_{A_2}^*(\msf{k},\msf{k}) \cong \Ext_{(A_1)_{\mf{m}_1} \otimes (A_2)_{\mf{m}_2}}^*(\msf{k},\msf{k}).
\]
Consider the maximal ideal $I=\mf{m}_1 \otimes_\msf{k} A_2 + A_1 \otimes_\msf{k} \mf{m}_2$ of $A_1 \otimes_\msf{k} A_2$. As $\mf{m}_i$ is the unique maximal ideal of $(A_i)_{\mf{m}_i}$, we can define $I_{(\mf{m}_1,\mf{m}_2)}=\mf{m}_1 \otimes_\msf{k} (A_2)_{\mf{m}_2} + (A_1)_{\mf{m}_1} \otimes_\msf{k} \mf{m}_2$ a maximal ideal of $(A_1)_{\mf{m}_1} \otimes (A_2)_{\mf{m}_2}$. We then define the following morphism:
\[
\begin{array}{cccc}
\varphi: & ((A_1)_{\mf{m}_1} \otimes (A_2)_{\mf{m}_2})_{I_{(\mf{m}_1,\mf{m}_2)}} & \longrightarrow &  (A_1 \otimes A_2)_{I} \\[5pt]
             &\displaystyle \frac{\frac{a_1}{s_1}\otimes \frac{a_2}{s_2}}{\frac{r_1}{b_1} \otimes \frac{r_2}{b_2}} & \longmapsto & \displaystyle \frac{a_1b_1 \otimes a_2b_2}{s_1r_1 \otimes s_2r_2}.
\end{array}
\]
By direct computations, one can verify that $\varphi$ is a ring isomorphism. By using the above isomorphism, as well as $\varphi$ and Theorem \ref{ExtofLocalisation}, we obtain the following:
\[
\begin{array}{ccl}
\Ext_{A_1}^*(\msf{k},\msf{k}) \otimes_{\msf{k}}  \Ext_{A_2}^*(\msf{k},\msf{k}) & \cong & \Ext_{(A_1)_{\mf{m}_1} \otimes_{\msf{k}}  (A_2)_{\mf{m}_2}}^*(\msf{k}, \msf{k}), \\
& \cong & \Ext_{((A_1)_{\mf{m}_1} \otimes_{\msf{k}}  (A_2)_{\mf{m}_2})_{I_{(\mf{m}_1, \mf{m}_2)}}}^*(\msf{k},\msf{k}), \\
& \cong & \Ext_{(A_1 \otimes_{\msf{k}}  A_2)_{I}}^*(\msf{k}, \msf{k}), \\
& \cong & \Ext_{A_1 \otimes_{\msf{k}}  A_2}^*(\msf{k}, \msf{k}).
\end{array}
\]
\end{proof}
}

\delete{

\section{Comparison between the Yoneda algebra and its even subalgebra}

In this section, we will determine the conditions on $R$ for the even part $\Ext_R^{ev}(A_x,A_x)=\bigoplus_{m=0}^{\infty}\Ext_R^{2m}(A_x,A_x)$ to be commutative, and then compare them with the commutativity conditions of Corollary \ref{cor:PresentationofExt} for the Yoneda algebra. We have seen in Lemma \ref{AlphaBetaCommute} that the $\beta_p$'s commute with one another and with the $\alpha_i$'s. The problem is then to consider the commutation of $\alpha_i \alpha_j$ with $\alpha_l \alpha_m$.


\begin{lemma}\label{4alphas}
Let $R=A[x_1,\dots, x_n]/(c_1,\dots, c_k)$ be a complete intersection ring, where $c_1, \dots, c_k$ are polynomials in the variables $x_1, \dots, x_n$. We assume that for all $1 \leq p \leq k$, we have $c_p \in \mf{m}_x^2$. Assume that $n \geq 4$. Let $ 1 \leq i, j, l, m \leq n$ be pairwise distinct. Then:
\begin{center}
$\begin{array}{cccl}
\alpha_m \alpha_l \alpha_j \alpha_i: & P_4 & \longrightarrow & \cc \\
           & T_m T_l T_j T_i & \longmapsto & 1, \\
           & S_p S_q & \longmapsto & n_j^{q, i}n_m^{p, l}+n^p_{i, j}n_m^{q, l}+n_m^{p, i}n_l^{q, j}-n_l^{p, i}n_m^{q, j}+n_m^{q, i}n_l^{p, j}-n_l^{q, i}n_m^{p, j}, \\
          & T_{i_1} T_{i_2} S_p & \longmapsto & \left\{\begin{array}{cl}
        -n_l^{p, j} & \text{ if } i_1=i, i_2=m,\\\
       n_m^{p, j} & \text{ if } i_1=i, i_2=l,\\
       n_m^{p, l} & \text{ if } i_1=j, i_2=i,\\
        n_l^{p, i} & \text{ if } i_1=j, i_2=m,\\
       -n_m^{p, i} & \text{ if } i_1=j, i_2=l,\\
        n^p_{i, j} & \text{ if } i_1=m, i_2=l,\\
         0 & \text{otherwise},
        \end{array}\right. \\
           & \text{other} & \longmapsto & 0.
\end{array}$ 
\end{center}
\end{lemma}

\begin{proof}
We have already determined $f=\alpha_j \alpha_i$ before Lemma \ref{RelationsinExt}. We lift this morphism using the following diagram:
\begin{center}
  \begin{tikzpicture}[scale=0.9,  transform shape]
  \tikzset{>=stealth}
  
\node (1) at ( 0,0){$P_3$};
\node (2) at ( 3,0){$P_2$};
\node (3) at ( 0,-2){$P_1$};
\node (4) at ( 3,-2) {$R$};
\node (5) at ( 6,-2){$\cc$};
\node (6) at (8,-2){0};

\node (7) at ( -3,0){$P_4$};
\node (8) at ( -3,-2){$P_2$};

\node (9) at ( -6,0){$\dots$};
\node (10) at ( -6,-2){$\dots$};

\draw [decoration={markings,mark=at position 1 with
    {\arrow[scale=1.2,>=stealth]{>}}},postaction={decorate}] (1) --  (2) node[midway, above] {$\partial_3$};
\draw [decoration={markings,mark=at position 1 with
    {\arrow[scale=1.2,>=stealth]{>}}},postaction={decorate}] (1)  --  (3) node[midway, left] {$\widetilde{f}_1$};

\draw [decoration={markings,mark=at position 1 with
    {\arrow[scale=1.2,>=stealth]{>}}},postaction={decorate}] (2)  --  (4) node[midway, left] {$\widetilde{f}_0$};
\draw [decoration={markings,mark=at position 1 with
    {\arrow[scale=1.2,>=stealth]{>}}},postaction={decorate}] (3)  --  (4) node[midway, above] {$\partial_1$};

\draw [decoration={markings,mark=at position 1 with
    {\arrow[scale=1.2,>=stealth]{>}}},postaction={decorate}] (2)  --  (5) node[midway, above right] {$\alpha_j \alpha_i$};
\draw [decoration={markings,mark=at position 1 with
    {\arrow[scale=1.2,>=stealth]{>}}},postaction={decorate}] (4)  --  (5) node[midway, above] {$\epsilon$};
    
\draw [decoration={markings,mark=at position 1 with
    {\arrow[scale=1.2,>=stealth]{>}}},postaction={decorate}] (5)  --  (6);
    
\draw [decoration={markings,mark=at position 1 with
    {\arrow[scale=1.2,>=stealth]{>}}},postaction={decorate}] (7)  --  (1) node[midway, above] {$\partial_4$};
\draw [decoration={markings,mark=at position 1 with
    {\arrow[scale=1.2,>=stealth]{>}}},postaction={decorate}] (8)  --  (3) node[midway, above] {$\partial_2$};
    
\draw [decoration={markings,mark=at position 1 with
    {\arrow[scale=1.2,>=stealth]{>}}},postaction={decorate}] (9)  --  (7) node[midway, above] {$\partial_5$};
\draw [decoration={markings,mark=at position 1 with
    {\arrow[scale=1.2,>=stealth]{>}}},postaction={decorate}] (10)  --  (8) node[midway, above] {$\partial_3$};    
\draw [decoration={markings,mark=at position 1 with
    {\arrow[scale=1.2,>=stealth]{>}}},postaction={decorate}] (7)  --  (8) node[midway, left] {$\widetilde{f}_2$};
\end{tikzpicture}
 \end{center}

We can define $\widetilde{f}_0$ such that $\widetilde{f}_0(T_j T_i)=1$, $\widetilde{f}_0(S_p)=n^p_{i, j}$ for all $p$, and $\widetilde{f}_0(\text{other})=0$. Through direct computations, one can show that the other morphisms of this diagram are given by:
\begin{center}
$\begin{array}[t]{cccl}
\widetilde{f}_1: & P_3 & \longrightarrow & P_1 \\
       & T_{i_1} T_j T_i & \longmapsto & T_{i_1} \text{ for all } {i_1} \neq i, j, \\
       & T_{i_1} S_p & \longmapsto & \left\{\begin{array}{cl}
        n^p_{i, j}T_i +C_{p, j} & \text{ if } {i_1}=i,\\[5pt]
        n^p_{i, j}T_j -C_{p, i} & \text{ if } {i_1}=j, \\[5pt]
        n^p_{i, j}T_{i_1}  & \text{ if } {i_1} \neq i, j,
        \end{array}\right. \\
        & \text{other} & \longmapsto & 0,
\end{array}$
\end{center}
and
\[ 
\begin{array}[t]{cccl}
\widetilde{f}_2: & P_4 & \longrightarrow & P_2 \\
       & T_{i_1} T_{i_2} T_j T_i & \longmapsto & T_{i_1} T_{i_2} \text{ for all } i_1, i_2 \neq i, j, \\
       & S_q S_p & \longmapsto & n^p_{i, j}S_q+n_{i,j}^{q}S_p+C_{q, i} C_{p, j}+C_{p, i} C_{q, j}, \\
       & T_{i_1} T_{i_2} S_p & \longmapsto & \left\{\begin{array}{cl}
        n^p_{i, j}T_j T_i +T_j C_{p, j}+T_i C_{p, i}+S_p & \text{ if } i_1=j, i_2=i,\\[5pt]
        n^p_{i, j}T_i T_{i_2} -T_{i_2} C_{p, j} & \text{ if } i_1=i, i_2 \neq j, \\[5pt]
        n^p_{i, j}T_j T_{i_2}  +T_{i_2} C_{p, i}& \text{ if } i_1=j, i_2 \neq i, \\[5pt]
        n^p_{i, j}T_{i_1} T_{i_2}  & \text{ if } i_1, i_2 \neq i, j,
        \end{array}\right. \\
        & \text{other} & \longmapsto & 0.
\end{array} 
\]
By our choice of notation, we have $C_{p, i}=\sum_{j=1}^n n^p_{i, j}T_j+\sum_{j=1}^n m_{i,j}^{p}T_j$, so the part of $C_{p, i} C_{q, j}$ with scalar coefficients is $\sum_{i_1 \neq i_2} n_{i,i_1}^{p} n_{j,i_2}^{q}T_{i_1} T_{i_2}$. Likewise we can express $T_j C_{p, j}$ using the $T_p T_q$'s. By computing the composition $\alpha_m \alpha_l \circ \widetilde{f}_2$ on each basis element of $P_4$ and taking into account that $i, j, l, m$ are pairwise distinct, we obtain the desired result.
\end{proof}

We can now investigate the commutativity of the even part of the Yoneda algebra.

\begin{proposition}\label{ExtCommutativeVariablesBigger3}
Let $R=A[x_1,\dots, x_n]/(c_1,\dots, c_k)$ be a complete intersection ring, where $c_1, \dots, c_k$ are polynomials in the variables $x_1, \dots, x_n$. We assume that for all $1 \leq p \leq k$, we have $c_p \in \mf{m}_x^2$, and that $n \geq 3$. Then:
\begin{align*}
\begin{array}{rcl}
\Ext_R^{ev}(\msf k,\msf k) \text{ is commutative } & \Longleftrightarrow & \left\{\begin{array}{l}
c_p \in \mf{m}_x^3 \text{ for all } 1 \leq p \leq k \text{ if } \on{char}\msf k \neq 2, \\[5pt]
c_p \in (\mathfrak{m}_x^3,x_i^2 \ | \ 1 \leq i \leq n) \text{ if } \on{char}\msf k = 2.
\end{array}\right.
\end{array}
\end{align*}

\end{proposition}

\begin{proof}
Let us assume that $n \geq 4$. Then there exist $1 \leq i, j, l, m \leq n$ pairwise distinct, and by Lemma \ref{4alphas}, we have:
\[
\begin{cases}
\alpha_m \alpha_l \alpha_j \alpha_i(T_i T_m S_p)=-n_{j,l}^{p}, \\[5pt]
\alpha_m \alpha_l \alpha_j \alpha_i(T_j T_l S_p)=-n_{i,m}^{p}.
\end{cases}
\]
By applying the permutations $m \leftrightarrow j$ and $l \leftrightarrow i$ in the previous equalities, we obtain:
\[
\begin{cases}
\alpha_j \alpha_i \alpha_m \alpha_l(T_l T_j S_p)=-n_{m,i}^{p}, \\[5pt]
\alpha_j \alpha_i \alpha_m \alpha_l(T_m T_i S_p)=-n_{l,j}^{p}.
\end{cases}
\]
As we want the algebra $\Ext_R^{ev}(\msf k,\msf k)$ to be commutative, we therefore need $n_{j,l}^{p}=-n_{l,j}^{p}$ and $n_{i,m}^{p}=-n_{m,i}^{p}$ for all pairwise distinct integers $i, j, l, m$ and for all $p$. It follows that $n_{i,j}^{p}+n_{j,i}^{p}=0$ for all $i \neq j$ and for all $p$.

Now let $1 \leq i, j, m \leq n$ be three pairwise distinct integers. We need to compare $\alpha_m \alpha_j \alpha_j \alpha_i$ and $\alpha_j \alpha_i \alpha_m \alpha_j$. Based on the computations of Lemma \ref{4alphas}, we have:
\[
\begin{cases}
\begin{array}{cccc}
\alpha_m \alpha_j \alpha_j \alpha_i: & T_i T_m S_p & \longmapsto & -n_{j,j}^{p},\\
                                                          & \text{other}    & \longmapsto & 0,
\end{array} \\
\begin{array}{cccc}
\alpha_j \alpha_i \alpha_m \alpha_j: & T_m T_i S_p & \longmapsto & -n_{j,j}^{p},\\
                                                          & \text{other}    & \longmapsto & 0.
\end{array}
\end{cases}
\]
To get the commutativity of $\Ext_R^{ev}(\msf k_x,\msf k_x)$ we need $2n^p_{j,j}=0$ for all $j$ and all $p$. If $\on{char}\msf k=2$, it is true, and if $2$ is invertible in $\msf k$, then we get $n^p_{j,j}=0$ for all $j$. Hence, for all $p$ and $i \neq j$, we have $n_{i,j}^{p}+n^p_{j,i}=0$ and if $\on{char}\msf k \neq 2$ we have $n^p_{j, j}=0$. This proves one direction of the equivalence. The other direction is a direct consequence of Corollary \ref{cor:PresentationofExt}.

Let us now assume that $n=3$, and let $1 \leq i, j, m \leq 3$ be three distinct integers. In the proof of Lemma \ref{4alphas}, if we compose $\alpha_m \alpha_i$ with $g_3$, we obtain:
\[ 
\begin{array}[t]{cccl}
\alpha_m \alpha_i \alpha_j \alpha_i: & P_4 & \longrightarrow & \cc \\
       & T_j T_i S_p & \longmapsto & 0, \\
       & T_i T_m S_p & \longmapsto & -n^p_{i, j}-n^p_{j, i} , \\
       & T_j T_m S_p & \longmapsto & n^p_{i, i}, \\
       & S_p S_q & \longmapsto & n_{i,j}^{q}n_{i,m}^{p}+n^p_{i, j}n_{i,m}^{q}+n_{i,m}^{p}n_{j,i}^{q}-n^p_{i, i}n_{j,m}^{q}+n_{i,m}^{q}n^p_{j, i}-n_{i,i}^{q}n_{j,m}^{p}, \\
        & \text{other} & \longmapsto & 0.
\end{array} 
\]
By symmetry, we have
\[ 
\begin{array}[t]{cccl}
\alpha_j \alpha_i \alpha_m \alpha_i: & P_4 & \longrightarrow & \cc \\
       & T_m T_i S_p & \longmapsto & 0, \\
       & T_i T_j S_p & \longmapsto & -n_{i,m}^{p}-n_{m,i}^{p} , \\
       & T_m T_j S_p & \longmapsto & n^p_{i, i}, \\
       & S_p S_q & \longmapsto & n_{i,m}^{q}n^p_{i, j}+n_{i,m}^{p}n_{i,j}^{q}+n^p_{i, j}n_{m,i}^{q}-n^p_{i, i}n_{m,j}^{q}+n_{i,j}^{q}n_{m,i}^{p}-n_{i,i}^{q}n_{m,j}^{p}, \\
        & \text{other} & \longmapsto & 0.
\end{array} 
\]
If $\Ext_R^{ev}(\msf k_x,\msf k_x)$ is commutative, then we have $n^p_{i, j}+n^p_{j, i}=0=2n^p_{i, i}$ for all $i \neq j$ and $p$, which proves the first direction of the equivalence. As above, the other direction is a direct consequence of Corollary \ref{cor:PresentationofExt}.
\end{proof}

\begin{proposition}
Let $R=\msf k[x_1, \dots, x_n]/(c_1,\dots, c_k)$ be a complete intersection ring, where $c_1, \dots, c_k$ are polynomials in the variables $x_1, \dots, x_n$. We assume that for all $1 \leq p \leq k$, we have $c_p \in \mf{m}_x^2$.  If $n \leq 2$, then $\Ext_R^{ev}(\msf k_x,\msf k_x)$ is commutative.
\end{proposition}

\begin{proof}
Assume that $n=2$. Let $1 \leq i \neq j \leq 2$. The elements of degree $2$ are $\alpha_i \alpha_j$, $\alpha_j \alpha_i$, $\alpha_i^2$, $\alpha_j^2$ and the $\beta_p$'s. We know from Lemma \ref{AlphaBetaCommute} that the $\beta_p$'s commute with everything. We can then verify by direct computations that the other elements of degree $2$ commute with one another. According to Corollary \ref{cor:PresentationofExt}, $\Ext_R^{2}(\msf k_x,\msf k_x)$ generates $\Ext_R^{ev}(\msf k_x, \msf k_x)$. As these generators commute, it follows that $\Ext_R^{ev}(\msf k_x, \msf k_x)$ is commutative.

If there is only one variable in $R$, then $R=\msf k[x]$ or $R=\msf k[x]/(f(x))$ where $f(x) \in (x^2)$. In both cases, we can directly check that the even part of the Yoneda algebra is commutative. 
\end{proof}

We can regroup all we found into the following theorem:

\begin{theorem}\label{ExtevCommutativeExtgradedCommutative}
Let $R=\msf k[x_1,\dots, x_n]/(c_1,\dots, c_k)$ be a complete intersection ring, where $c_1, \dots, c_k$ are polynomials in the variables $x_1, \dots, x_n$. We assume that for all $1 \leq p \leq k$, we have $c_p \in \mf{m}_x^2$. Then:
\begin{itemize}[leftmargin=*]\setlength\itemsep{5pt}
\item If $n=1, 2$, then $\Ext_R^{ev}(\msf k_x, \msf k_x)$ is commutative. 
\item If $n \geq 3$, $\begin{array}[t]{rcl}
\Ext_R^{ev}(\msf k_x, \msf k_x) \text{ is commutative } & \Longleftrightarrow & \left\{\begin{array}{l}
c_p \in \mf{m}_x^3 \text{ for all } 1 \leq p \leq k \text{ if } \on{char}\msf k \neq 2, \\[5pt]
c_p \in (\mathfrak{m}_x^3,x_i^2 \ | \ 1 \leq i \leq n) \text{ if } \on{char}\msf k = 2.
\end{array}\right. \\[20pt]
 &\Longleftrightarrow & \Ext_R^{*}(\msf k_x, \msf k_x) \text{ is graded commutative}.
\end{array}$
\end{itemize}
\end{theorem}

In particular, we see that the graded commutativity of $\Ext_R^*(\msf k_x, \msf k_x)$ when there are three or more variables is completely determined by the commutativity of the even degree subalgebra. However this is not true when there are less than three variables. For example, if $R=\msf k[x,y]/(x^2)$, then Corollary \ref{cor:PresentationofExt} states that
\[
\Ext_R^*(\msf k_x, \msf k_x)=\msf k \langle \alpha_x, \alpha_y \rangle / \left(\begin{array}{c} \alpha_x\alpha_y+\alpha_y\alpha_x  \\ \alpha_y^2 \end{array}\right).
\]
While $\Ext_R^{ev}(\msf k, \msf k_x)=\msf k[\alpha_x^2, \alpha_x\alpha_y]$ is commutative, we see that in $\on{char}k \neq 2$, the algebra $\Ext_R^*(\msf k_x, \msf k_x)$ is not graded commutative as $\alpha_x^2 \neq 0$.

\section{The dimensions of the summands of the Yoneda algebra}

We start with the same hypothesis as the previous section, i.e., $R=\cc[x_1,\dots, x_n]/(c_1,\dots, c_k)$ is a complete intersection ring such that for all $1 \leq p \leq k$, we have $c_p \in \mf{m}_x^2$. We are interested in the dimension of $\Ext_R^m(\cc,\cc)$, and how it is affected by the number of variables and the number of relations in $R$.

\begin{proposition}
Let $R$ be as above. We have the to following cases:
\begin{itemize}\setlength\itemsep{0.5em}
\item If $k=1$, then $\on{dim} \Ext_R^m(\cc,\cc)$ strictly increases until $m=n-1$, then it is constant.
\item If $k \geq 2$, then $\on{dim} \Ext_R^m(\cc,\cc)$ strictly increases everywhere.
\end{itemize}
\end{proposition}

\begin{proof}
We know that a basis of $\Ext_R^m(\cc,\cc)$ is given by the $(T_1^{a_1} \dots T_n^{a_n}S_1^{(b_1)} \cdots S_k^{(b_k)})^*$ where $a_i \leq 1$, $\sum_{i=1}^n a_i+2\sum_{p=1}^k b_p=m$. We define the following maps of sets:
\begin{align*}
\text{For }m \geq 1, \quad  \begin{array}[t]{ccc}
\text{Basis of } \Ext_R^{2m}(\cc,\cc) & \longrightarrow & \text{Basis of } \Ext_R^{2m+1}(\cc,\cc) \\
(T_{i_1} \dots T_{i_{2r}}S_1^{(b_1)} \cdots S_k^{(b_k)})^* & \longmapsto & \left\{\begin{array}{cl}
        (T_1 T_{i_1} \dots T_{i_{2r}}S_1^{(b_1)} \cdots S_k^{(b_k)})^* & \text{ if } i_1 \geq 2,\\[5pt]
        (T_{i_2} \dots T_{i_{2r}}S_1^{b_1+1} \dots S_k^{b_k})^* & \text{ if } i_1=1, 
         \end{array}\right. \\
 (S_1^{(b_1)} \cdots S_k^{(b_k)})^* & \longmapsto & (T_1 S_1^{(b_1)} \cdots S_k^{(b_k)})^*.
\end{array}
\end{align*}
and
\begin{align*}
\text{for }m \geq 0, \quad  \begin{array}[t]{ccc}
\text{Basis of } \Ext_R^{2m+1}(\cc,\cc) & \longrightarrow & \text{Basis of } \Ext_R^{2m+2}(\cc,\cc) \\
(T_{i_1} \dots T_{i_{2r+1}}S_1^{(b_1)} \cdots S_k^{(b_k)})^* & \longmapsto & \left\{\begin{array}{cl}
        (T_1 T_{i_1} \dots T_{i_{2r+1}}S_1^{(b_1)} \cdots S_k^{(b_k)})^* & \text{ if } i_1 \geq 2,\\[5pt]
        (T_{i_2} \dots T_{i_{2r+1}}S_1^{b_1+1} \dots S_k^{b_k})^* & \text{ if } i_1=1.
         \end{array}\right.
\end{array}
\end{align*}
We see that every element of the basis of $\Ext_R^{m}(\cc,\cc)$ contributes to a unique element of the basis of $\Ext_R^{m+1}(\cc,\cc)$, meaning that $\Ext_R^{m+1}(\cc,\cc) \geq \Ext_R^{m}(\cc,\cc)$ for all $m$. We now need to look at two cases separately. 

Assume $k=1$. If $1 \leq m \leq n-2$, then $\Ext_R^{m+1}(\cc,\cc)$ contains $(T_{2} \dots T_{m+2})^*$, which does not come from any element of the basis of $\Ext_R^{m}(\cc,\cc)$. Furthermore, $\on{dim} \Ext_R^{1}(\cc,\cc)=n > 1 = \on{dim} \Ext_R^{0}(\cc,\cc)$, so until $m=n-1$, the dimension is strictly increasing. If $m \geq n-1$, then any element of the basis of $\Ext_R^{m+1}(\cc,\cc)$ is either of the form $(T_1 \dots T_n)^*$, or of the from $(T_{i_1} \dots T_{i_{r}}S_1^{b_1})^*$ with $b_1 \geq 1$. In the first case, it comes from the element $(T_2 \dots T_n)^* \in \Ext_R^{m}(\cc,\cc)$. In the second case, if $i_1=1$, it comes from $(T_{i_2} \dots T_{i_{r}}S_1^{b_1})^* \in \Ext_R^{m}(\cc,\cc)$, and if $i_1 \geq 2$, then it comes form $(T_1T_{i_1} \dots T_{i_{r}}S_1^{b_1-1})^* \in \Ext_R^{m}(\cc,\cc)$. Therefore the maps given above are bijections between the basis of $\Ext_R^{m}(\cc,\cc)$ and $\Ext_R^{m+1}(\cc,\cc)$. The first statement of the proposition follows.

Assume $k \geq 2$. Depending on the parity of $m \geq 1$, $\Ext_R^{m+1}(\cc,\cc)$ contains either $T_2S_2^{b_2}$ are $S_2^{b_2}$ for some $b_2 \geq 1$. However, this element does not come from any element of the basis of $\Ext_R^{m}(\cc,\cc)$, meaning that $\on{dim} \Ext_R^{m}(\cc,\cc)$ strictly increases when $m \geq 1$. As $\on{dim} \Ext_R^{1}(\cc,\cc)=n > 1 = \on{dim} \Ext_R^{0}(\cc,\cc)$, the second statement of the proposition is proved.
\end{proof}

We now give a formula for the dimension of the summands of the Yoneda algebra.


\begin{theorem}\label{DimOfExt}
Let $R=\cc[x_1,\dots, x_n]/(c_1,\dots, c_k)$ be a complete intersection ring such that for all $1 \leq p \leq k$, we have $c_p \in \mf{m}_x^2$. We have:
\[
\sum_{m=0}^\infty (\on{dim} \Ext_R^m(\cc, \cc))t^m=\frac{(1+t)^n}{(1-t^2)^k}.
\]
The Gelfand-Kirillov dimension of $\Ext_R^*(\cc, \cc)$ is:
\[
\on{GKdim} \Ext_R^*(\cc, \cc)=k.
\]
For any $m \in \bb N$, we have:
\begin{align*}
\on{dim} \Ext_R^{m}(\cc,\cc)= \displaystyle \sum_{l=0}^{\lfloor \frac{m}{2} \rfloor} \binom{n}{2l+\overline{m}} \binom{\lfloor \frac{m}{2} \rfloor-l+k-1}{k-1},
\end{align*}
where $\overline{m}$ is the remainder of $m$ modulo $2$.
\end{theorem}

\begin{proof}
Let $R$ be as in the statement. We have already seen at the beginning of Section \ref{sec:Yoneda_CI} that $\Ext_R^m(\cc, \cc)=\on{Hom}_R(P_m, \cc)$. Therefore a basis of $\Ext_R^m(\cc, \cc)$ is given by the $(T_1^{a_1} \dots T_n^{a_n}S_1^{(b_1)} \cdots S_k^{(b_k)})^*$ where $0 \leq a_i \leq 1$, $b_p \geq 0$, and $\sum_{i=1}^n a_i+2\sum_{p=1}^k b_p=m$. It follows that we have the following isomorphism of vector spaces:
\[
\begin{array}{rcl}
\Ext_R^{*}(\cc,\cc) & \cong & \cc[S_1, \dots, S_k] \otimes \cc[T_1, \dots, T_n]/(T_1^2, \dots, T_n^2), \\[5pt]
 & \cong & \cc[S]^{\otimes k} \otimes (\cc[T]/(T^2))^{\otimes n}.
 \end{array}
\]
We write $g_A(t)$ for the generating function of a graded algebra $A$. It follows that $g_{\Ext_R^*(\cc, \cc)}(t)=g_{\cc[S]}(t)^k . g_{\cc[T]/(T^2)}(t)^n=\frac{(1+t)^n}{(1-t^2)^k}$ as $|T|=1$ and $|S|=2$. Finally, in the case of finitely generated algebras, the Gelfand-Kirillov dimension transforms tensor products into sums (cf. \cite[Proposition 3.11]{Krause-Lenagan}), so $\on{GKdim} \Ext_R^*(\cc, \cc)=k \on{GKdim} \cc[S]+n \on{GKdim} (\cc[T]/(T^2))=k$.

We first look at $\Ext_R^{2m}(\cc, \cc)$. As $2m$ is even, there is an even number of $T_i$'s in each element of the basis. 

If there are no $T_i$'s, then we need to find all solutions to $\sum_{p=1}^k b_p=m$, which is equivalent to finding all monomials $X_1^{b_1} \dots X_k^{b_k}$ of degree $m$ in $k$ variables. It is known that this number is $\binom{m+k-1}{k-1}$. 

If there are two $T_i$'s, then there are $\binom{n}{2}$ ways to choose them. Then we need to find the solutions to $\sum_{p=1}^k b_p=m-1$, of which there are $\binom{m-1+k-1}{k-1}$. 

If there are $2l$ $T_i$'s, $l \geq 1$, then there are $\binom{n}{2l}$ ways to choose them. Then we need to find the solutions to $\sum_{p=1}^k b_p=m-l$, of which there are $\binom{m-l+k-1}{k-1}$. 

We cannot have $l >m$, otherwise we would obtain a monomial of degree at least $2l > 2m$, which is not possible because we are in $\Ext_R^{2m}(\cc, \cc)$. We have thus found that $\on{dim} \Ext_R^{2m}(\cc, \cc)=\sum_{l=0}^m \binom{n}{2l} \binom{m-l+k-1}{k-1}=\sum_{l=0}^{\lfloor \frac{2m}{2} \rfloor} \binom{n}{2l} \binom{\lfloor \frac{2m}{2} \rfloor-l+k-1}{k-1}$.

The same strategy is used for $\Ext_R^{2m+1}(\cc, \cc)$. As $2m+1$ is odd, there is an odd number of $T_i$'s in each element of the basis. Thus if there are $2l+1$ $T_i$'s, $l \geq 0$, there are $\binom{n}{2l+1}$ ways to choose them. Then we need to find the solutions to $\sum_{p=1}^k b_p=m-l$, of which there are $\binom{m-l+k-1}{k-1}$. We cannot have $l >m$, otherwise we would obtain a monomial of degree at least $2l+1 > 2m+1$, which is not possible because we are in $\Ext_R^{2m+1}(\cc, \cc)$. We have thus found that $\on{dim} \Ext_R^{2m+1}(\cc, \cc)=\sum_{l=0}^{m} \binom{n}{2l+1} \binom{m-l+k-1}{k-1}=\sum_{l=0}^{\lfloor \frac{2m+1}{2} \rfloor} \binom{n}{2l+1} \binom{\lfloor \frac{2m+1}{2} \rfloor-l+k-1}{k-1}$.
\end{proof}

Using the formula obtained in the theorem, we can give a more precise characterisation of the evolution of $\Ext_R^m(\cc, \cc)$ for small values of $k$. We will need the following lemma.

\begin{lemma}\label{FormulasBinom}
Let $n$ be a positive integer. We have the following formulas:
\begin{itemize}\setlength\itemsep{0.5em}
\item For $n \geq 2$, $\displaystyle \sum_{l \geq 0}\left(\binom{n}{2l}-\binom{n}{2l+1} \right)l=2^{n-2}$.
\item For $n \geq 3$, $\displaystyle \sum_{l \geq 0}\left(\binom{n}{2l}-\binom{n}{2l+1} \right)l^2=(n-1)2^{n-3}$.
\end{itemize}
\end{lemma}

\begin{proof}
$\bullet$ We know that for a variable $x$ we have $(1+x)^n=\sum_{i=0}^n\binom{n}{i}x^i$. The derivative of this expression gives $n(1+x)^{n-1}=\sum_{i=0}^n\binom{n}{i}ix^{i-1}$. So for $x=-1$, we get 
\begin{align*}
\begin{array}{rcl}
0=\sum_{i=0}^n\binom{n}{i}i(-1)^{i-1} & = &\sum_{l \geq 0}\binom{n}{2l+1}(2l+1)-\sum_{l \geq 0}\binom{n}{2l}2l,\\[5pt]
 & = & 2\sum_{l \geq 0}\left(\binom{n}{2l+1}-\binom{n}{2l}\right)l+\sum_{l \geq 0}\binom{n}{2l+1}, \\[5pt]
 & = &2\sum_{l \geq 0}\left(\binom{n}{2l+1}-\binom{n}{2l}\right)l+2^{n-1}. 
\end{array}
\end{align*}
Therefore we obtain $\sum_{l \geq 0}\left(\binom{n}{2l}-\binom{n}{2l+1}\right)l=2^{n-2}$. By doing the same reasoning for $x=1$, we get $n2^{n-1}=2\sum_{l \geq 0}\left(\binom{n}{2l+1}+\binom{n}{2l}\right)l+2^{n-1}$. So $\sum_{l \geq 0}\left(\binom{n}{2l+1}+\binom{n}{2l}\right)l=(n-1)2^{n-2}$. It follows that $\sum_{l \geq 0}\binom{n}{2l+1}l=(n-2)2^{n-3}$.

$\bullet$ By differentiating the expression given above twice, we obtain $n(n-1)(1+x)^{n-2}=\sum_{i=0}^n\binom{n}{i}i(i-1)x^{i-2}$.  So for $x=-1$, we have $0=\sum_{i=0}^n\binom{n}{i}i(i-1)(-1)^{i-2}=\sum_{i=0}^n\binom{n}{i}i^2(-1)^{i}-\sum_{i=0}^n\binom{n}{i}i(-1)^{i}=\sum_{i=0}^n\binom{n}{i}i^2(-1)^{i}$, according to the previous paragraph. We can verify that:
\begin{align*}
\begin{array}{rcl}
0=\sum_{i=0}^n\binom{n}{i}i^2(-1)^i & = &\sum_{l \geq 0}\binom{n}{2l}(2l)^2-\sum_{l \geq 0}\binom{n}{2l+1}(2l+1)^2,\\[5pt]
 & = & 4\sum_{l \geq 0}\left(\binom{n}{2l}-\binom{n}{2l+1}\right)l^2-4\sum_{l \geq 0}\binom{n}{2l+1}l-\sum_{l \geq 0}\binom{n}{2l+1}, \\[5pt]
 & = & 4\sum_{l \geq 0}\left(\binom{n}{2l}-\binom{n}{2l+1}\right)l^2-4(n-2)2^{n-3}-2^{n-1}. 
\end{array}
\end{align*}
Therefore we obtain $\sum_{l \geq 0}\left(\binom{n}{2l}-\binom{n}{2l+1}\right)l^2=(n-1)2^{n-3}$. 

\end{proof}

\begin{lemma}\label{m>n-1}
For $m \geq n-1$, the integer $2 \lfloor \frac{m}{2} \rfloor +\overline{m}$ is bigger or equal to the biggest integer $s$ such that $s \leq n$ and $s \equiv \overline{m} \ \on{mod} 2$. 
\end{lemma}

\begin{proof}
The proof is done by direct computations. We split the question into four subproblems, depending on the parities of $m$ and $n$. We check directly that in those three cases, the statement is verified.
\end{proof}

A consequence of the previous lemma, is that, when $m \geq n-1$, the formula given in Theorem \ref{DimOfExt} can be taken over all $l \geq 0$. This is because $n >0$, so $\binom{n}{r}=0$ when $r >n$.

\begin{corollary}\label{FormulasFork=1,2,3}
Let $R$ be as in Theorem \ref{DimOfExt}. Then:
\begin{itemize}\setlength\itemsep{0.5em}
\item If $k=1$ and $m \geq n-1$, then $\on{dim} \Ext_R^{m}(\cc,\cc)=2^{n-1}$.
\item If $k=2<n$ and $m \geq n-3$, then $\on{dim} \Ext_R^{m}(\cc,\cc)=2^{n-2}(m-n+3)+\on{dim} \Ext_R^{n-3}(\cc,\cc)$. \\
 If $k=2=n$ and $m \geq n-2=0$, then $\on{dim} \Ext_R^{m}(\cc,\cc)=m+1$.
\item If $k=3$ and $m \geq n-2$, then $\on{dim} \Ext_R^{m}(\cc,\cc)=2^{n-3}\frac{m(m-1)-(n-2)(n-3)}{2} +(m-n+2)(7-n)2^{n-4}+\on{dim} \Ext_R^{n-2}(\cc,\cc)$.
\end{itemize}
\end{corollary}

\begin{proof}
To lighten notation, we will write $f(m)=\on{dim} \Ext_R^{m}(\cc,\cc)$ during this proof.

\underline{First case:} $k=1$. Based on Theorem \ref{DimOfExt}, we have $f(m)=\sum_{l=0}^{\lfloor \frac{m}{2} \rfloor} \binom{n}{2l+\overline{m}}$. If $m \geq n-1$, then by Lemma \ref{m>n-1} we have $f(m)=\sum_{l \geq 0} \binom{n}{2l+\overline{m}}=2^{n-1}$, whether $m$ is even or odd. 

\underline{Second case:} $k=2$ and $n >k$. According to Theorem \ref{DimOfExt}, we have $f(m)=\sum_{l=0}^{\lfloor \frac{m}{2} \rfloor} \binom{n}{2l+\overline{m}}(\lfloor \frac{m}{2} \rfloor-l+1)$. 

$\bullet$ When $m \geq n-1$, by Lemma \ref{m>n-1} we see that $f(m)=\sum_{l \geq 0} \binom{n}{2l+\overline{m}}(\lfloor \frac{m}{2} \rfloor-l+1)=2^{n-1}(\lfloor \frac{m}{2} \rfloor+1)-\sum_{l \geq 0} \binom{n}{2l+\overline{m}}l$. It follows that $f(m+1)-f(m)=2^{n-1}(\lfloor \frac{m+1}{2} \rfloor-\lfloor \frac{m}{2} \rfloor)+\sum_{l \geq 0}\left(\binom{n}{2l+\overline{m}}-\binom{n}{2l+\overline{m+1}}\right) l$. \\
If $m=2m'$, then $\lfloor \frac{m+1}{2} \rfloor-\lfloor \frac{m}{2} \rfloor=m'-m'=0$. Thus $f(m+1)-f(m)=\sum_{l \geq 0}\left(\binom{n}{2l}-\binom{n}{2l+1}\right) l=2^{n-2}$ according to Lemma \ref{FormulasBinom}. \\
If $m=2m'+1$, then $\lfloor \frac{m+1}{2} \rfloor-\lfloor \frac{m}{2} \rfloor=m'+1-m'=1$. Thus $f(m+1)-f(m)=2^{n-1}+\sum_{l \geq 0}\left(\binom{n}{2l+1}-\binom{n}{2l}\right) l=2^{n-1}-2^{n-2}=2^{n-2}$.

$\bullet$ Set $m=n-3$. We first handle the case $n=2n'$. In this case, we have $\lfloor \frac{m+1}{2} \rfloor =\lfloor \frac{2n'-2}{2} \rfloor=n'-1$, and $\lfloor \frac{m}{2} \rfloor =\lfloor \frac{2n'-3}{2} \rfloor=n'-2$. We thus obtain:
\begin{align*}
\begin{array}{rcl}
f(m+1)-f(m) & = &\displaystyle \sum_{l=0}^{n'-1} \binom{n}{2l}(n'-1-l+1)-\sum_{l=0}^{n'-2} \binom{n}{2l+1}(n'-2-l+1),\\[10pt]
 & = &\displaystyle n'\sum_{l=0}^{n'-1} \binom{n}{2l}-(n'-1)\sum_{l=0}^{n'-2} \binom{n}{2l+1} \\
  & &\displaystyle +\sum_{l=0}^{n'-2} \binom{n}{2l+1}l-\sum_{l=0}^{n'-1} \binom{n}{2l}l, \\[10pt]
 & = &\displaystyle n'(2^{n-1}-1)-(n'-1)(2^{n-1}-n) \\
  & &\displaystyle +\sum_{l \geq 0}\left(\binom{n}{2l+1}-\binom{n}{2l}\right)l-n(n'-1)+n', \\[10pt]
 & = &\displaystyle 2^{n-1}+\sum_{l \geq 0}\left(\binom{n}{2l+1}-\binom{n}{2l}\right)l, \\[10pt]
 & = &\displaystyle 2^{n-1}-2^{n-2}, \\[10pt]
  & = &\displaystyle 2^{n-2}.
\end{array}
\end{align*}
We now assume $n=2n'+1$. We have verify that $\lfloor \frac{m+1}{2} \rfloor =\lfloor \frac{2n'-1}{2} \rfloor=n'-1$, and $\lfloor \frac{m}{2} \rfloor =\lfloor \frac{2n'-2}{2} \rfloor=n'-1$. We thus obtain:
\begin{align*}
\begin{array}{rcl}
f(m+1)-f(m) & = &\displaystyle \sum_{l=0}^{n'-1} \binom{n}{2l+1}(n'-1-l+1)-\sum_{l=0}^{n'-1} \binom{n}{2l}(n'-1-l+1),\\[10pt]
 & = &\displaystyle n'\sum_{l=0}^{n'-1} \binom{n}{2l+1}-n'\sum_{l=0}^{n'-1} \binom{n}{2l} \\
  & &\displaystyle +\sum_{l=0}^{n'-1} \binom{n}{2l}l-\sum_{l=0}^{n'-1} \binom{n}{2l+1}l, \\[10pt]
 & = &\displaystyle n'(2^{n-1}-1)-n'(2^{n-1}-n) \\
  & &\displaystyle +\sum_{l \geq 0}\left(\binom{n}{2l}-\binom{n}{2l+1}\right)l-n'n+n', \\[10pt]
 & = &\displaystyle \sum_{l \geq 0}\left(\binom{n}{2l}-\binom{n}{2l+1}\right)l, \\[10pt]
  & = &\displaystyle 2^{n-2}.
\end{array}
\end{align*}

$\bullet$ For $m=n-2$, we proceed as the above case and obtain the same result.

We have therefore found that for $n>k=2$ and $m \geq n-3$, we have $f(m+1)-f(m)=2^{n-2}$. Thus for such $m$ we get the formula $f(m)=2^{n-2}(m-n+3)+f(n-3)$.

For $k=2=n$, we compute directly that $\begin{array}[t]{rcl}
f(m)& = & \left\{\begin{array}{cl}
        2\lfloor \frac{m}{2} \rfloor +2 & \text{ if } m \text{ odd},\\[5pt]
        2\lfloor \frac{m}{2} \rfloor +1 & \text{ if } m \text{ even},
         \end{array}\right. \\
& = &m+1.
\end{array}$

\underline{Third case:} $k=3$. First assume that $m \geq n-1$ and $m$ even. According to Theorem \ref{DimOfExt} and Lemma \ref{m>n-1}, we have: 
\begin{align*}
\begin{array}{rcl}
f(m+1)-f(m) &= &\displaystyle \sum_{l \geq 0} \binom{n}{2l+1}\frac{( \frac{m}{2} -l+2)(\frac{m}{2}-l+1)}{2}-\sum_{l \geq 0} \binom{n}{2l}\frac{( \frac{m}{2}-l+2)(\frac{m}{2}-l+1)}{2}, \\
&= &\displaystyle \frac{( \frac{m}{2} +2)(\frac{m}{2}+1)}{2}\sum_{l \geq 0} \left(\binom{n}{2l+1} -\binom{n}{2l}\right) \\
 & & \displaystyle -\frac{( \frac{m}{2} +2+\frac{m}{2}+1)}{2}\sum_{l \geq 0}\left( \binom{n}{2l+1}-\binom{n}{2l}\right)l \\
  & & \displaystyle +\frac{1}{2}\sum_{l \geq 0}\left( \binom{n}{2l+1}-\binom{n}{2l}\right)l^2, \\
& = & \displaystyle 2^{n-2}\frac{m+3}{2}-2^{n-3}\frac{n-1}{2}, \\
& = & \displaystyle 2^{n-4}(2m-n+7).
\end{array}
\end{align*} 
If $m$ is odd, a similar procedure leads to the expected result. For $m=n-2$ and $n-3$, we compute explicitly and check that the equality is satisfied.

\end{proof}

For any map $f: \bb N \longrightarrow \bb N$, we can define the discrete derivatives of $f$ recursively by $\Delta^0f(m)=f(m)$ and $\Delta^{r+1} f(m)=\Delta^r f(m+1)-\Delta^r f(m)$ for all $r \geq 0$. Consider the function $f(m)=\Ext_R(\cc,\cc)(m)$ (we write $\Ext_R(\cc,\cc)(m)$ instead of $\Ext_R^{m}(\cc,\cc)$ to emphasise that $m$ is a variable). From Corollary \ref{FormulasFork=1,2,3}, we see that:
\begin{itemize}\setlength\itemsep{0.5em}
\item If $k=1$, then $\Delta^{0}(\on{dim} \Ext_R(\cc,\cc)(m))=2^{n-1}$ for $m \geq n-1$.
\item If $k=2$ we have $\Delta^{1}(\on{dim} \Ext_R(\cc,\cc)(m))=2^{n-2}$ if $m \geq n-2$.
\item If $k=3$ we have $\Delta^{2}(\on{dim} \Ext_R(\cc,\cc)(m))=2^{n-3}$ if $m \geq n-3$.
\end{itemize}

After looking numerically at the behaviour of the dimensions of the summands of the Yoneda algebra for larger values of $k$, we make the following conjecture:

\begin{conjecture}
Let $R$ be a complete intersection ring with $k$ relations and $n$ variables, and with $k \geq 1$. We have the following:
\begin{itemize}\setlength\itemsep{0.5em}
\item If $k=1$, then for any $m \geq n-k$, we have $\Delta^{k-1}(\on{dim} \Ext_R(\cc,\cc)(m))=2^{n-k}$.
\item If $k \geq 2$, then for any $m \geq n-(k+1)$, we have $\Delta^{k-1}(\on{dim} \Ext_R(\cc,\cc)(m))=2^{n-k}$.
\end{itemize} 
\end{conjecture}
}

\section{Homotopy Lie algebras and complete intersections}\label{sec:5}

\subsection{The Yoneda algebra of a complete intersection algebra}\label{sec:Yoneda_CI}

Let $\msf k$ be a commutative Noetherian ring and set $A=\msf k[x_1,\dots, x_n]/(c_1,\dots, c_k)$ a complete intersection $\msf k$-algebra, i.e.,  $c_1, \dots, c_k \in \msf k[x_1, \dots, x_n]$ is a regular sequence. This means that $c_{i+1}$ is not a zero divisor in $\msf k[x_1,\dots, x_n]/( c_1,\dots, c_i)$ for all $i$. Set $\mf{m}=(x_1, \dots , x_n) \subset k[x_1, \dots, x_n]$. The polynomial ring $\msf k[x_1, \dots, x_n]$ is graded with $|x_i|=1$, and so is $\mathfrak{m}$. We will take about polynomial degree (not to be confused with the polynomial degree of Section \ref{sec:4}, but the difference will be clear by the context). We assume that for all $1 \leq p \leq k$, we have $c_p \in \mf{m}^2$ and write $c_p=\sum_{i=1}^n c_{p, i}x_i$ with $c_{p, i} \in \mf{m}$. Let $\mf m_x=(\overline{x_1}, \dots, \overline{x_{n}})$ be the ideal of $A$ which is the image of $ \mf m$. We are interested in the $A$-module $\msf k=A/\mf{m}_x$. Using a reasoning similar to that of \cite[Theorem 4]{Tate}, a corollary of Theorem \ref{thm:Tate} in the complete intersection case states that the free resolution $P^*$ of the $A$-module $\msf k$ is given by $P_2$, i.e., the complex $P_2$ is already exact. Hence $P^*$ is generated as a PD dg algebra by the set
\[
\{T_i, S_p \ | \ 1 \leq i \leq n, \, 1 \leq p \leq k, \, \on{deg}T_i=-1, \, \on{deg}S_p=-2\}
\]
and the differential $d$ is given by:
\begin{itemize}
\item $d(T_i) = \overline{x_i}$ for all $1 \leq i \leq n$,
\item $d(S_p) = \displaystyle \sum_{i=1}^n \overline{c_{p, i}} \, T_i$ for all $1 \leq p \leq k$.
\end{itemize}
To lighten notations, we will remove the $\overline{\color{white}x}$ so it is assumed that we are in the quotient ring. Moreover, we will write $S_i^{(k)}=\gamma_k(S_i)$ for the divided power structure. We see that $d(P^*) \subseteq \mf{m}_xP^{*}$, and so the resolution is $\mf m_x$-minimal, implying that $\Ext_A^*(\msf k,\msf k)=\mc Hom_{A}^*(P^*, \msf k)$.

For all  $1 \leq i \leq n$ and  $1 \leq p \leq k$, we define $\alpha_i=T_i^\vee$ and $\beta_p=S_p^\vee$ (see before Theorem \ref{thm:lift_pd} for the notation).

\delete{
\begin{align*}
\begin{array}[t]{cccc}
\alpha_i: & P_1 & \longrightarrow & \msf k \\
           & T_j & \longmapsto & \delta_{i, j}
\end{array} \quad \text{ and } \quad \begin{array}[t]{cccc}
\beta_p: & P_2 & \longrightarrow & \msf k \\
           & S_q & \longmapsto & \delta_{p, q}, \\
           & T_{i}T_{j} & \longmapsto & 0.
\end{array}
\end{align*}

In the remaining of this article, for a monomial $X \in P_m$, we will write $X^*$ for the dual of this element in $\on{Hom}_R(P_m,\msf k)$.

The proof of the next proposition can be found in the Appendix \ref{appendix:sec:Yoneda_CI}.

\begin{proposition}\label{BasisofExtdegree2} 
For any $m \in \mathbb{Z}_{>0}$, let $T_1^{a_1}\cdots T_n^{a_n}S_1^{(b_1)}\cdots S_k^{(b_k)} \in P_m$. Then the dual of $T_1^{a_1}\cdots T_n^{a_n}S_1^{(b_1)}\cdots S_k^{(b_k)}$ is a linear combination of elements of the form $\alpha_1^{a'_1}\cdots \alpha_n^{a'_n}\beta_1^{b'_1} \cdots \beta_k^{b'_k} \in \Ext_R^m(\msf k, \msf k)$. Moreover, for any $m \in \bb N$, the set $\{\alpha_1^{a_1}\cdots \alpha_n^{a_n}\beta_1^{b_1} \cdots \beta_k^{b_k} \ | \ a_1+\dots+a_n+2b_1+\dots+2b_k=m\}$ is a basis of $\Ext_R^m(\msf k, \msf k)$.
\end{proposition}

We can now give a presentation of the Yoneda algebra when the relations of $R$ are in $\mf{m}_x^2$.
}

\begin{theorem}\label{thm:HomotopyLiealgebra}
Let $A=\msf k[x_1,\dots, x_n]/(c_1,\dots, c_k)$ be a complete intersection $\msf k$-algebra with $\msf k$ is a commutative Noetherian ring. Consider the $A$-module $\msf k=A/\mf{m}_x$.  We assume that for all $1 \leq p \leq k$, we have $c_p \in \mf{m}^2$. By writing $c_p \equiv \sum_{i < j}(n^p_{i, j}+n^p_{j, i})x_ix_j+\sum_{i}n^p_{i, i}x_i^2 \mod \mathfrak{m}^3$, we have that
\begin{align*}
\pi^*(A, \msf k)=\bigoplus_{i=1}^n\msf k \alpha_i \oplus \bigoplus_{p=1}^k \msf k \beta_p
\end{align*}
is a restricted graded Lie algebra with $\on{deg}\alpha_i=1$, $\on{deg}\beta_p=2$, and the restricted Lie structure is given by:
\begin{empheq}[left = \empheqlbrace]{align}
    [\alpha_i, \alpha_j] & =  \displaystyle\sum_{p=1}^k (n^p_{i, j}+n^p_{j, i})\beta_p  \text{ for all } 1 \leq i \neq j \leq n ; \label{thm:HomotopyLiealgebra:eq:1}\\[5pt]
[\alpha_i, \alpha_i] & =   \sum_{p=1}^k 2n^p_{i, i}\beta_p \text{ for all } 1 \leq i \leq n ; \label{thm:HomotopyLiealgebra:eq:2}\\[5pt]
[\beta_p, \pi^*(A, \msf k)] & =  0 \text{ for all } 1 \leq p \leq k; \label{thm:HomotopyLiealgebra:eq:3}\\[5pt]
q(\alpha_i) & =\sum_{p=1}^k n^p_{i, i}\beta_p \text{ for all } 1 \leq i \leq n.\label{thm:HomotopyLiealgebra:eq:4}
\end{empheq}
\end{theorem}

\begin{proof}
We know from the discussion before Proposition \ref{prop:decomp_dual} and from Theorem \ref{thm:image_homotopy} that $\pi^*(A,\msf k)$ is a restricted graded Lie algebra spanned by the $\alpha_i$ and $\beta_p$, and the restricted map is given by the square. The Lie brackets $[\alpha_i,\alpha_j]$ and the quadratic map $q(\alpha_i)$ are obtained by Lemma \ref{lem:RelationsinExt}. The Lie brackets $[\beta_p,\alpha_i]$ and $[\beta_p,\beta_q]$ are zero as they would be elements of degree $3$ and $4$ respectively, and the highest degree in $\pi^*(A,\msf k)$ is $2$.

\end{proof}

\delete{
\begin{proof}
We verify directly that $\bigoplus_{i=1}^n\msf k \alpha_i \oplus \bigoplus_{p=1}^k \msf k \beta_p$ is indeed a restricted graded Lie algebra. Then using the relations \eqref{thm:HomotopyLiealgebra:eq:1}-\eqref{thm:HomotopyLiealgebra:eq:4} as well as the presentation given in Corollary \ref{cor:PresentationofExt}, we obtain the desired result.
\end{proof}
}

\delete{
\begin{corollary}\label{cor:PresentationofExt}
Let $R=\msf k[x_1,\dots, x_n]/(c_1,\dots, c_k)$ be a complete intersection ring with $\msf k$ a commutative Noetherian ring. Consider the $R$-module $\msf k=R/\mf{m}_x$. We assume that for all $1 \leq p \leq k$, we have $c_p \in \mf{m}^2$. By writing $c_p \equiv \sum_{i < j}(n^p_{i, j}+n^p_{j, i})x_ix_j+\sum_{i}n^p_{i, i}x_i^2 \text{ mod } \mathfrak{m}^3$, we have:
\begin{align*}
\Ext_R^*(\msf k, \msf k)=\msf k[\beta_1, \dots, \beta_k] \otimes_{\msf k} \msf k \langle \alpha_{1}, \dots, \alpha_{n} \rangle / \mathcal{I}, 
\end{align*}
where $\mathcal{I}$ is the two sided ideal generated by $\alpha_i \alpha_j +\alpha_j \alpha_i-\sum_{p=1}^k (n^p_{i, j}+n^p_{j, i})\beta_p$ for all $1 \leq i \neq j \leq n$, and by $\alpha_i^2-\sum_{p=1}^k n^p_{i, i}\beta_p$ for all $1 \leq i \leq n$. In particular, the Yoneda algebra is finitely generated. Furthermore, it is strictly graded commutative with $\on{deg}\alpha_i=1$, $\on{deg}\beta_p=2$ if and only if $c_p \in \mf{m}^3$ for all $1 \leq p \leq k$. In this case, we obtain:
\begin{align*}
\Ext_R^*(\msf k, \msf k)=\msf k[\beta_1, \dots, \beta_k] \otimes \Lambda[ \alpha_{1}, \dots, \alpha_{n}]. 
\end{align*}
\end{corollary}

\begin{proof}
{\color{red}We know that $\Ext_R^*(\msf k, \msf k)$ is the enveloping algebra of $\pi^*(R, \msf k)$ (only if $\msf k$ is a field of characteristic zero (Milnor-Moore theorem))}. From the Lie brackets given in Theorem \ref{thm:HomotopyLiealgebra}, the presentation of the Yoneda algebra follows.

Considering that $\on{deg}\alpha_i=1$ and $\on{deg}\beta_p=2$, if the Yoneda algebra is strictly graded commutative, then $\alpha_i \alpha_j+\alpha_j\alpha_i=0$ for all $i \neq j$, and $\alpha_i^2=0$ for all $i$. This would mean that $\sum_{p=1}^k (n^p_{i, j}+n^p_{j, i})\beta_p=0$ for all $i \neq j$ and $\sum_{p=1}^k n^p_{i, i}\beta_p=0$ for all $i$. But by construction, the $\beta_p$'s are linearly independent, and so we obtain $n^p_{i, j}+n^p_{j, i}=0=n^p_{i, i}$ for all $i \neq j$ and all $p$. As a consequence, we get that for all $p$, the relation $c_p$ has no term of degree $2$, i.e., $c_p \in \mf{m}^3$.

Conversely, if $c_p \in \mf{m}^3$ for all $p$, then $n^p_{i, j}+n^p_{j, i}=0=n^p_{i, i}$ for all $i \neq j$ and all $p$, and based on the relations defining $\mathcal{I}$, we see that $\Ext_R^*(\msf k, \msf k)$ is clearly strictly graded commutative and given by $\msf k[\beta_1, \dots, \beta_k] \otimes \Lambda[ \alpha_{1}, \dots, \alpha_{n}]$.    
\end{proof}
}

\delete{
\begin{proof}
Let $\mathcal A$ denote the algebra on the right hand side of the equality in the statement. We know from Proposition \ref{BasisofExtdegree2} that $\Ext_R^*(\msf k, \msf k)$ is generated by the $\alpha_i$'s and $\beta_p$'s. We have also seen in Lemma \ref{lem:RelationsinExt} that the relations in $\mathcal{I}$ are satisfied in $\Ext_R^*(\msf k, \msf k)$. Hence there is a natural surjective algebra morphism $\pi: \mathcal A \longrightarrow \Ext_R^*(\msf k, \msf k)$ corresponding to the quotient of $A$ by eventual additional relations. Furthermore, the Yoneda algebra is naturally graded, coming from a projective complex. If we grade $\mathcal A$ in the same way as $\Ext_R^*(\msf k, \msf k)$, then $\pi$ is homogeneous.

Based on the relations defining $\mathcal A$, any monomial in $\mathcal A$ can be expressed as a linear combination of terms of the form $\alpha_1^{a_1} \cdots \alpha_n^{a_n}\beta_1^{b_1} \cdots \beta_k^{b_k}$, with $b_p \in \bb N$ and $a_i=0, 1$. Thus the homogeneous summand of degree $m$ of $\mathcal A$, written $\mathcal A_m$, has a basis $\{\alpha_1^{a_1} \cdots \alpha_n^{a_n}\beta_1^{b_1} \cdots \beta_k^{b_k} \ | \ b_p \geq 0, a_i \leq 1,\sum_{i=1}^na_i+2\sum_{p=1}^kb_p=m\}$. However, according to Proposition \ref{BasisofExtdegree2}, this basis of $\mathcal A_m$ is sent to a basis of $\Ext_R^*(\msf k, \msf k)$, which makes $\pi$ into a linear isomorphism on the homogeneous subspaces. As we know that $\pi$ preserves the algebra structure, it follows that $\pi$ is an algebra isomorphism, which proves the presentation of the Yoneda algebra.

Considering that $\on{deg}\alpha_i=1$ and $\on{deg}\beta_p=2$, if the Yoneda algebra is strictly graded commutative, then $\alpha_i \alpha_j+\alpha_j\alpha_i=0$ for all $i \neq j$, and $\alpha_i^2=0$ for all $i$. This would mean that $\sum_{p=1}^k (n^p_{i, j}+n^p_{j, i})\beta_p=0$ for all $i \neq j$ and $\sum_{p=1}^k n^p_{i, i}\beta_p=0$ for all $i$. But by construction, the $\beta_p$'s are linearly independent, and so we obtain $n^p_{i, j}+n^p_{j, i}=0=n^p_{i, i}$ for all $i \neq j$ and all $p$. As a consequence, we get that for all $p$, the relation $c_p$ has no term of degree $2$, i.e., $c_p \in \mf{m}^3$.

Conversely, if $c_p \in \mf{m}^3$ for all $p$, then $n^p_{i, j}+n^p_{j, i}=0=n^p_{i, i}$ for all $i \neq j$ and all $p$, and based on the relations defining $\mathcal{I}$, we see that $\Ext_R^*(\msf k, \msf k)$ is clearly strictly graded commutative and given by $\msf k[\beta_1, \dots, \beta_k] \otimes \Lambda[ \alpha_{1}, \dots, \alpha_{n}]$. 
\end{proof}
}

\delete{
We have seen in Corollary \ref{cor:PresentationofExt} that $\Ext_R^*(\msf k, \msf k)$ is finitely generated by $n+k$ elements. We can however refine this number (see Appendix \ref{appendix:sec:Yoneda_CI} for the proof).

\begin{proposition}\label{prop:NumberofGenerators}
Let $R=\msf k[x_1,\dots, x_n]/(c_1,\dots, c_k)$ be a complete intersection ring with $\msf k$ a commutative ring. Consider the $R$-module $\msf k=R/\mf{m}_x$. We assume that for all $1 \leq p \leq k$, we have $c_p \in \mf{m}^2$. We write $c_p \equiv \sum_{i < j}(n^p_{i, j}+n^p_{j, i})x_ix_j+\sum_{i}n^p_{i, i}x_i^2 \text{ mod } \mathfrak{m}^3$. Then $\Ext_R^*(\msf k, \msf k)$ is generated by $n+k-\on{rank}A/I$ elements, where $A/I$ is an $\binom{n+1}{2} \times k$ matrix composed of the $n^p_{i, j}$'s.
\end{proposition}
}

\delete{
\begin{definition}
    The universal restricted enveloping algebra of a restricted graded Lie algebra $\mathfrak{g}$ is defined as the quotient
    \[
    \msf u(\mathfrak{g})=T(\mathfrak{g})/I(\mathfrak{g})
    \]
where $T(\mathfrak{g})$ is the tensor algebra of $\mathfrak{g}$ and $I(\mathfrak{g})$ is the two sided ideal generated by $a \otimes b-(-1)^{|a||b|}b \otimes a-[a, b]$ for all $a, b \in \mathfrak{g}$, and by $a \otimes a-q(a)$ for all $a \in \mathfrak{g}$ of odd degree.
\end{definition}

\begin{remark}
   The universal restricted enveloping algebra $\msf u(\mathfrak{g})$ of a restricted graded Lie algebra is referred to as the u-algebra of $\mathfrak{g}$ in \cite{Jacobson}.  
\end{remark}

The \textbf{homotopy Lie algebra} of an algebra $A$ (over $\msf k$) was introduced by Avramov (\cite[Theorem 1.2]{Avramov}) and is defined as the restricted graded Lie algebra $\pi^*(A)$ satisfying
\[
\on{Ext}_A^*(\msf k, \msf k)=\msf u(\pi^*(A)).
\]

}

\begin{remark}
Here we can make a connection with the cohomological variety introduced in \cite{CJL1}. Consider the case where $A$ is not a complete intersection, and set $\varpi: \widetilde{A} \to A$ the complete intersection approximation of $A$ (see \cite{CJL1} for details on the notations). Set $\widetilde{I}=\varpi^{-1}(I)$. Then $A/I$ is an $A$-module and an $\widetilde{A}$-module through the $\widetilde{A}$-module isomorphism $A/I \cong \widetilde{A}/\widetilde{I}$. We then have the following commutative diagram
 \begin{center}
  \begin{tikzpicture}[scale=0.9,  transform shape]
  \tikzset{>=stealth}
  
\node (1) at ( 0,0){$\on{Ext}_{\widetilde{A}}^*(A/I,A/I)$};
\node (2) at ( 5,0){$\on{Ext}_{A}^*(A/I,A/I)$};

\node (3) at (0,-3){$\on{Ext}_{\widetilde{A}}^*(A/I,A/I)^{ab}$};
\node (4) at ( 5,-3){$\on{Ext}_{A}^*(A/I,A/I)^{ab}$};

\draw [decoration={markings,mark=at position 1 with
    {\arrow[scale=1.2,>=stealth]{>}}},postaction={decorate}] (2) --  (1) node[midway, above] {$\varpi^\#$};
    \draw [decoration={markings,mark=at position 1 with
    {\arrow[scale=1.2,>=stealth]{>>}}},postaction={decorate}] (2) --  (4) node[midway, above] {};
    \draw [decoration={markings,mark=at position 1 with
    {\arrow[scale=1.2,>=stealth]{>>}}},postaction={decorate}] (1) --  (3) node[midway, above] {};
    \draw [decoration={markings,mark=at position 1 with
    {\arrow[scale=1.2,>=stealth]{>}}},postaction={decorate}] (4) --  (3) node[midway, below] {$(\varpi^\#)^{ab}$};

\end{tikzpicture}
 \end{center}
We know from Theorem \ref{thm:image_homotopy} that $\widetilde{\pi}^*=\pi^*(\widetilde{A},\widetilde{I})$ and $\pi^*=\pi^*(A,I)$ generate their respective Yoneda algebras. Then using Theorem \ref{thm:HomotopyLiealgebra}, we obtain
\begin{align*}
& \on{Ext}_{\widetilde{A}}^*(A/I,A/I)^{ab}=\on{Sym}_{A/I}\left(\widetilde{\pi}^2/[\widetilde{\pi}^1,\widetilde{\pi}^1]\bigoplus \widetilde{\pi}^1\right), \\
& \on{Ext}_{A}^*(A/I,A/I)^{ab}=\on{Sym}_{A/I}(\pi^*/[\pi^*,\pi^*]). 
\end{align*}
Set $A=R(L_{\mathfrak{\hat{g}}}(k,0))$ where $\mathfrak{g}$ is a finite dimensional simple Lie algebra with non-degenerate symmetric invariant bilinear form with level $k \in \mathbb{Z}_{\geq 2}$, as in \cite[Theorem 6.2]{CJL1}. Set also $I$ the ideal generated by a basis of $\mathfrak{g}$, so that $A/I=\msf k$. Then $[\widetilde{\pi}^1,\widetilde{\pi}^1]=0$ as the relations of $\widetilde{A}$ all satisfy $\widetilde{c_p} \in \mathfrak{m}^3 \backslash \mathfrak{m}^2$. Then the morphism of \cite[Corollary 4.9]{CJL1} becomes
\[
\on{Ext}_{A}^*(\msf k,\msf k) \twoheadrightarrow \on{Sym}_{\msf k}(\widetilde{\pi}^2).
\]
It follows that the results on the cohomological variety found in \cite{CJL1} focus on the degree $2$ part of the homotopy Lie algebra of $\widetilde{A}$ when $A=R(L_{\hat{\mathfrak{g}}}(k,0))$. Moreover, we also know from \cite[Lemma 6.1]{CJL1} that $\dim_{\msf k} \pi^2=\dim_{\msf k} \widetilde{\pi}^2=\dim_{\msf k} L((k+1)\theta)$, where $L((k+1)\theta)$ is the highest $\mathfrak{g}$-module of highest weight $(k+1)\theta$ with $\theta$ the highest root of $\mathfrak{g}$.
\end{remark}

\begin{example}
A simple singularity is an isolated hypersurface singularity in $\cc^3$ obtained as the quotient of the complex plane by the natural action of a finite subgroup of $\on{SU}_2$. Such singularities are linked to the finite dimensional simple Lie algebras through the resolution of the singular point, and they are given by the following equations:
\begin{align*}
\renewcommand\arraystretch{1.1}
\begin{array}{|l|l|}
\hline
\text{Type} & \text{Equation} \\
\hline
A_n \ (n \geq 1) & x^{n+1}+yz \\
\hline
D_{n+2} \ (n \geq 2)& x^{n+1}+xy^2+z^2 \\
\hline
E_6 & x^4+y^3+z^2 \\
\hline
E_7 & x^3+xy^3+z^2 \\
\hline
E_8 & x^5+y^3+z^2 \\
\hline
\end{array}
\end{align*}

We consider the singularities above over any commutative Noetherian ring, i.e., the ring of functions of the singularity $X$ is given by $A_X=\msf k[x,y,z]/(c)$ with $\msf k$ a commutative Noetherian ring and $c \in \mf{m}^2$. This is a complete intersection $\msf k$-algebra, and so we can apply Theorem \ref{thm:HomotopyLiealgebra}. In the table below, we give the degree $1$ and degree $2$ basis of $\pi^*(A_X, \msf k)$, as well as the non-zero values for the bracket and for $q$.

\[
\begin{array}{|c|c|c|c|c|}
\hline
    X & \pi^1(A_X, \msf k) & \pi^2(A_X, \msf k) & [\cdot, \cdot] & q \\
    \hline
    A_1 & \msf k \alpha_x \oplus \msf k \alpha_y \oplus \msf k \alpha_z & \msf k \beta &\begin{array}{l}
        [\alpha_x, \alpha_x]=2\beta \\[5pt]
        [\alpha_y, \alpha_z]=\beta 
    \end{array} & \begin{array}{l}
        q(\alpha_x)=\beta 
    \end{array} \\
    \hline
    A_n (n \geq 2) & \msf k \alpha_x \oplus \msf k \alpha_y \oplus \msf k \alpha_z & \msf k \beta & \begin{array}{l}
        [\alpha_y, \alpha_z]=\beta 
    \end{array} & q=0 \\
    \hline
    \begin{array}{c}
    D_{n+2} \ (n \geq 2),\\
    E_6, E_7, E_8
    \end{array}& \msf k \alpha_x \oplus \msf k \alpha_y \oplus \msf k \alpha_z & \msf k \beta & \begin{array}{l}
        [\alpha_z, \alpha_z]=2\beta 
    \end{array} & \begin{array}{l}
        q(\alpha_z)=\beta 
    \end{array} \\
    \hline
\end{array}
\]
\end{example}

The above example illustrates a particular property of $\pi^*(A,\msf k)$, in that it is not impacted by terms of degree higher than $2$ in the relations of $A$. We indeed see that for the simple singularities of types $D_{n+2}$, $E_6$, $E_7$ and $E_8$, we obtain the same $\pi^*(A,\msf k)$, even if the rings of functions are different. However, those ring differences appear in terms of polynomial degree $3$ and above in the relations. This particularity of $\pi^*(A,\msf k)$ is due to the fact that we are looking at the trivial $A$-module $\msf k$. Indeed, if we were to add to the relation $c$ of $A_X$ a term $x^{a_x}y^{a_y}z^{a_z}$ with $a_x+a_y+a_z \geq 3$ and $a_x \geq 1$, then we could choose the differential of the corresponding element $S$ of degree $-2$ in the resolution so that it contains $x^{a_x-1}y^{a_y}z^{a_z}T_x$. It follows that if we lift the morphism $\alpha_x$, then $(\widetilde{\alpha_x})_{-2}(S)$ contains $x^{a_x-2}y^{a_y}z^{a_z}T_x$ if $a_x \geq 2$, or $y^{a_y-1}z^{a_z}T_y$, or $y^{a_y}z^{a_z-1}T_z$ otherwise, depending on the values of $a_y$ and $a_z$. In all cases, we see that $(\widetilde{\alpha_x})_{-2}(S)$ contains a term of the form $p_iT_i$ where $p_i$ is a monomial of degree $\geq 1$. Hence $\alpha_i \circ (\widetilde{\alpha_x})_{-2}(S)$ will contain $p_i\alpha_i(T_i)=p_i1=0$ because $\msf k$ is the trivial $A$-module. Therefore if we were to modify the exponents in $x^{a_x}y^{a_y}z^{a_z}$, as long as their sum remains above $3$, the Yoneda product will remain the same. This can be seen more generally in Theorem \ref{thm:HomotopyLiealgebra}, because in the presentation of $\pi^*(A,\msf k)$, the relations are constructed from the terms of polynomial degree $2$ in the relations, and higher degree terms do not interfere.

Let $A=\msf k[x_1,\dots, x_n]/(c_1,\dots, c_k)$ be a $\msf k$-algebra such that for all $1 \leq p \leq k$, we have $c_p \in \mf{m}^2$. Let $\lfloor c_p \rfloor \in \mf{m}^2$ be the homogeneous polynomial degree $2$ part of $c_p$. We define $\lfloor A \rfloor= \msf k[x_1,\dots, x_n]/(\lfloor c_1 \rfloor,\dots, \lfloor c_k \rfloor)$, which is a quadratic algebra. Even if $A$ is a complete intersection, the algebra $\lfloor A \rfloor $ might not be a complete intersection. Due to the above discussion, we obtain a corollary to the previous theorem.

\begin{corollary}
Assume that both $A$ and $\lfloor A \rfloor$ are complete intersections. Then we have an isomorphism of restricted graded Lie algebras
\[
\pi^*(A, \msf k) \cong \pi^*(\lfloor A \rfloor, \msf k).
\]
\delete{an isomorphism of algebras
\[
\Ext_A^*(\msf k, \msf k) \cong \Ext_{\lfloor A \rfloor}^*(\msf k,\msf k)
\]
and}
\end{corollary}

We conclude this section by drawing a parallel between Theorem \ref{thm:HomotopyLiealgebra} and the following result of G. Sj\"odin:

\begin{theorem}\emph{\textbf{(\cite[Theorem 6]{Sjodin76}).}}\label{thm:Tate_fg}
Let $R$ be a local ring with maximal ideal $\mathfrak{m}_x$ and residue field $\msf k$. The algebra $\Ext_R^*(\msf{k}, \msf{k})$ is finitely generated and strictly graded commutative if and only if $R$ is a local complete intersection with
\[
\on{dim}_{\msf{k}} \mf{m}_x^2/\mf{m}_x^3=\binom{n+1}{2}
\]
where $n=\on{dim}_{\msf{k}} \mf{m}_x/\mf{m}_x^2$.
\end{theorem}

 When the algebra $A$ in Theorem \ref{thm:HomotopyLiealgebra} is a complete intersection with the relations in $\mf{m}^2$, we know from Theorem \ref{thm:image_homotopy} that the Yoneda algebra is finitely generated. So we want to focus on the numerical condition in Sj\"odin's result. The number $\binom{n+1}{2}$ is the dimension of the space of homogeneous polynomials of degree $2$ in $n$ variables. On the other hand, $\on{dim}_{\msf{k}} \mf{m}_x^2/\mf{m}_x^3$ is the dimension of the space of homogeneous polynomials of degree $2$ in $A$. Hence Sj\"odin's criteria means that the degree $2$ homogeneous subspace of $A$ is as big as possible. A consequence of Theorem \ref{thm:HomotopyLiealgebra} is that the Yoneda algebra is strictly graded commutative if and only if $c_p \in \mathfrak{m}^3$ for all $p$. However, if this condition is verified, then all monomials of degree $2$ in $A$ are linearly independent, and so Sj\"odin's criteria is verified. If a relation $c_p$ contains terms of degree $2$, it can be written as $c_p=\lfloor c_p \rfloor+(c_p)_{\geq3}$ with $(c_p)_{\geq3} \in \mathfrak{m}^3$. But as $\overline{c_p}=0$ in $A$, we see that $\overline{\lfloor c_p \rfloor}= -\overline{(c_p)_{\geq3}}$ and so $\overline{\lfloor c_p \rfloor} \in \mf{m}_x^3$, creating a linear dependency between elements in $\mf{m}_x^2/\mf{m}_x^3$. Thus $\mf{m}_x^2/\mf{m}_x^3$ cannot have the highest possible dimension. It follows that we have an equivalence:
\[
\on{dim}_{\msf{k}} \mf{m}_x^2/\mf{m}_x^3=\binom{n+1}{2} \Longleftrightarrow c_p \in \mf{m}^3 \ \forall \ p.
\]
Moreover, we know from Theorem \ref{thm:HomotopyLiealgebra} that for a complete intersection ring, having $c_p \in \mathfrak{m}^3$ for all $p$ is equivalent to $\pi^*(R, \msf k)$ having trivial Lie bracket. We thus obtain the following corollary to Theorem \ref{thm:Tate_fg}:

\begin{corollary}
Let $R$ be a local ring with maximal ideal $\mathfrak{m}_x$ and residue field $\msf k$. The algebra $\Ext_R^*(\msf{k}, \msf{k})$ is finitely generated and strictly graded commutative if and only if $\pi^*(R, \msf k)$ has trivial Lie bracket and is in degree $1$ and $2$. 
\end{corollary}

The advantage of Sj\"odin's theorem is that $R$ being a complete intersection is not a hypothesis but a part of the result. However, his result does not give any information about the structure of $\pi^*(R, \msf k)$ when the above equality is not verified. On the other hand, Theorem \ref{thm:HomotopyLiealgebra} gives a presentation of the restricted graded Lie algebra even if the equality is not verified.

\begin{remark}
As mentioned before, the Yoneda product is not impacted by the terms of degree at least $3$ in the relations of $A$, which illustrates the importance of the quotient $\mf{m}_x^2/\mf{m}_x^3$ in Sj\"odin's result. However, one may wonder how to differentiate between two Yoneda algebras when the Yoneda product is the same. This can be done using $A_\infty$-structures. Simply put, an $A_\infty$-structure on a graded vector space $A$ is a collection of products $b_n: A^{\otimes n} \longrightarrow A$ for all $n \geq 1$, verifying compatibility conditions. For example, for each $n \geq 3$, the ring $R_n=\cc[x]/(x^n)$ gives $\Ext_{R_n}^*(\cc, \cc)=\cc[\alpha, \beta]/(\alpha^2)$. But by adding an $A_\infty$-structure, we obtain that $b_2$ is the Yoneda product, $b_n$ is non-zero and $b_m=0$ for $m \neq 2, n$. Equipped with this extra structure, $\Ext_{R_n}^*(\cc, \cc) \not\cong \Ext_{R_{n'}}^*(\cc, \cc)$ for $n \neq n'$ as $A_\infty$-algebras. We will not be dealing with $A_\infty$-structures in the rest of this article, the reader is referred to \cite{Keller} and \cite{Lu-Palmieri-Wu-Zhang} for more details on the subject. 
\end{remark}

\delete{
\subsection{The weak graded commutativity of the Yoneda algebra}\label{Section3.3}

Let $A=\bigoplus_{i}A_i$ be a graded algebra over a field $\msf k$ of characteristic not $2$. It is equipped with a graded commutator given by $[a,b]_s=ab-(-1)^{ij}ba$ for $a \in A_i$, $b \in A_j$, and then extended to $A$ by bilinearity. In particular, the quotient $A/ \langle [A, A]_s \rangle$ is strictly graded commutative, and the two sided ideal $\langle [A, A]_s \rangle$ is the smallest two sided ideal such that the quotient becomes strictly graded commutative. We notice that if $a$ is of odd degree, then the graded commutator gives $[a, a]_s=2a^2$, and as $2$ is invertible in $\msf k$, we have $ \langle a^2 \ | \ a \text{ of odd degree} \rangle \subseteq \langle [A, A]_s \rangle$. In general, this inclusion is strict.

\begin{definition}
A graded algebra $A$ over a field of characteristic not $2$ is \textbf{weakly graded commutative} if:
\[
\langle a^2 \ | \ a \text{ of odd degree} \rangle = \langle [A, A]_s \rangle,
\]
or equivalently if $A/ \langle a^2 \ | \ a \text{ of odd degree} \rangle $ is a strictly graded commutative algebra.
\end{definition}

We have already seen that the strictly graded commutativity of $\Ext_R^*(\msf k, \msf k)$ when $R$ is a complete intersection with relations of least degree at least 2 is equivalent to $c_p \in \mf{m}^3$ for all $p$. We would like to see how those conditions change when we look at the weak graded commutativity.

Assume that $\msf k$ is a field of characteristic not 2. In order for $\Ext_R^*(\msf k, \msf k)$ to be weakly graded commutative, we need $\alpha_i \alpha_j +\alpha_j \alpha_i \in \langle a^2 \ | \ a \text{ of odd degree} \rangle$ for all $i, j$. But we have
\[
\alpha_i \alpha_j+\alpha_j \alpha_i=(\alpha_i+\alpha_j)^2-\alpha_i^2-\alpha_j^2.
\]
Thus for all $i, j$, we have $\alpha_i \alpha_j+\alpha_j \alpha_i \in \langle a^2 \ | \ a \text{ of odd degree} \rangle$. We know that the algebra $\Ext_R^*(\msf k, \msf k)$ is generated by the $\alpha_i$'s and $\beta_p$'s, and that the $\beta_p$'s are in the center. It follows that in $\Ext_R^*(\msf k, \msf k)/\langle a^2 \ | \ a \text{ of odd degree} \rangle$, we have:
\begin{itemize}
\item $\overline{\beta_p}$ is in the center, for all $p$;
\item $\overline{\alpha_i}$ and $\overline{\alpha_j}$ anti-commute, for all $i \neq j$;
\item $\overline{\alpha_i}^2=0$, for all $i$.
\end{itemize}
As the quotient is generated by the $\overline{\alpha_i}$'s and $\overline{\beta_p}$'s, it is a strictly graded commutative algebra. We have thus proved:

\begin{theorem}
The Yoneda algebra $\Ext_R^*(\msf k, \msf k)$ of a complete intersection polynomial ring $R$ over a field $\msf k$ of characteristic not $2$ and with relations of least degree at least $2$ is weakly graded commutative.
\end{theorem}
}

\subsection{Reconstruction of complete intersections} \label{sec:reconstruct_homLie}
\delete{
Given a positively graded $\msf k$-algebra $A=\bigoplus _{i\geq 0}A_i$ with $ \dim_{\msf k} A_0=1$, we determine the conditions on $A$ so that there exists a commutative algebra $R$ with an homomorphism $x: R \longrightarrow \msf k$ such that $A$ is isomorphic to $\Ext^{*}_R(\msf k, \msf k)$ as graded algebra. Koszul duality in quadratic algebras answers this question for quadratic algebras. In this paragraph, we give a complete description of the Yoneda algebra $\Ext^{*}_R(\msf k, \msf k)$.
}
In this section, we show how to reconstruct an algebra from its homotopy Lie algebra. We will first need results related to Gr\"obner bases.

Let $\on{Pol}=\msf k[x_1, \dots, x_n]$ be a polynomial ring with $\msf k$ a field. For $\mathbf{a}=(a_1,\dots,a_n) \in \bb N^n$, we write $x^\mathbf{a}=x_1^{a_1} \cdots x_n^{a_n} \in \on{Pol}$. We will refer to notations of \cite[Ch. 2]{Herzog-Hibi}. 

\begin{proposition}\label{prop:grob_basis}
Set $f_1, \dots, f_s \in \on{Pol}$ with $s \leq n$. For any $1 \leq i \leq s$, set $ l_i> \deg (f_i)$ (here $\deg$ refers to the total degree of the polynomial) and define $g_i=f_i+x_i^{l_i}$. Then $\{g_1, \dots, g_s\}$ is a Gr\"obner basis of $I_s=(g_1, \dots, g_s)$.
\end{proposition}

\begin{proof}
We will take the lexicographic order (see \cite[2.1.2(a)]{Herzog-Hibi}) on the set of monomials. With this order, $x^\mathbf{a}<x^\mathbf{b}$ for $\mathbf{a}, \mathbf{b}\in \bb N^n$ if either $|\mathbf{a}|=\sum_i a_i<|\mathbf{b}|=\sum_i b_i$, or $|\mathbf{a}|=|\mathbf{b}|$ and $a_i<b_i$ for the smallest $i$ with $ a_i\neq b_i$. In particular, we have $x_1 > x_2 > \dots > x_n$.

In our setting, the initial monomial of $g_i$ is $\on{in}_<(g_i)=x_i^{l_i}$ and the pair $ (\on{in}_<(g_i), \on{in}_<(g_j))$ is relatively prime if $i\neq j$. By \cite[Cor. 2.3.4] {Herzog-Hibi}, $\{g_1, \dots, g_s\}$ is a Gr\"obner basis for the ideal $ I_s$. In fact, it is a reduced Gr\"obner basis. 
\end{proof}

The following result seems to be well-known to the commutative algebra community, but we failed to find a reference. We provide a proof in Appendix \ref{appendix:sec_5} for the sake of completeness.

\begin{lemma}\label{lem:reg_seq} Let $\{g_1, \dots, g_s\}$ be a Gr\"obner basis such that $ (\on{in}_<(g_i), \on{in}_<(g_j))$ is a relative prime pair if $i\neq j$. Then $ \{g_1, \dots, g_s\}$ is a regular sequence in $\on{Pol}$. 
\end{lemma}

Let $\mathfrak{g}=\mathfrak{g}_1 \oplus \mathfrak{g}_2$ be a finite dimensional restricted graded Lie algebra over a field $\msf k$ such that $\dim_{\msf k} \mathfrak{g}_1 \geq \dim_{\msf k} \mathfrak{g}_2$. We write $\{\alpha_1, \dots, \alpha_n \}$ and $\{\beta_1, \dots, \beta_k\}$ for bases of $\mathfrak{g}_1$ and $\mathfrak{g}_2$ respectively. For each $1 \leq p \leq k$, we construct the matrix $N^p=(n^{p}_{i,j})\in M_n(\msf k)$ given by the structure constants of $\mathfrak{g}$, i.e., $[\alpha_i,\alpha_j]=\sum_{p=1}^k (n^p_{i,j}+n^p_{j,i})\beta_p$ for $i \neq j$ and $q(\alpha_i)=\sum_{p=1}^k n^p_{i,i}\beta$. We can now prove the following result:

\begin{theorem}\label{thm:reconstruct_homotopy}
Given a finite dimensional restricted graded Lie algebra $\mathfrak{g}=\mathfrak{g}_1 \oplus \mathfrak{g}_2$ over a field $\msf k$ with $\dim_{\msf k} \mathfrak{g}_1 \geq \dim_{\msf k} \mathfrak{g}_2$, there exists a finitely generated complete intersection $\msf k$-algebra $A$ such that $\pi^*(A, \msf k) \cong \mathfrak{g}$ as restricted graded Lie algebras. 
\end{theorem}

\begin{proof}
We first define $\on{Pol}=\msf k[x_1, \dots, x_n]$ and, using the matrices $N^p$ introduced previously, we set $c'_p=\sum_{i, j}n^{p}_{i,j}x_ix_j \in \on{Pol}$ and  $c_p=c_p'+x_p^3 \in \on{Pol}$ for all $1 \leq p \leq k$. 
Let $A=\on{Pol}/(c_1, \dots, c_k)$ and let $\mathfrak{m}_x$ be the kernel of the morphism $A\to \msf k$ defined by $\overline{x_i}\mapsto 0$. Then by Proposition \ref{prop:grob_basis} the family $\{c_1, \dots, c_k\}$ is a Gr\"obner basis of the ideal $(c_1, \dots, c_k)$, and by Lemma \ref{lem:reg_seq} it is also regular sequence in $\on{Pol}$. Using Theorem \ref{thm:HomotopyLiealgebra}, we can then conclude the proof.   
\end{proof}

\delete{
We say that two complete intersections $(R, \mathfrak{m}_x)$ and $(R',\mathfrak{m}_{x'})$ are Yoneda equivalent if $\Ext^*_R(\msf k, \msf k)$ and $\Ext^*_{R'}(\msf k_{x'}, \msf k_{x'})$ are isomorphic as graded algebras. Thus the Yoneda equivalence classes are defined by the isomorphism classes of the algebras $B(N)$. This will be addressed in a later work. 
}

\delete{
\section{Finite generation of the Yoneda algebra}

Let $ \msf k$ be a field of characteristic zero and let $R= \msf k[x_1, \dots, x_n]/\langle c_1, \dots, c_k \rangle$ be a finitely generated commutative algebra such that the Tate resolution $T$ of the trivial $R$-module $ \msf k=R/\langle x_1, \dots, x_n \rangle$ is obtained after adding variables of degree $N$ for some $N \in \bb N$. We write
\[
T=R \langle A_1, \dots, A_{m_1}, B_1, \dots, B_{m_2}, \dots, P_1, \dots, P_{m_{N}} \rangle, 
\]
where $|A_i|=1$, $|B_i|=2$, $\dots$, $|P_i|=N$. We also assume that $X$ is an $I$-minimal resolution of $ \msf k$, i.e., $\on{Im}\partial \subset \on{Rad} X$. A basis of $X$ as an $R$-module is composed of ordered monomials of the form $A_1^{a_1}\dots P_{m_N}^{p_{m_N}}$. For such a monomial $Z$, we write $Z^*$ for the element in $\on{Hom}_R(X, \msf k)$ sending $Z$ to $1$ and every other monomial to $0$.

We need to introduce some notation. Fix an ordered monomial $M$ in the variables $A, B, \dots, P$. If in $M$ we move a collection of variable $X, \dots, Y$ to the position of another variable $Z$, it might introduce a $-1$ factor. We write this factor as $(-1)^{(Z, X \cdots Y)_M}$. We also write $\widetilde{Z}^M$ for the number of odd degree variables on the left of $Z$ in $M$. Finally, if a variable $Z$ appears in $M$, we write $k_{M, Z}$ for the power in which $Z$ appears, $z_{X \cdots Y} \in \langle x_1, \dots, x_n \rangle$ for the coefficient of $X \cdots Y$ in $\partial(Z)$, and $Z_{X \cdots Y}$ for the element in $T_1$ such that $\partial_1(Z_{X \cdots Y})=z_{X \cdots Y}$.

Starting from an ordered monomial $M$, we will denote the changes in $M$ in the following way: the monomial obtained by adding to the left of $M$ variables $W_1, \dots, W_s$ will be written as $([+W_1\cdots+W_s] \ | \ M)$. If in $M$ we remove variables $X, \dots, Y$ and add a variable $Z$ (in the position corresponding to its degree), the monomial obtained will be denoted by $([-X\cdots-Y+Z] \ | \ M)$.

\begin{lemma}\label{lem:lift}
Let $M$ be an ordered monomial such that $M=K_i^{k_i-1} \cdots$ with $|K_i| \geq 4$ and $k_i \geq 1$. Consider $g_1$, $g_2$, $\dots$, the successive lifts of $M^*$. Then for $k \geq 3$, we have
\begin{align}
\begin{array}{rl}
g_k([+W_1 &\cdots+W_s]  [-X_1\cdots-Y_1+Z_1] \dots [-X_r\dots-Y_r+Z_r] \ | \ M)  =  \\[5pt]
&\displaystyle \sum_{\{i_1, \dots, i_r\}=\{1, \dots, r\}}(-1)^{C_{(i_1, 1), \dots, (i_r, r)}}(k_{M, Z_1}+1)\cdots(k_{M, Z_r}+1)\delta_{|Z_1|, |Z_{i_1}|} \cdots \delta_{|Z_r|, |Z_{i_r}|} \\[5pt]
&\quad \quad \quad  \times W_1\cdots W_s \cdot (Z_1)_{M, X_{i_1} \cdots Y_{i_1}} \cdots (Z_r)_{M, X_{i_r} \cdots Y_{i_r}} 
\end{array}
\end{align}
if $|W_j| < |K_i|$ for all $1 \leq j \leq s$. The constant $C_{(i_1, 1), \dots, (i_r, r)}$ is given by
\delete{
\begin{align*}
\resizebox{\textwidth}{!}{$\begin{array}{l}
C_{(i_1, 1), \dots, (i_r, r)} \\[5pt]
\displaystyle = \widetilde{Z}_1+\sum_{\on{var}^i < Z_1}|\on{var}^i|+(Z_1, X_1 \cdots Y_1)+|X_1| \sum_{\substack{X_1 < \on{var}^i < Z_1 \\ i\neq 1}}|\on{var}^i|+\cdots+|Y_1| \sum_{\substack{Y_1 < \on{var}^i < Z_1 \\ i \neq 1}}|\on{var}^i| \\[5pt]
\quad \displaystyle +\widetilde{Z}_2+\sum_{\substack{\on{var}^i < Z_2 \\ i \neq 1}}|\on{var}^i|+(Z_2, X_2 \cdots Y_2)+|X_2| \sum_{\substack{X_2 < \on{var}^i < Z_2 \\ i \neq 1, 2}}|\on{var}^i|+\cdots+|Y_2| \sum_{\substack{Y_2 < \on{var}^i < Z_2 \\ i \neq 1, 2}}|\on{var}^i| \\[5pt]
\quad \displaystyle +\widetilde{Z}_{r-2}+\sum_{\substack{\on{var}^i < Z_{r-2} \\ i \neq 1, \cdots, r-3}}|\on{var}^i|+(Z_{r-2}, X_{r-2} \cdots Y_{r-2})+|X_{r-2}| \sum_{\substack{X_{r-2} < \on{var}^i < Z_{r-2} \\ i \neq 1, 2, \cdots, r-2}}|\on{var}^i|+\cdots+|Y_{r-2}| \sum_{\substack{Y_{r-2} < \on{var}^i < Z_{r-2} \\ i \neq 1, 2, r-2}}|\on{var}^i| \\[5pt]
\quad \displaystyle +\widetilde{Z}_{r-1}+\sum_{\substack{\on{var}^i < Z_{r-1} \\ i \neq 1, \cdots, r-2}}|\on{var}^i|+(Z_{r-1}, X_{r-1} \cdots Y_{r-1})+|X_{r-1}| \sum_{\substack{X_{r-1} < \on{var}^i < Z_{r-1} \\ i \neq 1, 2, \cdots, r-1}}|\on{var}^i|+\cdots+|Y_{r-1}| \sum_{\substack{Y_{r-1} < \on{var}^i < Z_{r-1} \\ i \neq 1, 2, r-1}} |\on{var}^i|\\[5pt]
\quad \displaystyle +\widetilde{Z}_{r}+\sum_{\substack{\on{var}^i < Z_{r} \\ i \neq 1, \cdots, r-1}}|\on{var}^i|+(Z_r, X_r \cdots Y_r). 
\end{array}$}
\end{align*}}

\begin{align*}
\begin{array}{l}
C_{(i_1, 1), \dots, (i_r, r)} = \displaystyle \sum_{j=1}^{r}\Bigl( \widetilde{Z}_{j}+(Z_{j}, X_{i_j} \cdots Y_{i_j})_M +\sum_{\substack{|\on{var}^i| < |Z_{j}| \\ i \neq i_1, \dots, i_{j-1}}}|\on{var}^i| \Bigr) \\
\quad \quad \quad \quad +\displaystyle \sum_{j=1}^{r-1} \Bigl( |X_{i_j}| \sum_{\substack{|X_{i_j}| < |\on{var}^i| < |Z_{j}| \\ i \neq i_1, \dots, i_j}}|\on{var}^i|+\cdots+|Y_{i_j}| \sum_{\substack{|Y_{i_j}| < |\on{var}^i| < |Z_{j}| \\ i \neq i_1, \dots, i_j}} |\on{var}^i| ) \Bigr). 
\end{array}
\end{align*}
where $\on{var}^i$ denotes a variable $X_i$, \dots, $Y_i$, or $Z_i$.
\end{lemma}

\begin{proof}
The proof is done by direct computation. We check the formula directly for $k=1, \dots, 4$, and then proceed by induction on $k$.
\end{proof}

\begin{proposition}\label{prop:decomp_dual}
For any positive integers $b_1, \dots, p_{m_N}$, the element $(B_1^{b_1} \dots P_{m_N}^{p_{m_N}})^*$ can be expressed as a polynomial of the duals of the variables. In particular, if $|B_1^{b_1} \dots P_{m_N}^{p_{m_N}}| \geq N$, then
\begin{equation}\label{eq:split_formula}
(B_1^{b_1} \dots P_{m_N}^{p_{m_N}})^*=\frac{1}{b_1! \cdots p_{m_N}!}(B_1^*)^{b_1} \dots (P_{m_N}^*)^{p_{m_N}}.
\end{equation}
\end{proposition}

\begin{proof}
Consider a monomial $B_i^{b_i} \dots P_{m_N}^{p_{m_N}}$ with $b_i \geq 1$. In order to lift $M^*=(B_i^{b_i-1} \dots P_{m_N}^{p_{m_N}})^*$, we set 
\begin{align*}
\begin{array}[t]{cccl}
g_1: & B_i^{b_i-1} \dots P_{m_N}^{p_{m_N}} & \longmapsto & 1 \in R, \\
           & \text{other} & \longmapsto & 0,
\end{array} 
\end{align*}
The image by $\partial$ of a monomial of degree $|M|+1$ will contain $M$ if it is (possibly) a variable $Q$ of degree $|M|+1$, a monomial of the form $([+A_i] \ | \ M)$, or of the form $([-X\cdots-Y+Z] \ | \ M)$ with $|Z|-|X| \cdots-|Y|=1$. So we have
\begin{align*}
\begin{array}{rcl}
g_1\partial(Q)&=& q_M, \\[5pt]
g_1\partial([+A_i] \ | \ M)&=&\partial(A_i), \\[5pt]
g_1\partial([-X\cdots-Y+Z] \ | \ M)&=&(-1)^{\widetilde{Z}^M+|X|+\cdots+|Y|+(Z, X \cdots Y)_M}(k_{M, Z}+1)\partial(Z_{M, X\cdots Y})
\end{array}
\end{align*}
Hence we define
\begin{align*}
\begin{array}[t]{cccl}
g_2 :&Q& \longmapsto & Q_M, \\[5pt]
&([+A_i] \ | \ M)& \longmapsto & A_i, \\[5pt]
& ([-X\cdots-Y+Z] \ | \ M)& \longmapsto&(-1)^{\widetilde{Z}^M+|X|+\cdots+|Y|+(Z, X \cdots Y)}(k_{M, Z}+1)Z_{X\cdots Y}, \\[5pt]
    & \text{other} & \longmapsto & 0.
\end{array} 
\end{align*}
Likewise, we can compute $g_3$ and verify that $(B_i^{b_i} \dots P_{m_N}^{p_{m_N}})^*=\frac{1}{b_i} B_i^* \cdot M^*+\sum_W W^*$ with $W$ variables of degree $|M|+2$. By repeating this procedure for the other $B_i$'s and the $C_i$'s, we obtain
\begin{align}\label{eq:decomposition2}
(B_1^{b_1} \dots P_{m_N}^{p_{m_N}})^*=\frac{1}{b_1! \cdots c_{m_3}!}(B_1^*)^{b_1} \dots (C_{m_3}^*)^{c_{m_3}} (D_1^{d_1} \cdots P_{m_N})^*+\Xi
\end{align}
where $\Xi$ is a sum of monomials in the duals of the variables.

Let $M$ be an ordered monomial such that $M=K_i^{k_i-1} \cdots$ with $|K_i| \geq 4$ and $k_i \geq 1$. Using Lemma \ref{lem:lift}, we can determine that in the lift $g_{|K_i|+1}$ of $M^*$, the only elements whose images contain $K_i$ are $([+K_j] \ | \ M)$ and (possibly) variables of degree $|M|+|K_i|$. It follows that $K_i^* \cdot M^*=k_i (K_i^{k_i} \cdots)^*+\sum_U U^*$ with $U$ variables of degree $|M|+|K_i|$. By successively applying this result for all degrees, we prove the first statement of the proposition. The polynomial $\Xi$ in Formula~\eqref{eq:decomposition2} appears because when one lifts a monomial $M^*$, there might be variables of degree $|M|+1$ whose differentials contain the element not in $\on{Ker} g_1$, i.e., $M$, and then there might be variables of degree $|M|+2$ whose differential contain the elements of degree $|M|+1$ not in $\on{Ker}g_2$, and so on. But if $M$ is such that $|M| \geq N$ (for example is $M$ contains some variable of degree $N$), then there is no variable of degree strictly more than $|M|$, and thus we obtain Formula~\eqref{eq:split_formula}.
\end{proof}

We can now prove the main result of this section.

\begin{theorem}
Let $R$ be a finitely generated commutative algebra. Assume that Tate resolution $T$ of the trivial $R$-module is obtained after finitely many steps and is an $I$-minimal. Then $\Ext_{R}^*(\msf k, \msf k)$ is finitely generated by the duals of the variables in Tate resolution.
\end{theorem}

\begin{proof}
Consider the monomial $M=B_1^{b_1} \cdots P_{m_N}^{p_{m_N}}$. We lift the morphism $M^*$. We see that 
\begin{align*}
\begin{array}[t]{cccl}
g_1: & M & \longmapsto & 1 \in R, \\
           & \text{other} & \longmapsto & 0,
\end{array} 
\end{align*}
and
\begin{align*}
\begin{array}[t]{cccl}
g_2 :&Q& \longmapsto & Q_M, \\[5pt]
 &([+A_i] \ | \ M)& \longmapsto & A_i, \\[5pt]
& ([-X\cdots-Y+Z] \ | \ M)& \longmapsto&(-1)^{\widetilde{Z}^M+|X|+\cdots+|Y|+(Z, X \cdots Y)_M}(k_{M, Z}+1)Z_{M, X\cdots Y} \\[5pt]
    & \text{other} & \longmapsto & 0,
\end{array} 
\end{align*}
where the $Q$'s are variables of degree $|M|+1$ such that $q_M \neq 0$. Thus we obtain
\begin{align*}
A_i^* \cdot (B_1^{b_1} \cdots P_{m_N}^{p_{m_N}})^*=(A_i B_1^{b_1} \cdots P_{m_N}^{p_{m_N}})^*+\text{monomials with 0 }A_q\text{'s}.
\end{align*}

By doing the same construction and increasing the numbers of $A_i$'s, we can show that for any $1 \leq r \leq m_1$,
\begin{align}\label{eq:extract_A_(r)}
\begin{array}{rl}
A_{i_1}^* \cdot (A_{i_2} \cdots A_{i_r} B_1^{b_1} \cdots P_{m_N}^{p_{m_N}})^*= & \ (A_{i_1}A_{i_2} \cdots A_{i_r}  B_1^{b_1} \cdots P_{m_N}^{p_{m_N}})^* \\[5pt]
& +\text{monomials with strictly less than r }A_q\text{'s}.
\end{array}
\end{align}

We know that the terms without $A_q$'s can be expressed by the dual of the variables according to Proposition \ref{prop:decomp_dual}. Thus by doing an induction on $r$ and using Formula~\eqref{eq:extract_A_(r)}, we conclude the proof of the theorem.
\end{proof}
}

\subsection{Finite generation of the Yoneda algebra}\label{sec:finite_gen}

Let $ \msf k$ be a commutative ring and let $A=\msf k[x_1,\dots, x_n]/(c_1,\dots, c_k)$ be a finitely generated commutative $\msf k$-algebra such that the resolution $P^*$ of Theorem \ref{thm:Tate} of the trivial $A$-module $ \msf k=A/\mathfrak{m}_x$ is obtained after finitely many steps and each $F_n$ is of finite rank over $A$ (see the proof of Theorem \ref{thm:Tate} for the notation). It follows that $\mathcal B$ is a finite set and we write $N$ for its cardinal. We also assume that $P^*$ is an $\mf m_x$-minimal resolution of $ \msf k$, i.e., $d(P^*) \subset \mathfrak{m}_x P^*$, and so $\Ext_{A}^*(\msf k, \msf k)=\mc Hom^*_A(P^*, \msf k)$.
\delete{
As the resolution is $\mf m_x$-minimal, for all $1 \leq i \leq N$, we know that $\partial X_i$ is a sum of monomials in variables $X_j$ with $|X_j| \leq  |X_i|-1$ and with coefficients in $\mathfrak{m}_x$. Hence, for such a monomial $g$, there exists an element $T_{X_i,g} \in T_R(\msf k)_1$ of degree $1$ such that the coefficient of $g$ in $\partial X_i$ is $\partial T_{X_i, g}$. If $g$ does not appear in  $\partial X_i$, we set $T_{X_i, g}=0$.
}
We know from Proposition \ref{prop:decomp_dual} that for any $n >0$, the space $\on{Ext}_A^n(\msf k, \msf k)$ is spanned by elements of the form $(b^\vee)^f$ with $f \in \mathcal{B}(-n)$. In particular, we have seen that $\Ext_{A}^*(\msf k, \msf k)$ is generated over $\msf k$ (with the Yoneda product) by the subspace $V=\bigoplus_{b \in \mathcal B} \msf k b^\vee$, and any element of $\Ext_{A}^*(\msf k, \msf k)$ is an ordered polynomial in the $b^\vee$. Assuming $\msf k$ is a field, it follows that
\[
\bigoplus_{s=0}^k\dim_{\msf k} V^s   \leq  \displaystyle  \bigoplus_{s=0}^k \binom{N-1+s}{N-1} \leq \displaystyle  (k+1)(N+k)^{N-1},
\]
where $V^s$ is the subspace of $\Ext_{A}^*(\msf k, \msf k)$ spanned by the $(b^\vee)^f$ with $\on{poldeg}f=s$. It follows that $\on{GKdim}_{\msf k}\Ext_{A}^*(\msf k, \msf k) \leq N$ (cf. \cite{Krause-Lenagan}) and the Yoneda algebra has finite Gelfand-Kirillov dimension. Based on \cite[Theorem 2.3]{Gulliksen}\delete{ and Theorem \ref{ExtofLocalisation}}, it follows that the local ring $A_{\mathfrak{m}_x}$ is a complete intersection. Furthermore, as stated in Section \ref{sec:Yoneda_CI}, if $A$ is a complete intersection, then there exists a resolution $P^*$ obtained after adjoining variables of degree $-2$. This resolution is $\mf m_x$-minimal if the relations of $A$ are at least of polynomial degree $2$. By combining these results, we obtain:

\delete{
For an element $X^{(f)} \in \mathcal{B}$, we again write $(X^{(f)})^\vee$ for the dual element in $\on{Hom}_R(T_R(\msf k), \msf k)$.
given by
\begin{align*}
\begin{array}[t]{cccl}
(X^{(f)})^*: & T_R(\msf k) & \longrightarrow & \msf k \\[5pt]
& X^{(f')} \in \mathcal{B} & \longmapsto &  \delta_{f, f'}.
\end{array}
\end{align*}

We need to introduce some notations. Later on, we will need to reorder some monomials when the variables are not in ascending order. If in an ordered monomial $X^{(f)}$ there is an ordered sub-monomial $X^{(f')}$, and we wish to move it left of a variable $X_i$ in $X^{(f)}$, the sign change due to the different degrees will be written as $(-1)^{(f', f)_{X_i}}$, and the non-negative scalar factor due to the divider power structures will be denoted $s_{(f', f)_{X_i}}$. We illustrate this process in the following example.

\begin{example}
Consider odd variables $X_1$, $X_2$ and even variables $X_3$, $X_4$. Set
\[
X^{(f)}=X_{1}^{(1)}X_{2}^{(1)}X_{3}^{(3)}X_{4}^{(5)},
\]
\[
X^{(f')}=X_{1}^{(1)}X_{3}^{(1)},
\]
and single out the variable $X_4$. We want to rewrite the monomial $X^{(f)}$ but where the sub-monomial $X^{(f')}$ is left of $X_4$. We can verify that
\[
X_{2}^{(1)}X_{3}^{(2)}\underbrace{(X_{1}^{(1)}X_{3}^{(1)})}_\text{$X^{(f')}$}X_{4}^{(5)}=(-1)^{(f', f)_{X_4}}s_{(f', f)_{X_4}}X^{(f)}
\]
with
\[
(-1)^{(f', f)_{X_4}}=-1,
\]
and
\[
s_{(f', f)_{X_4}}=\frac{3!}{2!1!}=3 \in \bb N.
\]
\end{example}

Notice that the convention we use is that the coefficient appears in front of the ordered monomial $X^{(f)}$. So it is the reordering of the variables that creates the coefficient.

For an element $f \in \mathcal{B}$, we will write $\on{supp_{mul}}f$ for the support of $f$ with multiplicity, i.e., $\on{supp_{mul}}f$ is the set of variables $X \in \mathcal{X}$ such that $f(X) \neq 0$, and it contains $f(X)$ copies of $X$.

Set $f \in \mathcal{B}$ and $X_i \in \on{supp}f$. We define $(\on{deg} < X_i)_f=\sum_{j=1}^{i-1}|X_j|f(X_j)$. For example, if $X^{(f)}=X_{1}X_{2}^{(2)}X_{3}X_{5}X_{6}^{(4)}$ then
\[
\begin{array}{l}
(\on{deg} < X_3)_f=|X_1|+2|X_2|, \\[5pt]
(\on{deg} < X_6)_f=|X_1|+2|X_2|+|X_3|+|X_5|.
\end{array}
\]

For a variable $X_i \in \mathcal{X}$ and $f \in \mathcal{B}(|X_i|-1)$, the coefficient of $X^{(f)}$ in $\partial(X_i)$ will be written as $(x_i)_f \in \mathfrak{m}_x$. As $\partial_1$ is surjective on $\mathfrak{m}_x$ by assumption, there exists an element $T_{X_i, f}$ of degree $1$ in $T_R(\msf k)$ such that $\partial_1(T_{X_i, f})=(x_i)_f$. Moreover, if $f \notin \mathcal{B}(|X_i|-1)$, we set $T_{X_i, f}=0$

For any $X_i \in \mathcal{X}$, define $\delta_{X_i} : \mathcal{X} \longrightarrow \bb N$ such that $\delta_{X_i}(X_j)=\delta_{i, j}$. We now introduce operations on the basis $\mathcal{B}$. Consider a variable $X_i \in \mathcal{X}$ and define the maps
\[
\begin{array}{rccc}
e_{X_i} : & \mathcal{B}(n) & \longrightarrow & \mathcal{B}(n+|X_i|) \\[5pt]
 & X^{(f)} & \longmapsto & X^{(f+\delta_{X_i})}
\end{array}
\]
and
\[
\begin{array}{rccc}
f_{X_i} : & \mathcal{B}(n) & \longrightarrow & \mathcal{B}(n-|X_i|) \\[5pt]
 & X^{(f)} & \longmapsto &  \left\{\begin{array}{cl}
X^{(f-\delta_{X_i})} & \text{ if } f(X_i) \geq 1,\\[5pt] 
 0 & \text{ if } f(X_i)=0.
 \end{array}\right.
\end{array}
\]
Given a collection of variables $X_{i_1}, \dots, X_{i_k}$ and $X_{a}$ in $\mathcal{X}$ such that $|X_a|-\sum_{s=1}^k |X_{i_s}|=1$, we define the map
\[
\begin{array}{rccc}
e_{i_1, \dots, i_k; a} : & \mathcal{B}(n) & \longrightarrow & \mathcal{B}(n+1) \\[5pt]
 & X^{(f)} & \longmapsto &f_{X_{i_1}} \dots f_{X_{i_k}}e_{X_a}X^{(f)}.
\end{array}
\]

For any $1 \leq i \leq N$, we write $X_i^\vee=(X^{(\delta_{X_i})})^\vee$. Moreover, for $f \in \mathcal{B}$, we write $(X^\vee)^{f}=(X_1^\vee)^{f(X_1)} \cdots (X_N^\vee)^{f(X_N)}$ where the $X_i^\vee$'s are in ascending order (notice that there is no divided powers). The proof of the following result is in Appendix \ref{appendix:sec:finite_gen}.

\begin{proposition}\label{prop:decomp_dual}
Set $X^{(f)} \in \mathcal{B}$ such that $f(X_i)=0$ for any $X_i$ with $|X_i|=1$. Then its dual $(X^{(f)})^\vee$ can be expressed as a linear combination of elements of the form $(X^\vee)^{f'}$ with $f' \in \mathcal{B}(\on{deg}f)$. In particular, if $f(X_j) \neq 0$ for some $X_j \in \mathcal{X}$ with $|X_j|=|X_N|$, then
\begin{equation}\label{eq:split_formula}
(X^{(f)})^\vee=(X^\vee)^f.
\end{equation}
\end{proposition}

We can now prove the main result of this section.

\begin{theorem}\label{thm:Yoneda_fg}
Let $R$ be a finitely generated commutative $\msf k$-algebra. Assume that the Tate resolution $T_R(\msf k)$ of the trivial $R$-module $\msf k$ is obtained after adjoining finitely many variables, i.e,
\[
T_R(\msf k)=R \langle X_1, \dots, X_N \rangle, 
\]
and is an $I$-minimal. Then $\Ext_{R}^*(\msf k, \msf k)$ is finitely generated by the $X_i^\vee$'s. Furthermore, every element in $\Ext_{R}^*(\msf k, \msf k)$ of degree $n$ can be expressed as linear combination of elements of the form $(X^\vee)^{f}$ with $f \in \mathcal{B}(n)$.
\end{theorem}

{\color{red} Use earlier argument with the derivations. Then homotopy if finite dimensional and PBW gives the result.}

\begin{proof}
Consider the element $f \in \mathcal{B}$ with $f(X_i)=0$ for all $|X_i|=1$. We lift the morphism $(X^{(f)})^\vee$. We see that 
\begin{align*}
\begin{array}[t]{cccl}
(\widetilde{f})_0: & T_R(\msf k)_{\on{deg}f} & \longrightarrow & R \\[5pt]
& X^{(f')} \in \mathcal{B}(\on{deg}f) & \longmapsto &  \delta_{f, f'},
\end{array} 
\end{align*}
and
\begin{align*}
\begin{array}[t]{cccl}
(\widetilde{f})_1 :& T_R(\msf k)_{\on{deg}f+1} & \longrightarrow & T_R(\msf k)_{1} \\[5pt]
&X_s & \longmapsto & T_{X_s, f}, \\[5pt]
&e_{X_s} X^{(f)}& \longmapsto & X_s, \\[5pt]
& e_{i_1, \dots, i_k; a} X^{(f)}& \longmapsto&(-1)^{(\on{deg} < X_a)_f+\on{deg}f'+(f', f)_{X_a}}s_{(f', f)_{X_a}} T_{X_a, f'}, \\[5pt]
& \text{other} & \longmapsto & 0,
\end{array} 
\end{align*}
where $f'=\sum_{j=1}^k \delta_{X_{i_j}}$. Thus for any $X_i \in \mathcal{X}$ such that $|X_i|=1$, we obtain
\[
\begin{array}{ccl}
(e_{X_i}X^{(f)})^\vee & = &\displaystyle  X_i^\vee \cdot (X^{(f)})^\vee+\text{ linear combination of duals of elements} \\[5pt]
 &&  \text{in } \mathcal{B}(\on{deg}f+1) \text{ not containing variables of degree }1.
 \end{array}
\]
By doing the same construction and increasing the numbers of variables of degree $1$ in $X^{(f)}$, we can show that for an element $f \in \mathcal{B}$ containing $s$ variables of degree $1$ and such that $f(X_i)=0$ for some other $X_i$ with $|X_i|=1$, then
\begin{align}\label{eq:extract_A_(r)}
\begin{array}{rcl}
 (e_{X_i}X^{(f)})^\vee &=& X_i^\vee \cdot (X^{(f)})^\vee +\text{ linear combination of duals of elements in} \\[5pt]
&&\quad  \mathcal{B}(\on{deg}f+1)  \text{ with $s$ or less variables of degree } 1.
\end{array}
\end{align}

We know that the duals of the $X^{(f)}$ without variables of degree $1$ can be expressed as linear combination of the $(X^\vee)^{f'}$ according to Proposition \ref{prop:decomp_dual}. Thus by doing an induction on the number of variables of degree $1$ and using Formula~\eqref{eq:extract_A_(r)}, we conclude the proof of the theorem.
\end{proof}
}

\begin{theorem}\label{thm:Tate_finite}
Let $A$ be a finitely generated commutative $\msf k$-algebra with $\msf k$ a field and with relations of polynomial degree at least $2$. There exists a PD dg resolution $P^*$ of the trivial $A$-module $\msf k$ that is $\mf m_x$-minimal and such that $\mathcal B$ is finite if and only if (the localisation of) $A$ is a complete intersection. In this case, there exists a PD dg resolution such that $\on{min}\{|b| \ | \ b \in \mathcal B\}=-2$, and a presentation of $\pi^*(A,\msf k)$ is given in Theorem \ref{thm:HomotopyLiealgebra}.
\end{theorem}

\delete{
\begin{theorem}[Reformulation of the previous one]
Let $R$ be a Noetherian local ring. Then an $\mathfrak{m}$-minimal resolution is obtained after adjoining finitely many variables if and only if $R$ is a complete intersection. In that case, there exists an $\mathfrak{m}$-minimal resolution $T_R(\msf k)$ generated in degree $1$ and $2$ (so minimal in terms of generators), and a presentation of the Yoneda algebra is given in Corollary \ref{cor:PresentationofExt}.
\end{theorem}
}

We now assume that $A$ is local and not a complete intersection. Hence its $\mf m_x$-minimal PD dg resolution (if it exists) will need an infinite set $\mathcal B$. It follows that its Yoneda algebra will be generated by infinitely many variables. If we assume that there are relations so that $\on{Ext}_A^*(\msf k,\msf k)$ only needs a finite subset $\mathcal B'$ of $\mathcal B$, then the previous reasoning still applies and so any element will be a linear combination of elements of the form $(b'^\vee)^f$ with $b' \in \mathcal B'$. But then we can use the same reasoning as the proof of Theorem \ref{thm:Tate_finite} to see that $A$ is a complete intersection, which is a contradiction. Hence $\on{Ext}_A^*(\msf k,\msf k)$ is generated by an infinite set $\mathcal B^\vee=\{b^\vee \ | \ b \in \mathcal B\}$, and for each $b^\vee \in \mathcal B^\vee$, we have $b^\vee \notin (b'^\vee \ | \ b' \in \mathcal B\backslash \{b\})$. It follows that $\pi^*(A, \msf k)$ will have infinitely many homogeneous components. In particular, we get the following corollary:

\begin{corollary}
Let $A$ be a finitely generated local $\msf k$-algebra with $\msf k$ a field. Assume that there exists an $\mf m_x$-minimal PD dg resolution of the trivial module $\msf k$. Then $\pi^*(A, \msf k)$ is either in degree $1$ and $2$, or has infinitely many homogeneous components.
\delete{
Any restricted graded Lie algebra $\mathfrak{g}=\bigoplus_{i \geq 1} \mathfrak{g}_i$ with finitely many homogeneous degrees and with $\mathfrak{g}_i \neq 0$ for some $i>2$ is not the homotopy Lie algebra of a finitely generated local ring {\color{red}whose $I$-minimal Tate resolution exists}.
}
\end{corollary}

\delete{

\section{The non complete intersection case}\label{sec:non_CI}

Let $\msf k$ be a commutative ring and set $R=\msf k[x_1,\dots, x_n]/( c_1,\dots, c_k)$ where $c_1, \dots, c_k$ are homogeneous polynomials in the variables $x_1, \dots, x_n$ such that they form a minimal set of generators of the ideal $(c_1,\dots, c_k)$. By minimal we mean that $c_i \notin (c_1, \dots, \widehat{c_i}, \dots, c_k)$ for all $1 \leq i \leq k$, so there is no redundancy among the generators. As in the previous section, we write $c_p=\sum_{i=1}^n c_{p, i}x_i$ with $c_{p, i} \in \mf{m}$. 

Set $\mf{m}=(x_1, \dots , x_n) \subset k[x_1, \cdots, x_n]$ and assume that for all $1 \leq i \leq k$, we have $c_i \in \mf{m}^2$. Let $\mf m_x=(\overline{x_1}, \dots, \overline{x_{n}})$ the image of $\mf m$ in $R$, and We would like to determine the Yoneda algebra for the $R$-module $\msf k=R/\mf{m}_x$. However, as $R$ might not be a complete intersection ring, the construction of Theorem \ref{thm:Tate} might never stop, or we might need to add too many variables for the computation to be practical. We are thus going to approximate $R$ by a complete intersection ring $\widetilde{R}$, and from this new ring we will obtain a presentation of a quotient of the Yoneda algebra of $R$.

We order the $c_i$'s such that $c_1,\dots, c_r$, $r \leq k$, is the longest regular sequence among all sequences of $c_i$'s (if there are several such sequences, we just choose one). We define the following ring:
\begin{align*}
\widetilde{R}= \msf k[x_1,\dots, x_n, x_{n+1}, \dots, x_{n+k-r}]/(c_1,\dots, c_r, c_{r+1}-x_{n+1}^2, \dots, c_k-x_{n+k-r}^2).
\end{align*}
The ring $\widetilde{R}$ is clearly a complete intersection because $c_1,\dots, c_r, c_{r+1}-x_{n+1}^2, \dots, c_k-x_{n+k-r}^2$ is a regular sequence in $\msf k[x_1,\dots, x_{n+k-r}]$. Furthermore we have a surjective morphism:
\begin{align*}
\begin{array}[t]{cccl}
\pi: & \widetilde{R} & \longrightarrow & R \\
           & \overline{x_i} & \longmapsto & \left\{\begin{array}{cl}
        \overline{x_i} & \text{ if } i \leq n, \\
       0 & \text{ if } i \geq n+1.
        \end{array}\right.
\end{array}
\end{align*}
Here $\overline{x_i}$ means the class of $x_i$ modulo the ideal of relations. The above morphism induces $\pi^{\#} :\Ext_R^*(\msf k, \msf k) \longrightarrow \Ext_{\widetilde{R}}^*(\msf k, \msf k)$ where $\msf k$ is an $\widetilde{R}$-module through $(\overline{x_1}, \dots, \overline{x_{n+k-r}})$. As discussed at the beginning of Section \ref{sec:Yoneda_CI}, $\widetilde{R}$ is a complete intersection so a free resolution of the $\widetilde{R}$-module $\msf k$ is given as a PD dg $\widetilde{R}$-algebra by:
\begin{align*}
Q_*= \widetilde{R} \langle t_1, \dots, t_{n+k-r}, s_1, \dots, s_k \rangle,
\end{align*}
where the $t_i$'s have degree $1$, the $s_p$'s have degree $2$ and:
\begin{itemize}
\item $\partial_1 t_i = \overline{x_i}$ for all $1 \leq i \leq n+k-r$,
\item $\partial_2 s_p = \displaystyle \sum_{i=1}^n \overline{c_{p, i}} \, t_i$ if $1 \leq j \leq r$,
\item $\partial_2 s_p = \displaystyle \sum_{i=1}^n \overline{c_{p, i}} \, t_i-\overline{x_{n+p-r}} \, t_{n+p-r}$ if $r+1 \leq p \leq k$.
\end{itemize}
To lighten notations, we will remove the $\overline{\color{white}x}$ so it is assumed that we are in the quotient ring. We write $Q_i$ for the subspace of $Q$ of degree $i$. For all  $1 \leq i \leq n+k-r$ and  $1 \leq j \leq k$, we write $\alpha_i=t_i^\vee$ and $\beta_p=s_p^\vee$.
\delete{:
\begin{align*}
\begin{array}[t]{cccc}
\alpha_i: & Q_1 & \longrightarrow & \msf k \\
           & t_i & \longmapsto & 1, \\
           & \text{other} & \longmapsto & 0,
\end{array} \quad \text{ and } \quad \begin{array}[t]{cccc}
\beta_p: & Q_2 & \longrightarrow & \msf k \\
           & s_p & \longmapsto & 1, \\
           & \text{other} & \longmapsto & 0.
\end{array}
\end{align*}
}
Using Corollary \ref{cor:PresentationofExt}, we know that:
\begin{align*}
\Ext_{\widetilde{R}}^*(\msf k,\msf k)=\msf k[\beta_1, \dots, \beta_k] \otimes \msf k \langle \alpha_{1}, \dots, \alpha_{n+k-r} \rangle / \mathcal{I}, 
\end{align*}
where $\mathcal{I}$ is the two sided ideal given by the terms of degree $2$ in the relations of $\widetilde{R}$. In particular, $\alpha_{n+p-r}^2=-\beta_p$ in the quotient for $p \geq r+1$.

\subsection{The lift between the $I$-minimal resolutions}\label{sec:lift_min_res}

We now look into the free resolution of $\msf k$ as an $R$-module. As mentioned above, because $R$ is not a complete intersection if $r < k$, we cannot compute explicitly the whole resolution with a PD dg algebra. We can however compute the first two steps and it will be enough for our purpose. 

We apply Theorem \ref{thm:Tate} but stop after the degree $2$. Using Lemma \ref{lem:H_1}, we obtain a PD dg $R$-algebra generated by $n$ variables $T_1, \dots, T_n$ of degree $1$ such that $\partial_1(T_i)=x_i$, and by $k$ variables $S_1,\dots, S_k$ of degree $2$ such that $\partial_2(S_p)=\sum_{i=1}^n c_{p, i} T_i$. This complex is not exact. We would need to add variables of higher degrees to obtain the free resolution $P_*$ of $\msf k$ as an $R$-module. Hence the beginning of the $R$-algebra $P_*$ is as follows:
\[
  \begin{tikzpicture}[scale=0.9,  transform shape]
  \tikzset{>=stealth}
  
\node (1) at ( -8,0){$\dots$};
\node (2) at ( -4,0){$\displaystyle \bigoplus_{p=1}^k RS_p \oplus \bigoplus_{1 \leq i < j \leq n} R T_i T_j$};
\node (3) at ( 0,0){$\displaystyle \bigoplus_{i=1}^nRT_i$};
\node (4) at ( 2.5,0) {$R$};
\node (5) at ( 4.5,0){$\cc$};
\node (6) at ( 6.5,0){$0$,};

\draw [decoration={markings,mark=at position 1 with
    {\arrow[scale=1.2,>=stealth]{>}}},postaction={decorate}] (1) --  (2) node[midway, above] {$\partial_3$};
\draw [decoration={markings,mark=at position 1 with
    {\arrow[scale=1.2,>=stealth]{>}}},postaction={decorate}] (2)  --  (3) node[midway, above] {$\partial_2$};
\draw [decoration={markings,mark=at position 1 with
    {\arrow[scale=1.2,>=stealth]{>}}},postaction={decorate}] (3)  --  (4) node[midway, above] {$\partial_1$};
\draw [decoration={markings,mark=at position 1 with
    {\arrow[scale=1.2,>=stealth]{>}}},postaction={decorate}] (4)  --  (5) node[midway, above] {$\epsilon$};
\draw [decoration={markings,mark=at position 1 with
    {\arrow[scale=1.2,>=stealth]{>}}},postaction={decorate}] (5)  --  (6) ;
\end{tikzpicture}
\]
with $\epsilon$ the augmentation map. 

We call $\psi$ the lift between the free resolutions of $\msf k$ as an $R$-module and as an $\widetilde{R}$-module:
 \begin{center}
  \begin{tikzpicture}[scale=0.9,  transform shape]
  \tikzset{>=stealth}
  
\node (1) at ( -8,0){$\dots$};
\node (2) at ( -4,0){$\displaystyle \bigoplus_{p=1}^k \widetilde{R}s_p \oplus \bigoplus_{1 \leq i < j \leq n+k-r} \widetilde{R} t_i t_j$};
\node (3) at ( 1,0){$\displaystyle \bigoplus_{i=1}^{n+k-r} \widetilde{R}t_i$};
\node (4) at ( 4,0) {$\widetilde{R}$};
\node (5) at ( 6,0){$\msf k$};
\node (6) at ( 8,0){0};
  
\node (7) at ( -8,-2){$\dots$};
\node (8) at ( -4,-2){$\displaystyle \bigoplus_{p=1}^k RS_p \oplus \bigoplus_{1 \leq i < j \leq n} R T_i T_j$};
\node (9) at ( 1,-2){$\displaystyle \bigoplus_{i=1}^nRT_i$};
\node (10) at ( 4,-2) {$R$};
\node (11) at ( 6,-2){$\msf k$};
\node (12) at ( 8,-2){0};

\draw [decoration={markings,mark=at position 1 with
    {\arrow[scale=1.2,>=stealth]{>}}},postaction={decorate}] (1) --  (2) node[midway, above] {$\partial_3$};
\draw [decoration={markings,mark=at position 1 with
    {\arrow[scale=1.2,>=stealth]{>}}},postaction={decorate}] (2)  --  (3) node[midway, above] {$\partial_2$};
\draw [decoration={markings,mark=at position 1 with
    {\arrow[scale=1.2,>=stealth]{>}}},postaction={decorate}] (3)  --  (4) node[midway, above] {$\partial_1$};
\draw [decoration={markings,mark=at position 1 with
    {\arrow[scale=1.2,>=stealth]{>}}},postaction={decorate}] (4)  --  (5) node[midway, above] {$\widetilde{\epsilon}$};
\draw [decoration={markings,mark=at position 1 with
    {\arrow[scale=1.2,>=stealth]{>}}},postaction={decorate}] (5)  --  (6);
    
\draw [decoration={markings,mark=at position 1 with
    {\arrow[scale=1.2,>=stealth]{>}}},postaction={decorate}] (7) --  (8) node[midway, above] {$\partial_3$};
\draw [decoration={markings,mark=at position 1 with
    {\arrow[scale=1.2,>=stealth]{>}}},postaction={decorate}] (8)  --  (9) node[midway, above] {$\partial_2$};
\draw [decoration={markings,mark=at position 1 with
    {\arrow[scale=1.2,>=stealth]{>}}},postaction={decorate}] (9)  --  (10) node[midway, above] {$\partial_1$};
\draw [decoration={markings,mark=at position 1 with
    {\arrow[scale=1.2,>=stealth]{>}}},postaction={decorate}] (10)  --  (11) node[midway, above] {$\epsilon$};
\draw [decoration={markings,mark=at position 1 with
    {\arrow[scale=1.2,>=stealth]{>}}},postaction={decorate}] (11)  --  (12);  
    
\draw [decoration={markings,mark=at position 1 with
    {\arrow[scale=1.2,>=stealth]{>}}},postaction={decorate}] (2) --  (8) node[midway, right] {$\psi_2$};
\draw [decoration={markings,mark=at position 1 with
    {\arrow[scale=1.2,>=stealth]{>}}},postaction={decorate}] (3)  --  (9) node[midway, right] {$\psi_1$};
\draw [decoration={markings,mark=at position 1 with
    {\arrow[scale=1.2,>=stealth]{>}}},postaction={decorate}] (4)  --  (10) node[midway, right] {$\pi$};
\draw [decoration={markings,mark=at position 1 with
    {\arrow[scale=1.2,>=stealth]{>}}},postaction={decorate}] (5)  --  (11) node[midway, right] {$\operatorname{id}$};
\end{tikzpicture}
 \end{center}
where the arrows going down are morphisms of $\widetilde{R}$-modules. Indeed, there exists an action of $\widetilde{R}$ on the free resolution of $\msf k$ as an $R$-module induced by the morphism $\pi: \widetilde{R} \longrightarrow R$ and the resolution of the $\widetilde{R}$-module is projective.

In order to make the diagram commute, it is possible to define $\psi_1$ and $\psi_2$ as:
\begin{align*}
 \begin{array}[t]{cccl}
\psi_1: & Q_1& \longrightarrow & P_1 \\
           & t_i & \longmapsto & \left\{\begin{array}{l}
        T_i  \text{ if } i \leq n,\\
        0  \text{ if } i \geq n+1,
        \end{array}\right.
\end{array} \quad \text{and} \quad   \begin{array}[t]{cccl}
\psi_2: & Q_2& \longrightarrow & P_2 \\
           & t_i t_j & \longmapsto & \left\{\begin{array}{l}
        T_i T_j \text{ if } i, j \leq n,\\
        0 \text{ otherwise},
        \end{array}\right. \\
        & s_p & \longmapsto & S_p.
\end{array}
\end{align*}

The next proposition gives an explicit realisation of $\psi_n$ for any $n \in \mathbb{Z}_{>0}$ (proof in Appendix \ref{appendix:sec:lift_min_res}).

\begin{proposition}\label{prop:LiftofPsi}
Let $m \in \mathbb{Z}_{>0}$ and set $t_1^{a_1}\dots t_{n+k-r}^{a_{n+k-r}}s_1^{(b_1)}\dots s_k^{(b_k)} \in Q_{m}$. Then $\psi$ is a PD dg algebra homomorphism with:
\begin{align}
\psi_m(t_1^{a_1}\dots t_{n+k-r}^{a_{n+k-r}}s_1^{(b_1)}\dots s_k^{(b_k)})=T_1^{a_1} \dots T_n^{a_n}0^{a_{n+1}} \dots 0^{a_{n+k-r}}S_1^{(b_1)}\dots S_k^{(b_k)}.
\end{align}
\end{proposition}

\subsection{A finitely generated quotient in the non complete intersection case}\label{sec:fg_quotient}

The induced morphism $\pi^{\#}: \Ext_R^*(\msf k, \msf k) \longrightarrow \Ext_{\widetilde{R}}^*(\msf k, \msf k)$ is given by the precomposition with $\psi$ (see \cite[Chap. III, Thm.6.7]{MacLane}). Before computing $\pi^\#$ explicitly, we need to determine the Yoneda product in $\Ext_R^*(\msf k, \msf k)$ for elements of small degree.

For $1 \leq i \leq n$ and $1 \leq p \leq k$, let $\gamma_i=T_i^\vee$ and $\delta_p=S_p^\vee$ be the duals of $T_i$ and $S_p$, respectively.
\delete{i.e.:
\begin{align*}
\begin{array}[t]{cccl}
\gamma_i: & P_1& \longrightarrow & \msf k \\
           & T_l & \longmapsto & \delta_{i, l},
\end{array} \quad  \text{ and } \quad \begin{array}[t]{cccl}
\delta_p: & P_2& \longrightarrow & \msf k \\
           & S_l & \longmapsto & \delta_{p, l} \\
        & T_l T_m & \longmapsto & 0 \text{ for all }1 \leq l, m \leq n.
           \end{array} 
\end{align*}
}
We know from the construction of the beginning of the free resolution $P_*$ that $\partial_1^*=\partial_2^*=0$. It follows that $\Ext_R^1(\msf k, \msf k)=\text{Hom}_R(P_1,\msf k)=\bigoplus_{i=1}^n\msf k \gamma_i$, and $\Ext_R^2(\msf k, \msf k)=\text{Ker}(\partial_3^*)$. We cannot express $\text{Ker}(\partial_3^*)$ explicitly because we do not know the variables we may need to add to $P_3$. However, we can prove the following lemma (proof in Appendix \ref{appendix:sec:fg_quotient}):

\begin{lemma}\label{lem:ext2_non_CI}
The $\delta_p$'s defined above are elements of $\Ext_R^2(\cc, \cc)$.
\end{lemma} 

\begin{remark}
The necessity of the assumption on the degrees of the generators of the relation ideal comes from the previous lemma. Indeed, if some $c_p$ were to have monomials of different degrees, then the above result might not be true. We give the following counter example. Let $R'=\msf k[x, y]/(c_1, c_2)$ with $c_1=x^3$ and $c_2=x^2+x^2y$. It is not a complete intersection as $xc_2=0$ in $\msf k[x,y]/(c_1)$. The ideal $(c_1, c_2)$ is clearly minimally generated by $c_1$ and $c_2$, and $c_2$ is not homogeneous. If we compute the beginning of the free resolution of $\msf k=R'/(x, y)$, we obtain:
\begin{itemize}\setlength\itemsep{5pt}
\item $\partial_1 T_x = x$, $\partial_1 T_y = y$,
\item $\partial_2 S_x = x^2T_x$,  $\partial_2 S_y = xT_x+x^2T_y$.
\end{itemize}
We can verify that $(1+y)S_x \in \on{Ker}(\partial_2)$. Hence $\delta_x$ is not a cocycle, and so is not an element of $\Ext^2_{R'}(\msf k, \msf k)$.
\end{remark}

{\color{red} We need to assume that $P_*$ is $I$-minimal so that $\on{Ext}_R^*(\msf k,\msf k)=\on{Hom}_R(P_*,\msf k)$. Otherwise $(T_1^{a_1}\cdots T_{n}^{a_n}S_1^{(b_1)}\cdots S_k^{(b_k)})^*$ and $z$ might not be elements in $\on{Ext}^*(\msf k,\msf k)$, in which case we do not need Lemma \ref{lem:ext2_non_CI}.}

{\color{red}Not true if the relations have degree 2 terms (see Lemma \ref{lem:RelationsinExt}).} We can directly compute that, for all $1 \leq i < j \leq n$, we have $\gamma_i \gamma_j=-\gamma_j \gamma_i=(T_i T_j)^*$. Furthermore, we can check that $\gamma_i^2=0$ for all $i$. As we have just found a description of $\psi$, we can see that:
\[
\begin{array}[t]{cccl}
\pi^{\#}(\gamma_i)=\gamma_i \circ \psi_1: & Q_1 & \longrightarrow & \cc \\
           & t_j & \longmapsto & \delta_{i, j}.
\end{array}
\]
Therefore for all $1 \leq i \leq n$, we have $\pi^{\#}(\gamma_i)=\alpha_i$. We do the same for $\delta_p$ and find that $\pi^{\#}(\delta_p)= \beta_p$. As we do not know the basis of $P_3$ and $P_4$, we cannot determine explicitly the commutation relations between $\gamma_i$ and $\delta_p$, and between $\delta_p$ and $\delta_q$. However, it is known that the induced morphism $\pi^{\#}$ is a morphism of algebras preserving the Yoneda product (it is easily seen when the product is done by splicing $n$-fold exact sequences), thus $\pi^{\#}(\gamma_i \delta_p-\delta_p \gamma_i)= \alpha_i \beta_p-\beta_p \alpha_i=0$, based on the commutation relations in $\Ext_{\widetilde{R}}^*(\cc, \cc)$. Thus $\gamma_i \delta_p \equiv \delta_p \gamma_i \text{ mod } \text{Ker}(\pi^{\#})$. With the same approach we obtain $\delta_p \delta_q \equiv \delta_q \delta_p \text{ mod } \on{Ker}(\pi^{\#})$. It follows that $\on{Im}(\pi^{\#})$ contains the subalgebra of $\Ext_{\widetilde{R}}^*(\cc, \cc)$ generated by the $\alpha_i$'s and $\beta_p$'s, for $i \leq n$. We intend to prove that these two spaces are in fact equal.

\begin{proposition}\label{prop:MonomialContainingDual}
Set $m \in \mathbb{Z}_{>0}$ and $T_1^{a_1}\cdots T_{n}^{a_n}S_1^{(b_1)}\cdots S_k^{(b_k)} \in P_m$. Then there exist $a \in \msf k$ and $z \in \on{Ker}(\pi^{\#})$ such that $(T_1^{a_1}\cdots T_{n}^{a_n}S_1^{(b_1)}\cdots S_k^{(b_k)})^*=a \gamma_1^{a_1}\cdots \gamma_{n}^{a_n}\delta_1^{b_1}\cdots \delta_k^{b_k}+z$.
\end{proposition}

{\color{red}Put the proof in the appendix}
\begin{proof}[Proof of Proposition \ref{prop:MonomialContainingDual}]
We use the same notations as the proof of Proposition \ref{BasisofExtdegree2}. We write $\mathcal{P}(m)$ for the statement of the proposition for a fixed $m \in \mathbb{Z}_{>0}$. Based on previous computations, we know that $\mathcal{P}(1)$ and $\mathcal{P}(2)$ are true. We will proceed by induction.
 
$\bullet$ \underline{Even case:} Assume $\mathcal{P}(m)$ is true up to some $m$ even. Take $T_1^{a_1}\dots T_n^{a_n}S_1^{(b_1)}\dots S_k^{(b_k)} \in P_{m+1}$. Because $m+1$ is odd and the degrees of $T_i$'s and $S_p$'s are $1$ and $2$ respectively, the element is of the form $T_{i_1}T_{i_2} \dots T_{i_s}S_1^{(b_1)}\dots S_k^{(b_k)}$ with $i_1 < \dots < i_s$ and $s$ odd. Therefore $X=T_{i_2} \dots T_{i_s}S_1^{(b_1)}\dots S_k^{(b_k)} \in P_{m}$ and, using the assumption, its dual element is $f_X=a (\gamma_{i_2} \dots \gamma_{i_s}\delta_1^{b_1}\dots \delta_k^{b_k}+z)$ with $a \in \cc^*$ and $z \in \on{Ker}(\pi^{\#})$. We are going to lift this morphism.
 \begin{center}
  \begin{tikzpicture}[scale=0.9,  transform shape]
  \tikzset{>=stealth}
  
\node (1) at ( 0,0){$P_{m+1}$};
\node (2) at ( 3,0){$P_m$};
\node (3) at ( 0,-2){$P_1$};
\node (4) at ( 3,-2) {$R$};
\node (5) at ( 5,-2){$\cc$};
\node (6) at ( 7,-2){0};

\node (7) at ( -2,0){$\dots$};
\node (8) at ( -2,-2){$\dots$};

\draw [decoration={markings,mark=at position 1 with
    {\arrow[scale=1.2,>=stealth]{>}}},postaction={decorate}] (1) --  (2) node[midway, above] {$\partial_{m+1}$};
\draw [decoration={markings,mark=at position 1 with
    {\arrow[scale=1.2,>=stealth]{>}}},postaction={decorate}] (1)  --  (3) node[midway, left] {$g_2$};

\draw [decoration={markings,mark=at position 1 with
    {\arrow[scale=1.2,>=stealth]{>}}},postaction={decorate}] (2)  --  (4) node[midway, left] {$g_1$};
\draw [decoration={markings,mark=at position 1 with
    {\arrow[scale=1.2,>=stealth]{>}}},postaction={decorate}] (3)  --  (4) node[midway, above] {$\partial_1$};

\draw [decoration={markings,mark=at position 1 with
    {\arrow[scale=1.2,>=stealth]{>}}},postaction={decorate}] (2)  --  (5) node[midway, above right] {$f_X$};
\draw [decoration={markings,mark=at position 1 with
    {\arrow[scale=1.2,>=stealth]{>}}},postaction={decorate}] (4)  --  (5) node[midway, above] {$\epsilon$};
    
\draw [decoration={markings,mark=at position 1 with
    {\arrow[scale=1.2,>=stealth]{>}}},postaction={decorate}] (5)  --  (6);
    
\draw [decoration={markings,mark=at position 1 with
    {\arrow[scale=1.2,>=stealth]{>}}},postaction={decorate}] (7)  --  (1) node[midway, above] {$\partial_{m+2}$};
\draw [decoration={markings,mark=at position 1 with
    {\arrow[scale=1.2,>=stealth]{>}}},postaction={decorate}] (8)  --  (3) node[midway, above] {$\partial_2$};
\end{tikzpicture}
 \end{center}
Define $g_1:  P_m  \longrightarrow  R$ such that it sends $X$ to $1$ and everything else is sent to zero, and set $\on{I}'=\{i_2, \dots, i_s\}$ and $\on{I}=\{1, \dots, n\} \backslash \on{I}'$. These two sets have an empty intersection. For any $i \in \on{I}$, $g_1\partial_{m+1}(T_iX)=x_i=\partial_1(T_i)$. Then define $\on{J}$ the subset of $\{1, \dots, k \} \times \on{I}'$ given by the pairs $(p, i_p')$ where $\partial(S_p)$ contains $T_{i_p'}$. We see that for any $(p, i_p') \in \on{J}$, we have: 
\[
g_1\partial_{m+1}(T_{i_2} \dots \widehat{T_{i_p'}} \dots T_{i_s}S_1^{b_1}\dots S_p^{b_p+1} \dots S_k^{b_k})= (-1)^{d_{p, i_p'}}(b_p+1)c_{p,  i_p'}.
\]
where $d_{p, i_p'}$ is an integer. There might be other $Y \in P_{m+1}$ such that $\partial_{m+1}(Y)$ contains $X$, but then $Y$ would contain at least one variable different from $T_i, S_p$. We can thus define:
\begin{align*}
\begin{array}[t]{cccl}
g_2: & P_{m+1} & \longrightarrow & P_1 \\
           & T_iT_{i_2} \dots  T_{i_s}S_1^{(b_1)} \cdots S_k^{(b_k)}  & \longmapsto & T_i  \ \forall \ i \in \on{I}, \\
           & T_{i_2} \dots \widehat{T_{i_p'}} \dots T_{i_s}S_1^{b_1}\dots S_p^{b_p+1} \dots S_k^{b_k}  & \longmapsto &\displaystyle  (-1)^{d_{p, i_p'}}(b_p+1)C_{p,  i_p'}  \ \forall \ (p, i_p')  \in \on{J}, \\
           & Y & \longmapsto & g_2(Y), \\
           & \text{other} & \longmapsto & 0.
\end{array} 
\end{align*}
We have thus have:
\begin{align*}
\begin{array}{rcl}
\gamma_{i_1} f_X & = & (T_{i_1}T_{i_2}\dots T_{i_s}S_1^{(b_1)}\dots S_k^{(b_k)})^* \\[5pt]
 & &\displaystyle +\sum_{ (p, i_p') \in \on{J}} (-1)^{d_{p, i_p'}}(b_p+1)n_{i_1}^{p,  i_p'}(T_{i_2} \dots \widehat{T_{i_p'}} \dots T_{i_s}S_1^{b_1}\dots S_p^{b_p+1} \dots S_k^{b_k})^*  + \sum_Y Y^*,
\end{array}
\end{align*}
where the sum is taken over a subset (or the whole of) the $Y$'s above. Therefore each monomial in this sum contains a variable different from $T_i, S_p$.

We know by construction that $i_1 < i_p'$ for all $i_p' \in \on{I}'$, so $n^{p, i_p'}_{i_1}=0$. {\color{red} This is wrong because we do not have this assumption anymore. Change the statement to look like Prop.\ref{BasisofExtdegree2} by using linear combinations.} Furthermore, the kernel of $\pi^{\#}$ is an ideal of $\Ext_R^*(\cc, \cc)$ so $\gamma_i z \in \text{Ker}(\pi^{\#})$. Because each $Y$ in $\sum_Y Y^*$ is composed of monomials with a variable different from $T_i, S_p$, and the fact that $\on{Im}(\psi)$ is generated by the $T_i, S_p$ (Proposition \ref{prop:LiftofPsi}), we know that $\sum_Y Y^* \in \text{Ker}(\pi^{\#})$. We thus have:
\begin{align*}
(T_{i_1}T_{i_2}\dots T_{i_s}S_1^{(b_1)}\dots S_k^{(b_k)})^*= a(\gamma_{i_1} \gamma_{i_2} \dots  \gamma_{i_s}\delta_1^{b_1}\dots \delta_k^{b_k}+z')
\end{align*}
where $z'=\gamma_{i_1} z-\frac{1}{a}\sum_Y Y^* \in \text{Ker}(\pi^{\#})$. As any element of $P_{m+1}$ composed of $T_i, S_p$ is of this form, we see that $\mathcal{P}(m+1)$ is true.

$\bullet$ \underline{Odd case:} Assume $\mathcal{P}(m)$ is true up to some $m$ odd. An element $T_1^{a_1}\dots T_n^{a_n}S_1^{(b_1)}\dots S_k^{(b_k)} \in P_{m+1}$ has two possible configurations:
\begin{enumerate}
\item There exist $i \neq j$ such that $a_i=a_j=1$,
\item $a_1= a_2 = \dots =a_n=0$.
\end{enumerate}

In the first case, we can apply the same strategy as above and obtain the desired result.

In the second case we have $S_1^{(b_1)}\dots S_k^{(b_k)} \in P_{m+1}$. Thus there exists $1 \leq q \leq k$ such that $b_q \geq 1$. Using the assumption the $\mathcal{P}(m)$ is true up to $m$, we know that there exists $a \in \cc^*$ and $z \in \text{Ker}(\pi^{\#})$ such that $f_X=a(\delta_1^{b_1}\dots \delta_q^{b_q-1} \dots \delta_k^{b_k}+z)$ corresponds to $X=S_1^{b_1}\dots S_q^{b_q-1} \dots  S_k^{b_k} \in P_{m-1}$. By lifting $f_X$ we find that:
\begin{align*}
\begin{array}[t]{cccl}
g_1: & P_{m-1} & \longrightarrow & R\\
           & X & \longmapsto & 1,\\
           & \text{other} & \longmapsto & 0, \\
\end{array}  \hspace{4em}   \begin{array}[t]{cccl}
g_2: & P_{m} & \longrightarrow & P_1 \\
           & T_iX& \longmapsto & T_i \text{ for all } 1 \leq i \leq n,  \\
           & Y & \longmapsto & g_2(Y), \\
           & \text{other} & \longmapsto & 0, \\
\end{array} 
\end{align*}
where the $Y$'s are monomials of $P_m$ such that $\partial_m(Y)$ contains $X$, so they contain at least one variable different from $T_i, S_p$. Finally we obtain:
\begin{align*}
\begin{array}[t]{cccl}
g_3: & P_{m+1} & \longrightarrow & P_2 \\
           & T_iT_jX & \longmapsto & T_iT_j, \text{ for all } 1 \leq i <j \leq n,  \\
           & S_p X & \longmapsto & \left\{\begin{array}{cl}
        (b_p+1)S_p & \text{ if } p \neq q,\\[3pt]
        b_q S_q  & \text{ if } p=q,
        \end{array}\right. \\
           & \Omega & \longmapsto & g_3(\Omega), \\
           & \Psi & \longmapsto & g_3(\Psi), \\
           & \text{other} & \longmapsto & 0,
\end{array} 
\end{align*}
where the $\Omega$'s and $\Psi$'s are other monomials such that $\partial_{m+1}(\Omega)$ contains $T_iX$ for some $1 \leq i \leq n$, and such that $\partial_{m+1}(\Psi)$ contains one of the $Y$'s in the definition of $g_2$. We see that in both cases, $\Omega$ and $\Psi$ contain at least one variable different from $T_i, S_p$. We have thus found:
\begin{align*}
\frac{1}{b_q}\delta_q f_X=\frac{a}{b_q}\delta_q(\delta_1^{b_1}\dots \delta_q^{b_q-1} \dots \delta_k^{b_k}+z)=(S_1^{b_1}\dots S_q^{b_q} \dots S_k^{b_k})^*+\sum_\Omega \Omega^* +\sum_\Psi \Psi^*,
\end{align*}
where the sums are taken over subsets (or the whole of) the $\Omega$'s and $\Psi$'s above. Therefore each monomial in these sums contains a variable different from $T_i, S_p$.

The kernel of $\pi^{\#}$ is an ideal of $\Ext_R^*(\cc, \cc)$ so $\delta_q z \in \text{Ker}(\pi^{\#})$. Furthermore we know that the $\delta_i$'s commute modulo $\text{Ker}(\pi^{\#})$, so there exists $z' \in \text{Ker}(\pi^{\#})$ such that $\delta_q\delta_1^{b_1}\dots \delta_q^{b_q-1} \dots \delta_k^{b_k} =\delta_1^{b_1}\dots \delta_q^{b_q} \dots \delta_k^{b_k} +z'$.

Because each $\Omega$ and $\Psi$ in $\sum_\Omega \Omega^* +\sum_\Psi \Psi^*$ is composed of monomials with a variable different from $T_i, S_p$, and the fact that $\on{Im}(\psi)$ is generated by the $T_i, S_p$ (Proposition \ref{prop:LiftofPsi}), we know that $\sum_\Omega \Omega^* +\sum_\Psi \Psi^* \in \text{Ker}(\pi^{\#})$. As a consequence:
\begin{align*}
(S_1^{b_1}\dots S_q^{b_q} \dots S_k^{b_k})^*= \frac{a}{b_q}(\delta_1^{b_1}\dots \delta_q^{b_q} \dots \delta_k^{b_k}+z''),
\end{align*}
where $z''=(z'+\delta_q z-\frac{b_q}{a}\sum_\Omega \Omega^* -\frac{b_q}{a}\sum_\Psi \Psi^*) \in \text{Ker}(\pi^{\#})$, so $\mathcal{P}(m+1)$ is true.

We have found that $\mathcal{P}$ is hereditary and is true for $m=1,2$. Therefore $\mathcal{P}(m)$ is true for all $m \in \mathbb{Z}_{>0}$. 
\end{proof}

We know that $\on{Im}(\psi)$ is generated by the $T_i, S_p$, so any element of $\Ext_R^*(\cc, \cc)$ corresponding to a monomial containing variables different from $T_i, S_p$ is in the kernel of $\pi^\#$. It follows that:
\[
\begin{array}{rl}
\text{Ker}(\pi^{\#})= & \text{Span} \{\text{duals of monomials in } P \text{ containing at least} \\
 & \quad \quad \quad   \text{one variable different from }T_i, S_p\}.
\end{array}
\]
\indent A consequence of Proposition \ref{prop:MonomialContainingDual} is that 
\begin{align*}
\begin{array}{rcl}
\pi^\#((T_1^{a_1}\dots T_n^{a_n}S_1^{(b_1)}\dots S_k^{(b_k)})^*) & = & \pi^\#(a\gamma_1^{a_1}\dots \gamma_n^{a_n}\delta_1^{b_1}\dots \delta_k^{b_k}), \\[5pt]
& = & a\alpha_1^{a_1}\dots \alpha_n^{a_n}\beta_1^{b_1}\dots \beta_k^{b_k},
\end{array}
\end{align*}
for some $a \in \cc^*$, and so $\on{Im}(\pi^{\#})$ is generated by the $\alpha_i$'s ($1 \leq i \leq n)$ and the $\beta_p$'s. As we know the relations between the $\alpha_i$'s and the $\beta_p$'s from Corollary \ref{cor:PresentationofExt}, we have proved the following result:

{\color{red} Issue due to the linear combinations. $\on{Im} \pi^\#$ contains only linear combinations of monomials in $\alpha, \beta$. Also $\pi^\#$ is an algebra morphism and $\on{Im} \pi^\#$ contains $\alpha$ and $\beta$. Hence $\on{Im} \pi^\#=\msf k\langle \alpha, \beta \rangle$.}

\begin{theorem}\label{QuotientOfExt}
Set $R=\cc[x_1,\dots, x_n]/(c_1,\dots, c_k)$ where $c_1, \dots, c_k$ are homogeneous polynomials in the variables $x_1, \dots, x_n$ such that they form a minimal set of generators of the ideal $(c_1,\dots, c_k)$. We further assume that $c_p \in \mf{m}_x^2$ for all $1 \leq i \leq k$, and write $c_p \equiv \sum_{i \leq j}n^p_{i, j}x_ix_j \text{ mod } \mathfrak{m}_x^3$. We consider the $R$-module $\cc=R/\mf{m}_x$. Then there exists an ideal $\mathcal{J}$ of $\Ext_R^*(\cc, \cc)$ such that:
\begin{align*}
\Ext_R^*(\cc, \cc)/ \mathcal{J} \cong \cc[\beta_1, \dots, \beta_k] \otimes \cc \langle \alpha_{1}, \dots, \alpha_{n} \rangle / \mathcal{I}, 
\end{align*}
where $\mathcal{I}$ is the two sided ideal generated by $\alpha_i \alpha_j +\alpha_j \alpha_i-\sum_{p=1}^kn^p_{i, j}\beta_p$ for all $1 \leq i < j \leq n$, and by $\alpha_i^2-\sum_{p=1}^k n^p_{i, i}\beta_p$ for all $1 \leq i \leq n$. In particular, if for all $1 \leq p \leq k$ we have $c_p \in \mf{m}_x^3$, then $\Ext_R^*(\cc, \cc)/ \mathcal{J}$ is a strictly graded commutative quotient of $\Ext_R^*(\cc, \cc)$.
\end{theorem}

\begin{corollary}
Set $R=\cc[x_1,\dots, x_n]/(c_1,\dots, c_k)$ where $c_1, \dots, c_k$ are homogeneous polynomials in the variables $x_1, \dots, x_n$ such that they form a minimal set of generators of the ideal $(c_1,\dots, c_k)$. We further assume that $c_p \in \mf{m}_x^2$ for all $1 \leq i \leq k$, and write $c_p \equiv \sum_{i \leq j}n^p_{i, j}x_ix_j \text{ mod } \mathfrak{m}_x^3$. Then the quotient $\overline{\pi^*(R, \msf k)}=\pi^*(R, \msf k)/\pi^{>2}(R)$ is a restricted graded Lie algebra
\begin{align*}
\overline{\pi^*(R, \msf k)}=\bigoplus_{i=1}^n\msf k \overline{\alpha_i} \oplus \bigoplus_{p=1}^k \msf k \overline{\beta_p}
\end{align*}
such that $\on{deg}\overline{\alpha_i}=1$, $\on{deg}\overline{\beta_p}=2$, and the restricted Lie structure is given by:
\begin{empheq}[left = \empheqlbrace]{align}
    [\overline{\alpha_i}, \overline{\alpha_j}] & =  \displaystyle\sum_{p=1}^k (n^p_{i, j}+n^p_{j, i})\overline{\beta_p}  \text{ for all } 1 \leq i \neq j \leq n ;\\[5pt]
[\overline{\alpha_i}, \overline{\alpha_i}] & =   \sum_{p=1}^k 2n^p_{i, i}\overline{\beta_p} \text{ for all } 1 \leq i \leq n ; \\[5pt]
[\overline{\beta_p}, \overline{\pi^*(R, \msf k)}] & =  0 \text{ for all } 1 \leq p \leq k; \\[5pt]
q(\overline{\alpha_i}) & =\sum_{p=1}^k n^p_{i, i}\overline{\beta_p} \text{ for all } 1 \leq i \leq n.
\end{empheq}

\end{corollary}
As it is known that the Gelfand-Kirillov dimension of the quotient of an algebra is less or equal to the Gelfand-Kirillov dimension of the algebra (cf. \cite[Lemma 3.1]{Krause-Lenagan}). By the relations defining $\mathcal{I}$, we see that there exists $k' \leq k$ and a subset $\{\beta_{i_1}, \dots, \beta_{i_{k'}}\} \subset \{\beta_1,\dots \beta_k\}$ such that
\[
(\cc[\beta_1, \dots, \beta_k] \otimes \cc \langle \alpha_{1}, \dots, \alpha_{n} \rangle / \mathcal{I})/(\alpha_1,\dots, \alpha_n)=\cc[\beta_{i_1}, \dots, \beta_{i_{k'}}].
\]
We can then see:

\begin{corollary}
Let $R$ be a ring as in Theorem \ref{QuotientOfExt}. Then:
\[
\on{GKdim} \Ext_R^*(\cc, \cc) \geq k'.
\]
\end{corollary}

The next lemma states that the elements of a regular sequence form a minimal set of generators of the ideal generated by the sequence. Therefore, in the event where the generators are homogeneous, Corollary \ref{cor:PresentationofExt} becomes of special case of Theorem \ref{QuotientOfExt}, in which $\mathcal{J}=0$.

\begin{lemma}
Let $\{c_1, \dots, c_k\}$ be a regular sequence in $\cc[x_1,\dots, x_n]$. Then the $c_i$'s are minimal generators of the ideal $(c_1, \dots, c_k)$, i.e., there is no $1 \leq i \leq k$ such that $c_i \in (c_j \ | \ j \neq i)$.
\end{lemma}

\begin{proof}
Let $S=\cc[x_1,\dots, x_n]$ and $\{c_1, \dots, c_k\}$ be as in the statement. For $s \leq k$, we will write $I_s=(c_1, \dots, c_s)$. By contradiction, assume that there exists $1 \leq i \leq k$ such that $c_i=\sum_{j \neq i}f_j c_j$, with $f_j \in S$.

It is impossible to have all $j \leq i-1$, otherwise $c_i \equiv 0 \on{mod} I_{i-1}$, contradicting the regularity of the sequence.

We write $c_i=f_1c_1+\dots +f_j c_j$ with $j>i$ the highest index appearing in the sum. Then $f_j c_j \equiv 0 \on{mod} I_{j-1}$. If $f_j \notin I_{j-1} $, then $c_j$ is a divisor of zero in $S/I_{j-1}$, which is impossible. So there exist $g_1, \dots , g_{j-1}$ such that $f_j=g_1c_1+\dots+g_{j-1}c_{j-1}$. By writing $\tilde{f}_l=f_l+g_lc_j$, we obtain:
\[
c_i(1-g_ic_j)=\sum_{\substack{ l \leq j-1\\ l \neq i}} c_l \tilde{f}_l.
\]
By repeating the same procedure for $c_{j-1}\tilde{f}_{j-1}$, and then for the subsequent last terms of the expression, we arrive at the following expression:
\[
c_i(1-\sum_{l=i+1}^j p_lc_l)=\sum_{l=1}^{i-1}h_lc_l
\]
for some polynomials $p_l$ and $h_l$. If $(1-\sum_{l=i+1}^j p_lc_l) \in I_{i-1}$, then $1 \in I_{j}$, contradicting the fact that $I_j$ is an ideal. Hence $c_i$ is a divisor of zero in $S/I_{i-1}$, which is impossible because of the regularity of the sequence. Therefore the original assumption must be false, and the lemma is proved.
\end{proof}
}

\appendix

\section{Computations on PD dg algebras}\label{appendix:sec_5}

\delete{
\subsection{Computations for Section \ref{sec:4}}\label{appendix:sec:derivations}

For an element $X^{(f)} \in \mathcal{B}$, we again write $(X^{(f)})^\vee$ for the dual element in $\on{Hom}_R(T_R(\msf k), \msf k)$.
given by
\begin{align*}
\begin{array}[t]{cccl}
(X^{(f)})^\vee: & T_R(\msf k) & \longrightarrow & \msf k \\[5pt]
& X^{(f')} \in \mathcal{B} & \longmapsto &  \delta_{f, f'}.
\end{array}
\end{align*}

We need to introduce some notations. Later on, we will need to reorder some monomials when the variables are not in ascending order. If in an ordered monomial $X^{(f)}$ there is an ordered sub-monomial $X^{(f')}$, and we wish to move it left of a variable $X_i$ in $X^{(f)}$, the sign change due to the different degrees will be written as $(-1)^{(f', f)_{X_i}}$, and the non-negative scalar factor due to the divider power structures will be denoted $s_{(f', f)_{X_i}}$. We illustrate this process in the following example.

\begin{example}
Consider odd variables $X_1$, $X_2$ and even variables $X_3$, $X_4$. Set
\[
X^{(f)}=X_{1}^{(1)}X_{2}^{(1)}X_{3}^{(3)}X_{4}^{(5)},
\]
\[
X^{(f')}=X_{1}^{(1)}X_{3}^{(1)},
\]
and single out the variable $X_4$. We want to rewrite the monomial $X^{(f)}$ but where the sub-monomial $X^{(f')}$ is left of $X_4$. We can verify that
\[
X_{2}^{(1)}X_{3}^{(2)}\underbrace{(X_{1}^{(1)}X_{3}^{(1)})}_\text{$X^{(f')}$}X_{4}^{(5)}=(-1)^{(f', f)_{X_4}}s_{(f', f)_{X_4}}X^{(f)}
\]
with
\[
(-1)^{(f', f)_{X_4}}=-1,
\]
and
\[
s_{(f', f)_{X_4}}=\frac{3!}{2!1!}=3 \in \bb N.
\]
\end{example}

Notice that the convention we use is that the coefficient appears in front of the ordered monomial $X^{(f)}$. So it is the reordering of the variables that creates the coefficient.

For an element $f \in \mathcal{B}$, we will write $\on{supp_{mul}}f$ for the support of $f$ with multiplicity, i.e., $\on{supp_{mul}}f$ is the set of variables $X \in \mathcal{X}$ such that $f(X) \neq 0$, and it contains $f(X)$ copies of $X$.

Set $f \in \mathcal{B}$ and $X_i \in \on{supp}f$. We define $(\on{deg} < X_i)_f=\sum_{j=1}^{i-1}|X_j|f(X_j)$. For example, if $X^{(f)}=X_{1}X_{2}^{(2)}X_{3}X_{5}X_{6}^{(4)}$ then
\[
\begin{array}{l}
(\on{deg} < X_3)_f=|X_1|+2|X_2|, \\[5pt]
(\on{deg} < X_6)_f=|X_1|+2|X_2|+|X_3|+|X_5|.
\end{array}
\]

For a variable $X_i \in \mathcal{X}$ and $f \in \mathcal{B}(|X_i|-1)$, the coefficient of $X^{(f)}$ in $\partial(X_i)$ will be written as $(x_i)_f \in \mathfrak{m}_x$. As $\partial_1$ is surjective on $\mathfrak{m}_x$ by assumption, there exists an element $T_{X_i, f}$ of degree $1$ in $T_R(\msf k)$ such that $\partial_1(T_{X_i, f})=(x_i)_f$. Moreover, if $f \notin \mathcal{B}(|X_i|-1)$, we set $T_{X_i, f}=0$

For any $X_i \in \mathcal{X}$, define $\delta_{X_i} : \mathcal{X} \longrightarrow \bb N$ such that $\delta_{X_i}(X_j)=\delta_{i, j}$. We now introduce operations on the basis $\mathcal{B}$. Consider a variable $X_i \in \mathcal{X}$ and define the maps
\[
\begin{array}{rccc}
e_{X_i} : & \mathcal{B}(n) & \longrightarrow & \mathcal{B}(n+|X_i|) \\[5pt]
 & X^{(f)} & \longmapsto & X^{(f+\delta_{X_i})}
\end{array}
\]
and
\[
\begin{array}{rccc}
f_{X_i} : & \mathcal{B}(n) & \longrightarrow & \mathcal{B}(n-|X_i|) \\[5pt]
 & X^{(f)} & \longmapsto &  \left\{\begin{array}{cl}
X^{(f-\delta_{X_i})} & \text{ if } f(X_i) \geq 1,\\[5pt] 
 0 & \text{ if } f(X_i)=0.
 \end{array}\right.
\end{array}
\]
Given a collection of variables $X_{i_1}, \dots, X_{i_k}$ and $X_{a}$ in $\mathcal{X}$ such that $|X_a|-\sum_{s=1}^k |X_{i_s}|=1$, we define the map
\[
\begin{array}{rccc}
e_{i_1, \dots, i_k; a} : & \mathcal{B}(n) & \longrightarrow & \mathcal{B}(n+1) \\[5pt]
 & X^{(f)} & \longmapsto &f_{X_{i_1}} \cdots f_{X_{i_k}}e_{X_a}X^{(f)}.
\end{array}
\]



\delete{
\begin{lemma}\label{lem:lift}
Set $f \in \mathcal{B}$ such that $|X_i| \geq 4$ for all $X_i \in \on{supp}f$. Set variables $X_{m_1}, \dots, X_{m_l}$ such that, for any $1 \leq s \leq l$ we have $|X_{m_s}| < |X_i|$ for all $X_i \in \on{supp}f$. Finally, consider a collection of operators $e_{i_j, \dots, i_{k_j}; a_j}$ for $1 \leq j \leq r$. Then for $k \geq 1$, the successive lifts $(\widetilde{f})_k$ of $(X^{(f)})^\vee$ are given by
\begin{align}\label{eq:big_sum}
\begin{array}{l}
\displaystyle (\widetilde{f})_k\big(\prod_{i=1}^l e_{X_{m_i}}\prod_{j=1}^re_{i_j, \dots, i_{k_j}; a_j}X^{(f)} \big)   \\[15pt]
=\displaystyle \sum_{f_1, \dots, f_r}(-1)^{C_{f_1, \dots, f_r}}  \left(\prod_{i=1}^l X_{m_i}\right) \left(\prod_{i=1}^r  s_{(f_i, f)_{X_{a_i}}}T_{X_{a_i}, f_i}\right).
\end{array}
\end{align}
We write $\mathcal{S}$ for the set of variables involved in the $e_{i_j, \dots, i_{k_j}; a_j}$, $1 \leq j \leq r$ (if a variable appears in several places or with a divided power, then $\mathcal{S}$ contains as many copies of this variable). The sum in Formula~\eqref{eq:big_sum} is over all the tuples of monomials $(f_1, \dots, f_r) \in \mathcal{B}^r$ such that $\mathcal{S}\backslash \{X_{a_1}, \dots, X_{a_r} \}=\bigsqcup_{s=1}^r\on{supp_{mul}}f_s$. Given such a tuple of monomials $(f_1, \dots, f_r)$, for any $1 \leq s \leq r$, we write $\mathcal{S}^{(f_1,\dots,f_r)}_{\leq s}=\bigsqcup_{i=1}^s\big(\on{supp_{mul}}f_i \sqcup \{X_{a_i}\}\big) \subset \mathcal{S}$.

The constant $C_{f_1, \dots, f_r}$ is given by
\begin{align*}
\begin{array}{l}
C_{f_1, \dots, f_r} = \displaystyle \sum_{s=1}^{r}\left( (\on{deg} < X_{a_s})_{f}+(f_s, f)_{X_{a_s}} +\sum_{\substack{X_{i} \in \mathcal{S} \backslash \mathcal{S}^{(f_1,\dots,f_r)}_{\leq s-1} \\|X_{i}| < |X_{a_s}|}}|X_{i}| \right) \\
\quad \quad \quad \quad +\displaystyle \sum_{s=1}^{r-1} \left( \sum_{X_{i_s} \in \on{supp_{mul}}f_s}|X_{i_s}| \Bigl(\sum_{\substack{X_{i} \in \mathcal{S}  \backslash \mathcal{S}^{(f_1,\dots,f_r)}_{\leq s} \\ |X_{i_s}| < |X_{i} | < |X_{a_s}|}}|X_{i}|\Bigr)\right). 
\end{array}
\end{align*}

\begin{align*}
\resizebox{\textwidth}{!}{$\begin{array}{l}
C_{M_1, \dots, M_r} \\[5pt]
\displaystyle = \widetilde{Z}_1+\sum_{\on{var}^i < Z_1}|\on{var}^i|+(Z_1, X_1 \cdots Y_1)+|X_1| \sum_{\substack{X_1 < \on{var}^i < Z_1 \\ i\neq 1}}|\on{var}^i|+\cdots+|Y_1| \sum_{\substack{Y_1 < \on{var}^i < Z_1 \\ i \neq 1}}|\on{var}^i| \\[5pt]
\quad \displaystyle +\widetilde{Z}_2+\sum_{\substack{\on{var}^i < Z_2 \\ i \neq 1}}|\on{var}^i|+(Z_2, X_2 \cdots Y_2)+|X_2| \sum_{\substack{X_2 < \on{var}^i < Z_2 \\ i \neq 1, 2}}|\on{var}^i|+\cdots+|Y_2| \sum_{\substack{Y_2 < \on{var}^i < Z_2 \\ i \neq 1, 2}}|\on{var}^i| \\[5pt]
\quad \displaystyle +\widetilde{Z}_{r-2}+\sum_{\substack{\on{var}^i < Z_{r-2} \\ i \neq 1, \cdots, r-3}}|\on{var}^i|+(Z_{r-2}, X_{r-2} \cdots Y_{r-2})+|X_{r-2}| \sum_{\substack{X_{r-2} < \on{var}^i < Z_{r-2} \\ i \neq 1, 2, \cdots, r-2}}|\on{var}^i|+\cdots+|Y_{r-2}| \sum_{\substack{Y_{r-2} < \on{var}^i < Z_{r-2} \\ i \neq 1, 2, r-2}}|\on{var}^i| \\[5pt]
\quad \displaystyle +\widetilde{Z}_{r-1}+\sum_{\substack{\on{var}^i < Z_{r-1} \\ i \neq 1, \cdots, r-2}}|\on{var}^i|+(Z_{r-1}, X_{r-1} \cdots Y_{r-1})+|X_{r-1}| \sum_{\substack{X_{r-1} < \on{var}^i < Z_{r-1} \\ i \neq 1, 2, \cdots, r-1}}|\on{var}^i|+\cdots+|Y_{r-1}| \sum_{\substack{Y_{r-1} < \on{var}^i < Z_{r-1} \\ i \neq 1, 2, r-1}} |\on{var}^i|\\[5pt]
\quad \displaystyle +\widetilde{Z}_{r}+\sum_{\substack{\on{var}^i < Z_{r} \\ i \neq 1, \cdots, r-1}}|\on{var}^i|+(Z_r, X_r \cdots Y_r). 
\end{array}$}
\end{align*}
\end{lemma}

\begin{proof}
The proof is done by direct computation. We check the formula directly for $k=1, \dots, 4$, and then proceed by induction on $k$.
\end{proof}
}

\begin{proof}[Proof of Proposition \ref{prop:decomp_dual}]
Consider an element $f \in \mathcal{B}$ such that the smallest $i$ with $f(X_i) \neq 0$ satisfies $|X_i|=2$, and set $X^{(f')}=f_{X_i}X^{(f)}$. In order to lift $(X^{(f')})^\vee$  we set 
\begin{align*}
\begin{array}[t]{cccl}
(\widetilde{f'})_0: & T_R(\msf k)_{\on{deg}f'} & \longrightarrow & R \\[5pt]
& g \in \mathcal{B}(\on{deg}f') & \longmapsto &  \delta_{f', g}.
\end{array} 
\end{align*}
For the image by $\partial_{\on{deg}f'+1}$ of an element in $\mathcal{B}(\on{deg}f'+1)$ to contain $X^{(f')}$, there are three possibilities:
\begin{enumerate}[nosep]
\item it is a variable $X_s$ such that $\partial_{\on{deg}f'+1}(X_s)=\partial(T_{X_s,f'})X^{(f')}+\dots$.
\item it is monomial of the form $e_{X_s}X^{(f')}$ with $|X_s|=1$.
\item it is a monomial of the form $e_{i_1, \dots, i_k; a}X^{(f')}$. In this case, we write $f''$ for the element in $\mathcal{B}$ such that $f''=\sum_{j=1}^k \delta_{X_{i_j}}$.
\end{enumerate}
So we have
\begin{align*}
\begin{array}{rcl}
(\widetilde{f'})_0\partial(X_s)&=& \partial(T_{X_s,f'}), \\[5pt]
(\widetilde{f'})_0\partial(e_{X_s}X^{(f')})&=&\partial_1(X_s), \\[5pt]
(\widetilde{f'})_0\partial(e_{i_1, \dots, i_k; a}X^{(f')})&=&(-1)^{(\on{deg} < X_a)_{f'}+\on{deg}f''+(f'',f')_{X_a}}s_{(f'', f')_{X_a}}\partial_1(T_{X_a, f''})
\end{array}
\end{align*}
Hence we define
\begin{align*}
\begin{array}[t]{cccl}
(\widetilde{f'})_1 :& T_R(\msf k)_{\on{deg}f'+1} & \longrightarrow & T_R(\msf k)_{1} \\[5pt]
&X_s & \longmapsto & T_{X_{s, f'}}, \\[5pt]
&e_{X_s}M& \longmapsto & X_s, \\[5pt]
& e_{i_1, \dots, i_k; a}X^{(f')}& \longmapsto&(-1)^{(\on{deg} < X_a)_{f'}+\on{deg}f''+(f'',f')_{X_a}}s_{(f'', f')_{X_a}} T_{X_a, f''}, \\[5pt]
& \text{other} & \longmapsto & 0,
\end{array} 
\end{align*}
where ``other" indicates the other elements in $\mathcal{B}(\on{deg}f'+1)$. With the same reasoning, we can determine $(\widetilde{f'})_2$. \delete{ As $|X_i|=2$, the definition of the Yoneda product (cf. \cite[Theorem 4.2.2]{CTVZ}) states that
\[
X_i^* \cdot (X^{(f')})^*=X_i^* \circ (\widetilde{f'}^*)_2.
\]
}
We can then verify that
 \begin{align}\label{eq:decomposition}
 (X^{(f)})^\vee=X_i^\vee \cdot (X^{(f')})^\vee+\sum_{s} c_{s}X_{s}^\vee
 \end{align}
 where $ c_{s} \in \msf k$ and $|X_s|=\on{deg}f$ for all $s$. The sum over the variables in $\mathcal{B}(\on{deg}f)$ appears because when one lifts $(X^{(f')})^\vee$, there might be variables of degree $\on{deg}f'+1$ whose images by $\partial_{\on{deg}f'+1}$ contain the element in $\mathcal{B}(\on{deg}f') \backslash \on{Ker}(\widetilde{f'})_{0}$, i.e., $X^{(f')}$. By repeating this procedure for the other variables of degree $2$ and $3$, we see that
\begin{align}\label{eq:decomposition2}
\begin{array}{rcl}
(X^{(f)})^\vee&=&\displaystyle \left(\prod_{|X_i|=2, 3}(X_i^\vee)^{f(X_i)}\right) \cdot \left(X^{\left(\displaystyle f-\sum_{|X_i|=2, 3} f(X_i)\delta_{X_i}\right)}\right)^\vee \\[30pt]
&&\displaystyle +\sum_{f' \in \mathcal{B}(\on{deg}f)}c_{f'}(X^\vee)^{f'},
\end{array}
\end{align}
where $c_{f'} \in \msf k$ for all $f' \in \mathcal{B}(\on{deg}f)$.

Set $f \in \mathcal{B}$ such that the smallest $i$ with $f(X_i) \neq 0$ satisfies $|X_i| \geq 4$ {\color{red} Use the fact that the lift is a derivation. It works for any case $|X_i| \geq 1$ (always extract the variable of lowest degree)}. We write $X^{(f')}=f_{X_i}X{(f)}$. Using Lemma \ref{lem:lift}, we can determine that for the lift $(\widetilde{f'})_{|X_i|}: T_R(\msf k)_{|M|+|X_i|} \longrightarrow T_R(\msf k)_{|X_i|}$ of $(X^{(f')})^\vee$, the only elements whose images contain $X_i$ are $e_{X_i}X^{(f')}$ and (possibly) variables $X_s$. It follows that the Yoneda product $X_i^\vee \cdot (X^{(f')})^\vee= X_i^\vee \circ (\widetilde{f'})_{|X_i|}$ satisfies
\begin{align}\label{eq:extract_variable}
(X^{(f)})^\vee  =   X_i^\vee \cdot (X^{(f')})^\vee+\sum_{s} c_{s}X_s^\vee,
\end{align}
 where $ c_{s} \in \msf k$ and $|X_s|=\on{deg}f$ for all $s$. By successively applying this result for all degrees, we prove the first statement of the proposition.
 
Like previously mentioned, the sum over $\mathcal{B}(\on{deg}f)$ in Formula~\eqref{eq:decomposition2} appears because when one lifts a monomial $(X^{(f')})^\vee$, there might be variables of degree $\on{deg}f'+1$ whose images by $\partial_{\on{deg}f'+1}$ contain the element in $\mathcal{B}(\on{deg}f') \backslash \on{Ker}(\widetilde{f'})_{0}$, i.e., $X^{(f')}$. Then there might be variables of degree $\on{deg}f'+2$ whose images by $\partial_{\on{deg}f'+2}$ contain the elements in $\mathcal{B}(\on{deg}f'+1) \backslash \on{Ker}(\widetilde{f'})_1$, and so on. We assume that we start with $X^{(f)}$ such that there exists $X_j$ verifying $|X_j|=|X_N|$, $f(X_j) \neq 0$, and for any $X_i$ with $f(X_i)\neq 0$, we have $i \leq j$. Then each time we extract a variable as in Equation~\eqref{eq:extract_variable}, we see that $\on{deg}f' \geq |X_N|$ as $f'$ still contains $X_j$. Therefore there is no $X_s \in \mathcal{X}$ such that $|X_s| > \on{deg}f'$, and we obtain Formula~\eqref{eq:split_formula}.

Consider the element $f \in \mathcal{B}$ with $f(X_i)=0$ for all $|X_i|=1$. We lift the morphism $(X^{(f)})^\vee$. We see that 
\begin{align*}
\begin{array}[t]{cccl}
(\widetilde{f})_0: & T_R(\msf k)_{\on{deg}f} & \longrightarrow & R \\[5pt]
& X^{(f')} \in \mathcal{B}(\on{deg}f) & \longmapsto &  \delta_{f, f'},
\end{array} 
\end{align*}
and
\begin{align*}
\begin{array}[t]{cccl}
(\widetilde{f})_1 :& T_R(\msf k)_{\on{deg}f+1} & \longrightarrow & T_R(\msf k)_{1} \\[5pt]
&X_s & \longmapsto & T_{X_s, f}, \\[5pt]
&e_{X_s} X^{(f)}& \longmapsto & X_s, \\[5pt]
& e_{i_1, \dots, i_k; a} X^{(f)}& \longmapsto&(-1)^{(\on{deg} < X_a)_f+\on{deg}f'+(f', f)_{X_a}}s_{(f', f)_{X_a}} T_{X_a, f'}, \\[5pt]
& \text{other} & \longmapsto & 0,
\end{array} 
\end{align*}
where $f'=\sum_{j=1}^k \delta_{X_{i_j}}$. Thus for any $X_i \in \mathcal{X}$ such that $|X_i|=1$, we obtain
\[
\begin{array}{ccl}
(e_{X_i}X^{(f)})^\vee & = &\displaystyle  X_i^\vee \cdot (X^{(f)})^\vee+\text{ linear combination of duals of elements} \\[5pt]
 &&  \text{in } \mathcal{B}(\on{deg}f+1) \text{ not containing variables of degree }1.
 \end{array}
\]
By doing the same construction and increasing the numbers of variables of degree $1$ in $X^{(f)}$, we can show that for an element $f \in \mathcal{B}$ containing $s$ variables of degree $1$ and such that $f(X_i)=0$ for some other $X_i$ with $|X_i|=1$, then
\begin{align}\label{eq:extract_A_(r)}
\begin{array}{rcl}
 (e_{X_i}X^{(f)})^\vee &=& X_i^\vee \cdot (X^{(f)})^\vee +\text{ linear combination of duals of elements in} \\[5pt]
&&\quad  \mathcal{B}(\on{deg}f+1)  \text{ with $s$ or less variables of degree } 1.
\end{array}
\end{align}

We know that the duals of the $X^{(f)}$ without variables of degree $1$ can be expressed as linear combination of the $(X^\vee)^{f'}$ according to Proposition \ref{prop:decomp_dual}. Thus by doing an induction on the number of variables of degree $1$ and using Formula~\eqref{eq:extract_A_(r)}, we conclude the proof of the proposition.
\end{proof}

\subsection{Computations for Section \ref{sec:Yoneda_CI}}\label{appendix:sec:Yoneda_CI}
}

\delete{
An element in $\mathfrak{m}_x$ can be given a vector notation where each coordinate corresponds to the coefficient of the variable with the same index. Hence, for $c_p$, we write
\[
c_p=\begin{pmatrix}
    c_{p,1} \\
    \vdots \\
    c_{p,n}
\end{pmatrix}_{ \!\! \! x}
\]
where the subscript $x$ is used to indicated that we consider the coefficients of the $x_i$. Note that this vector representation is not unique.

We know that $c_{p, i}$ is an element of $\mf{m}_x$ for all $i, p$ because $c_p \in \mf{m}_x^2$. Therefore for all $1 \leq i, j \leq n$ and $1 \leq p \leq k$, there exists $n^p_{i,j} \in \msf k$ and $m^p_{i,j} \in \mf{m}_x$ such that:
\[
c_{p,i}=\begin{pmatrix}
    n^p_{i,1} \\
    \vdots \\
    n^p_{i,n}
\end{pmatrix}_{ \!\! \! x}+\begin{pmatrix}
    m^p_{i,1} \\
    \vdots \\
    m^p_{i,n}
\end{pmatrix}_{ \!\! \! x}
\]
Hence the degree $2$ part of $c_p$ is described by the matrix $N_p=\begin{pmatrix}
    n^p_{1,1} & \cdots & n^p_{1,n} \\
     \vdots & \ddots & \vdots \\
     n^p_{n,1} & \cdots & n^p_{n,n}
\end{pmatrix}$, which translates as
\[
c_p \equiv \sum_{1 \leq i <  j \leq n}(n^p_{i,j}+n^p_{j,i})x_ix_j+\sum_{1 \leq i \leq n}n^p_{i,i}x_i^2 \ \on{mod} \ \mathfrak{m}_x^3
\]

We can thus define $C_{p,i} \in P_1$ such that $\partial_1(C_{p,i})=c_{p,i}$:
\[
C_{p,i}=\begin{pmatrix}
    n^p_{i,1} \\
    \vdots \\
    n^p_{i,n}
\end{pmatrix}_{ \!\! \! T}+\begin{pmatrix}
    m^p_{i,1} \\
    \vdots \\
    m^p_{i,n}
\end{pmatrix}_{ \!\! \! T}
\]
where the subscript $T$ indicates that we consider the coefficients of the $T_i$ in $P_1$. In particular, for any $f_i \in R$ with $1 \leq i \leq n$, we have:
\[
\begin{pmatrix}
    f_1 \\
    \vdots \\
    f_n
\end{pmatrix}_{ \!\! \! T} \stackrel{\partial_1}{\longrightarrow} \begin{pmatrix}
    f_1 \\
    \vdots \\
    f_n
\end{pmatrix}_{ \!\! \! x}.
\]
Note that when the $T_i$ and the $c_{p,i}$ are fixed, the vector notation of $C_{p,i}$ is unique. The purpose of the $C_{p, i}$ is that they verify $\partial_1(\sum_{1 \leq i \leq n}C_{p,i}x_i)=c_p$, so they create an antecedent of $c_p$ in $P_1$.

\begin{remark}
We illustrate our notations with an example. Consider $c_p=x_1x_2$ and $\on{char} \msf k \neq 2$. We can choose $c_p=\begin{pmatrix}
    \frac{1}{2}x_2 \\
    \frac{1}{2}x_1
\end{pmatrix}_{ \!\! \! x}$. Thus $c_{p,1}=\begin{pmatrix}
    0 \\
    \frac{1}{2}
\end{pmatrix}_{ \!\! \! x}$ and $c_{p,2}=\begin{pmatrix}
    \frac{1}{2} \\
    0
\end{pmatrix}_{ \!\! \! x}$. It follows that $C_{p,1}=\begin{pmatrix}
    0 \\
    \frac{1}{2}
\end{pmatrix}_{ \!\! \! T}$ and $C_{p,2}=\begin{pmatrix}
    \frac{1}{2} \\
    0
\end{pmatrix}_{ \!\! \! T}$. Therefore $\sum_{i=1}^nC_{p, i}x_i=\frac{1}{2}x_1T_2+\frac{1}{2}x_2T_1 \stackrel{\partial_1}{\longmapsto} x_1x_2=c_p$. Hence when we integrate monomials of degree $2$, each variable contributes.
\end{remark}
}

Consider the settings and notation of Section \ref{sec:Yoneda_CI}.

\begin{lemma}\label{lem:RelationsinExt}
In the Yoneda algebra $\Ext_A^*(\msf k, \msf k)$, we have:
\begin{itemize}
\item $\alpha_i \alpha_j +\alpha_j \alpha_i= \displaystyle \sum_{p=1}^k(n^p_{i, j}+n^p_{j, i})\beta_p$ for all $1 \leq i \neq j \leq n$, 
\item $\alpha_i^2= \displaystyle \sum_{p=1}^k n^p_{i, i}\beta_p$ for all $1 \leq i \leq n$.
\end{itemize}
\end{lemma}

\begin{proof}
Fix $1 \leq i \leq n$ and consider the following diagram:
 \begin{center}
  \begin{tikzpicture}[scale=0.9,  transform shape]
  \tikzset{>=stealth}
  
\node (1) at (0,0){$\displaystyle \bigoplus_{p=1}^k A S_p \oplus \bigoplus_{1 \leq j<l \leq n} A T_j T_l$};
\node (2) at (6,-0.16){$\displaystyle \bigoplus_{1 \leq j \leq n} A T_j $};
\node (3) at (0,-3.16){$\displaystyle \bigoplus_{1 \leq j \leq n}A T_j$};
\node (4) at (6,-3) {$R$};
\node (5) at (8,-3){$\msf k$};
\node (6) at (10,-3){0};

\node (7) at ( -4,0){$\dots$};
\node (8) at ( -4,-3){$\dots$};

\draw [decoration={markings,mark=at position 1 with
    {\arrow[scale=1.2,>=stealth]{>}}},postaction={decorate}] (1) --  (5.1,0) node[midway, above] {$d$};
\draw [decoration={markings,mark=at position 1 with
    {\arrow[scale=1.2,>=stealth]{>}}},postaction={decorate}] (1)  --  (3) node[midway, left] {$(\widetilde{\alpha_i})_{-2}$};

\draw [decoration={markings,mark=at position 1 with
    {\arrow[scale=1.2,>=stealth]{>}}},postaction={decorate}] (2)  --  (4) node[midway, left] {$(\widetilde{\alpha_i})_{-1}$};
\draw [decoration={markings,mark=at position 1 with
    {\arrow[scale=1.2,>=stealth]{>}}},postaction={decorate}] (1,-3)  --  (4) node[midway, above] {$d$};

\draw [decoration={markings,mark=at position 1 with
    {\arrow[scale=1.2,>=stealth]{>}}},postaction={decorate}] (2)  --  (5) node[midway, above right] {$\alpha_i$};
\draw [decoration={markings,mark=at position 1 with
    {\arrow[scale=1.2,>=stealth]{>}}},postaction={decorate}] (4)  --  (5) node[midway, above] {$\epsilon$};
    
\draw [decoration={markings,mark=at position 1 with
    {\arrow[scale=1.2,>=stealth]{>}}},postaction={decorate}] (5)  --  (6);
    
\draw [decoration={markings,mark=at position 1 with
    {\arrow[scale=1.2,>=stealth]{>}}},postaction={decorate}] (7)  --  (1) node[midway, above] {$d$};
\draw [decoration={markings,mark=at position 1 with
    {\arrow[scale=1.2,>=stealth]{>}}},postaction={decorate}] (8)  --  (-1,-3) node[midway, above] {$d$};

\end{tikzpicture}
 \end{center}
where $\epsilon$ is the augmentation map corresponding to the ideal $\mf{m}_x$, and $(\widetilde{\alpha_i})_{-1}$, $(\widetilde{\alpha_i})_{-2}$ make the diagram commute (they exist because the $P^i$'s are projective $A$-modules).

We know that for all $1 \leq l \leq n$, $1 \leq p \leq k$, we have $c_{p,l}=\sum_{j=1}^n n^p_{l,j}x_j \in \mathfrak{m}$. So we can define $C_{p,l}=\sum_{j=1}^n n^p_{l,j}T_j \in P^{-1}$ satisfying $d(C_{p,l})=c_{p,l}$. The first two maps in the above commutative diagram are then given by:
\begin{align*}
\begin{array}[t]{cccc}
(\widetilde{\alpha_i})_{-1}: & P^{-1} & \longrightarrow & P^0 \\
           & T_j & \longmapsto & \delta_{i, j}
            \end{array} \quad \text{ and } \quad \begin{array}[t]{cccc}
(\widetilde{\alpha_i})_{-2}: & P^{-2} & \longrightarrow & P^{-1} \\
           & T_j T_l& \longmapsto & \left\{\begin{array}{cl}
        -T_l & \text{ if } j=i,\\
        T_j & \text{ if } l=i,\\
         0 & \text{ otherwise},
        \end{array}\right. \\
           & S_p & \longmapsto & C_{p, i}.
\end{array}
\end{align*}
The Yoneda product $\alpha_j  \alpha_i$ is then:
\[
\begin{array}[t]{cccl}
\alpha_j \alpha_i: & P^{-2} & \longrightarrow & \msf k \\
           & T_j T_i & \longmapsto & 1, \\
           & S_p & \longmapsto & n^p_{i, j}  \ \forall p, \\
           & \text{other} & \longmapsto & 0.
\end{array}
\]
The formulas of the lemma follow.
\end{proof}

\delete{
Hence $\partial_2(S_p)= \sum_{i<j}  (n^p_{i, j}x_j T_i+n^p_{j, i}x_i T_j)+\sum_{i=1}^n  n^p_{i, i}x_i T_i+\on{deg} \geq 2$. However, in the construction of Tate, the element $\partial_2(S_p)$ is chosen so that it is a generator of the first cohomology group before adding variables of degree $2$. So this element can be chosen modulo $\partial_2(T_iT_j)=x_i T_j-x_j T_i$. Hence, instead of the element previously mentioned as $\partial_2(S_p)$, we chose $\sum_{i=1}^n c_{p, i}T_i+ \sum_{i<j} n^p_{j, i}(x_j T_i-x_i T_j)$ and we will obtain the same resolution. With this choice, we obtain:
\[
 \begin{array}{ccl}
 \partial_2(S_p) & = & \displaystyle \sum_{i<j}  (n^p_{i, j}x_j T_i+n^p_{j, i}x_i T_j)+ \sum_{i<j} n^p_{j, i}(x_j T_i-x_i T_j)+\sum_{i=1}^n  n^p_{i, i}x_i T_i+\on{deg} \geq 2, \\[5pt]
  & = & \displaystyle \sum_{i<j}(n^p_{i, j}+n^p_{j, i})x_j T_i +\sum_{i=1}^n  n^p_{i, i}x_i T_i+\on{deg} \geq 2, \\[5pt]
  & = & \displaystyle \sum_{i \leq j}\widetilde{n}_j^{p, i} x_j T_i +\on{deg} \geq 2.
\end{array}
\]
In this new expression, we have $\widetilde{n}_j^{p, i}=0$ for $j<i$, and the resolution does not change. Hence we will use the convention that $n^p_{i, j}=0$ for $j <i$.

\begin{remark}\label{Convention}
Another way to look at this convention is that in $c_p$, if we have a term $a x_i x_j$ with $a \in \cc^*$, we do not decompose it as $b x_i x_j+c x_j x_i$ with $b+c=a$, $b x_i \in c_{p, j}$ and $c x_j \in c_{p, i}$. We either consider $ax_i \in c_{p, j}$ or $ax_j \in c_{p, i}$. By convention, we chose $ax_j \in c_{p, i}$ for all $i <j$.
\end{remark}
}

\delete{
We also determine the commutation relations of the $\beta_p$'s through direct computations.

\begin{lemma}\label{AlphaBetaCommute}
In the Yoneda algebra $\Ext_R^*(\msf k, \msf k)$, for all $1 \leq i \leq n$, $1 \leq p, q \leq k$, we have:
\[
\alpha_i \beta_p=\beta_p \alpha_i \text{ and } \beta_p \beta_q=\beta_q \beta_p.
\]
\end{lemma}

\begin{proof}
We keep lifting the morphism $\alpha_i$ using the previous diagram. The morphisms $(\widetilde{\alpha_i})_{-1}$ and $(\widetilde{\alpha_i})_{-2}$ have already been obtained. The morphism $g_3$ is then given by:
\[
\begin{array}[t]{cccc}
(\widetilde{\alpha_i})_{-3}: & P^{-3} & \longrightarrow & P^{-2} \\
           & T_{i_1} T_{i_2} T_{i_3}& \longmapsto & \left\{\begin{array}{cl}
        T_{i_2} T_{i_3} & \text{ if } i_1=i,\\
        -T_{i_1} T_{i_3} & \text{ if } i_2=i,\\
         T_{i_1} T_{i_2} & \text{ if } i_3=i,\\
         0 & \text{otherwise},
        \end{array}\right. \\
           & T_l S_p & \longmapsto & \left\{\begin{array}{cl}
        T_l C_{p, i} & \text{ if } l \neq i,\\
        T_i C_{p, i}+S_p & \text{ if } l=i.
              \end{array}\right.
\end{array}
\]
We can thus determine that:
\[
\begin{array}[t]{cccc}
\beta_p \alpha_i: & P^{-3} & \longrightarrow & \msf k \\
           & T_i S_p& \longmapsto & 1,\\
           & \text{other} & \longmapsto & 0.
\end{array}           
\]
By doing the same procedure but lifting $\beta_p$ instead, we can see that $\alpha_i \beta_p$ sends $T_i S_p$ to $1$ and everything else to $0$, so $\beta_p \alpha_i=\alpha_i \beta_p$. Furthermore for all $ 1 \leq p \neq q \leq k$, we have:
\[
\begin{array}[t]{cccc}
\beta_p \beta_q: & P^{-4} & \longrightarrow & \msf k \\
           & S_p S_q & \longmapsto & 1,\\
           & \text{other} & \longmapsto & 0
\end{array}  \quad \text{ and }    \begin{array}[t]{cccc}
\beta_p^2: & P^{-4} & \longrightarrow & \msf k \\
           & S_p^2 & \longmapsto & 2,\\
           & \text{other} & \longmapsto & 0.
\end{array}     
\]
As we know that the variables $S_p$ commute with one another, it follows that $\beta_p \beta_q=\beta_q \beta_p$ for all $p, q$.
\end{proof}
}

\delete{
\begin{proof}[Proof of Proposition \ref{BasisofExtdegree2}] 
Set $m \in \mathbb{Z}_{>0}$ and consider $T_{i_1}\cdots T_{i_s}S_1^{(b_1)}\cdots S_k^{(b_k)} \in P_{m}$  with $i_1 < \cdots < i_s$ and set $X=T_{i_2}\cdots T_{i_s}S_1^{(b_1)}\cdots S_k^{(b_k)} \in P_{m-1}$. To fix notations, we write $f_X=X^*$ for the dual morphism. We are going to lift this morphism.
 \begin{center}
  \begin{tikzpicture}[scale=0.9,  transform shape]
  \tikzset{>=stealth}
  
\node (1) at ( 0,0){$P_{m}$};
\node (2) at ( 3,0){$P_{m-1}$};
\node (3) at ( 0,-2){$P_1$};
\node (4) at ( 3,-2) {$R$};
\node (5) at ( 5,-2){$\msf k$};
\node (6) at ( 7,-2){0};

\node (7) at ( -2,0){$\dots$};
\node (8) at ( -2,-2){$\dots$};

\draw [decoration={markings,mark=at position 1 with
    {\arrow[scale=1.2,>=stealth]{>}}},postaction={decorate}] (1) --  (2) node[midway, above] {$\partial_{m}$};
\draw [decoration={markings,mark=at position 1 with
    {\arrow[scale=1.2,>=stealth]{>}}},postaction={decorate}] (1)  --  (3) node[midway, left] {$(\widetilde{f_X})_1$};

\draw [decoration={markings,mark=at position 1 with
    {\arrow[scale=1.2,>=stealth]{>}}},postaction={decorate}] (2)  --  (4) node[midway, left] {$(\widetilde{f_X})_0$};
\draw [decoration={markings,mark=at position 1 with
    {\arrow[scale=1.2,>=stealth]{>}}},postaction={decorate}] (3)  --  (4) node[midway, above] {$\partial_1$};

\draw [decoration={markings,mark=at position 1 with
    {\arrow[scale=1.2,>=stealth]{>}}},postaction={decorate}] (2)  --  (5) node[midway, above right] {$f_X$};
\draw [decoration={markings,mark=at position 1 with
    {\arrow[scale=1.2,>=stealth]{>}}},postaction={decorate}] (4)  --  (5) node[midway, above] {$\epsilon$};
    
\draw [decoration={markings,mark=at position 1 with
    {\arrow[scale=1.2,>=stealth]{>}}},postaction={decorate}] (5)  --  (6);
    
\draw [decoration={markings,mark=at position 1 with
    {\arrow[scale=1.2,>=stealth]{>}}},postaction={decorate}] (7)  --  (1) node[midway, above] {$\partial_{m+1}$};
\draw [decoration={markings,mark=at position 1 with
    {\arrow[scale=1.2,>=stealth]{>}}},postaction={decorate}] (8)  --  (3) node[midway, above] {$\partial_2$};
\end{tikzpicture}
 \end{center}
Define $(\widetilde{f_X})_0:  P_{m-1}  \longrightarrow  R$ such that it sends $X$ to $1$ and every other basis element to zero.

We see that for any $i <i_2$, we have
\[
(\widetilde{f_X})_0 \partial_{m}(T_i T_{i_2} \cdots T_{i_s}S_1^{(b_1)} \cdots S_k^{(b_k)})= \partial_1(T_i).
\]
If $i_2 < i <i_3$, then
\[
(\widetilde{f_X})_0 \partial_{m}(T_{i_2} T_i \cdots T_{i_s}S_1^{b_1}\cdots  S_k^{b_k})= -\partial_1(T_i).
\]
We proceed similarly  with the cases $i_3 < i < i_4$, $i_4 < i < i_5$, and so on. Finally, in the case $i_s <i$, we get 
\[
(\widetilde{f_X})_0 \partial_{m}(T_{i_2} \cdots T_{i_s}T_i S_1^{(b_1)}\cdots S_k^{(b_k)})= (-1)^s\partial_1(T_i).
\]
Moreover, we see for any $i_r \in \{i_2, \dots, i_s\}$ and any $1 \leq p \leq k$, we have
\[
(\widetilde{f_X})_0 \partial_{m}(T_{i_2} \cdots\widehat{T_{i_r}}\cdots T_{i_s} S_1^{(b_1)}\cdots S_p^{(b_p+1)} \cdots S_k^{(b_k)})= (-1)^r\partial_1(C_{p, i_r}).
\]
We can thus define $(\widetilde{f_X})_1$ as:
\begin{align*}
\begin{array}[t]{cccl}
(\widetilde{f_X})_1: & P_{m} & \longrightarrow & P_1 \\
           & T_iT_{i_2} \cdots  T_{i_s}S_1^{(b_1)} \cdots S_k^{(b_k)}  & \longmapsto & T_i  \ \forall \ i< i_1, \\
           & \vdots& \vdots &\vdots \\
           & T_{i_2} \cdots  T_{i_s}T_iS_1^{(b_1)} \cdots S_k^{(b_k)}  & \longmapsto & (-1)^sT_i  \ \forall \ i_s < i, \\
           & T_{i_2} \cdots \widehat{T_{i_r}} \cdots T_{i_s}S_1^{(b_1)}\cdots S_p^{(b_p+1)} \cdots S_k^{(b_k)}  & \longmapsto &\displaystyle  (-1)^rC_{p,  i_r}  \ \forall \ 1 \leq p \leq k, i_r \in\{i_2, \dots, i_s \}, \\
           & \text{other} & \longmapsto & 0.
\end{array} 
\end{align*}
As $i_1 <i_2$ by definition, we have:
\begin{align*}
\begin{array}[t]{cccl}
\alpha_{i_1}f_X: & P_{m} & \longrightarrow & \msf k \\
           & T_{i_1}T_{i_2} \cdots  T_{i_s}S_1^{(b_1)} \cdots S_k^{(b_k)}  & \longmapsto & 1, \\
           & T_{i_2} \cdots \widehat{T_{i_r}} \cdots T_{i_s}S_1^{(b_1)}\cdots S_p^{(b_p+1)} \cdots S_k^{(b_k)}  & \longmapsto &\displaystyle  (-1)^r n^p_{i_r, i_1}  \ \forall \ 1 \leq p \leq k, i_r \in\{i_2, \dots, i_s \}, \\
           & \text{other} & \longmapsto & 0.
\end{array} 
\end{align*}
We can then write
\begin{align*}
\begin{array}[t]{rl}
(T_{i_1}T_{i_2} \cdots   T_{i_s}S_1^{(b_1)} & \mkern-18mu \  \cdots S_k^{(b_k)})^* = \alpha_{i_1}(T_{i_2} \cdots  T_{i_s}S_1^{(b_1)} \cdots S_k^{(b_k)})^* \\
& \displaystyle -\sum_{p=1}^n\sum_{r=2}^s(-1)^r n^p_{i_r, i_1}( T_{i_2} \cdots \widehat{T_{i_r}} \cdots T_{i_s}S_1^{(b_1)}\cdots S_p^{(b_p+1)} \cdots S_k^{(b_k)})^*.
\end{array} 
\end{align*}
On the right hand side, all the dual terms contain at most $s-1$ variables $T_l$ of degree $1$. We can thus proceed by induction on the number of variables of degree $1$ to prove the first statement of the proposition. It remains only to start the induction process.

Using the same type of construction we did above, i.e., lifting a morphism to compute a Yoneda product, we prove that for any $b_1, \dots, b_k \in \bb N$ (at least one is non zero) and any $1 \leq i \leq n$, we have
\[
(S_1^{(b_1)}\cdots  S_k^{(b_k)})^*=\beta_1^{b_1} \cdots \beta_k^{b_k} 
\]
and
\[
(T_i S_1^{(b_1)}\cdots  S_k^{(b_k}))^*=\alpha_i\beta_1^{b_1} \cdots \beta_k^{b_k}. 
\]
This concludes the proof of the first statement of the proposition.

For any $m \in \mathbb{Z}_{>0}$, we know that $\{(T_1^{a_1}\cdots T_n^{a_n}S_1^{(b_1)}\cdots S_k^{(b_k)})^* \ | \  a_1+\dots+a_n+2b_1+\dots+2b_k=m\}$ is a basis of the vector space $\Ext_R^m(\msf k, \msf k)$ because the resolution is $I$-minimal. Based on the result above, it follows that $\{(\alpha_1^{a_1}\cdots \alpha_n^{a_n}\beta_1^{b_1} \cdots \beta_k^{b_k} \ | \  a_1+\dots+a_n+2b_1+\dots+2b_k=m\}$ is a generating family of the vector space $\Ext_R^m(\msf k, \msf k)$. But as both sets have the same cardinality, it follows that the latter set is also a basis of $\Ext_R^m(\msf k, \msf k)$.

\end{proof}

}

\delete{
\begin{proof}[Proof of Proposition \ref{prop:NumberofGenerators}]
From what we found in Corollary \ref{cor:PresentationofExt}, we have $n+k$ generators. However, they are not independent. We know that in $\Ext_R^*(\msf k, \msf k)$, we have $\alpha_i \alpha_j +\alpha_j \alpha_i=\sum_{p=1}^k(n^p_{j,i}+ n^p_{i, j})\beta_p$ for all $i <j$, and $\alpha_i^2=\sum_{p=1}^k n^p_{i, i}\beta_p$ for all $i$. So we can write:
\[
\begin{pmatrix}
           \alpha_1^2 \\
           \alpha_1 \alpha_2+\alpha_2 \alpha_1 \\
           \alpha_1 \alpha_3+\alpha_3 \alpha_1 \\
           \vdots \\
           \alpha_2^2 \\
           \alpha_2 \alpha_3+\alpha_3 \alpha_2 \\
           \vdots \\
           \alpha_n^2
         \end{pmatrix}=M \begin{pmatrix}
           \beta_1 \\
           \vdots \\
           \beta_k
         \end{pmatrix}.
\]
We put the matrix $M$ in reduced row echelon form $M'$, and applying the same procedure to the vector on the left hand side of the equality. In each line $i$ of $M'$, the first $1$ that appears is in the position $(i, j_i)$ for some $j_i$. We know that $\on{rank}M=\on{rank}M'$, so the last non-zero line of $M'$ has index $\on{rank}M$. It follows that $\beta_{j_{\on{rank}M}}$ can be written as a linear combination of the $\alpha_i \alpha_j$'s and the $\beta_p$'s with $p>j_{\on{rank}M}$. By going up the rows of $M'$, we see that for all $1 \leq s \leq \on{rank} M$, $\beta_{j_s}$ being written as a linear combination of the $\alpha_i \alpha_j$'s and the $\beta_p$'s with $p>j_s$, $p \neq j_{s+1}, \dots, j_{\on{rank}M}$. We have therefore obtained the following equality:
\[
\on{Span}\{\beta_1, \dots, \beta_k\} \subseteq \on{Span}\{\alpha_i \alpha_j+\alpha_j \alpha_j, \alpha_i^2, \beta_p \text{ with } i < j, \ p \neq  j_{1}, \dots, j_{\on{rank}M} \}.
\]
It follows that $\Ext_R^*(\msf k, \msf k)$ is generated by $\alpha_i$'s and $\beta_p$'s where $p \neq  j_{1}, \dots, j_{\on{rank}M}$. There are thus $n+k-\on{rank}M$ generators.
\end{proof}

\subsection{Computations for Section \ref{sec:reconstruct_homLie}}\label{appendix:sec:reconstruct_homLie}
}

We give below the proof of a lemma in Section \ref{sec:reconstruct_homLie}.

\begin{proof}[Proof of Lemma \ref{lem:reg_seq}] We prove that for all $s$, the image $\overline{g_s}$ is not a zero divisor in $\on{Pol}/I_{s-1}$, with $I_{s-1}=(g_1, \dots, g_{s-1})$. 

By the division algorithm \cite[Thm 2.2.1]{Herzog-Hibi}, any $ f\in \on{Pol}$ can be written as $f\equiv f' \on{ mod } I_{s-1}$ such that none of $\on{in}_<(g_i)$ divides any $u\in \on{supp}(f')$ if $f'\neq 0$. 

Assume that $\overline{g_s}$ is a zero divisor in $\on{Pol}/I_{s-1}$, i.e., there exists $\overline{f} \neq 0$ such that $\overline{f} \overline{g_s}=0$. As $\overline{f} \neq 0$, we have $f' \neq 0$ and $f'g_s \in I_{s-1}$. Once again, the division algorithm applied to $f'g_s$ with respect to $g_1, \dots, g_{s-1}$ gives 
\[ f'g_s=h_1g_1+\cdots+h_{s-1}g_{s-1}.\]
If $h_i\neq 0$, then $ \on{in}_<(f'g_s)\geq \on{in}_<(h_i g_i)$ by the division algorithm. In particular, 
\[ \on{in}_<(f'g_s)\geq \max\{  \on{in}_<(h_i g_i) \;|\; h_i\neq 0,\;  i<s\}.\]
On the other hand, by \cite[Lemma 2.1.4(iv)]{Herzog-Hibi} we have 
 \[ \on{in}_<(f'g_s)\leq \max\{  \on{in}_<(h_i g_i) \;|\; h_i\neq 0,\;  i<s\}.\]
 It follows that $\on{in}_<(f'g_s) = \max\{  \on{in}_<(h_i g_i) \;|\; h_i\neq 0,\;  i<s\}$, and so there exists an $i<s$ such that $ \on{in}_<(f'g_s)=\on{in}_<(h_i g_i)$. Using \cite[Lemma 2.1.4(iii)]{Herzog-Hibi}, we have 
 \[ \on{in}_<(f')\on{in}_<(g_s)=\on{in}_<(h_i)\on{in}_<(g_i).\]
 Since $ (\on{in}_<(g_s),\on{in}_<(g_i))$ is a relatively prime pair in $\on{Pol}$ by assumption, it follows that $\on{in}_<(g_i) $ must divide $\on{in}_<(f')$. This contradict the assumption that, since $f'\neq 0$, $ \on{in}_<(g_i)$ does not divide any $ u\in \on{supp}(f')$. 
\end{proof}

\delete{
\subsection{Computations for Section \ref{sec:finite_gen}}\label{appendix:sec:finite_gen}

\subsection{Computations for Section \ref{sec:lift_min_res}}\label{appendix:sec:lift_min_res}

\begin{lemma}\label{lem:H_1}
Let $X_1$ be the complex of $R$-modules obtained after adjoining the $n$ variables of degree $1$. Then the $\sum_{i=1}^n \overline{c_{p, i}} T_i$  ($1 \leq p \leq k$) generate $H_1(X_1)=\operatorname{Ker}\partial_1/\on{Im}\partial_2$ as an $R$-module.
\end{lemma}

\begin{proof}
Let $\overline{r}=\sum_{i=1}^n \overline{r_i} T_i$ with $r_i \in \msf k[x_1,\dots,x_n]$ and $\overline{\color{white}x}$ is the class modulo $(c_1,\dots, c_k)$. If $\partial_1(\overline{r})=0$ in $R$, then $\sum_{i=1}^n r_i x_i=p(c_1,\dots,c_k)$ in $\msf k[x_1,\dots,x_n]$ with $p$ a polynomial without constant term. We write $p(c_1, \dots, c_k)=\sum_{p} a_p c_p+\sum_{p, l}a_{p, l}c_p c_l+\sum_{p, l, m}a_{p, l, m}c_p c_l c_m+...$ where the $a_p$, $a_{p, l}, \dots$ are in $\msf k[x_1,\dots,x_n]$. By regrouping the terms we obtain:
\[
\sum_{i=1}^n (r_i -\sum_{p} a_p c_{p,i}-\sum_{p, l}a_{p, l}c_p c_{l, i}-\sum_{p, l, m}a_{p, l, m}c_p c_l c_{m, i}+\dots)x_i=0
\]
in $\msf k[x_1, \dots, x_n]$. As $r_i$ is a representative of $\overline{r_i}$ modulo the relations, we see that we can choose $r_i$ so that $p$ has only monomials of degree $1$. So after this choice, we get:
\[
\sum_{i=1}^n (r_i -\sum_{p=1}^k a_p c_{p,i})x_i=0
\]
in $\msf k[x_1, \dots, x_n]$. Furthermore, we know that the projective resolution of $\msf k$ as a $\msf k[x_1, \dots, x_n]$-module is given by the Koszul complex with generators $\mathcal T_1, \dots, \mathcal T_n$, so in order to have $\partial_1(\sum_{j=1}^n p_ j \mathcal T_j)=p_1x_1+\dots+p_n x_n=0$ for polynomials $p_i$ in $\msf k[x_1, \dots, x_n]$, we need $\sum_{j=1}^n p_j \mathcal T_j \in \operatorname{Ker}\partial_1=\on{Im}\partial_2$, which is generated by the $x_i \mathcal T_j-x_j \mathcal T_i$. It follows that in the complex $X_1$, we have:
\[
\sum_{i=1}^n (r_i -\sum_{p=1}^k a_p c_{p,i})\mathcal T_i \in \sum_{1 \leq i \neq j \leq n}\msf k[x_1,\dots, x_n](x_i \mathcal T_j-x_j \mathcal T_i).
\]
But we know that $X_1$ is the image of the Koszul complex by sending $\mathcal T_i$ on $T_i$ and $\msf k[x_1,\dots, x_n]$ to $R$. It follows that
\[
\sum_{i=1}^n (\overline{r_i} -\sum_{p=1}^k \overline{a_p c_{p,i}}) T_i \in \sum_{1 \leq i \neq j \leq n}\msf R(\overline{x_i} T_j-\overline{x_j} T_i).
\]
We have found that $\overline{r} \in \operatorname{Ker}\partial_1$ if and only if there exist polynomials $a_p$ such that $\overline{r}=\sum_{p=1}^k \overline{a_p} (\sum_{i=1}^n \overline{c_{p, i}}T_i)$ in $(X_1)_1/\on{Im}\partial_2$. The assertion of the lemma follows.
\end{proof}

\begin{proof}[Proof of Proposition \ref{prop:LiftofPsi}]
We write $\mathcal{P}(m)$ for the statement of the proposition for a fixed $m \in \mathbb{Z}_{>0}$. We have just seen that $\mathcal{P}(1)$ and $\mathcal{P}(2)$ are true. As with the previous proof, we will proceed by induction.

Assume $\mathcal{P}(m)$ is true for some $m \in \mathbb{Z}_{>0}$. We fix $t_1^{a_1}\dots t_{n+k-r}^{a_{n+k-r}}S_1^{(b_1)}\dots S_k^{(b_k)} \in Q_{m+1}$. By direct computation we see that
\begin{equation*}
\resizebox{\hsize}{!}{$
\begin{array}{l}
\psi_m \partial_{m+1}(t_1^{a_1}\dots t_{n+k-r}^{a_{n+k-r}}S_1^{(b_1)}\dots S_k^{(b_k)}) =\displaystyle  \sum_{i=1}^n (-1)^{a_1+a_2+\dots +a_{i-1}} a_ix_iT_1^{a_1} \dots \widehat{T_i^{a_i}} \dots T_n^{a_n}0^{a_{n+1}} \dots 0^{a_{n+k-r}}S_1^{(b_1)}\dots S_k^{(b_k)}  \\
+ (-1)^{a_1+\dots a_{n+k-r}} \displaystyle \sum_{p=1}^k \sum_{i=1}^n(-1)^{a_{i+1}+a_{i+2}+\dots+a_{n+k-r}}b_p c_{p, i}T_1^{a_1}\dots T_i^{a_i+1} \dots T_n^{a_n}0^{a_{n+1}} \dots 0^{a_{n+k-r}}S_1^{b_1}\dots S_p^{b_p-1} \dots S_k^{b_k},  \\
 = \partial_{m+1}(T_1^{a_1} \dots T_n^{a_n}0^{a_{n+1}} \dots 0^{a_{n+k-r}}S_1^{(b_1)}\dots S_k^{(b_k)}).
 \end{array}$}
\end{equation*}
We can check that the second equality is true for any values of the $a_i$, $i \geq n+1$. Indeed, if one of the $a_i$ is equal to $1$, then $\partial_{m+1} (T_1^{a_1} \dots T_n^{a_n}0^{a_{n+1}} \dots 0^{a_{n+k-r}}S_1^{(b_1)}\dots S_k^{(b_k)})=\partial_{m+1}(0)=0$. And if $a_{n+1}= \dots = a_{n+k-r}=0$, then a direct computation verifies the equality.

Thus we found that
\[
\psi_m \partial_{m+1}(t_1^{a_1}\dots t_{n+k-r}^{a_{n+k-r}}S_1^{(b_1)}\dots S_k^{(b_k)})=\partial_{m+1} (T_1^{a_1} \dots T_n^{a_n}0^{a_{n+1}} \dots 0^{a_{n+k-r}}S_1^{(b_1)}\dots S_k^{(b_k)}),
\]
and because all elements of the basis of $Q_{m+1}$ are of this form, we can define:
\begin{align*}
\begin{array}[t]{cccc}
\psi_{m+1}: & Q_{m+1} & \longrightarrow & P_{m+1} \\
           & t_1^{a_1}\dots t_{n+k-r}^{a_{n+k-r}}S_1^{(b_1)}\dots S_k^{(b_k)} & \longmapsto & T_1^{a_1} \dots T_n^{a_n}0^{a_{n+1}} \dots 0^{a_{n+k-r}}S_1^{(b_1)}\dots S_k^{(b_k)}.
\end{array}
\end{align*}
This makes the diagram commute and so $\mathcal{P}(m+1)$ is true. As $\mathcal{P}(1)$ and $\mathcal{P}(2)$ are true, we have proved that $\mathcal{P}(m)$ is true for all $m \in \mathbb{Z}_{>0}$.
\end{proof}

\subsection{Computations for Section \ref{sec:fg_quotient}}\label{appendix:sec:fg_quotient}

\begin{proof}[Proof of Lemma \ref{lem:ext2_non_CI}]
We need to show that for all $1 \leq p \leq k$, we have $0=\partial_3^*(\delta_p)=\delta_p \circ \partial_3$. But when the complex is made acyclic, we know that $\on{Im}(\partial_3)=\text{Ker}(\partial_2)$. So $\delta_p \in \Ext_R^2(\msf k, \msf k)$ if and only if $\delta_p$ is zero on $\text{Ker}(\partial_2)$. As $\delta_p$ is the dual of $S_p$, this is equivalent to saying that no element of $\text{Ker}(\partial_2)$ contains a scalar multiple of $S_p$.

Let us assume that for $1 \leq l \leq k$, $1 \leq i, m \leq n$ there exist $r_l, r_{i, m} \in \msf k[x_1, \dots, x_n]$ such that $r_p S_p+\sum_{l \neq p}r_l S_l+\sum_{i <m}r_{i, m}T_iT_m$ is in $\text{Ker}(\partial_2)$, with $r_p \equiv 1 \text{ mod } \mathfrak{m}_x$. It follows that:
\begin{align*}
\renewcommand{\arraystretch}{2}
\begin{array}{cl}
0 &\displaystyle  =\sum_{i=1}^n r_p c_{p, i}T_i+\sum_{l \neq p}r_l (\sum_{i=1}^nc_{l, i}T_i)+\sum_{i <m}r_{i, m}(x_iT_m-x_mT_i), \\
 & \displaystyle = \sum_{i=1}^n(r_p c_{p, i}+\sum_{l \neq p}r_l c_{l, i}+\sum_{m<i}x_m r_{m, i}-\sum_{i<m}x_m r_{i, m})T_i.
 \end{array}
\end{align*}
As the $T_i$'s are free over $R$, we get that for all $1 \leq i \leq n$:
\begin{align*}
r_p c_{p, i}+\sum_{l \neq p}r_l c_{l, i}+\sum_{m<i}x_m r_{m, i}-\sum_{i<m}x_m r_{i, m}=0 \text{ in } R.
\end{align*}
Hence for all $1 \leq i \leq n$, there exists a polynomial $p_i \in \msf k[x_1, \dots, x_n][y_1, \dots, y_k]$ without constant term such that:
\begin{align*}
r_p c_{p, i}+\sum_{l \neq p}r_l c_{l, i}+\sum_{m<i}x_m r_{m, i}-\sum_{i<m}x_m r_{i, m}=p_i(c_1, \dots, c_k) \text{ in } \msf k[x_1, \dots, x_n].
\end{align*}
By multiplying each expression by $x_i$, and then summing over $i$, the last two terms cancel each other and we get 
\begin{align*}
r_p c_p+\sum_{l \neq p}r_l c_l=\sum_{i=1}^n x_i p_i(c_1, \dots, c_k).
\end{align*}
The coefficient of $c_p$ in the right hand side is clearly in $\mathfrak{m}_x$. We pass it to the left side and regroup all the $c_l$ with $l \neq p$ on the right hand side to obtain:
\begin{align*}
(1+\widetilde{r_p})c_p=p(c_1, \dots, \widehat{c_p}, \dots, c_k).
\end{align*}
with $\widetilde{r_p} \in \mathfrak{m}_x$ and $p \in \msf k[x_1, \dots, x_n][y_1, \dots, y_{k-1}]$. But as the $c_l$'s are homogeneous, this implies that $c_p=\widetilde{p}(c_1, \dots, \widehat{c_p}, \dots, c_k)$ for some polynomial $\widetilde{p}$ without a constant term, contradicting the minimality of the $c_l$'s as generators of $(c_1, \dots, c_k)$. Therefore no element of $\text{Ker}(\partial_2)$ is of the assumed form.
\end{proof}
}

\section{Some homological exercises}
Let $ \mc C$ be an abelian category. We denote by $\mc Ch(\mc C)$  the category of cochain complexes in $\mc C$ with morphisms being chain maps, and  $ \mc H(\mc C)$ the homotopy category. The proof of the following lemma is a straightforward exercise by doing a downward induction on the cohomological degree.

\begin{lemma}\label{lem:homotopy}
Let \(\mathcal C\) be an abelian category.  Let
\(P^\ast\) and \(X^\ast\) be cochain complexes in \(\mathcal C\) with the following properties:
\begin{enumerate}
\item There is an $l\in \bb Z$ such that \(P^n=0\) for all \(n>l\), and each \(P^n\) is projective in \(\mathcal C\).
\item \(X^\ast\) is exact in all negative degrees: \(H^n(X^\ast)=0\) for every \(n<0\).
\end{enumerate}
Then for every integer \(r<-l\) every cocycle
\[
f\in Z^{r}\bigl(\mc Hom^*_{\mathcal C}(P^\ast,X^\ast)\bigr)=\Hom_{\mc Ch(\mc C)}(P^*,X^*[r])
\]
is null-homotopic, equivalently $ \Hom_{\mc H(\mc C)}(P^*, X^*[r])=0$ for all $ r<-l$. Consequently
\[
H^{r}\bigl(\mc Hom^*_{\mathcal C}(P^\ast,X^\ast)\bigr)=0\qquad(\forall r<-l).
\]
\end{lemma}

\begin{proof}
Write the differential of a complex \(Y^\ast\) by \(d_Y\).  A homogeneous
cochain \(f\) of degree \(r<-l\) is a collection of morphisms
\(f^n:P^n\to X^{n+r}\) satisfying the cocycle condition
\begin{align}\label{eq:C}
d_X\circ f^n - (-1)^r f^{n+1}\circ d_P = 0\qquad\text{for all }n.
\tag{C}
\end{align}
Equivalently $f\in \Hom_{\mc Ch(\mc C)}(P^*, X^*[r])$
We will construct a homotopy \(h\) of degree \(r-1\) with components
\(h^n:P^n\to X^{n+r-1}\) such that
\begin{align}\label{eq:H}
(-1)^{r}d_X\circ h^n +  h^{n+1}\circ d_P \;=\;(-1)^{r} f^n
\qquad\text{for all }n,
\tag{H}
\end{align}
which is exactly the equation \(d_{\Hom}(h)=f\).  The construction is by
downward induction on \(n\) (recall \(P^n=0\) for \(n>l\)).

\medskip\noindent\textbf{Base step.}  Start at \(n=l\).  Since \(P^{l+1}=0\)
the cocycle identity \eqref{eq:C} for \(n=l\) reduces to
\(d_X\circ f^n = 0\).  Because \(r<-l\) we have \(n+r< 0\), so
\(d_X\circ f^n\) lands in a negative degree of \(X^\ast\). By
exactness of \(X^\ast\) at degree \(n+r\) we have
\(\on{Ker}(d_X|_{X^{n+r}})=\on{Im}(d_X|_{X^{n+r-1}})\). Hence the
map \(f^n:P^n\to X^{r+n}\) factors through \(d_X:X^{n+r-1}\to X^{n+r}\); i.e.
there exists a morphism \(h^n:P^n\to X^{n+r-1}\) such that
\[
d_X\circ h^n \;=\; f^n.
\]
since \(P^n\) is projective in $\mathcal{C}$.  Thus \eqref{eq:H} holds for \(n=l\) (the term involving \(h^{n+1}\)
vanishes since \(P^{n+1}=0\)).

\medskip\noindent\textbf{Inductive step.} Let $ m\le l$. Suppose \(h^{m+1},h^{m+2},\dots\)
have been constructed so that \eqref{eq:H} holds on all degrees \(>m\).  Define
the map
\[
u^m \;:=\; f^m + (-1)^{\,r-1} h^{m+1}\circ d_P \;:\; P^m \longrightarrow X^{m+r}.
\]
We check that \(u^m\) takes values in \(\on{Ker}(d_X)\). Indeed
\[
\begin{aligned}
d_X\circ u^m
&= d_X\circ f^m + (-1)^{\,r-1} d_X\circ h^{m+1}\circ d_P \\
&= (-1)^r f^{m+1}\circ d_P + (-1)^{\,r-1}\bigl( f^{m+1} + (-1)^{\,r-1} h^{m+2}\circ d_P \bigr)\circ d_P
\end{aligned}
\]
where we used the cocycle identity \eqref{eq:C} for \(f\) and the inductive
hypothesis \eqref{eq:H} for \(h^{m+1}\). The two terms cancel and one obtains
\(d_X\circ u^m=0\). By exactness of \(X^\ast\) at degree \(m+r\) we have
\(\on{Ker}(d_X|_{X^{m+r}})=\on{Im}(d_X|_{X^{m+r-1}})\), so \(u^m\)
factors through \(d_X:X^{m+r-1}\to X^{m+r}\). Equivalently there exists
a morphism \(h^m:P^m\to X^{m+r-1}\) with
\[
d_X\circ h^m \;=\; u^m
\]
since \(P^m\) is projective in
\(\mathcal C\).  By construction the equality \eqref{eq:H} then holds on degree
\(m\), completing the induction.

\medskip\noindent Since \(P^n=0\) for \(n > l\) the downward induction
defines \(h^n\) for all \(n\).  The family \(h=\{h^n\}\) satisfies
\eqref{eq:H} on every degree, hence \(f=d_{\Hom}(h)\) and so \(f\) is
null-homotopic. 
\end{proof}

\section{Square of PD derivations of odd degrees}\label{appendix:sec:square_pd_der}


\begin{lemma}\label{lem:D2_pd_der}
Let $P^*$ be a strictly graded commutative PD dg algebra over a commutative ring $R$ 
and let
\[
D\colon P^* \longrightarrow P^{*+r}
\]
be a PD derivation of odd degree $r$. Then $D^2:=D\circ D$ is a PD derivation of degree $2r$.
\end{lemma}

\begin{proof}
We check the two required properties.

\medskip

\noindent\textbf{(Derivation property).}
Let $a,b\in P^*$ be homogeneous. By the graded Leibniz rule for $D$,
\[
D(ab)=D(a)\,b + (-1)^{r|a|}a\,D(b).
\]
By applying $D$ again and expanding each term we get:
\[
\begin{aligned}
D^2(ab)
&= D\big(D(a)\,b\big) + (-1)^{r|a|} D\big(a\,D(b)\big) \\
&= D^2(a)\,b + (-1)^{r|D(a)|} D(a)D(b)
   + (-1)^{r|a|}\big( D(a)D(b) + (-1)^{r|a|} a D^2(b)\big).
\end{aligned}
\]
But we know that $|D(a)| = |a|+r$, so
\[
(-1)^{r|D(a)|} = (-1)^{r|a|}(-1)^{r^2}=-(-1)^{r|a|}
\]
as $r$ is odd. Substituting this into the previous equation, the two middle terms cancel each other and we obtain
\[
D^2(ab)=D^2(a)\,b + a\,D^2(b),
\]
which is the graded Leibniz rule for $D^2$ (degree $2r$). In particular, $D^2$ is a derivation.

\medskip

\noindent\textbf{(PD rule).}
Let $x\in P^*$ be homogeneous with $x \in I_{ev}$, and fix $m\ge1$. Using the PD rule for $D$ and the Leibniz rule,
\[
\begin{aligned}
D^2(\gamma_m(x))
&= D\big(\gamma_{m-1}(x)D(x)\big) \\
&= D(\gamma_{m-1}(x))\,D(x) + (-1)^{r\cdot|\gamma_{m-1}(x)|}\,\gamma_{m-1}(x)\,D^2(x) \\
&= \gamma_{m-2}(x)\,D(x)^2 + (-1)^{r\cdot|\gamma_{m-1}(x)|}\,\gamma_{m-1}(x)\,D^2(x) \\
&= \gamma_{m-2}(x)\,D(x)^2+\gamma_{m-1}(x)\,D^2(x).
\end{aligned}
\]
as $|\gamma_{m-1}(x)|=(m-1)|x|$ is even. Because $r$ is odd and $|x|$ is even, $|D(x)|=|x|+r$ is odd, hence $D(x)^2=0$ due to the strict graded commutativity. Thus the previous equation becomes
\[
D^2(\gamma_m(x))=\gamma_{m-1}(x)\,D^2(x),
\]
and $D^2$ satisfies the PD rule on divided powers.
\end{proof}

\begin{lemma}
Let \(P^*\) be a PD dg algebra over a commutative ring $R$, and let
\[
\mc Der^{*,\on{pd}}_{R}(P^*,P^*)
\]
be the cochain complex of PD derivations with differential
\[
\partial(D)= [d, D] \;=\; d\circ D - (-1)^{|D|} D\circ d,
\]
where \(d\) is the internal differential of \(P^*\) and \(|D|\) is the degree of \(D\).  
Assume \(q\colon \mc Der^{odd,pd}_R(P^*,P^*) \to \mc Der^{even,pd}_R(P^*,P^*)\) is given by \(q(D)=D\circ D\). 
If \(D\) satisfies \(\partial(D)=0\) (i.e., \(D\) is a cocycle), then \(\partial(q(D))=0\). Hence \(q\) sends cocycles to cocycles and induces a map (for odd $r$)
\[
q\colon H^{r}\big(\mc Der^{*,\on{pd}}_R(P^*,P^*)\big)\to 
H^{2r}\big(\mc Der^{*,\on{pd}}_R(P^*,P^*)\big).
\]
\end{lemma}

\begin{proof}
Consider the graded commutator \([X,Y]=X\circ Y - (-1)^{|X||Y|} Y\circ X\).  We have the identity
\[
[X, Y \circ Z] \;=\; [X,Y] \circ \,Z + (-1)^{|X||Y|} Y \circ \,[X,Z],
\]
valid for any homogeneous \(X,Y,Z \in \mc Der^{*,\on{pd}}_R(P^*,P^*)\).  Setting \(X=d\), \(Y=Z=D\), and using \(|d|=1\) and \(|D|=r\),
\[
[d, q(D)] =  [d,D] \circ \,D + (-1)^r D \circ \,[d,D].
\]
If \(D\) is a cocycle then \([d,D]=0\), so both terms on the right vanish and hence \([d,q(D)]=0\). Equivalently \(\partial(q(D))=0\), as required.
\end{proof}

\begin{lemma}
Let $\mc Der^{*,\on{pd}}_R(P^*,P^*)$ be the dg Lie algebra of PD derivations equipped with the graded commutator
\[
[X,Y]=X\circ Y - (-1)^{|X||Y|}Y\circ X.
\]
Let $D\in\mc Der^{r,pd}_R(P^*,P^*)$ be homogeneous of odd degree $r$, and let $D'\in\mc Der^{*,\on{pd}}_R(P^*,P^*)$ be arbitrary homogeneous. Then
\[
[D^2, D'] \;=\; [D,[D,D']].
\]
\end{lemma}
\begin{proof}
For homogeneous $X,Y,Z$ one has
\[
[X \circ Y,Z]=X \circ [Y,Z]+(-1)^{|Y||Z|}[X,Z] \circ Y.
\]
Taking $X=Y=D$ and $Z=D'$ yields
\[
[D^2,D']=D \circ \,[D,D']+(-1)^{|D||D'|}[D,D'] \circ D.
\]
On the other hand
\[
[D,[D,D']]=D \circ \,[D,D']-(-1)^{|D|(|D|+|D'|)}[D,D'] \circ D.
\]
Since $|D|$ is odd, $(-1)^{|D|(|D|+|D'|)}=-(-1)^{|D||D'|}$,
and so
\[
[D,[D,D']]=D \circ \,[D,D']+(-1)^{|D||D'|}[D,D'] \circ D=[D^2,D'].
\]
\end{proof}

\delete{
\subsection{Conceptual Meaning of the Identity 
\texorpdfstring{$[D^{2},-]=[D,[D,-]]$}{[D2,-]=[D,[D,-]]}}\label{appendix:C2}

Let $(\mathfrak{g},[\cdot , \cdot])$ be a graded Lie algebra.  For any homogeneous 
$x\in \mathfrak{g}$ the adjoint operator
\[
\on{ad}_x \colon \mathfrak{g} \to \mathfrak{g},\qquad 
\on{ad}_x(y)=[x,y],
\]
is a graded derivation of degree $|x|$.  A standard computation shows that
for homogeneous $x$,
\begin{equation}\label{eq:ad-ad}
[\on{ad}_x,\on{ad}_x] \;=\; \on{ad}_{[x,x]},
\end{equation}
where the left side is the graded commutator of endomorphisms.

\begin{lemma}
If $x$ has odd degree in a $p$-restricted Lie algebra, and $ \frac{1}{2}\in R$, then
\[
\on{ad}_{x^{[2]}} \;=\; \frac{1}{2}[\on{ad}_x,\on{ad}_x].
\]
Equivalently, for all $y\in \mathfrak{g}$,
\[
[x^{[2]},y]=[x,[x,y]].
\]
{\color{red}$x^2$ makes no sense here as $\mathfrak{g}$ is a Lie algebra.}
\end{lemma}

\begin{proof}
Since $|x|$ is odd, we have $(-1)^{|x|^{2}}=-1$.  
Expanding the graded commutator,
\[
[\on{ad}_x,\on{ad}_x]
= \on{ad}_x\circ \on{ad}_x 
  - (-1)^{|x||x|} \on{ad}_x\circ \on{ad}_x
= 2(\on{ad}_x)^2.
\]
By \eqref{eq:ad-ad} this equals $2\on{ad}_{x^{2}}$ {\color{red}WHY?}.  Thus
$\on{ad}_{x^{2}}=\tfrac12[\on{ad}_x,\on{ad}_x]$, and applying both sides to a
homogeneous $y$ gives
\[
[x^{2},y]=[x,[x,y]]. \qedhere
\]
\end{proof}

\begin{remark}
In the context of a PD dg algebra $P^*$, let $\mathfrak{g}=\mc Der^{*,\on{pd}}_R(P^*, P^*)$ and let
$D\in\mc Der^{*,\on{pd}}_{R}(P^*, P^*)$ be a derivation of odd degree.  Then the lemma shows
that the failure of the odd operator $\on{ad}_D$ to square to zero is
itself an inner derivation:
\[
(\on{ad}_D)^2=\on{ad}_{D^{2}}.
\]
In particular, if $D$ is a cocycle (i.e., $[d,D]=0$), then
$[d,D^{2}]=0$ as well, so $D^{2}$ is again a cocycle.  This is the key
fact ensuring that the assignment $q([D])=[D^{2}]$ defines a map
\[
q \colon H^{r}(\mc Der^{*,\on{pd}}_{R}(P^*,P^*)) \longrightarrow
H^{2r}(\mc Der^{*,\on{pd}}_{R}(P^*,P^*)).
\]
\end{remark}
}

\setlength\bibitemsep{7pt}

\end{document}